\documentclass[11pt, oneside]{article}
\usepackage[margin=20truemm]{geometry}
\usepackage{graphicx}	
\usepackage{amssymb}
\usepackage{mleftright}
\usepackage{amsmath,amsthm,mathtools,color}
\usepackage{usebib}
\usepackage{esint}
\newtheorem{theorem}{Theorem}[section]
\newtheorem{corollary}[theorem]{Corollary}
\newtheorem{lemma}[theorem]{Lemma}
\newtheorem{proposition}[theorem]{Proposition}
\newtheorem{remark}[theorem]{Remark}
\newtheorem{definition}[theorem]{Definition}

\newcommand{\divx}{\mathop{\mathrm{div}}}
\newcommand{\esssup}{\mathop{\mathrm{ess~sup}}}
\newcommand{\essinf}{\mathop{\mathrm{ess~inf}}}

\newcommand{\D}{\mathbf{D}}
\newcommand{\bu}{\mathbf{u}}
\newcommand{\bv}{\mathbf{v}}
\newcommand{\bw}{\mathbf{w}}
\newcommand{\buf}{\mathbf{f}}
\newcommand{\bg}{\boldsymbol\gamma}
\newcommand{\bG}{\boldsymbol\Gamma}
\newcommand{\A}{\mathbf{A}}
\newcommand{\G}{\boldsymbol{\mathcal G}}
\newcommand{\V}{\boldsymbol{\mathcal V}}
\newcommand{\bphi}{\boldsymbol\varphi}
\newcommand{\bxi}{\boldsymbol\xi}
\newcommand{\bta}{\boldsymbol\eta}
\newcommand{\bz}{\boldsymbol\zeta}
\allowdisplaybreaks[4]

\renewcommand{\d}{\mathrm{d}}

\numberwithin{equation}{section}

\def\Yint#1{\mathchoice
    {\YYint\displaystyle\textstyle{#1}}%
    {\YYint\textstyle\scriptstyle{#1}}%
    {\YYint\scriptstyle\scriptscriptstyle{#1}}%
    {\YYint\scriptscriptstyle\scriptscriptstyle{#1}}%
      \!\iint}
\def\YYint#1#2#3{{\setbox0=\hbox{$#1{#2#3}{\iint}$}
    \vcenter{\hbox{$#2#3$}}\kern-.51\wd0}}
\def\longdash{{-}\mkern-3.5mu{-}} 
\def\tiltlongdash{\rotatebox[origin=c]{15}{$\longdash$}}
\def\fiint{\Yint\tiltlongdash}

\title{Gradient continuity for the parabolic $(1,\,p)$-Laplace system}
\author{Shuntaro Tsubouchi \footnote{Graduate School of Mathematical Sciences, The University of Tokyo, 3-8-1 Komaba, Meguro-ku, Tokyo, 153-8914, Japan. \textit{Email}: \texttt{tsubos@g.ecc.u-tokyo.ac.jp}}}
\date{}
\begin{document}
\maketitle 
\begin{abstract}
This paper deals with the parabolic $(1,\,p)$-Laplace system, a parabolic system that involves the one-Laplace and $p$-Laplace operators with $p\in(1,\,\infty)$. We aim to prove that a spatial gradient is continuous in space and time. An external force term is treated in a parabolic Lebesgue space under the optimal regularity assumption. We also discuss a generalized parabolic system with the Uhlenbeck structure. A main difficulty is that the uniform ellipticity of the $(1,\,p)$-Laplace operator is violated on a facet, or the degenerate region of a spatial gradient. The gradient continuity is proved by showing local H\"{o}lder continuity of a truncated gradient, whose support is far from the facet. This is rigorously demonstrated by considering approximate parabolic systems and deducing various regularity estimates for approximate solutions by classical methods such as De Giorgi's truncation, Moser's iteration, and freezing coefficient arguments. A weak maximum principle is also utilized when $p$ is not in the supercritical range.
\end{abstract}

\tableofcontents
\section{Introduction}\label{Section: Introduction}
 In this paper, we consider an $N$-dimensional vector-valued function $\bu=(u^{j}(x,\,t))_{j}$, which is defined in $\Omega_{T}\coloneqq \Omega\times (0,\,T)$ with $\Omega\subset{\mathbb R}^{n}$ being an $n$-dimensional bounded Lipschitz domain, and satisfies
 \begin{equation}\label{Eq (Section 1): Parabolic (1,p)-Laplace}
    \partial_{t}u^{j}+\partial_{x_{\beta}}\left(\lvert\D\bu\rvert^{-1}\partial_{x_{\alpha}} u^{j}+\lvert\D\bu\rvert^{p-2}\partial_{x_{\alpha}} u^{j} \right)=f^{j}\quad \text{in}\quad \Omega_{T}
 \end{equation}
 for each $j\in\{\,1,\,\dots\,,\,N\,\}$ in a weak sense.
 Here, we let $p\in(1,\,\infty)$, $T\in(0,\,\infty)$, $n\ge 2$ and $N\ge 1$ be fixed, and use the convention to sum all of $\alpha,\,\beta\in\{\,1,\,\dots\,,\,n\,\}$.  
 For the left-hand side of (\ref{Eq (Section 1): Parabolic (1,p)-Laplace}), the time derivative and the space derivative of $u^{j}$ are respectively denoted by $\partial_{t}u^{j}$ and $\nabla u^{j}=(\partial_{x_{\alpha}}u^{j}(x,\,t))_{\alpha}$ for each $j\in\{\,1,\,\dots\,,\,n\,\}$.
 The symbol $\D\bu=(\partial_{x_{\alpha}}u^{j})_{\alpha,\,j}$ denotes the spatial gradient, or the $N\times n$ Jacobian matrix of $\bu$.
 Finally, we assume that $\buf=(f^{j}(x,\,t))_{j}$ is a given external force term that belongs to the Lebesgue space $L^{r}(0,\,T;\,L^{q}(\Omega))^{N}$ with the pair $(q,\,r)\in(n,\,\infty\rbrack\times (2,\,\infty\rbrack$ satisfying
 \begin{equation}\label{Eq (Section 1): Condition for q,r}
    \frac{n}{q}+\frac{2}{r}<1.
 \end{equation}
 In this paper, we call (\ref{Eq (Section 1): Parabolic (1,p)-Laplace}) as the parabolic $(1,\,p)$-Laplace system, which consists of the one-Laplace and $p$-Laplace operators.
 The main aim of this paper is to show that $\D\bu$ is continuous over $\Omega_{T}$.
 It should be noted that (\ref{Eq (Section 1): Condition for q,r}) is optimal when one treats the external force term in Lebesgue spaces, and considers the gradient continuity for classical heat equations or systems (see \cite{LSU MR0241822} for a classical monograph).
 The regularity assumptions of $\bu$ differ, depending on whether $p\in(1,\,\infty)$ is greater than the critical exponent $p_{\mathrm c}\coloneqq 2n/(n+2)\in\lbrack 1,\,2)$ or not (see \cite{Choe MR1135917, DiBenedetto-monograph MR1230384, DiBenedetto-Friedman MR743967,DiBenedetto-Friedman MR783531,DiBenedetto-Friedman MR814022}).
 When $p$ is in the supercritical range $p\in(p_{\mathrm c},\,\infty)$, no additional regularity assumption is required. 
 In the remaining case (i.e., $p\in(1,\,p_{\mathrm c}\rbrack$ and $n\ge 3$), however, we additionally let
 \begin{equation}\label{Eq (Section 1): Higher Integrability of u for subcritical case}
     \bu\in L_{\mathrm{loc}}^{\varsigma}(\Omega_{T})^{N}\quad \textrm{with}\quad \varsigma>\varsigma_{\mathrm c}\coloneqq \frac{(2-p)n}{p}\ge 2,
 \end{equation}
 so that we can improve the interior regularity of solutions. 
 This paper aims to establish a qualitative $C^{0}$-regularity result for (\ref{Eq (Section 1): Parabolic (1,p)-Laplace}) by adapting the truncation method that is discussed in \cite{T-scalar, T-system} for the elliptic problems.
 Although this paper is a vectorial version of the regularity results for parabolic equations \cite{T-subcritical, T-supercritical}, we remove some restrictions on the external force term $\buf$ in the previous works.

 More generally, this paper deals with a general parabolic system of the form
 \begin{equation}\label{Eq (Section 1): General System}
     \partial_{t}u^{j}-\partial_{x_{\beta}}\left(\gamma_{\alpha\beta}\left[a_{1}(x,\,t)\lvert \D\bu\rvert_{\bg}^{-1}+a_{p}(x,\,t)g_{p}(\lvert\D\bu\rvert_{\bg}^{2})\right]\partial_{x_{\alpha}}u^{j} \right)=f^{j}\quad \text{in}\quad \Omega_{T},
 \end{equation}
 where $g_{p}\colon (0,\,\infty)\to(0,\,\infty)$ is a positive function that includes $g_{p}(\sigma)=\sigma^{p/2-1}$ for $\sigma\in(0,\,\infty)$. 
 The scalar-valued functions $a_{1}(x,\,t)$ and $a_{p}(x,\,t)$ are respectively non-negative and positive, and the matrix-valued function $\bg=(\gamma_{\alpha\beta}(x,\,t))_{\alpha,\,\beta}\colon \Omega_{T}\to{\mathbb R}^{n\times n}$ is symmetric and positive definite for all $(x,\,t)\in\Omega_{T}$, whence $\bg=\bg(x,\,t)$ naturally provides an inner product.
 Hereinafter, for each $(x,\,t)\in\Omega_{T}$, let $\lvert\,\cdot\,\rvert_{\bg}=\lvert \,\cdot\,\rvert_{\bg(x,\,t)}$ denote the norm that is induced by this inner product.
 The detailed conditions are explained later in Subsection \ref{Subsect: Structural conditions}, where the definition of weak solutions to (\ref{Eq (Section 1): General System}) are also given.
 \subsection{Some specific mathematical sources in fluid mechanics}
 The $(1,\,p)$-Laplace operator is found in some mathematical of the crystal surface growth under the roughening temperature for $p=3$ \cite{spohn1993surface}, and of the motion of Bingham fluids for $p=2$ \cite{Duvaut-Lions MR0521262}.
 Among them, we mainly explain the latter model.

 A Bingham flow is a non-Newtonian visco-plastic fluid that contains the two completely different aspects of plasticity and viscosity, which are respectively reflected by the one-Laplace and the Laplace operator.
 The model equation is a modified incompressible parabolic Navier-Stokes system of the form
 \[\left\{\begin{array}{rcl}
     \partial_{t}\bu-(\bu\cdot \nabla)\bu-\frac{\delta}{\delta \bu}\left(\int_{V}\lvert \D\bu\rvert+\frac{\lvert\D\bu \rvert^{2}}{2}\,\d x\right)+\nabla\pi&=&0,\\ 
     \divx \bu&=&0,
 \end{array}  \right. \quad \text{for}\quad (x,\,t)\in V\times (0,\,\infty),\]
 where $V\subset {\mathbb R}^{3}$ is a given domain, and the ${\mathbb R}^{3}$-valued function $\bu=(u_{1},\,u_{2},\,u_{3})$ and the scalar-valued function respectively denote the velocity of the Bingham fluids, and the pressure function.
 This problem can be more simplified under suitable settings. 
 When $\bu$ is sufficiently small, it will not be restrictive to discard the convective acceleration term $(\bu\cdot \nabla)\bu$.
 Then, the resulting system becomes (\ref{Eq (Section 1): Parabolic (1,p)-Laplace}) with $p=2$ and $\buf=-\nabla\pi$. 
 If $\nabla \pi$ is in $L^{r}(0,\,\infty;\,L^{q}(\Omega))^{3}$ with $3/q+2/r<1$, then our main theorem implies that $\D\bu$ is continuous.
 This system could be simplified under the laminar flow condition, which is found in \cite[Chapter VI]{Duvaut-Lions MR0521262} for the elliptic case. 
 More precisely, let $V$ be a cylinder pipe of the form $V=\Omega\times {\mathbb R}$ with $\Omega\subset {\mathbb R}^{2}$ being a domain, and the velocity $\bu$ be of the uni-directional form $\bu=(0,\,0,\,u(x_{1},\,x_{2},\,t))$.
 Then, we have $\partial_{x_{1}}\pi=\partial_{x_{2}}\pi=0$ and therefore $\pi$ is of the form $\pi=\pi(x_{3},\,t)$.
 Moreover, both the incompressible condition $\divx U=0$ and the identity $(\bu\cdot \nabla)\bu=0$ automatically follow.
 As a result, the given parabolic system boils down to the parabolic equation $\partial_{t}u-\Delta_{1}u-\Delta_{2}u=-\partial_{x_{3}}\pi$, where $\Delta_{s}u\coloneqq \divx (\lvert \nabla u\rvert^{s-2}\nabla u)$ stands for the $s$-Laplace operator with $s\in\lbrack 1,\,\infty)$.
 We keep in mind that the terms on the left-hand side depend at most on $x_{1},\,x_{2}$ and $t$, while $-\partial_{x_{3}}\pi$ depends at most on $x_{3}$ and $t$.
 This implies that we may write $-\partial_{x_{3}}\pi=f(t)$ for some function $f\colon (0,\,\infty)\to {\mathbb R}$.
 Applying our result to this problem with $q=\infty$, we conclude that the slope of the component $u$ is continuous, provided $f\in L^{r}(0,\,\infty;\,L^{\infty}(\Omega))$ with $r>2$.

 It is also worth mentioning the Plandtl--Eyring fluids, which also have some plastic structure that is however suitably hardened by logarithmic power.
 The mathematical difference between Bingham and Plandtl--Eyring fluids, particularly on regularity problems, is when uniform ellipticity breaks.
 As is already explained, the $(1,\,p)$-Laplace operator becomes no longer uniformly elliptic as a gradient shrinks, while the uniform ellipticity for the latter problem gets violated as a gradient blows up.
 Hence, the latter problem is rather one of the classical non-uniformly elliptic problems.
 The full regularity results of the minimizers of nearly linear growth functionals are found in \cite{Fuchs-Mingione}, where the gradient H\"{o}lder regularity is shown.
 Also, more generalizations to non-autonomous problems are discussed in \cite{CDFM, FDFP}. 
 Regularity theories in non-uniformly elliptic problems have been well-established under various settings, see \cite{Mingione-survey MR2291779, Mingione-R survey} for the detailed surveys.

 The mathematical analyses concerning stationary or non-stationary Bingham fluid problems at least trace back to the classical textbook by Duvaut--Lions in the 1970s \cite{Duvaut-Lions MR0521262}, where the mathematical formulations are mainly based on variational inequalities.
 See also \cite{Fuchs-Seregin Monograph} as a related material that provides some mathematical analysis concerning a Plandtl--Eyring fluid, as well as a Bingham fluid.
 This paper aims to show the gradient continuity for the $(1,\,p)$-parabolic systems, which has been open for almost fifty years.

 \subsection{Literature overview and our strategy}
 We briefly overview mathematical research on the $p$-Laplace regularity theory.
 Also, we would like to mention some recent progress on regularity for very singular and degenerate problems.

 For the $p$-Laplace problem with $p\in (1,\,\infty)$, it is well-known that a weak solution admits its H\"{o}lder gradient continuity.
 In other words, the spatial gradient, which is treated in the $L^{p}$-sense, is indeed H\"{o}lder continuous.
 Such a regularity result was shown by many experts for the elliptic case; see \cite{Evans MR672713, Uhlenbeck, Ural'ceva MR0244628} for $p\ge 2$ and \cite{Acerbi-Fusco, Lewis MR721568, Dibenedetto-elliptic MR709038, Tolksdorf MR727034} for $p>1$.
 For the parabolic case, the H\"{o}lder gradient was shown by DiBenedetto--Friedmann \cite{DiBenedetto-Friedman MR783531,DiBenedetto-Friedman MR814022} in the supercritical range $p\in(p_{\mathrm c},\,\infty)$ in 1985 (see also \cite{Alikakos-Evans,DiBenedetto-Friedman MR743967,Evans MR672713,Wiegner MR0886719} for weaker regularity results, and \cite{BDLS-parabolic} for a generalized result).
 Later in 1991, Choe \cite{Choe MR1135917} proved the same regularity result for general $p\in(1,\,\infty)$, where (\ref{Eq (Section 1): Higher Integrability of u for subcritical case}) is assumed particularly for $p\in(1,\,p_{\mathrm c}\rbrack$.
 The requirement of (\ref{Eq (Section 1): Higher Integrability of u for subcritical case}) in the case $p\in(1,\,p_{\mathrm c}\rbrack$ seems essential, since no improved regularity is generally expected for $p$-Laplace flows, provided $p$ is close to one (see e.g., \cite{DiBenedetto-Herrero MR1066761} or \cite[Chapter XII, \S 13]{DiBenedetto-monograph MR1230384}).
 To show improved regularity for $p\le p_{\mathrm c}$, we have to require some better conditions, such as a higher integrability assumption (\ref{Eq (Section 1): Higher Integrability of u for subcritical case}) or an $L^{\infty}$-bound of the parabolic boundary datum.

 For the $(1,\,p)$-Laplace problems, the $C^{1,\,\alpha}$-regularity is the best possibly expexted smoothness of solutions.
 Indeed, the Fenschel dual of the energy density $E(z)\coloneqq \lvert z\rvert+\lvert z\rvert^{p}/p\,(z\in{\mathbb R}^{n})$ is given by $u(x)\coloneqq (\lvert x\rvert-1)_{+}^{p^{\prime}}/p^{\prime}\,(x\in{\mathbb R}^{n})$, where $p^{\prime}\coloneqq p/(p-1)$ stands for the H\"{o}lder conjugate exponent of $p$.
 This special function satisfies the scalar-valued stationary problem $\Delta_{1}u+\Delta_{p}u=n$, and it is in the class $C^{1,\,\alpha}$ with $\alpha\coloneqq \max\{\,1,\,1/(p-1)\,\}$, which fact may indicate the best possible regularity of a weak solution to $(1,\,p)$-Laplace problems.
 However, the H\"{o}lder continuity of a spatial gradient is an open problem even for stationary problems, since the $(1,\,p)$-Laplace operator violates its uniform ellipticity as a gradient vanishes.
 This is explained by formally computing the ellipticity ratio, the ratio defined as the largest eigenvalue of $\nabla^{2}E(z)$ divided by the smallest one.
 The ellipticity ratio for the $(1,\,p)$-Laplace problem makes sense as long as $z\neq 0$, and it contains a singular term $\lvert z\rvert^{1-p}$.
 In other words, the Hessian matrix $\nabla^{2}E(\nabla u)$ becomes no longer uniformly elliptic or uniformly parabolic as $\nabla u\to 0$.
 This is due to the well-known fact on the one-Laplace operator; that is, its ellipticity always degenerates in the gradient direction and has singularity in any others.
 Such a purely anisotropic structure is similarly found in the general system (\ref{Eq (Section 1): General System}), and it strongly appears on the degenerate region $\{\D\bu=0\}$, which is often called the facet of $\bu$. 
 In other words, it appears difficult to show any quantitative continuity estimates of $\D\bu$ across the facet $\{\D\bu=0\}$, since (\ref{Eq (Section 1): General System}) will be no longer uniformly parabolic as the norm $\lvert\D\bu\rvert_{\bg}$ tends to zero.

 When it comes just to show the gradient continuity, where we do not necessarily try to show quantitative continuity estimates, we can provide an affirmative answer by classical analyses.
 More precisely, to prove $\D\bu\in C^{0}$ for the problem (\ref{Eq (Section 1): General System}), we introduce a truncated gradient, defined as
 \[\G_{\delta}(\D\bu)\coloneqq (\lvert\D\bu\rvert_{\bg}-\delta)_{+}\frac{\D\bu}{\lvert\D\bu\rvert_{\bg}}\]
 with $\delta\in(0,\,1)$ denoting a truncation parameter. 
 The most important viewpoint is that (\ref{Eq (Section 1): General System}) can be regarded as uniformly parabolic in $\{\lvert \D\bu\rvert_{\bg}\ge \delta\}$, the support of $\G_{\delta}(\D\bu)$. 
 In other words, the truncation parameter $\delta$ plays a fine role of suitably neglecting a non-uniformly elliptic structure of the $(1,\,p)$-Laplace operator.
 Hence, the continuity of $\G_{\delta}(\D\bu)$ is naturally expected, although the H\"{o}lder exponent may depend on $\delta$.
 Still, we are able to conclude $\D\bu\in C^{0}$ from $\G_{\delta}(\D\bu)\in C^{0}$, since $\G_{\delta}(\D\bu)$ uniformly converges to $\D\bu$ as $\delta\to 0$.
 This truncation approach is found in widely degenerate elliptic problems; see \cite{BDGPdN, Colombo-Figalli MR3133426,Santambrogio Vespri MR2728558}.
 Among them, the paper \cite{BDGPdN} by B\"{o}gelein--Duzaar--Giova--Passarelli di Napoli gives continuity estimates of truncated gradients by the classical methods such as De Giorgi's truncation and the freezing coefficient method.
 This strategy is successfully extended to parabolic problems in the recent paper \cite{BDGPdN-p} (see also \cite{Ambrosio-PdN} for another regularity result).
 Highly inspired by \cite{BDGPdN}, the author showed the gradient continuity for elliptic $(1,\,p)$-Laplace problems in the previous works; \cite{T-scalar} for the scalar-valued case, and \cite{T-system} for the vector-valued case (see also \cite{Giga-Tsubouchi MR4408168} for a weaker result).
 The author also would like to note that before these papers \cite{Giga-Tsubouchi MR4408168, T-scalar, T-system} appeared, the special case $n=p=2$ had already been discussed in \cite[Theorems 3.3.3 \& 3.4.3]{Fuchs-Seregin Monograph} by a different approach.

 To rigorously deduce the continuity of $\G_{\delta}(\D\bu)$, we need to consider an approximate parabolic system.
 For the general system (\ref{Eq (Section 1): General System}), where we let $g_{p}(\sigma)\equiv \sigma^{p/2-1}$ and $a_{1}=a_{p}\equiv 1$ for simplicity, the approximate system is given by
 \[\partial_{t}u_{\varepsilon}^{j}+\partial_{x_{\beta}}\left((\varepsilon^{2}+\lvert \D\bu_{\varepsilon}\rvert_{\bg}^{2})^{-1/2}\partial_{x_{\alpha}}u_{\varepsilon}^{j}+\left(\varepsilon^{2}+\lvert\D\bu_{\varepsilon}\rvert_{\bg}^{2} \right)^{(p-2)/2}\partial_{x_{\alpha}}u_{\varepsilon}^{j} \right)=f_{\varepsilon}^{j}.\]
 Here $\varepsilon\in(0,\,1)$ stands for the approximation parameter, and $\buf_{\varepsilon}=(f_{\varepsilon}^{1},\,\dots\,,\,f_{\varepsilon}^{N})$ converges to $\buf$ in some weak sense.
 In particular, the diffusion coefficients $\lvert \D\bu\rvert_{\bg}^{-1}$ and $\lvert \D\bu\rvert_{\bg}^{p-2}$ are suitably regularized, so that their possible singularities at $\D\bu=0$ are avoided (see also Remark \ref{Rmk: Section 4}).
 Along with this approximation, we need to introduce another truncated gradient of the form 
 \[\G_{\delta,\,\varepsilon}(\D\bu_{\varepsilon})\coloneqq\left(v_{\varepsilon}-\delta \right)_{+}\frac{\D\bu_{\varepsilon}}{\lvert\D\bu_{\varepsilon}\rvert_{\bg}},\quad \text{where}\quad v_{\varepsilon}\coloneqq \sqrt{\varepsilon^{2}+\lvert \D\bu_{\varepsilon}\rvert_{\bg}^{2}}\quad \text{for}\quad \varepsilon\in(0,\,\delta),\]
 since the uniform ellipticity of the approximate system is measured by $v_{\varepsilon}$.
 The strategy broadly consists of the following two parts; the demonstration of the strong congergence of $\D\bu_{\varepsilon}$, and the deduction of local a priori H\"{o}lder estimates of the truncated gradients
 \[\G_{2\delta,\,\varepsilon}(\D\bu_{\varepsilon})=\left(v_{\varepsilon}-2\delta \right)_{+}\frac{\D\bu_{\varepsilon}}{\lvert\D\bu_{\varepsilon}\rvert_{\bg}},\]
 uniformly for $\varepsilon\in(0,\,\delta/4)$. From these results, the continuity $\G_{2\delta}(\D\bu)$ is easily concluded by the Aezel\`{a}--Ascoli theorem.
 In the rest of this subsection, we would like to briefly explain each part, and express the novelty of this paper.
 
 In showing the strong convergence of $\D\bu_{\varepsilon}$, we should keep in mind that the one-Laplace operator lacks any fine properties that the $p$-Laplace operator has.
 More precisely, the $p$-Laplace operator has so-called strong monotonicity, while the one-Laplace operator is merely monotone.
 This is easily noticed by the fact that the one-Laplace operator is always degenerate elliptic in the direction of a gradient, which implies that no quantitative monotonicity estimates appear to be expected.
 We mainly appeal to the fact that the $(1,\,p)$-Laplace operator or its approximate operator still has strong monotonicity.
 The detailed proof becomes different when $p$ is in the supercritical range or not.
 To be precise, when $\buf_{\varepsilon}$ is non-trivial (i.e., $\buf_{\varepsilon}\not\equiv 0$), we have to check the strong convergence of $\bu_{\varepsilon}$.
 This assertion is shown by two different approaches, depending on $p\in(p_{\mathrm c},\,\infty)$ or $p\in(1,\,p_{\mathrm c}\rbrack$.
 The former case is easier; indeed, the compact embedding $W_{0}^{1,\,p}(\Omega)\subset L^{2}(\Omega)$ and the continuous embedding $L^{2}(\Omega)\subset W^{-1,\,p^{\prime}}(\Omega)$ allow us to use the Aubin--Lions lemma.
 In the latter case, this strategy will no longer work, since the compact embedding $W_{0}^{1,\,p}(\Omega)\subset L^{2}(\Omega)$ is violated.
 Instead, we utilize parabolic regularity estimates, including a weak maximum principle for the approximate systems.
 More precisely, by (\ref{Eq (Section 1): Higher Integrability of u for subcritical case}) and the weak maximum principle, we would like to show uniform $L^{\infty}$-bounds of $\bu_{\varepsilon}$ under a suitable setting.
 From this regularity result, we also deduce local uniform $L^{\infty}$-bounds of $\D\bu_{\varepsilon}$, and local $C^{1,\,1/2}$-bounds of $\bu_{\varepsilon}$.
 Utilizing these a priori estimates, we conclude the strong convergence of $\bu_{\varepsilon}$ from the bounded convergence theorem.

 A priori estimates of $\bu_{\varepsilon}$ are shown by the classical methods, including De Giorgi's truncation, Moser's iteration, and the freezing coefficient argument.
 To be precise, we prove the uniform bound of $\D\bu_{\varepsilon}$ by Moser's iteration.
 Here we require uniform $L^{\infty}$-bounds of $\bu_{\varepsilon}$ when $p\in(1,\,p_{\mathrm c}\rbrack$, which is verified by the weak maximum principle.
 The key estimate of $\G_{2\delta,\,\varepsilon}(\D\bu_{\varepsilon})$ is shown by careful modifications of \cite{BDLS-parabolic}.
 Roughly speaking, we mainly consider the two cases where the super-level set is suitably large (i.e., the non-degenerate case) or not (i.e., the degenerate case).
 In the latter case, we use De Giorgi's truncation to deduce an oscillation lemma.
 In the former case, we compare $\bu_{\varepsilon}$ with a comparison function that solves some sort of classical second-order heat system.
 There, we also verify that the average integral of $\D\bu_{\varepsilon}$ cannot degenerate.
 Although these divisions by cases are found in the classical monograph \cite[Chapter IX]{DiBenedetto-monograph MR1230384} or the recent paper \cite{BDLS-parabolic}, the main difference is that we carefully use the truncation parameter $\delta$.
 In fact, instead of $\G_{\delta,\,\varepsilon}(\D\bu_{\varepsilon})$, we often consider a more strictly truncated gradient $\G_{2\delta,\,\varepsilon}(\D\bu_{\varepsilon})$.
 Such a replacement helps us to avoid some delicate cases where a spatial gradient could vanish.
 This strategy is completely different from the celebrated intrinsic scaling method (see \cite{DiBenedetto-monograph MR1230384, DiBenedetto-Gianazza-Vespri MR2865434} for the materials), which plays an important role in deducing quantitative everywhere regularity estimates for parabolic $p$-Laplace problems.
 Our method rather avoids any mathematical analysis, particularly around the degenerate region of $v_{\varepsilon}$.
 As a sacrifice, our H\"{o}lder estimates of $\G_{2\delta,\,\varepsilon}(\D\bu_{\varepsilon})$ depends on the truncation parameter $\delta$.

 This paper provides a parabolic extension to \cite{T-system}.
 Although parabolic regularity results are already shown in \cite{T-subcritical,T-supercritical} for $N=1$, this paper mainly provides the two novelties.
 The first is that the external force term $\buf$ is treated for all $p\in(1,\,\infty)$ and $(q,\,r)\in(n,\,\infty\rbrack\times (2,\,\infty\rbrack$ under the optimal condition (\ref{Eq (Section 1): Condition for q,r}), while the previous works \cite{T-subcritical,T-supercritical} require some technical restrictions on $\buf$.
 More precisely, \cite{T-supercritical} treats an external force term in the class $L^{q}(\Omega_{T})$ with $q>n+2$ for $p\in(p_{\mathrm c},\,\infty)$ (i.e., $q=r$ is assumed), and \cite{T-subcritical} deals with no external force term for $p\in(1,\,p_{\mathrm c}\rbrack$ (i.e., $\buf\equiv 0$ is technically required).
 As already explained, the absence of the external force term in \cite{T-subcritical} is due to the lack of parabolic compact embeddings.
 We remove this technical problem by utilizing parabolic regularity estimates, including a weak maximum principle.
 The second is that the general matrix $\bg$ is treated by following the strategy of \cite{BDLS-parabolic}.
 This generalization is motivated by the possible applications to boundary regularity estimates, found in \cite{BDLS-boundary}.
 Since the main purpose of this paper is \textit{everywhere} $C^{0}$-regularity of a spatial gradient for a parabolic system, we treat a generalized $(1,\,p)$-Laplace operator that has the Uhlenbeck structure \cite{Uhlenbeck}.
 This symmetric structure is essentially used to deduce various regularity estimates of vector-valued solutions in this paper, while \cite{T-subcritical,T-supercritical} treats energy densities without the Uhlenbeck structure.
 However, it is worth noting that compared with the previous elliptic regularity result \cite{T-system}, where $\bg$ is assumed to be the identity matrix, this paper provides a generalization of the one-Laplace operator.

 \subsection{Notations}
 Before stating the main result, we fix some notations in this subsection.
 
 The symbols  ${\mathbb N}=\{\,1,\,2,\,\dots\,\}$ and ${\mathbb Z}_{\ge 0}\coloneqq \{0\}\cup {\mathbb N}$ respectively denote the collection of natural numbers and non-negative integers.
 For given real numbers $a,\,b\in{\mathbb R}$, we write $a\wedge b\coloneqq \min\{\,a,\,b\,\}\in{\mathbb R}$ and $a\vee b\coloneqq \max\{\,a,\,b\,\}\in{\mathbb R}$.
 We often use the abbreviations ${\mathbb R}_{\ge 0}\coloneqq \lbrack 0,\,\infty)\subset {\mathbb R}$ and ${\mathbb R}_{>0}\coloneqq (0,\,\infty)\subset {\mathbb R}$.
 For the pair $(q,\,r)\in(n,\,\infty\rbrack\times (2,\,\infty\rbrack$ satisfying (\ref{Eq (Section 1): Condition for q,r}), we fix the exponents
 \begin{equation}\label{Eq (Section 1): Exponents beta q-hat r-hat}
     \beta\coloneqq \left\{ \begin{array}{cc}
        \beta_{0} & (q=r=\infty), \\
        1-\displaystyle\frac{n}{q}-\displaystyle\frac{2}{r} & (\text{otherwise}),
     \end{array} \right.\quad 
     \widehat{q}\coloneqq \left(\frac{q}{2} \right)^{\prime}=\frac{q}{q-2},\quad \widehat{r}\coloneqq \left(\frac{r}{2} \right)^{\prime}=\frac{r}{r-2},
 \end{equation}
 where $\beta_{0}\in(0,\,1)$ is an arbitrarily fixed constant.

 For given $x_{0}\in{\mathbb R}^{n}$, $t_{0}\in{\mathbb R}$, $R\in(0,\,\infty)$, and $c_{0}\in(0,\,\infty)$, we set an open ball $B_{R}(x_{0})\coloneqq \{x\in{\mathbb R}^{n}\mid \lvert x-x_{0}\rvert<R\}$, half intervals $I_{R}(c_{0};\,t_{0})\coloneqq (t_{0}-c_{0}R^{2},\,t_{0}\rbrack$, $I_{R}(t_{0})\coloneqq I_{R}(1;\,t_{0})=(t_{0}-R^{2},\,t_{0}\rbrack$, and a standard parabolic cylinder $Q_{R}(x_{0},\,t_{0})\coloneqq B_{R}(x_{0})\times I_{R}(t_{0})$.
 The center points $x_{0}$ and $t_{0}$ are often omitted when they are clear.
 In the similar manner, we define $\widetilde{I}_{R}(c_{0};\,t_{0})\coloneqq \lbrack t_{0},\,t_{0}+c_{0}R^{2})$ and $\widetilde{Q}_{R}(c_{0};\,x_{0},\,t_{0})\coloneqq B_{R}(x_{0})\times \widetilde{I}_{R}(c_{0};\,t_{0})$.
 The space ${\mathbb R}^{n}$ is equipped with the canonical inner product and the Euclidean norm, defined as 
 \[\langle x\mid y\rangle\coloneqq x_{1}y_{1}+\cdots+x_{n}y_{n}\in{\mathbb R},\quad \lvert x\rvert\coloneqq\sqrt{\langle x\mid x\rangle}\]
 for $x=(x_{1},\,\dots\,,\,x_{n}),\,y=(y_{1},\,\dots\,,\,y_{n})\in{\mathbb R}^{n}$.
 In the same manner, we may introduce the inner product and the Euclidean norm for the spaces ${\mathbb R}^{Nn}$ and ${\mathbb R}^{Nn^{2}}$.

 For a $k$-dimensional Lebesgue measurable set $U\subset {\mathbb R}^{k}$ with $k\in{\mathbb N}$, the symbol $\lvert U\rvert$ stands for the $k$-dimensional Lebesgue measure of $U$.
 For ${\mathbb R}^{m}$-valued functions $\mathbf{g}=\mathbf{g}(x)$, and $\mathbf{h}=\mathbf{h}(x,\,t)$ that are respectively integrable in $U\subset {\mathbb R}^{n}$, and $U\times I\subset{\mathbb R}^{n+1}$ with $0<\lvert U\rvert,\,\lvert I\rvert<\infty$, let 
 \(\fint_{U}\mathbf{g}(x)\,{\mathrm d}x\coloneqq \lvert U\rvert^{-1}\int_{U}\mathbf{g}(x)\,{\mathrm d}x\in{\mathbb R}^{m}\), \(\fiint_{U\times I}\mathbf{h}(x,\,t)\,{\mathrm d}X\coloneqq \lvert U\rvert^{-1}\lvert I\rvert^{-1}\iint_{U}\mathbf{h}(x,\,t)\,{\mathrm d}X\in{\mathbb R}^{m}\)
 denote the average integrals.
 These average integrals are often written as $(\mathbf{f})_{U}$ and $({\mathbf h})_{U\times I}$ for simplicity.

 For a given interval $I\subset {\mathbb R}$, a given exponent $\pi\in\lbrack 1,\,\infty\rbrack$, and a given Banach space $E$ equipped with the norm $\lVert\,\cdot\,\rVert_{E}$, let $L^{s}(I;\,E)$ denote the standard Bochner space, equipped with the norm 
 \[\lVert u\rVert_{L^{s}(I;\,E)}\coloneqq \left\{\begin{array}{cc} \left(\displaystyle\int_{I}\lVert u(t)\rVert_{E}^{\pi}\,\d t\right)^{1/\pi} & (1\le \pi<\infty), \\ \esssup\limits_{t\in I}\,\lVert u(t)\rVert_{E} & (\pi=\infty), \end{array}  \right.\quad \text{for }u\in L^{\pi}(I;\,E).\]
 Following the notations in \cite{DiBenedetto-monograph MR1230384}, we abbreviate $L^{\pi_{1},\,\pi_{2}}(U\times I)\coloneqq L^{\pi_{2}}(I;\,L^{\pi_{1}}(U))$ for given exponents $\pi_{1},\,\pi_{2}\in\lbrack 1,\,\infty\rbrack$, and Lebesgue measurable sets $U\subset {\mathbb R}^{n}$, $I\subset {\mathbb R}$.
 We often write $L^{\pi}(U\times I)\coloneqq L^{\pi,\,\pi}(U\times I)$ for $\pi\in\lbrack 1,\,\infty\rbrack$.
 For a Lipschitz domain $U\subset {\mathbb R}^{n}$, the symbol $W^{1,\,\pi}(U)$ stands for the Sobolev space, equipped with the standard norm $\lVert u \rVert_{W^{1,\,\pi}(U)}\coloneqq \lVert u\rVert_{L^{\pi}(U)}+\lVert \nabla u\rVert_{L^{\pi}(U)}$ for $u\in W^{1,\,\pi}(U)$.
 The closed subspace $W_{0}^{1,\,\pi}(\Omega)\subset W^{1,\,\pi}(U)$ is defined as the closure of $C_{\mathrm c}^{1}(U)\subset W^{1,\,\pi}(U)$ with respect to the strong topology induced by this norm.

 Following \cite[Chapitre 2]{Lions-monotone MR0259693} or \cite[Chapter III]{Showalter MR1422252}, we would like to introduce parabolic function spaces.
 We fix the function space
 \[V_{0}=V_{0}(\Omega)\coloneqq \left\{\begin{array}{cc}    W_{0}^{1,\,p}(\Omega)& (p_{\mathrm c}<p<\infty),\\ W_{0}^{1,\,p}(\Omega)\cap L^{2}(\Omega)& (1<p\le p_{\mathrm c}), \end{array}  \right.\] 
 and equip them with the standard norms \[\lVert u \rVert_{V_{0}}\coloneqq \left\{\begin{array}{cc} \lVert \nabla u\rVert_{L^{p}(\Omega)} & (p_{\mathrm c}<p<\infty), \\ \lVert \nabla u \rVert_{L^{p}(\Omega)}+\lVert u\rVert_{L^{2}(\Omega)}, & (1<p\le p_{\mathrm c}), \end{array}\right.\quad \text{for}\quad u\in V_{0}(\Omega).\]
 Then, the continuous embeddings $V_{0}(\Omega)\subset L^{2}(\Omega)\subset V_{0}^{\prime}(\Omega)$ always holds, where $V_{0}^{\prime}=V_{0}^{\prime}(\Omega)$ denotes the continuous dual space of $V_{0}(\Omega)$ with respect to the equipped norm as above.
 Hereinafter, $\langle \mathbf{F},\, \bphi\rangle \in{\mathbb R}$ stands for the duality pairing for ${\mathbf F}=(F^{1},\,\dots\,,\,F^{N})\in V_{0}^{\prime}(\Omega)^{N}$ and $\bphi=(\varphi^{1},\,\dots\,,\,\varphi^{N})\in V_{0}(\Omega)^{N}$.

 The parabolic function spaces are now given as follows;
 \[\begin{array}{rcl}   X_{0}^{p}(0,\,T;\,\Omega)&\coloneqq& \left\{ u\in L^{p}(0,\,T;\,V_{0}(\Omega))\mid \partial_{t}u\in L^{p^{\prime}}(0,\,T;\,V_{0}^{\prime}(\Omega)) \right\},\\     X^{p}(0,\,T;\,\Omega)&\coloneqq& \left\{ u\in L^{p}(0,\,T;\,V(\Omega))\mid \partial_{t}u\in L^{p^{\prime}}(0,\,T;\,V_{0}^{\prime}(\Omega)) \right\}.\end{array}\]
 The inclusion $X_{0}^{p}(0,\,T;\,\Omega)\subset C(\lbrack 0,\,T\rbrack;\,L^{2}(\Omega))$ follows from the Lions--Magenes lemma \cite[Chapter III, Proposition 1.2]{Showalter MR1422252}.
 Moreover, when $p>p_{\mathrm c}$, the compact embedding $V_{0}(\Omega)=W_{0}^{1,\,p}(\Omega)\subset L^{2}(\Omega)$ allows us to apply the Aubin--Lions lemma \cite[Proposition 1.3]{Showalter MR1422252}.
 In particular, the compact embedding $X_{0}^{p}(0,\,T;\,\Omega)\subset L^{p}(0,\,T;\,L^{2}(\Omega))$ is useful in the supercritical range $p\in(p_{\mathrm c},\,\infty)$.
 We must keep in mind that these compact embeddings no longer hold for $p\in(1,\,p_{\mathrm c}\rbrack$.

 Throughout this paper, we treat (\ref{Eq (Section 1): General System}) under the classical setting found in \cite{Lions-monotone MR0259693,Showalter MR1422252}.
 In other words, the parabolic system (\ref{Eq (Section 1): General System}) is treated in the sense of the functional space $L^{p^{\prime}}(0,\,T;\,V_{0}^{\prime}(\Omega))^{N}$.
 For this reason, we let $\buf\in L^{p^{\prime}}(0,\,T;\,V_{0}^{\prime}(\Omega))^{N}$ in defining a weak solution to (\ref{Eq (Section 1): General System}).

 \subsection{Structural assumptions and the definition of a weak solution}\label{Subsect: Structural conditions}
 After providing structural conditions of $a_{1}$, $a_{p}$, $g_{p}$, and $\bg=(\gamma_{\alpha\beta})$, we would like to define a weak solution to (\ref{Eq (Section 1): General System}).

 Throughout this paper, we let $\bg=(\gamma_{\alpha\beta})_{\alpha,\,\beta}\colon \Omega_{T}\to{\mathbb R}^{n\times n}$ be a matrix-valued function that is symmetric and positive definite.
 In other words, $\gamma_{\alpha\beta}=\gamma_{\beta\alpha}$ holds for all $\alpha,\,\beta\in\{\,1,\,\dots\,,\,n\,\}$, and there exists a universal constant $\gamma_{0}\in(0,\,1)$ such that
 \begin{equation}\label{Eq (Section 1): Matrix Gamma}
     \gamma_{0}\lvert \bz\rvert^{2}\le \gamma_{\alpha\beta}(x,\,t)\zeta_{\alpha}\zeta_{\beta}\le \gamma_{0}^{-1}\lvert \bz\rvert^{2}
 \end{equation}
 for all $(x,\,t)\in\Omega_{T}$ and $\bz=(\zeta_{\alpha})\in{\mathbb R}^{n}$, where we use the convention to sum over $\alpha$ and $\beta$.
 For this $\bg$ and $k\in{\mathbb N}$, we introduce the inner product and the norm over ${\mathbb R}^{kn}$ as  \(\langle \bz\mid \bta  \rangle_{\bg(x,\,t)}\coloneqq \gamma_{\alpha\beta}(x,\,t)\zeta_{\alpha}^{j}\eta_{\beta}^{j}\) and \(\lvert \bz \rvert_{\bg(x,\,t)}\coloneqq \langle \bz\mid \bz \rangle_{\bg(x,\,t)}^{1/2}\) for $\bz=(\zeta_{\alpha}^{j}),\,\bta=(\eta_{\eta}^{j})\in{\mathbb R}^{kn}$, where we omit the summation symbol over $j\in\{\,1,\,\dots\,,\,k\,\}$, as well as $\alpha,\,\beta\in\{\,1,\,\dots\,,\,n\}$. 
 The point $(x,\,t)$ is often omitted for notational simplicity; in other words, we write $\langle \bz\mid \bta\rangle_{\bg}$ and $\lvert \bz\rvert_{\bg}$ for short.
 In this paper, we mainly treat $k=1,\,N,\,Nn$.

 For the ellipticity, we let $g_{p}$ admit another universal constant $\kappa_{0}\in(0,\,\infty)$ such that 
 \begin{equation}\label{Eq (Section 1): Ellipticity of p-Laplace-type operator}
     g_{p}(\varepsilon^{2}+\sigma)+2\sigma \min\{\,0,\,g_{p}^{\prime}(\varepsilon^{2}+\sigma)\,\} \ge \kappa_{0}\left(\varepsilon^{2}+\sigma\right)^{p/2-1}
 \end{equation}
 for all $\sigma\in\lbrack 0,\,\infty)$ and $\varepsilon\in(0,\,1)$.
 For the smoothness of $a_{1}$, $a_{p}$, $g_{p}$ and $\bg=(\gamma_{\alpha\beta})$, we only require $a_{1},\,a_{s}\in L^{\infty}(\Omega_{T})$, $\nabla a_{1},\,\nabla a_{p}\in L^{\infty}(\Omega_{T};\,{\mathbb R}^{n})$, $g_{p}\in C^{1}(0,\,\infty)$ and $\gamma_{\alpha\beta}\in W^{1,\,\infty}(\Omega_{T})$.
 More precisely, there exists a universal constant $\Gamma_{0}\in(1,\,\infty)$ such that 
 \begin{equation}\label{Eq (Section 1): Growth of p-Laplace-type operator}
     g_{p}(\sigma)+\sigma\lvert g_{p}^{\prime}(\sigma) \rvert\le \Gamma_{0}\sigma^{p/2-1}\quad \text{for all }\sigma\in(0,\,\infty),
 \end{equation} 
 \begin{equation}\label{Eq (Section 1): Bound Assumptions for Coefficients}
    \esssup_{\Omega_{T}}\left(a_{1}+ a_{p}+
        \lvert\nabla a_{1} \rvert+\lvert\nabla a_{p} \rvert+\lvert \gamma_{\alpha\beta} \rvert+\lvert\nabla\gamma_{\alpha\beta} \rvert+\lvert \partial_{t}\gamma_{\alpha\beta} \rvert\right)\le \Gamma_{0},
 \end{equation}
 \begin{equation}\label{Eq (Section 1): Positivity Assumptions for Coefficients}
   0\le \essinf_{\Omega_{T}}a_{1}, \quad \textrm{and}\quad \Gamma_{0}^{-1}\le \essinf_{\Omega_{T}}a_{p}.
 \end{equation} 
 The continuity of $g_{p}^{\prime}\in C^{0}(0,\,\infty)$ is assumed to be locally controlled by some family of the non-decreasing and concave functions.
 More precisely, for given $0<c_{1}<c_{2}<\infty$, there exists a continuous function $\omega_{p,\,c_{1},\,c_{2}}\colon {\mathbb R}_{\ge 0}\to{\mathbb R}_{\ge 0}$ that is non-decreasing and concave, and satisfies $\omega_{c_{1},\,c_{2}}(0)=0$ and  
 \begin{equation}\label{Eq (Section 1): Modulus of Continuity for p-Laplace-type-operator}
     \left\lvert g_{p}^{\prime}(\sigma_{1})-g_{p}^{\prime}(\sigma_{2})\right\rvert\le \omega_{p,\,c_{1},\,c_{2}}(\lvert \sigma_{1}-\sigma_{2}\rvert)\quad \textrm{for all }\sigma_{1},\,\sigma_{2}\in\lbrack c_{1},\,c_{2}\rbrack.
 \end{equation}
 A typical example of $g_{p}$ is 
 \begin{equation}\label{Eq (Section 1): gp}
    g_{p}(\sigma)\coloneqq \sigma^{p/2-1}\quad \text{for}\quad \sigma\in(0,\,\infty).
 \end{equation} 
 For this choice, the structural conditions (\ref{Eq (Section 1): Ellipticity of p-Laplace-type operator})--(\ref{Eq (Section 1): Growth of p-Laplace-type operator}) are easy to check by direct computations. 
 In particular, (\ref{Eq (Section 1): Ellipticity of p-Laplace-type operator}) holds with $\kappa_{0}\coloneqq \min\{\,1,\,p-1\,\}$, which clearly becomes $0$ when $p=1$.
 Since the second order derivative $g_{p}^{\prime\prime}(\sigma)=(p/2-1)(p/2-2)\sigma^{p/2-3}$ is locally bounded in $(0,\,\infty)$, (\ref{Eq (Section 1): Modulus of Continuity for p-Laplace-type-operator}) holds with $\omega_{p,\,c_{1},\,c_{2}}(\sigma)=\left(\max\limits_{\lbrack c_{1},\,c_{2} \rbrack}\lvert g_{p}^{\prime\prime}\rvert\right)\sigma$. 
 In this paper, we fix 
 \begin{equation}\label{Eq (Section 1): g1}
    g_{1}(\sigma)\coloneqq \sigma^{-1/2}\quad \text{for}\quad \sigma\in(0,\,\infty).
 \end{equation}
 Similarly to (\ref{Eq (Section 1): gp}), this $g_{1}$ satisfies the following (\ref{Eq (Section 1): Ellipticity of 1-Laplace-type operator})--(\ref{Eq (Section 1): Growth of 1-Laplace-type operator});
 \begin{equation}\label{Eq (Section 1): Ellipticity of 1-Laplace-type operator}
    g_{1}(\varepsilon^{2}+\sigma)+2\sigma \min\{\,0,\,g_{1}^{\prime}(\varepsilon^{2}+\sigma)\,\} \ge 0\quad \text{for all }\sigma\in\lbrack 0,\,\infty),\,\varepsilon\in(0,\,1), 
 \end{equation}
 \begin{equation}\label{Eq (Section 1): Growth of 1-Laplace-type operator}
    g_{1}(\sigma)+\sigma\lvert g_{1}^{\prime}(\sigma) \rvert\le \frac{3}{2}\sigma^{-1/2}\quad \text{for all }\sigma\in(0,\,\infty).
 \end{equation} 
 Also, for given $0<c_{1}<c_{2}<\infty$, we have
 \begin{equation}\label{Eq (Section 1): Modulus of Continuity for 1-Laplace-type-operator}
     \left\lvert g_{1}^{\prime}(\sigma_{1})-g_{1}^{\prime}(\sigma_{2})\right\rvert\le \omega_{1,\,c_{1},\,c_{2}}(\lvert \sigma_{1}-\sigma_{2}\rvert)\quad \textrm{for all }\sigma_{1},\,\sigma_{2}\in\lbrack c_{1},\,c_{2}\rbrack,
 \end{equation}
 where $\omega_{1,\,c_{1},\,c_{2}}(\sigma)\coloneqq (4c_{1})^{-1}\sigma$.
 When we let $a_{1}=a_{p}=1$, set $g_{1}$ and $g_{p}$ as (\ref{Eq (Section 1): gp})--(\ref{Eq (Section 1): g1}), and choose $\bg$ as the identity matrix, the parabolic system (\ref{Eq (Section 1): General System}) becomes (\ref{Eq (Section 1): Parabolic (1,p)-Laplace}).

 To define a weak solution to (\ref{Eq (Section 1): General System}), we treat the term $\lvert \D\bu\rvert_{\bg(x,\,t)}^{-1}\D\bu$ in the sense of a subgradient.
 For given $(x,\,t)\in\Omega_{T}$, the subdifferential of $\lvert\,\cdot\, \rvert_{\bg(x,\,t)}\colon{\mathbb R}^{Nn}\to{\mathbb R}_{\ge 0}$ at each point $\bz\in{\mathbb R}^{Nn}$ is defined as
 \[\partial_{\bg(x,\,t)}\lvert\,\cdot\,\rvert_{\bg(x,\,t)}(\bz)\coloneqq \left\{{\mathbf Z}\in{\mathbb R}^{Nn}\mathrel{}\middle| \mathrel{} \lvert \bxi\rvert_{\bg(x,\,t)}\ge \lvert \bz\rvert_{\bg(x,\,t)}+\langle{\mathbf Z}\mid \bxi-\bz \rangle_{\bg(x,\,t)} \text{ for all }\bxi\in{\mathbb R}^{Nn} \right\}.\]
 This set is a singleton $\left\{\lvert \bz\rvert_{\bg(x,\,t)}^{-1}\bz\right\}$ for $\bz\neq 0$, since the convex function $\lvert\,\cdot\, \rvert_{\bg(x,\,t)}$ is differentiable except at the origin.
 Otherwise, this set is the closed unit ball 
 \(\left\{{\mathbf Z}\in{\mathbb R}^{Nn}\mathrel{}\middle| \mathrel{} \lvert {\mathbf Z} \rvert_{\bg(x,\,t)}\le 1 \right\}\) (see \cite[Theorem 1.8]{Andreu-Vaillo et al}).
 For a given mapping ${\mathbf Z}={\mathbf Z}(x,\,t)\colon \Omega_{T}\to{\mathbb R}^{Nn}$, we define 
 \begin{equation}\label{Eq (Section 1): Flux Term}
    {\mathbf A}(x,\,t,\,\bz,\,{\mathbf Z})\coloneqq a_{1}(x,\,t){\mathbf Z}+a_{p}(x,\,t)g_{p}(\lvert \bz\rvert_{\bg}^{2})\bz.
 \end{equation}
 \begin{definition}
    \upshape \label{Definition (Section 1): A weak solution}
    Fix $p\in(1,\,\infty)$ and $\buf\in L^{2,\,1}(\Omega_{T})^{N}\cap L^{p^{\prime}}(0,\,T;\,V_{0}^{\prime}(\Omega))^{N}$.
    A function $\bu$ in the class $X^{p}(0,\,T;\,\Omega)^{N}\cap C(\lbrack 0,\,T\rbrack;\,L^{2}(\Omega))^{N}$ is called a \textit{weak} solution to (\ref{Eq (Section 1): General System}) if there exists ${\mathbf Z}\in L^{\infty}(\Omega_{T};\,{\mathbb R}^{Nn})$ such that
    \[{\mathbf Z}(x,\,t)\in \partial_{\bg(x,\,t)}\lvert\, \cdot\, \rvert_{\bg(x,\,t)}(\D\bu(x,\,t))\]
    for a.e.~$(x,\,t)\in\Omega_{T}$, and the term $\A(x,\,t,\,\D\bu,\,{\mathbf Z})$, defined as (\ref{Eq (Section 1): Flux Term}), satisfies
    \[\int_{0}^{T}\langle \partial_{t}\bu,\,\bphi \rangle\,\d t+\iint_{\Omega_{T}}\langle \A(x,\,t,\,\D\bu,\,{\mathbf Z})\mid\D\bphi \rangle_{\bg}\,\d x\d t=\iint_{\Omega_{T}}\langle \buf\mid \bphi \rangle\,\d x\d t\]
    for all $\bphi\in L^{p}(0,\,T;\,V_{0}(\Omega))^{N}$. 
 \end{definition}
 Finally, we introduce the datum set 
 \[{\mathcal D}\coloneqq \{\,n,\,N,\,p,\,q,\,r,\,\gamma_{0},\,\kappa_{0},\,\Gamma_{0},\,\{\omega_{p,\,c_{1},\,c_{2}}\}_{0<c_{1}<c_{2}<\infty}\}.\]
 We often use the abbreviation $C=C({\mathcal D})$ when the constant $C$ found in our estimates may depend on some members of $\mathcal{D}$.
 \subsection{Main result and plan of the paper}
 In this paper, we would like to prove Theorem \ref{Theorem (Section 1): Gradient continuity}.
 \begin{theorem}\label{Theorem (Section 1): Gradient continuity}
     Let $p\in(1,\,\infty)$, $\buf\in L^{p^{\prime}}(0,\,T;\,V_{0}^{\prime})^{N}\cap L^{r}(0,\,T;\,L^{q}(\Omega))^{N}$ with $(q,\,r)\in(n,\,\infty\rbrack\times (2,\,\infty\rbrack$ satisfying (\ref{Eq (Section 1): Condition for q,r}).
     Assume that $\bu\in X^{p}(0,\,T;\,\Omega)^{N}\cap  C(\lbrack 0,\,T\rbrack;\,L^{2}(\Omega))^{N}$ is a weak solution to (\ref{Eq (Section 1): General System}).
     When $p\in(1,\,p_{\mathrm c}\rbrack$, let (\ref{Eq (Section 1): Higher Integrability of u for subcritical case}) be additionally assumed.
     Then, the spatial gradient $\D\bu$ is continuous in $\Omega_{T}$.
 \end{theorem}

 This paper is organized as follows.
 Section \ref{Section: Preliminaries} provides preliminaries. 
 There we particularly introduce an approximate system for (\ref{Eq (Section 1): General System}).
 Section \ref{Section: L-infty} is focused on deducing bound estimates for weak solutions in the singular range $p\in(1,\,2)$.
 In particular, we use Moser's iteration to show a weak maximum principle for the approximate problems, and local $L^{\infty}$-estimates for (\ref{Eq (Section 1): General System}).
 In Section \ref{Section: Weak Formulations}, we would like to list the basic regularity estimates for approximate problems.
 In particular, local gradient bounds (Theorem \ref{Theorem (Section 4): Gradient bounds}), a De Giorgi-type oscillation lemma (Proposition \ref{Proposition (Section 4): Degenerate Case}), and a Campanato-type growth estimate (Proposition \ref{Proposition (Section 4): Non-Degenerate Case}) are shown in Sections \ref{Section: Gradient Bounds}--\ref{Section: Non-Degenerate}.
 Various weak formulations and energy estimates are also discussed in Section \ref{Section: Weak Formulations} as preliminaries for Sections \ref{Section: Gradient Bounds}--\ref{Section: Non-Degenerate}.
 Section \ref{Section: Gradient Bounds} aims to show local gradient bounds (Theorem \ref{Theorem (Section 4): Gradient bounds}) by Moser's iteration.
 The analytic approaches therein become different, depending on whether $p>p_{\mathrm c}$ or $p\le p_{\mathrm c}$.
 In Section \ref{Section: Degenerate}, we use De Giorgi's truncation to show an oscillation lemma for a certain subsolution (Proposition \ref{Proposition (Section 4): Degenerate Case}).
 Section \ref{Section: Non-Degenerate} shows Campanato-type growth estimates by the comparisons with some sort of heat flows.
 The proofs in Sections \ref{Section: Degenerate}--\ref{Section: Non-Degenerate} are carefully carried out by the truncation method, so that non-uniformly elliptic structures are suitably discarded. 
 The resulting estimates therein depend on the truncation parameter $\delta$.
 Section \ref{Section: Convergence Result} aims to prove the strong convergence of approximate solutions.
 More precisely, we would like to prove that a spatial gradient converges strongly in $L^{p}$.
 There, the treatment of external force terms will differ, depending on whether $p>p_{\mathrm c}$ or not.
 Section \ref{Section: Convergence Result} also includes some solvability results of the $(1,\,p)$-Laplace parabolic Dirichlet boundary problems, where the external force term is in $L^{2,\,1}(\Omega_{T})^{N}$.
 The proof of Theorem \ref{Theorem (Section 1): Gradient continuity} is finally given in Section \ref{Section: Gradient Continuity}.
 There, the local H\"{o}lder estimate for truncated gradients of approximate solutions (Theorem \ref{Theorem (Section 4): Holder Truncated-Gradient Continuity}) is also shown by using Propositions \ref{Proposition (Section 4): Degenerate Case}--\ref{Proposition (Section 4): Non-Degenerate Case}.

 \begin{remark}\upshape
    In Sections \ref{Section: L-infty}--\ref{Section: Convergence Result}, we often provide formal computations, in the sense that the time derivatives $\partial_{t}\bu,\,\partial_{t}\bu_{\varepsilon}$ are treated as some sort of \textit{functions}, although they are merely \textit{functionals}.
    These formal computations are justified by using the Steklov average, found in \cite{LSU MR0241822} (see also \cite[Appendix B]{BDM MR3073153} for an alternative approach).
 \end{remark}

 \subsection*{Acknowledgments}
 The author is supported by JSPS KAKENHI Grant No.~JP24K22828.
 The author would like to thank Prof.~Takahito Kashiwabara for kindly informing him of the literature \cite{Fuchs-Seregin Monograph}. 
 \section{Preliminaries}\label{Section: Preliminaries}
 \subsection{Approximate systems and a convergence lemma}
 We introduce an approximation parameter $\varepsilon\in(0,\,1)$, and assume that $\buf_{\varepsilon}\in L^{\infty}(\Omega_{T})^{N}$ satisfies
 \begin{equation}\label{Eq (Section 2): Weak Conv of f-epsilon}
     \buf_{\varepsilon} \stackrel{\star}{\rightharpoonup} \buf\quad \text{in}\quad L^{q,\,r}(\Omega_{T})^{N}\quad \text{and}\quad \buf_{\varepsilon}\rightharpoonup \buf \quad \text{in}\quad L^{p^{\prime}}(0,\,T;\,V_{0}^{\prime})^{N}.
 \end{equation}
 This condition is not restrictive by straightforward approximations.
 In fact, extending $\buf\in L^{q,\,r}(\Omega_{T})^{N}\cap L^{p^{\prime}}(0,\,T;\,V_{0}^{\prime}(\Omega))$ in ${\mathbb R}^{n+1}\setminus \Omega_{T}$ by the zero function, and convoluting this extended function with the $(n+1)$-dimensional Friedrichs mollifier, we easily show (\ref{Eq (Section 2): Weak Conv of f-epsilon}).
 Moreover, when both $q$ and $r$ are finite, the first weak convergence of (\ref{Eq (Section 2): Weak Conv of f-epsilon}) is replaced by the strong convergence in $L^{q,\,r}(\Omega_{T})^{N}$.
 For each $\varepsilon\in(0,\,1)$, we consider a parabolic system of the form
 \begin{equation}\label{Eq (Section 2): Approximate System}
    \partial_{t}u^{j}-\partial_{x_{\beta}}\left(\gamma_{\alpha\beta}a_{s}(x,\,t)g_{s}(\varepsilon^{2}+\lvert\D\bu_{\varepsilon}\rvert_{\bg}^{2})\partial_{x_{\alpha}}u_{\varepsilon}^{j} \right)=f_{\varepsilon}^{j}\quad \text{in}\quad \Omega_{T}
 \end{equation}
 for every $j\in\{\,1,\,\dots\,,\,N\,\}$, where we sum over $\alpha,\,\beta\in\{\,1,\,\dots\,,\,n\,\}$ and $s\in\{\,1,\,p\,\}$.
 In addition to (\ref{Eq (Section 2): Approximate System}), we treat the parabolic Dirichlet boundary condition
 \begin{equation}\label{Eq (Section 2): Parabolic Dirichlet Boundary}
     \bu_{\varepsilon}=\bu_{\star}\quad \textrm{on}\quad \partial_{\mathrm p}\Omega_{T}
 \end{equation}
 with $\bu_{\star}\in X^{p}(0,\,T;\,\Omega)^{N}\cap C(\lbrack 0,\,T\rbrack;\,L^{2}(\Omega))^{N}$.
 In other words, we consider $\bu_{\varepsilon}\in \bu_{\star}+X_{0}^{p}(0,\,T;\,\Omega)^{N}\subset C(\lbrack 0,\,T\rbrack;\,L^{2}(\Omega))^{N}$ that satisfies
 \begin{equation}\label{Eq (Section 2): Approximate Weak Form}
     \int_{0}^{T}\langle\partial_{t}\bu_{\varepsilon},\,\bphi \rangle\,\d t+\iint_{\Omega_{T}}\langle \A_{\varepsilon}(x,\,t,\,\D\bu)\mid \D\bphi \rangle_{\bg} =\iint_{\Omega_{T}}\langle \buf_{\varepsilon}\mid\bphi\rangle\,\d x\d t
 \end{equation}
 for all $\bphi\in X_{0}^{p}(0,\,T;\,\Omega)^{N}$, and $(\bu_{\varepsilon}-\bu_{\star})(\,\cdot\,,\,0)=0$ in $L^{2}(\Omega)^{N}$.
 Here the mapping ${\mathbf A}_{\varepsilon}\colon \Omega_{T}\times {\mathbb R}^{Nn}\to{\mathbb R}^{Nn}$ is defined as
 \[{\mathbf A}_{\varepsilon}(x,\,t,\,\bz)\coloneqq {\mathbf A}_{1,\,\varepsilon}(x,\,t,\,\bz)+{\mathbf A}_{p,\,\varepsilon}(x,\,t,\,\bz)\] with \(\A_{s,\,\varepsilon}(x,\,t,\,\bz)\coloneqq a_{s}(x,\,t)g_{s}(\varepsilon^{2}+\lvert\bz\rvert_{\bg(x,\,t)}^{2})\bz\)
 for $s\in\{\,1,\,p\,\}$, $(x,\,t)\in \Omega_{T}$, and $\bz\in{\mathbb R}^{Nn}$.
 The unique existence of the boundary value problem (\ref{Eq (Section 2): Approximate System})--(\ref{Eq (Section 2): Parabolic Dirichlet Boundary}) is in the scope of the classical monotone operator theory.
 More precisely, since (\ref{Eq (Section 1): Ellipticity of p-Laplace-type operator}), (\ref{Eq (Section 1): Positivity Assumptions for Coefficients}) and (\ref{Eq (Section 1): Ellipticity of 1-Laplace-type operator}) imply that $\A_{\varepsilon}$ satisfies 
 \begin{equation}\label{Eq (Section 2): Coercivity of A-epsilon}
     \langle \A_{\varepsilon}(x,\,t,\,\bz)\mid \bz \rangle_{\bg(x,\,t)}\ge \lambda_{0}\left(\varepsilon^{2}+\lvert \bz\rvert_{\bg}^{2} \right)^{p/2-1}\ge \left\{\begin{array}{cc}
        \lambda_{0}\lvert \bz\rvert_{\bg}^{p} & (2\le p<\infty), \\ \lambda_{0}\left(\lvert \bz\rvert_{\bg}^{p}-\varepsilon^{p} \right) & (1<p<2),
     \end{array} \right.
 \end{equation} 
 with $\lambda_{0}\coloneqq \kappa_{0}\Gamma_{0}^{-1}>0$,
 we can construct the weak solution of (\ref{Eq (Section 2): Approximate System})--(\ref{Eq (Section 2): Parabolic Dirichlet Boundary}) by the Faedo--Galerkin method (see e.g., \cite[Chapitre 2]{Lions-monotone MR0259693}, \cite[Chapter III]{Showalter MR1422252}).
 We note that the classical method therein might not work directly for (\ref{Eq (Section 2): Approximate System}), since the total variation energy lacks its differentiability, in the sense of G\^{a}teaux or Fr\'{e}chet derivative.
 In Section \ref{Section: Convergence Result}, we will later justify that we can construct weak solutions to (\ref{Eq (Section 1): General System}) as the limit of weak solutions to (\ref{Eq (Section 2): Approximate System}).
 In this sense, we would like to call (\ref{Eq (Section 2): Approximate System}) as an \textit{approximate} system, and a weak solution to (\ref{Eq (Section 2): Approximate System}) as an \textit{approximate} solution.
 Lemma \ref{Lemma (Section 2): Fundamental Lemma for convergence of solutions} below is useful in showing the convergence of approximate solutions.
 \begin{lemma}\label{Lemma (Section 2): Fundamental Lemma for convergence of solutions}
    For the mapping ${\mathbf A}_{\varepsilon}={\mathbf A}_{1,\,\varepsilon}+{\mathbf A}_{p,\,\varepsilon}$, we have the following.
    \begin{enumerate}
        \item \label{Item 1/3} For each fixed $\bv\in L^{p}(\Omega_{T};\,{\mathbb R}^{Nn})$, we have
        \begin{equation}\label{Eq (Section 3): Dummy Conv}
            {\mathbf A}_{\varepsilon}(x,\,t,\,\bv)\to{\mathbf A}_{0}(x,\,t,\,\bv)\quad \text{in}\quad L^{p^{\prime}}(\Omega_{T};\,{\mathbb R}^{Nn}).
        \end{equation}
        Here, for $(x,\,t)\in\Omega_{T}$ and $\bz\in{\mathbb R}^{Nn}$, ${\mathbf A}_{0}(x,\,t,\,\bz)$ is defined as $\A_{0}(x,\,t,\,\bz)\coloneqq a_{1}(x,\,t)g_{1}(\lvert \bz\rvert^{2})\bz +a_{p}(x,\,t)g_{p}(\lvert \bz\rvert^{2})\bz$ when $\bz\neq 0$, and otherwise $\A_{\varepsilon}(x,\,t,\,0)\coloneqq 0$.
        \item \label{Item 2/3} Let $\bv_{\varepsilon}\in L^{p}(\Omega_{T};\,{\mathbb R}^{Nn})$ be given for each $\varepsilon\in(0,\,1)$. 
        Assume that there hold $\varepsilon_{k}\to 0$ and $\bv_{\varepsilon_{k}}\to \bv_{0}$ in $L^{p}(\Omega_{T};\,{\mathbb R}^{Nn})$ for some $\bv_{0}\in L^{p}(\Omega_{T};\,{\mathbb R}^{Nn})$, as $k\to\infty$.
        Then, by taking a subsequence if necessary, we may let
        \begin{equation}\label{Eq (Section 3): Conv on A-p-epsilon}
            {\mathbf A}_{p,\,\varepsilon_{k}}(x,\,t,\,\bv_{\varepsilon_{k}})\to g_{p}(\lvert \bv_{0} \rvert_{\bg}^{2})\bv_{0}\quad \text{in}\quad L^{p^{\prime}}(\Omega_{T};\,{\mathbb R}^{Nn}),
        \end{equation}
        and
        \begin{equation}\label{Eq (Section 3): Conv on A-1-epsilon}
            {\mathbf Z}_{\varepsilon_{k}}\coloneqq \frac{\bv_{\varepsilon_{k}}}{\sqrt{\varepsilon^{2}+\lvert \bv_{\varepsilon_{k}}\rvert_{\bg}^{2}}}\stackrel{\star}{\rightharpoonup} {\mathbf Z}_{0}\quad \text{in}\quad L^{\infty}(\Omega_{T};\,{\mathbb R}^{Nn}),
        \end{equation}
        where the limit ${\mathbf Z}_{0}\in L^{\infty}(\Omega_{T};\,{\mathbb R}^{Nn})$ satisfies
        \begin{equation}\label{Eq (Section 3): Subgradient}
            {\mathbf Z}_{0}(x,\,t)\in \partial_{\bg(x,\,t)}\lvert \,\cdot\,\rvert_{\bg(x,\,t)}(\bv_{0}(x,\,t))\quad \text{for a.e.~}(x,\,t)\in\Omega_{T}.
        \end{equation}
        In particular, along with this subsequence, we have
        \[{\mathbf A}_{\varepsilon_{k}}(x,\,t,\,\bv_{\varepsilon_{k}}) \rightharpoonup {\mathbf A}(x,\,t,\,\bv_{0},\,{\mathbf Z}_{0})\quad \textrm{in}\quad L^{p^{\prime}}(\Omega_{T};\,{\mathbb R}^{Nn}).\]
        \item \label{Item 3/3} For every $m\in{\mathbb N}$, let the pair $(\bv_{m},\,{\mathbf Z}_{m})\in L^{1}(\Omega_{T};\,{\mathbb R}^{Nn})\times L^{\infty}(\Omega_{T};\,{\mathbb R}^{Nn})$ satisfy ${\mathbf Z}_{m}(x,\,t)\in \partial_{\bg (x,\,t)}\lvert\,\cdot\, \rvert_{\bg(x,\,t)}(\bv_{m}(x,\,t))$ for a.e.~$(x,\,t)\in\Omega_{T}$. 
        If $\bv_{m}\to \bv$ a.e. in $\Omega_{T}$ as $m\to\infty$, then by taking a subsequence if necessary, we may let
        \begin{equation}\label{Eq (Section 3): Conv on Z-m}
            {\mathbf Z}_{m}\stackrel{\star}{\rightharpoonup}{\mathbf Z}\quad \textrm{in}\quad L^{\infty}(\Omega_{T};\,{\mathbb R}^{N}),
        \end{equation}
        where the weak limit ${\mathbf Z}$ satisfies 
        \begin{equation}\label{Eq (Section 3): Subgradient again}
            {\mathbf Z}(x,\,t)\in\partial_{\bg(x,\,t)}\lvert \,\cdot\,\rvert_{\bg(x,\,t)}(\bv(x,\,t))\quad \text{for a.e.~}    (x,\,t)\in\Omega_{T}.
        \end{equation}
    \end{enumerate}
 \end{lemma}
 We briefly outline the proof of Lemma \ref{Lemma (Section 2): Fundamental Lemma for convergence of solutions}. 
 The detailed discussions are found in \cite[Lemma 2.8]{T-scalar} and \cite[Lemma 5]{Giga-Tsubouchi MR4408168}.
 \begin{proof}
    The strong convergence results (\ref{Eq (Section 3): Dummy Conv})--(\ref{Eq (Section 3): Conv on A-p-epsilon}) are easy consequences from Lebesgue's dominated convergence theorem (see \cite[Lemma 2.8 (1)--(2)]{T-scalar}).
    Although (\ref{Eq (Section 3): Conv on A-1-epsilon}) and (\ref{Eq (Section 3): Conv on Z-m}) are clear by a weak compactness argument, we need to verify the inclusion properties (\ref{Eq (Section 3): Subgradient}) and (\ref{Eq (Section 3): Subgradient again}) respectively.
    The claim (\ref{Eq (Section 3): Subgradient}) is equivalent to the two assertions; ${\mathbf Z}=\bv_{0}/\lvert\bv_{0}\rvert_{\bg}$ a.e.~in $D\coloneqq \{\Omega_{T}\mid \bv_{0}\neq 0\}$, and $\lvert {\mathbf Z}\rvert_{\bg}\le 1$ a.e.~in $\Omega_{T}$.
    We remark that $\bv_{\varepsilon_{k}}\to \bv_{0}$ a.e.~in $\Omega_{T}$, which is not restrictive by relabelling a sequence, implies ${\mathbf Z}_{\varepsilon_{k}} \to\bv_{0}/\lvert \bv_{0}\rvert_{\bg}$ a.e.~in $D$.
    Combining this fact and (\ref{Eq (Section 3): Conv on A-1-epsilon}) yields the first assertion.
    The second one is easy, since $\lvert {\mathbf Z}_{\varepsilon_{k}}\rvert_{\bg}\le 1$ clearly holds a.e.~in $\Omega_{T}$ and this inequality is preserved under limit passage (\ref{Eq (Section 3): Conv on A-1-epsilon}). 
    Indeed, the mapping $L^{\infty}(\Omega_{T};\,{\mathbb R}^{Nn}) \ni {\mathbf W}\mapsto \esssup\limits_{\Omega_{T}}\, \lvert {\mathbf W}\rvert_{\bg}\in{\mathbb R}_{\ge 0}$ is sequentially lower semicontinuous with respect to the weak${}^{\star}$ topology (see \cite[Lemma 2.8 (3)]{T-scalar} for a general result).
    In the same way, (\ref{Eq (Section 3): Subgradient again}) is shown (see also \cite[Lemma 5]{Giga-Tsubouchi MR4408168}). 
 \end{proof}
 Among the three convergence results listed in Lemma \ref{Lemma (Section 2): Fundamental Lemma for convergence of solutions}, the second one plays an important role.
 Roughly speaking, we can construct a weak solution to (\ref{Eq (Section 1): General System}) as the limit function of weak solutions to (\ref{Eq (Section 2): Approximate System}), as long as the strong convergence of a spatial derivative is verified. 
 Lemma \ref{Lemma (Section 2): Fundamental Lemma for convergence of solutions} is used later in Section \ref{Section: Convergence Result}.
 \subsection{Basic structures of the approximate operators and some related mappings}
 We introduce basic bilinear forms related to $\A_{\varepsilon}(x,\,t,\,\bz)$.
 For $(x,\,t)\in \Omega_{T}$ and $\bz\in {\mathbb R}^{Nn}$, we define
 \[
 {\mathcal A}_{\varepsilon}(x,\,t,\,\bz)(\bxi,\,\bta)\coloneqq a_{s}(x,\,t)\left[g_{s}(\varepsilon^{2}+\lvert \bz\rvert_{\bg}^{2})\gamma_{\alpha\beta}\delta^{ij}+2g_{s}^{\prime}(\varepsilon^{2}+\lvert\bz\rvert_{\bg}^{2})\gamma_{\alpha\kappa}\zeta_{\kappa}^{i}\gamma_{\beta\lambda}\zeta_{\lambda}^{j} \right]\gamma_{\mu\nu}\xi_{\alpha\mu}^{i}\eta_{\beta\nu}^{j}
 \]
 for $\bxi=(\xi_{\alpha\mu}^{i})$, $\bta=(\eta_{\beta\nu}^{j})\in{\mathbb R}^{Nn^{2}}$,
 \[
 {\mathcal B}_{\varepsilon}(x,\,t,\,\bz)(\bxi,\,\bta)\coloneqq a_{s}(x,\,t)\left[g_{s}(\varepsilon^{2}+\lvert \bz\rvert_{\bg}^{2})\gamma_{\alpha\beta}\delta^{ij}+2g_{s}^{\prime}(\varepsilon^{2}+\lvert\bz\rvert_{\bg}^{2})\gamma_{\alpha\kappa}\zeta_{\kappa}^{i}\gamma_{\beta\lambda}\zeta_{\lambda}^{j} \right]\xi_{\alpha}^{i}\eta_{\beta}^{j}
 \]
 for $\bxi=(\xi_{\alpha}^{i})$, $\bta=(\eta_{\beta}^{j})\in{\mathbb R}^{Nn}$, and
 \[
 {\mathcal C}_{\varepsilon}(x,\,t,\,\bz)(\bxi,\,\bta)\coloneqq a_{s}(x,\,t)\left[g_{s}(\varepsilon^{2}+\lvert \bz\rvert_{\bg}^{2})\gamma_{\alpha\beta}+2g_{s}^{\prime}(\varepsilon^{2}+\lvert\zeta\rvert_{\bg}^{2})\gamma_{\alpha\kappa}\zeta_{\kappa}^{i}\gamma_{\beta\lambda}\zeta_{\lambda}^{j} \right]\xi_{\alpha}\eta_{\beta}
 \]
 for $\bxi=(\xi_{\alpha})$, $\bta=(\eta_{\beta})\in{\mathbb R}^{n}$.
 Here $\delta^{ij}$ denotes the Kronecker delta, and we use the convention to sum all of $\alpha,\,\beta,\,\kappa,\,\lambda,\,\mu,\,\nu\in\{\,1,\,\dots\,n\,\}$, $i,\,j\in\{\,1,\,\dots\,,\,N\,\}$, and $s\in\{\,1,\,p\,\}$.
 In this subsection, we would like to deduce basic estimates concerning these symmetric bilinear forms or some other mappings.
 These estimates are used later in showing a priori estimates of approximate solutions.

 Firstly, we check the ellipticity estimates of these symmetric bilinear forms.
 \begin{lemma}\label{Lemma (Section 2): Ellipticity of Bilinear Forms}
    For $\varepsilon\in(0,\,1)$, $(x,\,t)\in \Omega_{T}$, $\bz\in{\mathbb R}^{Nn}$, the symmetric bilinear forms ${\mathcal A}_{\varepsilon}(x,\,t,\,\bz\,)$, ${\mathcal B}_{\varepsilon}(x,\,t,\,\bz)$, and ${\mathcal C}_{\varepsilon}(x,\,t,\,\bz)$ satisfy
    \begin{equation}\label{Eq (Section 2): Ellipticity of Bilinear Forms}
        \begin{array}{ccccc}
            \lambda_{0}h_{p,\,\varepsilon}(\bz)\lvert \bxi_{1}\rvert_{\bg}^{2} &\le & {\mathcal A}_{\varepsilon}(x,\,t,\,{\bz})(\bxi_{1},\,\bxi_{1})&\le &\Lambda_{0}\left(h_{1,\,\varepsilon}(\bz)+h_{p,\,\varepsilon}(\bz)\right)\lvert\bxi_{1}\rvert_{\bg}^{2},\\
            \lambda_{0}h_{p,\,\varepsilon}(\bz)\lvert \bxi_{2}\rvert_{\bg}^{2} &\le & {\mathcal B}_{\varepsilon}(x,\,t,\,{\bz})(\bxi_{2},\,\bxi_{2})&\le &\Lambda_{0}\left(h_{1,\,\varepsilon}(\bz)+h_{p,\,\varepsilon}(\bz)\right)\lvert\bxi_{2}\rvert_{\bg}^{2},\\ 
            \lambda_{0}h_{p,\,\varepsilon}(\bz)\lvert \bxi_{3}\rvert_{\bg}^{2} &\le & {\mathcal C}_{\varepsilon}(x,\,t,\,{\bz})(\bxi_{3},\,\bxi_{3})&\le &\Lambda_{0}\left(h_{1,\,\varepsilon}(\bz)+h_{p,\,\varepsilon}(\bz)\right)\lvert\bxi_{3}\rvert_{\bg}^{2}, 
        \end{array}
    \end{equation}
    for all $\bxi_{1}\in{\mathbb R}^{Nn^{2}}$, $\bxi_{2}\in{\mathbb R}^{Nn}$, $\bxi_{3}\in{\mathbb R}^{n}$.
    Here $\lambda_{0}\coloneqq \kappa_{0}\Gamma_{0}^{-1}$, $h_{s,\,\varepsilon}(\bz)\coloneqq \left(\varepsilon^{2}+\lvert \bz\rvert_{\bg}^{2} \right)^{s/2-1}$ for $s\in\lbrack 1,\,\infty)$, $\varepsilon\in(0,\,1)$, $\bz\in{\mathbb R}^{Nn}$, and the constant $\Lambda_{0}\in(1,\,\infty)$ depends at most on the datum set ${\mathcal D}$.
 \end{lemma}
 \begin{proof}
    The right-hand-side inequalities are easily shown by (\ref{Eq (Section 1): Growth of p-Laplace-type operator})--(\ref{Eq (Section 1): Bound Assumptions for Coefficients}).
    The left-hand-side ones are shown by (\ref{Eq (Section 1): Ellipticity of p-Laplace-type operator}), (\ref{Eq (Section 1): Positivity Assumptions for Coefficients}), and the Cauchy--Schwarz inequality for the positive definite matrix $\bg$ (see \cite[Lemma 2.7]{BDLS-parabolic}).
 \end{proof}
 Secondly, we prove a few estimates related to $\A_{\varepsilon}$ or $\lvert\,\cdot\,\rvert_{\bg}$, as in Lemmata \ref{Lemma (Section 2): Continuity of Hessian}--\ref{Lemma (Section 2): Continuity on Norms}.
 \begin{lemma}\label{Lemma (Section 2): Continuity of Hessian}
     Let the positive parameters $\delta$ and $\varepsilon$ satisfy $\delta\in(0,\,1)$ and $\varepsilon\in(0,\,\delta/4)$ respectively.
     Fix the point $(x,\,t)\in\Omega_{T}$ and the constants $0<c_{1}<c_{2}<\infty$.
     Then, we have
     \begin{equation}\label{Eq (Section 2): Ellipticity outside the origin}
         \langle \A_{\varepsilon}(x,\,t,\,\bz_{1})-\A_{\varepsilon}(x,\,t,\,\bz_{0})\mid \bz_{1}-\bz_{0} \rangle_{\bg(x,\,t)}\ge C({\mathcal D},\,c_{1},\,c_{2})\lvert\bz_{1}-\bz_{0}\rvert^{2},
     \end{equation}
     \begin{equation}\label{Eq (Section 2): Continuity outside the origin}
         \lvert \langle \A_{\varepsilon}(x,\,t,\,\bz_{1})-\A_{\varepsilon}(x,\,t,\,\bz_{2})\mid \bxi\rangle_{\bg(x,\,t)} \rvert\le C({\mathcal D},\,c_{1},\,c_{2})\lvert \bz_{1}-\bz_{0}\rvert \lvert \bxi\rvert,
     \end{equation}
     \begin{equation}\label{Eq (Section 2): Continuity of Hessian Matrix}
         \left\lvert {\mathcal B}_{\varepsilon}(x,\,t,\,\bz_{0})(\bz_{1}-\bz_{0},\,\bxi)-\langle \A_{\varepsilon}(x,\,t,\,\bz_{1})-\A_{\varepsilon}(x,\,t,\,\bz_{0})\mid \bxi \rangle_{\bg(x,\,t)} \right\rvert\le \omega(\lvert \bz_{1}-\bz_{0} \rvert)\lvert \bz_{1}-\bz_{0}\rvert \lvert \bxi \rvert,
     \end{equation}
     for any $\bz_{0},\,\bz_{1}\,,\,\bxi\in{\mathbb R}^{Nn}$ satisfying $c_{1}\le \lvert \bz_{0}\rvert_{\bg(x,\,t)}\le c_{2}$ and $\lvert \bz_{1}\rvert_{\bg(x,\,t)}\le c_{2}$.
     Here the concave function $\omega\colon \lbrack 0,\,\infty)\to \lbrack 0,\,\infty)$ is of the form 
     \[\omega(\sigma)\coloneqq C({\mathcal D},\,\delta,\,c_{1},\,c_{2})\left(\omega_{1,\,\widetilde{c}_{1},\,\widetilde{c}_{2}} ( (c_{1}+2c_{2})\sigma/\gamma_{0})+\omega_{p,\,\widetilde{c}_{1},\,\widetilde{c}_{2}} ( (c_{1}+2c_{2})\sigma/\gamma_{0})+\sigma\right)\quad \textrm{for}\quad \sigma\in\lbrack 0,\,\infty)\]
     with $\widetilde{c}_{1}\coloneqq c_{1}^{2}/4$, $\widetilde{c}_{2}\coloneqq (c_{1}/2+c_{2})^{2}+\delta^{2}/16$.
    \end{lemma}
 \begin{proof}
     For given $\bz_{0},\,\bz_{1}\in{\mathbb R}^{Nn}$, we define $\bz_{\tau}\coloneqq \bz_{0}+\tau(\bz_{1}-\bz_{0})$ for $\tau\in(0,\,1)$.
     Then, we have
     \[
        \langle \A_{\varepsilon}(x,\,t,\,\bz_{1})-\A_{\varepsilon}(x,\,t,\,\bz_{0})\mid \bxi \rangle_{\bg(x,\,t)}=\int_{0}^{1}{\mathcal B}_{\varepsilon}(x,\,t,\,\bz_{\tau})(\bz_{1}-\bz_{0},\,\bxi)\,\d\tau.
     \]
     Using this identity and (\ref{Eq (Section 2): Ellipticity of Bilinear Forms}), we can prove (\ref{Eq (Section 2): Ellipticity outside the origin})--(\ref{Eq (Section 2): Continuity outside the origin}) are proved, similarly to \cite[Lemma 2.6]{T-scalar}.
     To prove (\ref{Eq (Section 2): Continuity of Hessian Matrix}), we first consider $\lvert \bz_{1}-\bz_{0} \rvert_{\bg(x,\,t)}\le c_{1}/2$. 
     Then, the triangle inequality implies
     \begin{equation}\label{Eq (Section 2): range of zeta-tau}
        c_{1}/2 \le \lvert \bz_{\tau}\rvert_{\bg(x,\,t)}\le c_{1}/2+c_{2}\quad \text{and} \quad  \widetilde{c}_{1} \le \varepsilon^{2}+\lvert \bz_{\tau}\rvert_{\bg(x,\,t)}^{2}\le \widetilde{c}_{2}
     \end{equation}
     for all $\tau\in\lbrack 0,\,1\rbrack$.
     We use the above identity to compute 
     \begin{align*}
        &\left\lvert {\mathcal B}_{\varepsilon}(x,\,t,\,\bz_{0})(\bz_{1}-\bz_{0},\,\bxi)-\langle \A_{\varepsilon}(x,\,t,\,\bz_{1})-\A_{\varepsilon}(x,\,t,\,\bz_{0})\mid \bxi \rangle_{\bg(x,\,t)} \right\rvert\\ 
        &\le \int_{0}^{1}\lvert {\mathcal B}_{\varepsilon}(x,\,t,\,\bz_{\tau})(\bz_{1}-\bz_{0},\,\bxi)-{\mathcal B}_{\varepsilon}(x,\,t,\,\bz_{0})(\bz_{1}-\bz_{0},\,\bxi)\rvert\,\d\tau\\ 
        &\le C({\mathcal D},\,\delta,\,c_{1},\,c_{2})\lvert \bz_{1}-\bz_{0}\rvert \lvert \bxi\rvert \sum_{l=0,\,1}\sum_{s=1,\,p} \int_{0}^{1}\left\lvert g_{s}^{(l)}(\varepsilon^{2}+\lvert \bz_{\tau}\rvert_{\bg(x,\,t)}^{2})-g_{s}^{(l)}(\varepsilon^{2}+\lvert \bz_{0}\rvert_{\bg(x,\,t)}^{2}) \right\rvert \,\d \tau.
     \end{align*}
     With (\ref{Eq (Section 2): range of zeta-tau}) and $\lvert \bz_{\tau}-\bz_{0}\rvert\le \lvert \bz_{1}-\bz_{0}\rvert$ in mind, we can easily deduce (\ref{Eq (Section 2): Continuity of Hessian Matrix}) by using (\ref{Eq (Section 1): Matrix Gamma}), (\ref{Eq (Section 1): Growth of p-Laplace-type operator}), (\ref{Eq (Section 1): Modulus of Continuity for p-Laplace-type-operator}) and (\ref{Eq (Section 1): Growth of 1-Laplace-type operator})--(\ref{Eq (Section 1): Modulus of Continuity for 1-Laplace-type-operator}). 
     In the remaining case $\lvert \bz_{1}-\bz_{0} \rvert_{\bg(x,\,t)}>c_{1}/2$, we simply compute 
     \[
        \left\lvert {\mathcal B}_{\varepsilon}(x,\,t,\,\bz_{0})(\bz_{1}-\bz_{0},\,\bxi)-\langle \A_{\varepsilon}(x,\,t,\,\bz_{1})-\A_{\varepsilon}(x,\,t,\,\bz_{0})\mid \bxi \rangle_{\bg(x,\,t)} \right\rvert\le C({\mathcal D},\,\delta,\,c_{1},\,c_{2})\lvert\bz_{1}-\bz_{0} \rvert\lvert\bxi \rvert
     \]
     and use $2(\gamma_{0}c_{1})^{-1}\lvert \bz_{1}-\bz_{0}\rvert>1$ to conclude (\ref{Eq (Section 2): Continuity of Hessian Matrix}).
 \end{proof}
 \begin{lemma}\label{Lemma (Section 2): Continuity on Norms}
    Let $(x,\,t)\in Q_{\rho}(x_{0},\,t_{0})\subset \Omega_{T}$ with $\rho\in(0,\,1)$.
    Then, we have
    \begin{equation}\label{Eq (Section 2): Continuity of Norm}
        \left\lvert \lvert \bz\rvert_{\bg(x,\,t)}-\lvert \bz\rvert_{\bg(x_{0},\,t_{0})}\right\rvert\le c_{\dagger}\rho\lvert \bz\rvert,
    \end{equation}
    \begin{equation}\label{Eq (Section 2): Continuity of A-epsilon in x}
        \left\lvert \left\langle {\mathbf A}_{\varepsilon}(x,\,t,\,\bz)\mathrel{}\middle|\mathrel{}\bxi \right\rangle_{\bg(x,\,t)}-\left\langle {\mathbf A}_{\varepsilon}(x_{0},\,t,\,\bz)\mathrel{}\middle|\mathrel{}\bxi \right\rangle_{\bg(x_{0},\,t)} \right\rvert\le C\left(1+h_{p,\,\varepsilon}(\bz)\right)\rho\lvert \bxi\rvert
    \end{equation}
    for all $\bxi,\,\bz\in{\mathbb R}^{Nn}$.
 \end{lemma}
 \begin{proof}
    We use (\ref{Eq (Section 1): Matrix Gamma}) and (\ref{Eq (Section 1): Bound Assumptions for Coefficients}) to get
    \begin{align*}
        \lvert \bz\rvert_{\bg(x,\,t)}\left\lvert \lvert \bz\rvert_{\bg(x,\,t)}-\lvert \bz\rvert_{\bg(x_{0},\,t_{0})}\right\rvert
        &\le\left(\lvert \bz\rvert_{\bg(x,\,t)}+\lvert \bz\rvert_{\bg(x_{0},\,t_{0})}\right) \left\lvert \lvert \bz\rvert_{\bg(x,\,t)}-\lvert \bz\rvert_{\bg(x_{0},\,t_{0})}\right\rvert\\ 
        &=\left\lvert \lvert \bz\rvert_{\bg(x,\,t)}^{2}-\lvert \bz\rvert_{\bg(x_{0},\,t_{0})}^{2} \right\rvert=\left\lvert(\gamma_{\alpha\beta}(x,\,t)-\gamma_{\alpha\beta}(x_{0},\,t_{0}))\zeta_{\alpha}^{j}\zeta_{\beta}^{j}\right\rvert\\ 
        &\le c\rho\lvert \bz\rvert^{2}\le c_{\dagger}\rho\lvert\bz\rvert\lvert \bz\rvert_{\bg(x,\,t)},
    \end{align*}
    which implies (\ref{Eq (Section 2): Continuity of Norm}).
    Since the left-hand-side of (\ref{Eq (Section 2): Continuity of A-epsilon in x}) is equal to
    \[\left\lvert \sum_{s=1,\,p}\sum_{j=1}^{N}\sum_{\alpha,\,\beta=1}^{n}(a_{s}(x,\,t)\gamma_{\alpha\beta}(x,\,t)-a_{s}(x_{0},\,t)\gamma_{\alpha\beta}(x_{0},\,t))g_{s}(\varepsilon^{2}+\lvert\bz\rvert_{\bg}^{2})\zeta_{\alpha}^{j}\xi_{\beta}^{j} \right\rvert,\]
    (\ref{Eq (Section 2): Continuity of A-epsilon in x}) is straightforwardly shown by (\ref{Eq (Section 1): Growth of p-Laplace-type operator})--(\ref{Eq (Section 1): Bound Assumptions for Coefficients}), and (\ref{Eq (Section 1): Growth of 1-Laplace-type operator}).
 \end{proof}
 Thirdly, following \cite[Lemma 2.3]{T-scalar} and using (\ref{Eq (Section 1): Matrix Gamma}), we can straightforwardly show Lemma \ref{Lemma (Section 2): G-p-epsilon}.
 \begin{lemma}\label{Lemma (Section 2): G-p-epsilon}
     Fix $(x,\,t)\in\Omega_{T}$ and $\varepsilon\in(0,\,1)$, and define the bijective vector field ${\mathbf G}_{p,\,\varepsilon}(x,\,t;\,\cdot\,)\colon{\mathbb R}^{Nn}\to{\mathbb R}^{Nn}$ as ${\mathbf G}_{p,\,\varepsilon}(x,\,t;\,\bz)\coloneqq h_{2p,\,\varepsilon}(\bz)\bz$, or shortly ${\mathbf G}_{p,\,\varepsilon}(\bz)$, for $\bz\in{\mathbb R}^{Nn}$, where $h_{2p,\,\varepsilon}$ is given by Lemma \ref{Lemma (Section 2): Ellipticity of Bilinear Forms}.
     Then, there exists a constant $c=c({\mathcal D})\in(0,\,1)$ such that
     \begin{equation}
         \lvert {\mathbf G}_{p,\,\varepsilon}(\bz_{1})-{\mathbf G}_{p,\,\varepsilon}(\bz_{2})\rvert\ge c({\mathcal D})\left(\lvert \bz_{1}\rvert\vee\lvert \bz_{2}\rvert\right)^{p-1} \lvert \bz_{1}-\bz_{2} \rvert
     \end{equation}
     holds for all $\bz_{1},\,\bz_{2}\in{\mathbb R}^{Nn}$. In particular, there also holds
     \begin{equation}
         \lvert {\mathbf G}_{p,\,\varepsilon}^{-1}(\bta) \rvert\le c({\mathcal D})^{-1}\lvert \bta\rvert^{1/p}
     \end{equation}
     for all $\bta\in{\mathbb R}^{Nn}$.
 \end{lemma}
 Finally, we introduce the mapping
 \begin{equation}\label{Eq (Section 2): Def of G-2delta-epsilon}
 \G_{2\delta,\,\varepsilon}(x,\,t;\,\bz)\coloneqq \left(\sqrt{\varepsilon^{2}+\lvert \bz\rvert_{\bg(x,\,t)}^{2}}-2\delta \right)_{+}\frac{\bz}{\lvert \bz\rvert_{\bg(x,\,t)}}\in{\mathbb R}^{Nn}
 \end{equation}
 for $(x,\,t)\in\Omega_{T}$ and $\bz\in{\mathbb R}^{Nn}$, or $\G_{2\delta,\,\varepsilon}(\bz)$ for short.
 We note that the mapping $\G_{\delta,\,\varepsilon}(\bz)$, defined in the same manner, makes sense as long as $\delta>\varepsilon$ holds.
 Without a proof, we infer to Lemma \ref{Lemma (Section 2): Continuity of G-2delta-epsilon}, which is shown completely similarly to \cite[Lemma 2.4]{T-scalar}.
 \begin{lemma}\label{Lemma (Section 2): Continuity of G-2delta-epsilon}
 Let $\delta\in(0,\,1)$ and $\varepsilon\in(0,\,\varsigma\delta)$ for some fixed $\varsigma\in(0,\,2)$.
 Then, there exists a constant $c_{\dagger\dagger}\in(1,\,\infty)$, depending at most on $\varsigma$ and $\gamma_{0}$, such that the mapping $\G_{2\delta,\,\varepsilon}$, defined as (\ref{Eq (Section 2): Def of G-2delta-epsilon}), satisfies
 \[\left\lvert \boldsymbol{\mathcal G}_{2\delta,\,\varepsilon}(x,\,t;\,\bz_{1})-\boldsymbol{\mathcal G}_{2\delta,\,\varepsilon}(x,\,t;\,\bz_{2}) \right\rvert\le c_{\dagger\dagger}\lvert \bz_{1}-\bz_{2}\rvert\]
 for any $(x,\,t)\in \Omega_{T}$, and $\bz_{1},\,\bz_{2}\in{\mathbb R}^{Nn}$.
 \end{lemma}
 
 \subsection{Composite mappings}
 Throughout this paper, let $\psi\colon {\mathbb R}_{\ge 0}\to{\mathbb R}_{\ge 0}$ be a globally Lipschitz function that is non-decreasing and continuously differentiable except at finitely many points. 
 For this $\psi$, we consider a convex composite function $\Psi\colon {\mathbb R}_{\ge 0}\to{\mathbb R}_{\ge 0}$ of the form
 \begin{equation}\label{Eq (Section 2): Def of Psi}
    \Psi(\sigma)\coloneqq \int_{0}^{\sigma}\tau\psi(\tau)\,{\mathrm d}\tau+C\quad \text{for }\sigma\in{\mathbb R}_{\ge 0}
 \end{equation}
 with $C\in{\mathbb R}_{\ge 0}$ denoting the constant of integration.
 In other words, $\Psi=\Psi(\sigma)$ is an antiderivative of the function $\sigma\psi(\sigma)$.
 Since $\psi$ is non-decreasing, it is easy to check that $\Psi$ defined as (\ref{Eq (Section 2): Def of Psi}) satisfies
 \begin{equation}\label{Eq (Section 2): Control of Psi}
     \Psi(\sigma)\le \sigma^{2}\psi(\sigma)\quad \text{for all }\sigma\in{\mathbb R}_{\ge 0},\quad \textrm{provided}\quad C=0.
 \end{equation}
 We list some choices of $\psi$ adopted in this paper as follows.
    \begin{itemize}
    \item For given $k\in(0,\,\infty)$, we choose $\psi_{1,\,k}\coloneqq 2\chi_{(k,\,\infty)}$ by considering a piecewise linear function
    \begin{equation}\label{Eq (Section 2): Approximation}
    \psi_{1,\,k,\,{\widetilde \varepsilon}}(\sigma)\coloneqq \min\left\{\,(\sigma-k)_{+}/\widetilde{\varepsilon},\,1\,\right\}  \quad \text{for small }\widetilde{\varepsilon}>0,
    \end{equation}
    and letting $\widetilde{\varepsilon}\to 0$.
    The corresponding composite function $\Psi_{1,\,k}$ is of the form $\Psi_{1,\,k}(\sigma)\coloneqq (\sigma^{2}-k^{2})_{+}+C$.
    \item For given $k\in(0,\,\infty)$, we choose $\psi_{2,\,k}(\sigma)\coloneqq 2(1-k/\sigma)_{+}$, so that the corresponding composite is $\Psi_{2,\,k}(\sigma)\coloneqq (\sigma-k)_{+}^{2}+C$. 
    For this choice, we easily check the following identity;
    \begin{equation}\label{Eq (Section 2): psi-2}
        \psi_{2,\,k}(\sigma)+\psi_{2,\,k}^{\prime}(\sigma)\sigma=2\chi_{\{\sigma>k\}}(\sigma)\quad \text{for all }\sigma\in{\mathbb R}_{>0}.
    \end{equation}
    Here, for a measurable set $A\subset {\mathbb R}_{\ge 0}$, $\chi_{A}\colon {\mathbb R}_{\ge 0}\to \{\,0,\,1\,\}$ is the characteristic function of $A$.
    \item For given $m\in\lbrack 0,\,\infty)$ and $l\in(0,\,\infty)$, we set $\psi_{3,\,m,\,l}(\sigma)\coloneqq (\sigma\wedge l)^{m}$. 
    Then, we have  
    \begin{equation}\label{Eq (Section 2): Psi-3}
        \frac{(\sigma\wedge l)^{m+2}}{m+2}\le \Psi_{3,\,m,\,l}(\sigma)\coloneqq \int_{0}^{\sigma}\tau\psi_{3,\,m,\,l}(\tau)\,{\mathrm d}\tau\le \frac{\sigma^{m+2}}{m+2}\quad \textrm{for all }\sigma\in{\mathbb R}_{\ge 0}.
    \end{equation}
    As $l\to\infty$, $\Psi_{3,\,m,\,l}(\sigma)$ monotonically converges to the right-hand side of (\ref{Eq (Section 2): Psi-3}).
    \item For given $m\in\lbrack 0,\,\infty)$ and $l\in(1,\,\infty)$, we set $\psi_{4,\,m,\,l}(\sigma)\coloneqq (1-1/\sigma)_{+}\sigma^{m}\wedge (1-1/l)_{+}l^{m}$, and often consider the limit $\psi_{4,\,m}(\sigma)\coloneqq \lim\limits_{l\to\infty}\psi_{4,\,m,\,l}(\sigma)=(1-1/\sigma)_{+}\sigma^{m}$.
    Then, for every $\sigma\in{\mathbb R}_{\ge 0}$, we have 
    \begin{equation}\label{Eq (Section 2): Psi-4}
        \Psi_{4,\,m,\,l}(\sigma)\coloneqq \int_{0}^{\sigma}\tau\psi_{4,\,m,\,l}(\tau)\,\d\tau\uparrow\int_{0}^{\sigma}\tau\psi_{4,\,m}(\tau)\,\d\tau\ge \frac{(\sigma-1)_{+}^{m+2}}{m+2}
    \end{equation}
    as $l\to \infty$.
    Moreover, the function $\psi_{4,\,m,\,l}$ satisfies
    \begin{equation}\label{Eq (Section 2): Growth of psi-4}
        \psi_{4,\,m,\,l}(\sigma)+\psi_{4,\,m,\,l}^{\prime}(\sigma)\sigma\le (m+1)\sigma^{m}\quad \text{for all }\sigma\in{\mathbb R}_{>0},
    \end{equation}
    which is easy to check by direct computations.
   \end{itemize}
   These choices are to appear in the proof of various parabolic regularity estimates of this paper.
 \subsection{Basic lemmata for parabolic regularity}
 This subsection provides basic lemmata that are used to prove parabolic regularity.

 Lemma \ref{Lemma (Section 2): Interpolation Abosorbing lemma} is a well-known lemma, the proof of which was first given in \cite[Lemma 1.1]{Giaquinta-Giusti}.
 \begin{lemma}\label{Lemma (Section 2): Interpolation Abosorbing lemma}
     Fix a bounded closed interval $\lbrack R_{1},\,R_{2}\rbrack\subset{\mathbb R}$.
     Let $F\colon \lbrack R_{1},\,R_{2}\rbrack\to \lbrack 0,\,\infty)$ be a non-decreasing and bounded function that satisfies
     \[F(r_{1})\le \theta F(r_{2})+\left[\frac{A}{(r_{2}-r_{1})^{m}}+B \right]\quad \text{for any}\quad R_{1}\le r_{1}<r_{2}\le R_{2},\]
     where $A,\,B,\,m\in(0,\,\infty)$, and $\theta\in(0,\,1)$ are constant.
     Then, there exists a constant $C=C(m,\,\theta)\in(0,\,\infty)$ such that  
     \[F(R_{1})\le C\left[\frac{A}{(R_{2}-R_{1})^{m}}+B\right].\]
 \end{lemma}
 Lemma \ref{Lemma (Section 2): Moser Iteration Lemma} is no more than a short-cut lemma to easily deduce $L^{\infty}$-estimates by Moser's iteration. 
 Although the proof is provided in \cite[Lemma 4.2]{T-supercritical}, the basic computations therein are naturally found when one carries out Moser's iteration (see e.g., \cite[Lemma 2.3]{BDLS-boundary}, \cite[Theorem 1.2]{BDLS-parabolic}).
 \begin{lemma}\label{Lemma (Section 2): Moser Iteration Lemma}
    Fix $A,\,B,\,\kappa\in(1,\,\infty)$, and $\mu\in(0,\,\infty)$.
    Let the sequences $\{p_{l}\}_{l=0}^{\infty}\subset (0,\,\infty),\,\{Y_{l}\}_{l=0}^{\infty}\subset \lbrack0,\,\infty)$ satisfy
    \(Y_{l+1}^{p_{l+1}}\le \left(AB^{l}Y_{l}^{p_{l}}\right)^{\kappa}\), \(p_{l}\ge \mu(\kappa^{l}-1)\)
    for all $l\in{\mathbb Z}_{\ge 0}$, and $\kappa^{-l}p_{l}\to \mu$ as $l\to\infty$.
    Then, we have \(\limsup\limits_{l\to\infty}Y_{l}\le A^{\frac{\kappa^{\prime}}{\mu}}B^{\frac{(\kappa^{\prime})^{2}}{\mu}}Y_{0}^{\frac{p_{0}}{\mu}}\).
 \end{lemma}
 We infer a well-known lemma (see \cite[Chapter II, Lemma 5.7]{LSU MR0241822} for the proof).
 \begin{lemma}\label{Lemma (Section 2): Geometric Decay Lemma}
     Let the sequence $\{Y_{l}\}_{l=0}^{\infty},\,\{Z_{l}\}_{l=0}^{\infty}\subset \lbrack 0,\,\infty)$ satisfy the following recursive inequalities;
     \[Y_{l+1}\le AB^{l}\left(Y_{l}^{1+\upsilon}+Y_{l}^{\upsilon}Z_{l}^{1+\varkappa} \right),\quad Z_{l+1}\le AB^{l}\left( Y_{l}+Z_{l}^{1+\varkappa} \right)\quad \text{for all }l\in{\mathbb Z}_{\ge 0},\]
     where $B\in(1,\,\infty)$ and $A,\,\upsilon,\,\varkappa\in(0,\,\infty)$ are constants. If both $Y_{0}\le \varTheta$ and $Z_{0}\le \varTheta^{1/(1+\varkappa)}$ hold with
     \[\varpi\coloneqq \min\left\{\,\upsilon,\,\frac{\varkappa}{1+\varkappa} \, \right\},\quad \text{and}\quad\varTheta\coloneqq\min\left\{\,(2A)^{-\upsilon^{-1}}B^{-(\upsilon\varpi)^{-1}},\, (2A)^{-(1+\varkappa)\varkappa^{-1}}B^{-(\varkappa\varpi)^{-1}}\, \right\},\]
     then $Y_{l}\to 0$ and $Z_{l}\to 0$ as $l\to\infty$.
 \end{lemma}
 As well as Lemma \ref{Lemma (Section 2): Geometric Decay Lemma}, we need Lemma \ref{Lemma (Section 2): Dimensionless Lemma} in showing the De Giorgi-type oscillation lemma.
 \begin{lemma}\label{Lemma (Section 2): Dimensionless Lemma}
     Let $B_{\rho}(x_{0})\subset {\mathbb R}^{n}$, $I_{\rho}(\gamma;\,t_{0})\coloneqq (t_{0}-\gamma\rho^{2},\,t_{0}\rbrack\subset {\mathbb R}$ for some $\gamma\in(0,\,1\rbrack$, $x_{0}\in{\mathbb R}^{n}$, $t_{0}\in{\mathbb R}$, and $\rho\in(0,\,\infty)$.
     For a non-negative measurable function $\varphi\colon Q_{\rho}(\gamma;\,x_{0},\,t_{0})\coloneqq B_{\rho}(x_{0})\times I_{\rho}(\gamma;\,t_{0})\to {\mathbb R}_{\ge 0}$, we define $A\colon I_{\rho}(\gamma;\,t_{0})\to{\mathbb R}_{\ge 0}$ as $A(t)\coloneqq \lvert \{x\in B_{\rho}(x_{0})\mid \varphi(x,\,t)>0\}\rvert$ for $t\in I_{\rho}(\gamma;\,t_{0})$. 
     Let $(q,\,r)$ satisfy (\ref{Eq (Section 1): Condition for q,r}), and define $c_{q,\,r} \coloneqq 1/{\widehat q}-1/{\widehat r}\in\lbrack -1,\,1\rbrack$, where $({\widehat q},\,{\widehat r})$ is defined as (\ref{Eq (Section 1): Exponents beta q-hat r-hat}).
     Then, we have the following;
     \begin{itemize}
        \item The dimensionless quantities
        \[Y\coloneqq \frac{\lVert A \rVert_{L^{1}(I_{\rho}(\gamma;\,t_{0}) )}}{\lvert Q_{\rho}(\gamma;\,x_{0},\,t_{0})\rvert}\in\lbrack 0,\,1\rbrack,\quad Z\coloneqq \gamma^{\frac{c_{q,\,r}n}{n+2}}\frac{\left\lVert A^{1/{\widehat q}} \right\rVert_{L^{{\widehat r}}(I_{\rho}(\gamma;\,t_{0}))}}{\lvert Q_{\rho}(\gamma;\,x_{0},\,t_{0})\rvert^{\frac{n+2\beta}{n+2}}}\in\lbrack 0,\,1\rbrack\]
        admit a constant $C=C(n,\,q,\,r)\in(0,\,\infty)$ such that 
        \begin{equation}\label{Eq (Section 2): Dimensionless Comparisons}
           Y\le C Z^{\min\{\,{\widehat q},\,{\widehat r}\,\}},\quad Z\le CY^{\min\{\,1/{\widehat q},\, 1/{\widehat r}\,\}}.
        \end{equation}     
         \item If $\varphi\in L^{2,\,\infty}(I_{\rho}(\gamma;\,t_{0})\times B_{\rho}(x_{0}))\cap L^{2}(I_{\rho}(\gamma;\,t_{0});\,W_{0}^{1,\,2}(B_{\rho}(x_{0})))$, then $\varphi\in L^{2{\widehat q},\,2{\widehat r}}(Q_{\rho}(\gamma;\,x_{0},\,t_{0}))$. Moreover, we have
         \begin{align}\label{Eq (Section 2): Embedding for level set}
            &\lVert \varphi \rVert_{L^{2{\widehat q},\,2{\widehat r}}(Q_{\rho}(\gamma;\,x_{0},\,t_{0}))}^{2}\nonumber\\ &\le C(n,\,q,\,r)\lVert A^{1/{\widehat q}}\rVert_{L^{{\widehat r}}(I_{\rho}(\gamma;\,t_{0}))}^{\frac{2\beta}{n+2\beta}}\left(\lVert \varphi\rVert_{L^{2,\,\infty}(Q_{\rho}(\gamma;\,x_{0},\,t_{0}) )}^{2}+\lVert\nabla\varphi \rVert_{L^{2}(Q_{\rho}(\gamma;\,x_{0},\,t_{0}))}^{2} \right).
         \end{align}
     \end{itemize}
 \end{lemma}
 As related topics to Lemma \ref{Lemma (Section 2): Dimensionless Lemma}, we refer to \cite[Chapter II, \S 3 \& 7]{LSU MR0241822}. 
 In particular, (\ref{Eq (Section 2): Embedding for level set}) is found as a special case of \cite[Chapter II, (3.6)]{LSU MR0241822}, and (\ref{Eq (Section 2): Dimensionless Comparisons}) is a dimensionless version of \cite[Chapter II, (7.9)]{LSU MR0241822}. 
 For the reader's convenience, we provide the proof of (\ref{Eq (Section 2): Dimensionless Comparisons}).
 \begin{proof}
     By the definition of $c_{q,\,r}$, the identity 
     \begin{equation}\label{Eq (Section 2): c-q-r}
         \frac{n+2\beta}{n+2}=\frac{1}{\widehat{r}}+\frac{c_{q,\,r}n}{n+2}
     \end{equation}
     holds. From (\ref{Eq (Section 2): c-q-r}), we easily check $0\le Y,\,Z\le 1$.
     When $\widehat{q}\le \widehat{r}$, we use H\"{o}lder's inequality to get
     \[\int_{I_{\rho}(\gamma;\,t_{0})}A(t)\,\d t\le \lvert I_{\rho}(\gamma;\,t_{0}) \rvert^{1-{\widehat q}/{\widehat r}}\left(\int_{I_{\rho}(\gamma;\,t_{0})} A^{{\widehat r}/{\widehat q}}(t)\,\d t\right)^{{\widehat q}/{\widehat r}}=\left(\gamma\rho^{2}\right)^{1-{\widehat q}/{\widehat r}}\left(\lvert Q_{\rho}(\gamma,\,x_{0},\,t_{0}) \rvert^{\frac{n+2\beta}{n+2}} Z \right)^{{\widehat q}}.\]
     Dividing this inequality by $\lvert Q_{\rho}(\gamma;\,x_{0},\,t_{0})\rvert=c(n)\gamma\rho^{n+2}$, and noting (\ref{Eq (Section 2): c-q-r}), we have $Y\le C(n,\,q,\,r)Z^{\widehat{q}}$.
     Keeping $A(t)\le c(n)\rho^{n}$ and (\ref{Eq (Section 2): c-q-r}) in mind, we also compute
     \begin{align*}
        \gamma^{\frac{c_{q,\,r}n}{n+2}}\left(\int_{I_{\rho}(\gamma;\,t_{0})}A(t)^{\widehat{r}/\widehat{q}}\,\d t\right)^{1/\widehat{r}}
        &\le \gamma^{\frac{c_{q,\,r}n}{n+2}}\left[\left(c(n)\rho^{n}\right)^{1/\widehat{q}-1/\widehat{r}}\int_{I_{\rho}(\gamma;\,t_{0})}A(t)\,\d t\right]^{1/\widehat{r}}\\  
        &=c(n,\,q,\,r)\gamma^{\frac{n+2\beta}{n+2}}\rho^{n+2\beta} Y^{1/\widehat{r}}.
     \end{align*}
     Dividing this inequality by $\lvert Q_{\rho}(\gamma;\,x_{0},\,t_{0})\rvert^{\frac{n+2\beta}{n+2}}=(c(n)\gamma)^{\frac{n+2\beta}{n+2}}\rho^{n+2\beta}$, we have $Z\le C(n,\,q,\,r)Y^{1/\widehat{r}}$.
     The remaining case $\widehat{r}\le \widehat{q}$ is similarly shown.
 \end{proof}

 Lemma \ref{Lemma (Section 2): Campanato Integral Growth Lemma}, concerned with Campanato spaces, is easily shown by straightforward computations.
 \begin{lemma}\label{Lemma (Section 2): Campanato Integral Growth Lemma}
     Let $\bv\in L^{2}(Q_{R}(x_{0},\,t_{0});\,{\mathbb R}^{k})$ admit the constants $A\in(0,\,\infty),\,\beta\in(0,\,1)$ such that
     \[\fiint_{Q_{\tau R}(x_{0},\,t_{0})}\lvert \bv-(\bv)_{Q_{\tau R}(x_{0},\,t_{0})}\rvert^{2}\,\d x\d t\le A\tau^{2\beta}\quad \text{for all $\tau \in (0,\,1\rbrack$.}\]
     Then, the limit $\V(x_{0},\,t_{0})\coloneqq \lim\limits_{\tau \to 0}(\bv)_{Q_{\tau R}(x_{0},\,t_{0})}\in{\mathbb R}^{k}$ exists.
     Moreover, there exists a constant $c_{\dagger\dagger\dagger}=c_{\dagger\dagger\dagger}(\beta,\,n)$ such that
     \[
         \fiint_{Q_{\tau R}(x_{0},\,t_{0})}\lvert \bv-\V(x_{0},\,t_{0})  \rvert^{2}\,\d x\d t\le c_{\dagger\dagger\dagger}A\tau ^{2\beta}\quad \text{for all }\tau \in(0,\,1\rbrack.
     \]
 \end{lemma}
 We often use Lemma \ref{Lemma (Section 2): Average = L-2 minimizer}, which is easily proved by direct computations.
 \begin{lemma}\label{Lemma (Section 2): Average = L-2 minimizer}
     Let $k,\,m\in{\mathbb N}$.
     Fix a $k$-dimensional Lebesgue measurable set $X\subset {\mathbb R}^{k}$ with finite measure. 
     For any ${\mathbb R}^{m}$-valued measurable function $\bv=\bv(z)$ in the class $L^{2}(X;\,{\mathbb R}^{m})$, we have 
     \[
         \fint_{X}\lvert \bv-(\bv)_{U}\rvert^{2}\,\d z=\min_{\bxi\in{\mathbb R}^{m}}\fint_{X}\lvert\bv-\bxi \rvert^{2} \,\d z.
     \]
 \end{lemma}

 Finally, we recall the Poincar\'{e}--Sobolev inequality for parabolic function spaces.
 \begin{lemma}\label{Lemma (Secion 2): Interpolation among parabolic function spaces}
     Fix $s\in(1,\,\infty)$, and choose $\kappa=\kappa_{s}\in(1,\,2)$ as $\kappa_{s}\coloneqq 1+s/n$ when $s<n$, and otherwise as an arbitrary exponent in $(1,\,2)$.
     For arbitrary scalar-valued functions $\varphi_{1}\in L^{s,\,\infty}(\Omega_{T})$, $\varphi_{2}\in L^{s}(0,\,T;\,W_{0}^{1,\,s}(\Omega))$, we have the following;
     \begin{itemize}
         \item There exists a constant $C=C(n,\,s,\,\kappa,\,\Omega)\in(0,\,\infty)$ such that 
         \begin{equation}\label{Eq (Section 2): Parabolic Sobolev}
             \iint_{\Omega_{T}}\lvert\varphi_{1}\rvert^{(\kappa-1)s}\lvert \varphi_{2}\rvert^{s}\,{\mathrm d}x{\mathrm d}t\le C\lVert \varphi_{1} \rVert_{L^{s,\,\infty}(\Omega_{T})}^{s(\kappa-1)}\lVert \nabla\varphi_{2} \rVert_{L^{s}(\Omega_{T})}^{s}
         \end{equation}
         \item Let the exponents $(\pi_{1},\,\pi_{2})\in (n/s,\,\infty\rbrack\times (1,\,\infty\rbrack$ satisfy
         \[
             \frac{1}{(\kappa_{s}-1)\pi_{1}}+\frac{1}{\pi_{2}}<1, \text{ and set } \pi_{3}\coloneqq \left(1-\frac{1}{(\kappa_{s}-1)\pi_{1}}-\frac{1}{\pi_{2}} \right)^{-1}-1\in(0,\,\infty).
         \]
         If $\lvert \varphi_{1}\rvert\le \lvert\varphi_{2}\rvert$ holds a.e.~in $\Omega_{T}$, then we have 
         \begin{equation}\label{Eq (Section 2): Parabolic absorbing}
             \lVert \varphi_{1} \rVert_{L^{s\pi_{1}^{\prime},\,s\pi_{2}^{\prime}}(\Omega_{T})}^{s}\le \sigma\left(\lVert \varphi_{1} \rVert_{L^{s,\,\infty}(\Omega_{T})}^{s}+\lVert \nabla\varphi_{2} \rVert_{L^{s}(\Omega_{T})}^{s} \right)+C\sigma^{-\pi_{3}}\lVert \varphi_{1} \rVert_{L^{s}(\Omega_{T})}^{s}
         \end{equation}
         for any $\sigma\in(0,\,\infty)$ with $C=C(n,\,s,\,\pi_{1},\,\pi_{2},\,\Omega)$.
        \end{itemize}
    \end{lemma}
    The inequality (\ref{Eq (Section 2): Parabolic Sobolev}) is easy to prove by H\"{o}lder's inequality and the continuous embedding $W_{0}^{1,\,s}(\Omega)\hookrightarrow L^{\frac{s}{2-\kappa}}(\Omega)$.
    In showing local or global bounds of solutions, we use (\ref{Eq (Section 2): Parabolic absorbing}) to treat force terms in the class $L^{r}(0,\,T;\,L^{q}(\Omega))$.
    For the reader's convenience, we provide the proof of (\ref{Eq (Section 2): Parabolic absorbing}).
    \begin{proof}
        Choose $\theta\coloneqq (\pi_{1}^{\prime}-1)\cdot\frac{2-\kappa}{\kappa-1}\in (0,\,1)$, which satisfies $1=\frac{1-\theta}{\pi_{1}^{\prime}}+\frac{\theta}{\pi_{1}^{\prime}(2-\kappa)}$.
        We note that the exponents $(c_{1},\,c_{2},\,c_{3})\coloneqq \left(\frac{1}{\pi_{2}},\,\frac{1}{(\kappa-1)\pi_{1}},\,\frac{1}{\pi_{3}+1} \right)\in(0,\,1)^{3}$ satisfy
        \(c_{1}+c_{3}=\frac{1-\theta}{\pi_{1}^{\prime}}\), \(c_{2}\pi_{1}^{\prime}=\frac{\theta}{2-\kappa}\), and \(c_{2}+c_{3}=\frac{1}{\pi_{2}^{\prime}}\).
        We use the first and the second identities to compute
        \[\lVert \varphi_{1}(t) \rVert_{L^{s\pi_{1}^{\prime}}(\Omega)}^{s\pi_{2}^{\prime}}\le C\lVert \varphi_{1}(t) \rVert_{L^{s}(\Omega)}^{(1-\theta)s\cdot \frac{\pi_{2}^{\prime}}{\pi_{1}^{\prime}}}\lVert\nabla\varphi_{2}(t) \rVert_{L^{s}(\Omega_{T})}^{\frac{s\theta}{2-\kappa}\cdot \frac{\pi_{2}^{\prime}}{\pi_{1}^{\prime}}}\le C\lVert \varphi_{1}(t) \rVert_{L^{s}(\Omega)}^{(c_{1}+c_{3})s\pi_{2}^{\prime}}\lVert\nabla\varphi_{2}(t) \rVert_{L^{s}(\Omega_{T})}^{c_{2}s\pi_{2}^{\prime}},\]
        where we note the continuous embedding $W_{0}^{1,\,s}(\Omega)\hookrightarrow L^{\frac{s}{2-\kappa}}(\Omega)$, and the interpolation among $L^{s}(\Omega)\subset L^{s\pi_{1}^{\prime}}(\Omega)\subset L^{\frac{s}{2-\kappa}}(\Omega)$.
        By the third identity and H\"{o}lder's inequality in the time variable, we get $\lVert \varphi \rVert_{L^{s\pi_{1}^{\prime},\,s\pi_{2}^{\prime}}(\Omega_{T})}^{s}\le C\lVert \varphi_{1}\rVert_{L^{s,\,\infty}(\Omega_{T})}^{sc_{1}}\lVert \nabla\varphi_{2}\rVert_{L^{s}(\Omega_{T})}^{sc_{2}} \lVert \varphi_{1}\rVert_{L^{s}(\Omega_{T})}^{sc_{3}}$.
        Since \(c_{1}+c_{2}+c_{3}=1\) holds, we use Young's inequality to deduce (\ref{Eq (Section 2): Parabolic absorbing}).
    \end{proof}
    \begin{remark}\label{Remark (Section 2): Scaling Constant} \upshape
        The constant $C$ found in (\ref{Eq (Section 2): Parabolic Sobolev})--(\ref{Eq (Section 2): Parabolic absorbing}) depends on the best possible constant $S=S(n,\,s,\,\kappa_{s},\,\Omega)$ satisfying $\lVert \varphi \rVert_{L^{\frac{s}{2-\kappa_{s}}}}\le S\lVert \nabla\varphi\rVert_{L^{s}(\Omega)}$ for all $\varphi\in W_{0}^{1,\,s}(\Omega)$.
        In the critical case $n=s$, the constant $C$ found in (\ref{Eq (Section 2): Parabolic Sobolev}) depends on the diameter of $\Omega$.
        In particular, when $s=n=2$ and the domain $\Omega\subset{\mathbb R}^{2}$ is an open ball $B_{\rho}\subset {\mathbb R}^{2}$ with its radius $\rho\in(0,\,1\rbrack$, a standard scaling argument implies (\ref{Eq (Section 2): Parabolic Sobolev}) with $C=c(\kappa,\,B_{1})\rho^{2\widetilde\kappa}$, where ${\widetilde\kappa}\coloneqq 2-\kappa_{s}\in\lbrack 0,\,1)$.

        It is worth noting that in the special case $\varphi_{1}=\varphi_{2}$, we have a parabolic Poincar\'{e} inequality that is valid for each $s\in(1,\,\infty)$. 
        More precisely, for $\varphi\in L^{s}(I;\,W_{0}^{1,\,s}(B))\cap L^{s,\,\infty}(B\times I)$ with $s\in(1,\,\infty)$, where $I\subset {\mathbb R}$ and $B\subset {\mathbb R}^{n}$ are respectively a bounded open interval and an open ball, there holds
        \begin{equation}\label{Eq (Section 2): PS-Ineq for any s}
            \iint_{B\times I}\lvert \varphi \rvert^{s+\frac{s^{2}}{n}}\,{\mathrm d}x{\mathrm d}t\le C(n,\,s)\lVert \varphi \rVert_{L^{s,\,\infty}(B\times I)}^{\frac{s^{2}}{n}}\lVert \nabla\varphi\rVert_{L^{s}(B\times I)}^{s}.
        \end{equation}
        In fact, by H\"{o}lder's inequality and the Sobolev embedding $W_{0}^{1,\,1}(B)\hookrightarrow L^{n^{\prime}}(B)$, we have 
        \begin{align*}
            \int_{B}\lvert \phi \rvert^{s+\frac{s^{2}}{n}}\,\d x&\le \left(\int_{B}\lvert\phi\rvert^{s}\,\d x\right)^{1/n}\left(\int_{B}\lvert \phi\rvert^{n^{\prime}\cdot \frac{s}{n}(n+s-1)}\,\d x \right)^{1/n^{\prime}}\\ 
            &\le C(n)\left(\int_{B}\lvert\phi\rvert^{s}\,\d x\right)^{1/n}\left(\int_{B}\left\lvert \nabla\lvert \phi\rvert^{\frac{s}{n}(n+s-1)} \right\rvert\,\d x \right)\\ 
            &\le C(n,\,s)\left(\int_{B}\lvert\phi\rvert^{s}\,\d x\right)^{1/n}\left(\int_{B}\lvert \nabla\phi \rvert^{s}\,\d x \right)^{1/s}\left(\int_{B}\lvert \phi\rvert^{s+\frac{s^{2}}{n}}\,\d x \right)^{1-1/s}
        \end{align*}
        for any $\phi\in C_{\mathrm c}^{1}(B)$.
        From this, (\ref{Eq (Section 2): PS-Ineq for any s}) is easily deduced.
    \end{remark}

\section{Boundedness of a solution for $p\in(1,\,2)$}\label{Section: L-infty}
 In Section \ref{Section: L-infty}, we consider $p\in(1,\,2)$.
 Instead of (\ref{Eq (Section 1): Condition for q,r}), we assume a weaker condition
 \begin{equation}\label{Eq (Section 3): Weak conditions for q,r}
    \frac{n}{pq}+\frac{1}{r}<1.
 \end{equation}

 \subsection{A weak maximum principle}
 We prove a weak maximum principle for (\ref{Eq (Section 2): Approximate System})--(\ref{Eq (Section 2): Parabolic Dirichlet Boundary}).
 \begin{proposition}\label{Proposition (Section 3): A Weak Maximum Principle}
     Fix $p\in(1,\,2)$.
     Assume that $\buf_{\varepsilon}\in L^{\infty}(\Omega_{T})^{N}$, and the pair $(q,\,r)\in(n/p,\,\infty\rbrack\times (1,\,\infty\rbrack$ satisfies (\ref{Eq (Section 3): Weak conditions for q,r}).
     Let $\bu_{\varepsilon}\in \bu_{\star}+X_{0}^{p}(0,\,T;\,\Omega)^{N}$ be the weak solution of (\ref{Eq (Section 2): Approximate System})--(\ref{Eq (Section 2): Parabolic Dirichlet Boundary}) with $\bu_{\star}\in L^{\infty}(\Omega_{T})$.
     Then, $\bu_{\varepsilon}\in L^{\infty}(\Omega_{T})^{N}$. Moreover, for any $\pi\in(2-p,\,\infty)$, there exists a constant $C=C({\mathcal D},\,\pi,\,\Omega,\,T)\in(1,\,\infty)$ such that the following estimate holds;
     \[\esssup_{\Omega_{T}}\lvert \bu_{\varepsilon}\rvert\le C\left(\lVert \bu_{\varepsilon} \rVert_{L^{\pi}(\Omega_{T})}^{\frac{\pi}{\pi-2+p}}+ \left[\lVert \bu_{\star} \rVert_{L^{\infty}(\Omega)}+\lVert \buf_{\varepsilon}\rVert_{L^{q,\,r}(\Omega_{T})}^{p-1}+1\right]^{\frac{p\pi}{\pi-2+p}} \right).\]
 \end{proposition}
 \begin{proof}
 We set $k\coloneqq 1+M_{0}+\lVert \buf_{\varepsilon}\rVert_{L^{q,\,r}(\Omega_{T})}^{p-1}\in\lbrack 1,\,\infty)$.
 It suffices to prove that ${\widehat U}_{k}\coloneqq (\lvert \bu_{\varepsilon}\rvert-k)_{+}$ satisfies
 \[
    \esssup_{\Omega_{T}}{\widehat U}_{k}\le C({\mathcal D},\,\pi,\,\Omega,\,T)\left(\lVert {\widehat U}_{k} \rVert_{L^{\pi}(\Omega_{T})}^{\frac{\pi}{\pi-2+p}}+k^{\frac{p\pi}{\pi-2+p}} \right)
 \]
 for any $\pi\in(2-p,\,\infty)$. 
 We introduce a parameter $\alpha\in\lbrack 0,\,\infty)$, and test $\bphi\coloneqq \zeta\psi_{1,\,k,\,{\widetilde \varepsilon}}(\lvert\bu_{\varepsilon}\rvert)\bu_{\varepsilon}$
 into (\ref{Eq (Section 2): Approximate Weak Form}) with $\zeta\coloneqq {\widehat U}_{k,\,l}^{p\alpha}\phi(t)$, where $\psi_{1,\,k,\,\widetilde{\varepsilon}}$ is defined as (\ref{Eq (Section 2): Approximation}), $\widehat{U}_{k,\,l}\coloneqq \widehat{U}_{k}\wedge l$, and $\phi\colon\lbrack 0,\,T\rbrack\to \lbrack 0,\,1\rbrack$ is a non-increasing Lipschitz function satisfying $\phi(T)=0$. 
 We note that $\lvert \varphi^{j}\rvert\le l^{p\alpha}\phi(t)(\lvert \bu_{\varepsilon}\rvert-k)_{+}\in L^{p}(0,\,T;\,W_{0}^{1,\,p}(\Omega))$ holds for every $j\in\{\,1,\,\dots\,,\,N\,\}$, and hence $\bphi\in L^{p}(0,\,T;\,V_{0})$.
 The non-negative integral $\iint_{\Omega_{T}}a_{s}g_{s}(v_{\varepsilon}^{2})\langle \nabla \lvert \bu_{\varepsilon}\rvert^{2},\,\nabla \lvert \bu_{\varepsilon}\rvert \rangle_{\bg} \psi_{1,\,k,\,\widetilde{\varepsilon}}^{\prime}(\lvert \bu_{\varepsilon}\rvert)\,\d x\d t$ may be discarded, and letting ${\widetilde \varepsilon}\to 0$ yields
 \begin{align*}
     &\iint_{\Omega_{T}} \partial_{t}{\widetilde U}_{k}\cdot {\widehat U}_{k,\,l}^{p\alpha}\phi \,\d x\d t+\iint_{\Omega_{T}}a_{s}g_{s}(v_{\varepsilon}^{2})\langle \D\bu_{\varepsilon},\,\D\bu_{\varepsilon} \rangle_{\bg}{\widehat U}_{k,\,l}^{p\alpha}\phi\,\d x\d t\\ 
     &\quad +\iint_{\Omega_{T}} a_{s}g_{s}(v_{\varepsilon}^{2})\left\langle \nabla {\widetilde U}_{k}\mathrel{}\middle|\mathrel{} \nabla {\widehat U}_{k,\,l}^{p\alpha} \right\rangle_{\bg}\phi\,\d x\d t\le \iint_{\Omega_{T}}\lvert \buf_{\varepsilon}\rvert (\widehat{U}_{k}+k) {\widehat U}_{k,\,l}^{p\alpha}\phi\,\d x\d t,
 \end{align*}
 where ${\widetilde U}_{k}\coloneqq \left(\lvert \bu_{\varepsilon}\rvert^{2}-k^{2}\right)_{+}$.
 We may delete the third integral on the left-hand side, since it is non-negative.
 We rewrite $\partial_{t}{\widetilde U}_{k}\cdot U_{k,\,l}^{p\alpha}=2({\widehat U}_{k}+k)({\widehat U}_{k}\wedge l)^{p\alpha}\partial_{t}{\widehat U}_{k}=\partial_{t}({\widehat F}_{\alpha,\,l}({\widehat U}_{k}))$ with
 \[{\widehat F}_{\alpha,\,l}(\sigma)\coloneqq 2\int_{0}^{\sigma}(\tau +k)(\tau\wedge l)^{p\alpha}\,\d\tau\ge 2\int_{0}^{\sigma\wedge l}(\tau\wedge l)^{p\alpha+1}\,\d\tau=\frac{2}{p\alpha+2}(\sigma\wedge l)^{p\alpha+2}\]
 for $\sigma\in(0,\,\infty)$.
 By (\ref{Eq (Section 2): Coercivity of A-epsilon}), and the inequality $\lvert \nabla {\widehat U}_{k,\,l}\rvert\le \lvert \D\bu_{\varepsilon}\rvert$, we obtain
 \begin{align*}
    &\esssup_{0<t<T}\int_{\Omega}{\widehat F}_{\alpha,\,l}({\widehat U}_{k})\,\d x+\iint_{\Omega_{T}}\lvert \nabla {\widehat U}_{k,\,l}\rvert^{p}{\widehat U}_{k,\,l}^{p\alpha}\,\d x\d t\\
    &\le C({\mathcal D})\iint_{\Omega_{T}}\left({\widehat U}_{k,\,l}^{p\alpha}+{\widehat h}_{k}({\widehat U}_{k}+k)^{p}{\widehat U}_{k,\,l}^{p\alpha}\right)\,\d x\d t\\ 
    &\le C({\mathcal D})\iint_{\Omega_{T}}\left[\left(1+k^{p}\widehat{h}_{k}\right)\left(1+\widehat{U}_{k,\,l}^{p\alpha+p}\right)+\widehat{h}_{k}\widehat{U}_{k,\,l}^{p\alpha}\widehat{U}_{k}^{p}\right] \,\d x\d t
 \end{align*}
 with ${\widehat h}_{k}\coloneqq ({\widehat U}_{k}+k)^{-(p-1)}\lvert \buf_{\varepsilon} \rvert\in L^{q,\,r}(\Omega_{T})$.
 The inequality $\lVert {\widehat h}_{k}\rVert_{L^{q,\,r}(\Omega_{T})}\le k^{-(p-1)}\lVert \buf_{\varepsilon}\rVert_{L^{q,\,r}(\Omega_{T})}\le 1$ is easy to check by our choice of $k$.
 Hence, we use Young's inequality and H\"{o}lder's inequality to get
 \begin{equation}\label{Eq (Section 3): Before Absorbing}
    \esssup_{0<t<T}\int_{\Omega}{\widehat U}_{k,\,l}^{p\alpha+2}\,\d x+\iint_{\Omega_{T}}\left\lvert \nabla {\widehat U}_{k,\,l}^{1+\alpha} \right\rvert^{p}\,\d x\d t\le C(1+\alpha)^{p}\left[k^{p}+\left\lVert {\widehat U}_{k}{\widehat U}_{k,\,l}^{\alpha}\right\rVert_{L^{pq^{\prime},\,pr^{\prime}}(\Omega_{T})}^{p} \right]
 \end{equation}
 with $C=C({\mathcal D},\,\lvert\Omega\rvert,\,T)\in(0,\,\infty)$.
 Since $\buf_{\varepsilon}\in L^{\infty}(\Omega_{T})^{N}$, we may firstly consider $q=r=\infty$, where we use (\ref{Eq (Section 2): Parabolic Sobolev}) and (\ref{Eq (Section 3): Before Absorbing}) to get
 \[\iint_{\Omega_{T}}{\widehat U}_{k,\,l}^{(\kappa-1)(p\alpha+2)+p(1+\alpha)}\,\d x\d t\le C(1+\alpha)^{p\kappa}\left(k^{p}+\iint_{\Omega_{T}}{\widehat U}_{k}^{p}{\widehat U}_{k,\,l}^{p\alpha}\,\d x\d t \right)^{\kappa}\]
 with $\kappa\coloneqq 1+p/n\in(1,\,2)$.
 By this estimate and ${\widehat U}_{k}\in L^{p}(\Omega_{T})$, we carry out finitely many iterations to prove that ${\widehat U}_{k}\in L^{m}(\Omega_{T})$ holds for any $m\in(1,\,\infty)$.
 We also conclude the inclusions ${\widetilde \varphi}_{1}\coloneqq {\widehat U}_{k}^{\alpha+2/p}\in L^{p,\,\infty}(\Omega_{T})$ and ${\widetilde \varphi}_{2}\coloneqq {\widehat U}_{k}^{\alpha+1}\in L^{p}(0,\,T;\,W_{0}^{1,\,p}(\Omega))$.
 Moreover, letting $l\to\infty$ in (\ref{Eq (Section 3): Before Absorbing}), we have 
 \[\lVert {\widetilde\varphi}_{1}\rVert_{L^{p,\,\infty}(\Omega_{T})}^{p}+\lVert \nabla{\widetilde \varphi}_{2}\rVert_{L^{p}(\Omega_{T})}^{p} \le C(1+\alpha)^{p}\left[k^{p}+\lVert {\widetilde \varphi}_{2}\rVert_{L^{pq^{\prime},\,pr^{\prime}}(\Omega_{T})}^{p}\right].\]
 Noting $\widetilde{\varphi}_{2}\le \widetilde{\varphi}_{1}+1$ by Young's inequality, we apply (\ref{Eq (Section 2): Parabolic absorbing}) with $\varphi_{1}=\varphi_{2}=\widetilde{\varphi}_{2}$, $s\coloneqq p$, and $(\pi_{1},\,\pi_{2})\coloneqq (q,\,r)$.
 Then, by a standard absorbing argument and (\ref{Eq (Section 2): Parabolic Sobolev}), we have 
 \[\iint_{\Omega_{T}}{\widehat U}_{k}^{(\kappa-1)(p\alpha+2)+p(1+\alpha)}\,\d x\d t\le C(1+\alpha)^{\gamma\kappa}\left(\iint_{\Omega_{T}}\left({\widehat U}_{k}^{p\alpha+p}+k^{p}\right)\,\d x\d t\right)^{\kappa}\]
 for some $\gamma=\gamma(n,\,p,\,q)\in\lbrack p,\,\infty)$.
 By adding $\lvert \Omega_{T}\rvert$ into this estimate, and setting $\beta\coloneqq p\alpha+p\in\lbrack p,\,\infty)$, we have
 \[\iint_{\Omega_{T}}\left({\widehat U}_{k}^{\kappa\beta+(\kappa-1)(2-p)}+k^{p}\right)\,{\mathrm d}x{\mathrm d}t\le C({\mathcal D},\,\Omega,\,T)\beta^{\gamma \kappa}\left(\iint_{\Omega_{T}}\left({\widehat U}_{k}^{\beta}+k^{p} \right)\,\d x\d t \right)^{\kappa}\quad \text{for all }\beta\in\lbrack p,\,\infty).\]

 We first consider the case $\pi\in\lbrack p,\,\infty)$.
 For each $l\in{\mathbb Z}_{\ge 0}$, we define $Y_{l}\coloneqq \left(\iint_{\Omega_{T}}({\widehat U}_{k}^{p_{l}}+k^{p})\,\d x\d t \right)^{1/p_{l}}$ with $p_{l}\coloneqq (\pi-2+p)\kappa^{l}+2-p\in\lbrack \pi,\,\infty)\subset \lbrack p,\,\infty)$.
 By the last estimate, we are allowed to apply Lemma \ref{Lemma (Section 2): Moser Iteration Lemma} with $\mu\coloneqq \pi-(2-p)\in(0,\,\infty)$, $A=A({\mathcal D},\,\Omega,\,T)\in(1,\,\infty)$, and $B\coloneqq \kappa^{\gamma\kappa}\in(1,\,\infty)$.
 Finally, we have
 \[\esssup_{\Omega_{T}}{\widehat U}_{k}\le \limsup_{l\to\infty}Y_{l}\le C({\mathcal D},\,\pi,\,\Omega,\,T)\left(\lVert {\widehat U}_{k}\rVert_{L^{\pi}(\Omega_{T})}^{\pi/\mu}+k^{\frac{p\pi}{\pi-2+p}}\right),\]
 whence $\pi\in\lbrack p,\,\infty)$ is completed.
 For $\pi\in(2-p,\,p)$, we use the interpolation among $L^{\pi}(\Omega_{T})\subset L^{2}(\Omega_{T})\subset L^{\infty}(\Omega_{T})$ and Young's inequality to get 
 \[\esssup_{\Omega_{T}}{\widehat U}_{k}\le C\left(\lVert {\widehat U}_{k}\rVert_{L^{2}(\Omega_{T})}^{2/p}+k^{2} \right)\le \frac{1}{2}\esssup_{\Omega_{T}}{\widehat U}_{k}+C({\mathcal D},\,\pi,\,\Omega,\,T)\left(\lVert {\widehat U}_{k} \rVert_{L^{\pi}(\Omega_{T})}^{\frac{\pi}{\pi-2+p}}+k^{\frac{p\pi}{\pi-2+p}} \right),\]
 from which $\pi\in(2-p,\,p)$ is also completed.
 \end{proof}
 \subsection{A priori local bounds for the parabolic $(1,\,p)$-Laplace system}
 We provide local $L^{\infty}$-bounds of weak solutions to (\ref{Eq (Section 1): General System}) for $p\in(1,\,2)$.
 Here we only require $\buf\in L^{2,\,1}(\Omega_{T})^{N}\cap  L^{q,\,r}(\Omega_{T})^{N}$ with $(q,\,r)$ satisfying (\ref{Eq (Section 3): Weak conditions for q,r}), and we do not necessarily assume $\buf\in L^{\infty}(\Omega_{T})^{N}$.
 \begin{proposition}\label{Proposition (Section 3): Local Boundedness for subcritical cases}
    Let $\buf\in L^{2,\,1}(\Omega_{T})^{N} \cap L^{q,\,r}(\Omega_{T})^{N}$ with (\ref{Eq (Section 3): Weak conditions for q,r}), and $\bu$ be a weak solution to (\ref{Eq (Section 1): General System}) with $p\in(1,\,2)$. 
    For $p\in(1,\,p_{\mathrm c})$, let (\ref{Eq (Section 1): Higher Integrability of u for subcritical case}) be also in force.
    Then, for any $Q_{R}\Subset \Omega_{T}$ with $R\in(0,\,1)$, we have 
    \begin{equation}\label{Eq (Section 3): Local Bounds for sup-singular p}
        \esssup_{Q_{R/2}}\lvert\bu \rvert\le C({\mathcal D}) R^{-\frac{n}{2-\varsigma_{\mathrm c}}\cdot \frac{2-p}{p}} \left(1+\lVert \buf\rVert_{L^{q,\,r}(Q_{R})}^{2(p-1)}+\fiint_{Q_{R}}\lvert \bu\rvert^{2}\,\d x\d t \right)^{1/(2-\varsigma_{\mathrm c})}
    \end{equation}
    for $p\in(p_{\mathrm c},\,2)$, and
    \begin{equation}\label{Eq (Section 3): Local Bounds for sub-singular p}
        \esssup_{Q_{R/2}}\lvert\bu \rvert\le C({\mathcal D},\,\varsigma)R^{-\frac{n}{\varsigma-\varsigma_{\mathrm c}}\cdot \frac{2-p}{p}}\left(1+\lVert \buf \rVert_{L^{q,\,r}(Q_{R})}^{\varsigma(p-1)}+\fiint_{Q_{R}}\lvert \bu\rvert^{\varsigma}\,\d x\d t \right)^{1/(\varsigma-\varsigma_{\mathrm c})}
    \end{equation}
    for $p\in(1,\,p_{\mathrm c}\rbrack$.
 \end{proposition}
 \begin{proof}
     We fix $Q_{R}=B_{R}\times I_{R}\Subset \Omega_{T}$, and omit the radius $R$ in the proof for the notational simplicity.
     Choose $k\coloneqq 1+\lVert \buf \rVert_{L^{q,\,r}(Q_{R})}^{p-1}\in\lbrack 1,\,\infty)$, and define $h_{k}\coloneqq k^{-(p-1)}\lvert \buf\rvert\in L^{2,\,1}(Q_{R})\cap L^{q,\,r}(Q_{R})$.
     We prove that the function $U_{k}\coloneqq (\lvert \bu \rvert-k)_{+}+k\ge k$ satisfies the following reversed H\"{o}lder inequality;
     \begin{equation}\label{Eq (Section 3): Reversed Holder for U-k}
         \iint_{Q_{R_{1}}}U_{k}^{\kappa\beta+p-2}\,{\mathrm d}x{\mathrm d}t \le  \left[\frac{C({\mathcal D})\beta^{\gamma}}{(R_{2}-R_{1})^{2}}\iint_{Q_{R_{2}}}U_{k}^{\beta}\,{\mathrm d}x{\mathrm d}t\right]^{\kappa},
     \end{equation}
     provided $U_{k}\in L^{\beta}(Q_{R_{2}})$ with $\beta\in\lbrack 2,\,\infty)$ and $0<R_{1}<R_{2}\le R$. Here $\kappa=\kappa_{p}\coloneqq 1+p/n\in(1,\,2)$, and $\gamma=\gamma(n,\,p,\,q,\,r)\in\lbrack p,\,\infty)$ are constant.
     To prove (\ref{Eq (Section 3): Reversed Holder for U-k}), we test $\varphi\coloneqq \phi_{1,\,k,\,\widetilde{\varepsilon}}(\lvert \bu\rvert)\zeta\bu$ into (\ref{Eq (Section 2): Approximate Weak Form}), where we abbreviate a bounded function $\zeta\coloneqq U_{k,\,l}^{p\alpha}\eta^{p}\phi$ with $\alpha\coloneqq (\beta-2)/p\in\lbrack0,\,\infty)$. 
     Here we note 
     \(\psi_{1,\,k,\,\widetilde{\varepsilon}}(\lvert \bu \rvert)\langle \mathbf{Z}\mid \D\bu\rangle_{\bg}+\psi_{1,\,k,\,\widetilde{\varepsilon}}^{\prime}(\lvert \bu \rvert)\gamma_{\alpha\beta}Z_{\alpha}^{j}u^{j}\partial_{x_{\beta}}\lvert \bu\rvert \ge 0.\)
     Indeed, the first term is $\psi_{1,\,k,\,\widetilde{\varepsilon}}(\lvert \bu \rvert)\lvert \D\bu\rvert_{\bg}\ge 0$, and the second term is $0$ when $\D\bu=0$, and otherwise it is $\psi_{1,\,k,\,\widetilde{\varepsilon}}^{\prime}(\lvert \bu \rvert)\lvert \bu\rvert\lvert \nabla \lvert \bu\rvert\rvert_{\bg}^{2}\lvert \D\bu\rvert_{\bg}^{-1}\ge 0$.
     We also keep in mind that $\partial_{t}\left(\lvert\bu\rvert^{2}-k^{2}\right)=2U_{k}\partial_{t}U_{k}\chi_{\{\lvert \bu_{\varepsilon}\rvert>k\}}=\partial_{t}U_{k}^{2}$, and $\nabla \left(\lvert\bu\rvert^{2}-k^{2}\right)=\nabla U_{k}^{2}$.
     Discarding some non-negative integrals, we deduce
     \begin{align*}
         &-\iint_{Q}\zeta \partial_{t}U_{k}^{2}\,{\mathrm d}x{\mathrm d}t+\iint_{Q}a_{p}g_{p}(\lvert \D\bu\rvert_{\bg}^{2})\langle\D\bu,\,\D\bu\rangle_{\bg}\zeta\chi_{\{\lvert \bu\rvert>k\}}\,\d x\d t+\iint_{Q}a_{1}\gamma_{\alpha\beta}Z_{\alpha}^{j}\partial_{x_{\beta}}\zeta\chi_{\{\lvert \bu\rvert>k\}}\,\d x\d t \nonumber\\ 
         &+\iint_{Q}a_{p}g_{p}(\lvert \D\bu\rvert_{\bg}^{2})\left\langle \nabla U_{k}^{2}\mathrel{}\middle|\mathrel{}\nabla\zeta \right\rangle_{\bg}\chi_{\{\lvert \bu\rvert>k\}}\,\d x\d t\le \iint_{Q}\lvert\buf\rvert U_{k}\zeta\chi_{\{\lvert \bu\rvert>k\}} \,\d x\d t,
     \end{align*}
     where the last integral makes sense by $\lvert \buf\rvert\in L^{2,\,1}(Q)$ and $U_{k}\in L^{2,\,\infty}(Q)$.
     In computing the third integral on the left-hand side, we note $\gamma_{\alpha\beta}Z_{\alpha}^{j}\psi_{1,\,k}(\lvert \bu\rvert)u^{j}\partial_{x_{\beta}}U_{k,\,l}^{p\alpha}\ge 0$. Indeed, the left-hand side is $0$ when $\D\bu=0$, and otherwise it is $\lvert\D\bu\rvert_{\bg}^{-1}\left\langle \nabla U_{k}^{2}\mathrel{}\middle|\mathrel{}\nabla U_{k,\,l}^{p\alpha} \right\rangle_{\bg}\ge 0$. As a consequence, we have 
     \begin{align*}
        &\iint_{Q}a_{1}\gamma_{\alpha\beta}Z_{\alpha}^{j}u^{j}\partial_{x_{\beta}}\zeta\chi_{\{\lvert \bu\rvert>k\}}\,\d x\d t\ge \iint_{Q}a_{1}U_{k,\,l}^{p\alpha}\gamma_{\alpha\beta}Z_{\alpha}^{j}u^{j}\partial_{x_{\beta}}\eta^{p}\phi\chi_{\{\lvert \bu\rvert>k\}}\,\d x\d t\\&\ge -C\iint_{Q}U_{k}U_{k,\,l}^{p\alpha}\eta^{p-1} \lvert \nabla\eta\rvert\phi\,\d x \d t\ge -C\iint_{Q}\left(\eta^{p}+\lvert\nabla\eta\rvert^{p}\right)U_{k}^{p}U_{k,\,l}^{p\alpha}\phi\,\d x\d t,
     \end{align*}
     where we keep in mind $U_{k}\ge 1$ to deduce the last inequality.
     Similarly to the above computations, we discard some non-negative integrals and use Young's inequality to estimate the second and fourth integrals as follows;
     \begin{align*}
         &\iint_{Q}a_{p}g_{p}(\lvert \D\bu\rvert_{\bg}^{2})\langle\D\bu,\,\D\bu\rangle_{\bg}\zeta\chi_{\{\lvert \bu\rvert>k\}}\,\d x\d t+\iint_{Q}a_{p}g_{p}(\lvert \D\bu\rvert_{\bg}^{2})\left\langle \nabla U_{k}^{2}\mathrel{}\middle|\mathrel{}\nabla\zeta \right\rangle_{\bg}\chi_{\{\lvert \bu\rvert>k\}}\,\d x\d t\\ 
         &\ge c\iint_{Q}\lvert \D\bu\rvert^{p}U_{k,\,l}^{p\alpha}\eta^{p}\phi\,\d x\d t-C\iint_{Q}\lvert \D\bu\rvert^{p-2} \lvert \nabla U_{k}\rvert \lvert\nabla\eta\rvert U_{k} \eta^{p-1} U_{k,\,l}^{p\alpha}\,\d x\d t\\ 
         &\ge \frac{c}{2}\iint_{Q}\lvert \nabla U_{k} \rvert^{p}U_{k,\,l}^{p\alpha}\eta^{p}\phi\,\d x\d t-C\iint_{Q}U_{k}^{p}U_{k,\,l}^{p\alpha} \lvert \nabla\eta\rvert^{p}\phi\,\d x\d t
        \end{align*}
     where we note $\lvert \nabla U_{k}\rvert\le \lvert \D\bu\rvert$.
     For the first integral on the left-hand side, we keep in mind the identities
     \[\partial_{t}\left(U_{k}^{2}U_{k,\,l}^{p\alpha}\right)=2U_{k,\,l}^{p\alpha}U_{k}\partial_{t}U_{k}+p\alpha U_{k,\,l}^{p\alpha+1}\partial_{t}U_{k,\,l}=\frac{2}{p\alpha+2}\partial_{t}[\Psi_{3,\,p\alpha,\,l}(U_{k})]+\frac{p\alpha}{p\alpha+2}\partial_{t}U_{k,\,l}^{p\alpha+2},\]
     where $\Psi_{3,\,p\alpha,\,l}$ is defined as (\ref{Eq (Section 2): Psi-3}).
     Hence, we have 
     \begin{align*}
         &-\iint_{Q}U_{k,\,l}^{p\alpha}U_{k}^{2}\phi\partial_{t}\eta^{p}\,\d x\d t-\iint_{Q}U_{k,\,l}^{p\alpha}U_{k}^{2}\eta^{p}\partial_{t}\phi\,\d x\d t
         \\&=-\iint_{Q}\left[\frac{2}{p\alpha+2}\Psi_{3,\,p\alpha,\,l}(U_{k})+\frac{p\alpha}{p\alpha+2}U_{k,\,l}^{p\alpha+2} \right]\partial_{t}(\eta^{p}\phi)\,\d x\d t
         \\&\le \iint_{Q}\Psi_{3,\,p\alpha,\,l}(U_{k})\lvert \partial_{t}\eta^{p}\rvert\phi\,\d x\d t-\iint_{Q}\Psi_{3,\,p\alpha,\,l}(U_{k}) \eta^{p}\partial_{t}\phi\,\d x\d t 
     \end{align*}
     where we use $U_{k,\,l}^{p\alpha}\partial_{t}U_{k}^{2}=2\partial_{t}\Phi_{3,\,p\alpha,\,l}(U_{k})$ and (\ref{Eq (Section 2): Psi-3}).
     Combining these computations, we have
     \begin{align*}
         &-\iint_{Q}{\widehat\varphi}_{1}^{p}\partial_{t}\phi\,\d x\d t+\iint_{Q}\lvert\nabla {\widehat\varphi}_{2} \rvert^{p}\phi\,\d x\d t\\ 
         &\le C(1+\alpha)^{p}\left[\iint_{Q}\left(\eta^{p}+\lvert\nabla \eta\rvert^{p}\right)U_{k}^{p}U_{k,\,l}^{p\alpha}\,\d x\d t +\iint_{Q}U_{k}^{2}U_{k,\,l}^{p\alpha}\lvert \partial_{t}\eta^{p}\rvert\,\d x\d t+\lvert\buf\rvert U_{k}U_{k,\,l}^{p\alpha}\,\d x\d t \right]
     \end{align*}
     where ${\widehat \varphi}_{1}\coloneqq \eta U_{k,\,l}^{\alpha}U_{k}^{2/p}\in L^{p,\,\infty}(Q_{R})$, ${\widehat\varphi}_{2}\coloneqq \eta U_{k,\,l}^{\alpha}U_{k}\in L^{p}(I_{R};\, W_{0}^{1,\,p}(B_{R}))$.
     Choosing $\phi$ suitably and recalling the definition of $h_{k}$, we have
     \begin{align}\label{Eq (Section 3): Local Reversed Holder for U-k}
        &\lVert \widehat{\varphi}_{1}\rVert_{L^{p,\,\infty}(Q)}^{p}+\lVert \nabla \widehat{\varphi}_{2}\rVert_{L^{p}(Q)}^{p}\nonumber\\ 
        &\le C(1+\alpha)^{p}\left[\iint_{Q}\left(\eta^{p}+\lvert\nabla \eta\rvert^{p}\right)U_{k}^{p}U_{k,\,l}^{p\alpha}\,\d x\d t +\iint_{Q}U_{k}^{2}U_{k,\,l}^{p\alpha}\lvert \partial_{t}\eta^{p}\rvert\,\d x\d t+\iint_{Q}\lvert \buf \rvert U_{k}U_{k,\,l}^{p\alpha} \,\d x\d t\right].
     \end{align}
     Since ${\widehat\varphi}_{2}\le {\widehat \varphi}_{1}$ holds by $U_{k}\ge 1$, we easily find the qualitative bound of $\lVert {\widehat\varphi}_{2}\rVert_{L^{pq^{\prime},\,pr^{\prime}}(Q)}^{p}$, similarly to the proof of (\ref{Eq (Section 2): Parabolic absorbing}).
     In particular, the last integral in (\ref{Eq (Section 3): Local Reversed Holder for U-k}) is computed as
     \[\iint_{Q}\lvert \buf \rvert U_{k}U_{k,\,l}^{p\alpha} \,\d x\d t=\iint_{Q}h_{k}U_{k}U_{k,\,l}^{p\alpha}\,\d x\d t \le \lVert h_{k}\rVert_{L^{q,\,r}(Q_{R})}\lVert{\widehat\varphi}_{2} \rVert_{L^{pq^{\prime},\,pr^{\prime}}(Q_{R})}^{p}\le \lVert{\widehat\varphi}_{2} \rVert_{L^{pq^{\prime},\,pr^{\prime}}(Q_{R})}^{p}\]
     by recalling the definition of $k$ and $h_{k}$.
     With ${\widehat\varphi}_{2}\le {\widehat \varphi}_{1}$ in mind, we apply (\ref{Eq (Section 2): Parabolic absorbing}) with $\varphi_{1}=\varphi_{2}=\widehat{\varphi}_{2}$, $s\coloneqq p$, and $(\pi_{1},\,\pi_{2})\coloneqq (q,\,r)$.
     By a standard absorbing argument and (\ref{Eq (Section 2): Parabolic Sobolev}), we have 
     \[\iint_{Q}\eta^{\kappa p}U_{k}^{2\kappa+p-2}U_{k,\,l}^{\kappa (\beta -2)}\, \d x\d t\le C\beta^{\gamma\kappa}\left[\iint_{Q}\left(\eta^{p}+\lvert\nabla\eta \rvert^{p}+\lvert\partial_{t}\eta^{p} \rvert\right)U_{k}^{\beta}\,\d x\d t\right]^{\kappa}\]
     for some $\gamma=\gamma(n,\,p,\,q,\,r)\in\lbrack p,\,\infty)$. 
     Letting $l\to\infty$ in the resulting estimate, and suitably choosing a cut-off function $\eta$, we conclude (\ref{Eq (Section 3): Reversed Holder for U-k}).

     To prove (\ref{Eq (Section 3): Local Bounds for sup-singular p})--(\ref{Eq (Section 3): Local Bounds for sub-singular p}), we choose $p_{l}\coloneqq \kappa^{l}(p_{0}-\varsigma_{\mathrm c})+\varsigma_{\mathrm c}$, which satisfies the recursive identity $p_{l+1}=\kappa p_{l}+p-2$.
     We choose $p_{0}\coloneqq 2$ when $p\in(p_{\mathrm c},\,2\rbrack$, and otherwise we let $p_{0}\coloneqq \varsigma$.
     In both choices, we have $\mu\coloneqq p_{0}-\varsigma_{\mathrm c}\in(0,\,\infty)$.
     We define $R_{l}\coloneqq R/2+2^{-l-1}R\in(R/2,\,R\rbrack$, and $Y_{l}\coloneqq \left(\iint_{Q_{R_{l}}}U_{k}^{p_{l}}\,\d x\d t\right)^{1/p_{l}}$.
     By (\ref{Eq (Section 3): Reversed Holder for U-k}), it is easy to check that the sequence $Y_{l}$ satisfies all of the assumptions of Lemma \ref{Lemma (Section 2): Moser Iteration Lemma} with $\mu=p_{0}-\varsigma_{\mathrm c}\in(0,\,\infty)$, $A=C({\mathcal D})R^{-2}$, $B\coloneqq 4\kappa^{\gamma\kappa}\in(1,\,\infty)$.
     The resulting inequality of Lemma \ref{Lemma (Section 2): Moser Iteration Lemma} implies $U_{k}\in L^{\infty}(Q_{R/2})$ and 
     \[\esssup_{Q_{R/2}}\,U_{k}\le  \frac{C({\mathcal D},\,p_{0})}{R^{2\cdot\frac{\kappa^{\prime}}{p_{0}-\varsigma_{\mathrm c}}}}\cdot R^{\frac{n+2}{p_{0}-\varsigma_{\mathrm c}}}\left(\fiint_{Q_{R}} U_{k}^{p_{0}}\,\d x\d t \right)^{\frac{1}{p_{0}-\varsigma_{\mathrm c}}}\le \frac{C({\mathcal D},\,p_{0})}{R^{\frac{n}{p_{0}-\varsigma_{\mathrm c}}\cdot\left(\frac{2}{p}-1\right)}} \left[k^{p_{0}}+\fiint_{Q_{R}}\lvert\bu\rvert^{p_{0}}\,\d x\d t \right]^{\frac{1}{p_{0}-\varsigma_{\mathrm c}}}  .\]
     where $U_{k}\le \lvert \bu\rvert+k$ is used. Recalling $\lvert \bu\rvert\le U_{k}$ and the definition of $k$, we conclude (\ref{Eq (Section 3): Local Bounds for sup-singular p})--(\ref{Eq (Section 3): Local Bounds for sub-singular p}).      
    \end{proof} 
    \begin{remark}\label{Rmk: Alternative Approach for L-infty} \upshape
        When $p\in(p_{\mathrm c},\,2)$, local $L^{\infty}$-bounds of an approximate solution $\bu_{\varepsilon}$ can be verified without using Proposition \ref{Proposition (Section 3): A Weak Maximum Principle}.
        Instead, similarly to Proposition \ref{Proposition (Section 3): Local Boundedness for subcritical cases}, we can deduce local $L^{\infty}$--$L^{2}$ estimates of $\bu_{\varepsilon}$.
    \end{remark}
 \section{Regularity estimates and weak formulations}\label{Section: Weak Formulations}
 Section \ref{Section: Weak Formulations} provides a priori estimates for a solution $\bu_{\varepsilon}$ to (\ref{Eq (Section 2): Approximate System}) in a subcylinder ${\mathcal Q}\Subset \Omega_{T}$ or more smaller one $\widetilde{\mathcal Q}\Subset {\mathcal Q}$. 
 Since (\ref{Eq (Section 2): Weak Conv of f-epsilon}) is assumed and we will later discuss the convergence of approximate solutions in Section \ref{Section: Convergence Result}, it is not restrictive to let 
 \begin{equation}\label{Eq (Section 4): Uniform Bounds F and U}
    \lVert \buf_{\varepsilon}\rVert_{L^{q,\,r}({\mathcal Q})}\le F,\quad \text{and}  \quad \lVert \D\bu_{\varepsilon}\rVert_{L^{p}({\mathcal Q})}\le U
 \end{equation}
 for some constants $F,\,U\in(0,\,\infty)$ that are independent of $\varepsilon\in(0,\,1)$.
 Also when $p\in(1,\,2)$, we require
 \begin{equation}\label{Eq (Section 4): Weak Sol Bound for Singular p}
     \esssup_{\mathcal Q}\,\lvert \bu_{\varepsilon}\rvert\le M_{0}
 \end{equation}
 for some constant $M_{0}\in(1,\,\infty)$.
 The assumption (\ref{Eq (Section 4): Weak Sol Bound for Singular p}) is not restrictive by Propositions \ref{Proposition (Section 3): A Weak Maximum Principle}--\ref{Proposition (Section 3): Local Boundedness for subcritical cases} or Remark \ref{Rmk: Alternative Approach for L-infty} in Section \ref{Section: L-infty}. 
 \subsection{A priori regularity estimates for approximate solutions}
 We outline what to prove on the interior a priori estimates for $\bu_{\varepsilon}$, which broadly consists of the two parts. 
 
 Firstly, we would like to show
 \begin{equation}\label{Eq (Section 4): Local Gradient bound}
     \esssup_{\widetilde{\mathcal Q}}v_{\varepsilon}\le \mu_{0},
 \end{equation}
 for some $\mu_{0}=\mu_{0}({\mathcal D},\,F,\,U,\,{\mathcal Q},\,\widetilde{\mathcal Q},\,M_{0})$ that is independent of $\varepsilon\in(0,\,1)$ (Theorem \ref{Theorem (Section 4): Gradient bounds}).
 In other words, we aim to find the uniform bound of $v_{\varepsilon}=\sqrt{\varepsilon^{2}+\lvert \D\bu_{\varepsilon}\rvert_{\bg}^{2}}$.
 \begin{theorem}\label{Theorem (Section 4): Gradient bounds}
     Let $\bu_{\varepsilon}$ be a weak solution to (\ref{Eq (Section 2): Approximate System}) in ${\mathcal Q}\Subset \Omega_{T}$ with (\ref{Eq (Section 4): Uniform Bounds F and U}) satisfied independent of $\varepsilon\in(0,\,1)$.
     For $p\in(1,\,2)$, let (\ref{Eq (Section 4): Weak Sol Bound for Singular p}) be in force.
     Then, (\ref{Eq (Section 4): Local Gradient bound}) holds, uniformly for $\varepsilon\in(0,\,1)$.
     Moreover, the uniform bound $\mu_{0}\in(1,\,\infty)$ depends on ${\mathcal D}$, $F$, $U$, ${\mathcal Q}$, and $\widetilde{\mathcal Q}$ when $p\in(p_{\mathrm c},\,\infty)$, and this bound also depends on $M_{0}$ when $p\in(1,\,p_{\mathrm c})$.
 \end{theorem}
 Secondly, under the assumption (\ref{Eq (Section 4): Local Gradient bound}), we prove the uniform H\"{o}lder continuity of $\G_{2\delta,\,\varepsilon}(\D\bu_{\varepsilon})$ for $\varepsilon\in(0,\,\delta/4)$, as stated in Theorem \ref{Theorem (Section 4): Holder Truncated-Gradient Continuity}.
 \begin{theorem}\label{Theorem (Section 4): Holder Truncated-Gradient Continuity}
     Let $\bu_{\varepsilon}$ be a weak solution to (\ref{Eq (Section 2): Approximate System}) in $\widetilde{\mathcal Q}\Subset {\mathcal Q} \Subset \Omega_{T}$ with (\ref{Eq (Section 4): Uniform Bounds F and U}) and (\ref{Eq (Section 4): Local Gradient bound}) guaranteed.
     Then, the limit  
     \begin{equation}\label{Eq (Section 4): Limit Average of G-2delta-epsilon}
         \bG_{2\varepsilon,\,\delta}(x_{0},\,t_{0})\coloneqq \lim_{R \to 0}(\G_{2\delta,\,\varepsilon}(\D\bu_{\varepsilon}))_{Q_{R}(x_{0},\,t_{0})}\in{\mathbb R}^{Nn}
     \end{equation}
     is well-defined for every $(x_{0},\,t_{0})\in \widetilde{\mathcal Q}$, and the mapping $\bG_{2\delta,\,\varepsilon}$ is locally $(\alpha,\,\alpha/2)$-H\"{o}lder continuous in $\widetilde{\mathcal Q}$ for some $\alpha=\alpha({\mathcal D},\,F,\,\mu_{0},\,\delta)\in(0,\,\beta/2)$, whose continuity estimate is independent of $\varepsilon\in(0,\,\delta/4)$.
     In particular, ${\G}_{2\delta,\,\varepsilon}(\D\bu_{\varepsilon})$ also has the same continuity estimate.
 \end{theorem}
 Theorem \ref{Theorem (Section 4): Holder Truncated-Gradient Continuity} is shown by division by cases based on the size of the super-level set of $v_{\varepsilon}$.
 More precisely, for each $Q_{2\rho}(x_{0},\,t_{0})\Subset \widetilde{\mathcal Q}$, we assume 
 \begin{equation}\label{Eq (Section 4): Bound Assumption of v-epsilon}
    \esssup_{Q_{2\rho}(x_{0},\,t_{0})}v_{\varepsilon}\le \mu+\delta\le M
 \end{equation}
 or equivalently
 \begin{equation}\label{Eq (Section 4): Bound Assumption of G-delta-epsilon}
     \esssup_{Q_{2\rho}(x_{0},\,t_{0})}\lvert\G_{\delta,\,\varepsilon}(\D\bu_{\varepsilon}) \rvert_{\bg}\le \mu\le \mu+\delta\le M    
 \end{equation} 
 for some constant $M\in(1,\,\infty)$ and some parameter $\mu\in\lbrack 0,\,M-\delta\rbrack$.
 If $\mu\le \delta$, then we have $\G_{2\delta,\,\varepsilon}(\D\bu_{\varepsilon})\equiv 0$ in $Q_{2\rho}(x_{0},\,t_{0})$.
 With this in mind, we mainly consider the non-trivial case
 \begin{equation}\label{Eq (Section 4): delta vs mu}
     \delta<\mu.
 \end{equation}
 Our approach is essentially different from intrinsic scaling arguments found in $p$-Laplace regularity theory, since the delicate case $\mu\le \delta$, where a spatial gradient could vanish, is automatically ruled out.
 Under the assumptions (\ref{Eq (Section 4): Bound Assumption of v-epsilon})--(\ref{Eq (Section 4): delta vs mu}), we introduce the super-level set 
 \[S_{\rho,\,\mu}(x_{0},\,t_{0})\coloneqq \{(x,\,t)\in Q_{\rho}(x_{0},\,t_{0})\mid v_{\varepsilon}(x,\,t)-\delta>(1-\nu)\mu\},\]
 where the sufficiently small constant $\nu\in(0,\,1)$ is chosen later in Section \ref{Section: Non-Degenerate}.
 We often write this super-level set as $S_{\rho}$ or $S_{\rho,\,\mu}$ for simplicity.

 We mainly divide by cases, depending on whether the ratio of the sub-level set $\lvert Q_{\rho}\setminus S_{\rho,\,\mu}\rvert/\lvert Q_{\rho}\rvert$ exceeds the value $\nu$ or not.
 In the former and latter cases, which we respectively call degenerate and non-degenerate cases, we would like to show respectively the De Giorgi-type oscillation lemma (Proposition \ref{Proposition (Section 4): Degenerate Case}) and the Campanato-type growth estimates (Proposition \ref{Proposition (Section 4): Non-Degenerate Case}).
 
 In the degenerate case, we appeal to the fact that the scalar-valued function \[U_{\delta,\,\varepsilon}\coloneqq (v_{\varepsilon}-\delta)_{+}^{2}=\lvert \G_{\delta,\,\varepsilon}(\D\bu_{\varepsilon})\rvert_{\bg}^{2}\] is a weak subsolution to a uniformly parabolic equation.
 Here we should keep in mind that the support of $U_{\delta,\,\varepsilon}$ is in the non-degenerate region $\{v_{\varepsilon}\ge \delta\}$, where the approximate system (\ref{Eq (Section 2): Approximate System}) can be treated as uniformly parabolic in the classical sense.
 Combining this result with a level set assumption, we prove Proposition \ref{Proposition (Section 4): Degenerate Case} by standard iteration arguments.
 \begin{proposition}\label{Proposition (Section 4): Degenerate Case}
     Assume that $\bu_{\varepsilon}$ is a weak solution to (\ref{Eq (Section 2): Approximate System}) in $Q_{2\rho}(x_{0},\,t_{0})\subset \widetilde{\mathcal Q} \Subset {\mathcal Q} \Subset \Omega_{T}$ with $\varepsilon\in(0,\,\delta/4)$.
     In addition to (\ref{Eq (Section 4): Uniform Bounds F and U}) and (\ref{Eq (Section 4): Bound Assumption of v-epsilon})--(\ref{Eq (Section 4): delta vs mu}), let 
     \begin{equation}\label{Eq (Section 4): Measure Assumption: Degenerate}
        \lvert S_{\rho,\,\mu}\rvert<(1-\nu)\lvert Q_{\rho}\rvert
     \end{equation}
     hold for some $\nu\in(0,\,1)$. Then, there exists a constant $\kappa=\kappa({\mathcal D},\,\delta,\,M,\,\nu)\in\lbrack (\sqrt{\nu}/6)^{\beta},\, 1)$ such that we have 
     \begin{equation}\label{Eq (Section 4): Oscillation result}
        \esssup_{Q_{\sqrt{\nu}\rho/3}}\, \lvert \G_{\delta,\,\varepsilon}(\D\bu_{\varepsilon}) \rvert_{\bg} \le \kappa\mu,
     \end{equation}
     provided $\rho\le \widetilde{\rho}$ for some sufficiently small $\widetilde{\rho}=\widetilde{\rho}({\mathcal D},\,\delta,\,M,\,\nu)\in(0,\,1)$.
 \end{proposition}
 In the non-degenerate case, we expect that a gradient $\D\bu_{\varepsilon}$ may not degenerate, which is indeed rigorously shown by deducing some non-trivial energy bounds to estimate $\lvert (\D\bu_{\varepsilon})_{Q_{\rho}}\rvert_{\bg(x_{0},\,t_{0})}$ by below.
 From this starting point, we compare $\bu_{\varepsilon}$ with a weak solution of some sort of heat system to deduce the classical Campanato-type integral growth estimate (\ref{Eq (Section 4): Campanato-Growth}). 
 \begin{proposition}\label{Proposition (Section 4): Non-Degenerate Case}
     Assume that $\bu_{\varepsilon}$ is a weak solution to (\ref{Eq (Section 2): Approximate System}) in $Q_{2\rho}(x_{0},\,t_{0})\subset \widetilde{\mathcal Q} \Subset {\mathcal Q} \Subset \Omega_{T}$ with $\varepsilon\in(0,\,\delta/4)$.
     Let (\ref{Eq (Section 4): Uniform Bounds F and U}) and (\ref{Eq (Section 4): Bound Assumption of v-epsilon})--(\ref{Eq (Section 4): delta vs mu}) be in force.
     Then, there exist a sufficiently small number $\nu=\nu({\mathcal D},\,\delta,\,M,\,F)\in (0,\,10^{-23}\gamma_{0}^{16})$ and a sufficiently small radius $\widehat{\rho}=\widehat{\rho}({\mathcal D},\,\delta,\,M,\,F)\in (0,\,1)$ such that if $\rho\le \widehat{\rho}$ and
     \begin{equation}\label{Eq (Section 4): Measure Assumption: Non-Degenerate}
         \lvert S_{\rho,\,\mu}\rvert\ge (1-\nu)\lvert Q_{\rho}\rvert
     \end{equation}
     hold, then the limit $\bG_{2\delta,\,\varepsilon}(x_{0},\,t_{0})$ as in (\ref{Eq (Section 4): Limit Average of G-2delta-epsilon}) is well-defined. 
     Moreover, following (\ref{Eq (Section 4): Growth of Gamma-2delta-epsilon})--(\ref{Eq (Section 4): Campanato-Growth}) hold;
     \begin{equation}\label{Eq (Section 4): Growth of Gamma-2delta-epsilon}
         \lvert \bG_{2\delta,\,\varepsilon}(x_{0},\,t_{0}) \rvert\le \frac{\mu}{\gamma_{0}},
     \end{equation}
     \begin{equation}\label{Eq (Section 4): Campanato-Growth}
         \fiint_{Q_{\tau\rho}}\lvert \G_{2\delta,\,\varepsilon}(\D\bu_{\varepsilon})-\bG_{2\delta,\,\varepsilon}(x_{0},\,t_{0}) \rvert^{2}\,\d x\d t\le \tau^{2\beta}\quad \text{for all }\tau\in(0,\,1\rbrack.
     \end{equation}
 \end{proposition}
 
 After Theorem \ref{Theorem (Section 4): Gradient bounds} is shown in Section \ref{Section: Gradient Bounds}, the proofs of Propositions \ref{Proposition (Section 4): Degenerate Case}, \ref{Proposition (Section 4): Non-Degenerate Case}, and Theorem \ref{Theorem (Section 4): Holder Truncated-Gradient Continuity} are given respectively in Sections \ref{Section: Degenerate}, \ref{Section: Non-Degenerate}, and \ref{Section: Gradient Continuity}.
 In the remaining parts of Section \ref{Section: Weak Formulations}, we would like to deduce various weak formulations and energy estimates. These results are fully used in Sections \ref{Section: Gradient Bounds}--\ref{Section: Non-Degenerate}.
 \subsection{Basic weak formulations} 
 We would like to deduce a weak formulation of $v_{\varepsilon}$ in a systematic approach. 
 We differentiate the system (\ref{Eq (Section 2): Approximate System}) in space to deduce a weak formulation, and hence we have to use some smoothness structures of $a_{1}$ and $a_{p}$ in space variables.
 In particular, the Lipschitz assumptions $\nabla a_{1},\,\nabla a_{p},\,\nabla\gamma_{\alpha,\,\beta}\in L^{\infty}$ are used.
 Another regularity assumption $\partial_{t}\gamma_{\alpha,\,\beta}\in L^{\infty}$ is also used in computing the time derivative $\partial_{t}v_{\varepsilon}$.
 Following \cite[Proposition 3.1]{BDLS-parabolic} (see also \cite[Chapter VIII]{DiBenedetto-monograph MR1230384}), we would like to deduce Lemma \ref{Lemma (Section 4): Basic Weak Form}.
 \begin{lemma}\label{Lemma (Section 4): Basic Weak Form}
    Let $\psi\colon{\mathbb R}_{\ge 0}\to{\mathbb R}_{\ge 0}$ be a non-decreasing Lipschitz function, and define $\Psi\colon {\mathbb R}_{\ge 0}\to{\mathbb R}_{\ge 0}$ as (\ref{Eq (Section 2): Def of Psi}).
    Assume that $\bu_{\varepsilon}$ is a weak solution to (\ref{Eq (Section 2): Approximate System}) with $\varepsilon\in(0,\,1)$, and that $\zeta\in C_{\mathrm c}^{1}(Q;\,{\mathbb R}_{\ge 0})$ be arbitrarily given.
    We set the integrals
    \[\begin{array}{rcl} 
        E_{0}&\coloneqq &-\displaystyle\iint_{Q}\Psi(v_{\varepsilon})\partial_{t}\zeta\,\d x\d t,\\ E_{1}&\coloneqq & \displaystyle\iint_{Q}{\mathcal C}_{\varepsilon}(x,\,t,\,\D\bu_{\varepsilon})(\nabla v_{\varepsilon},\,\nabla\zeta)\psi(v_{\varepsilon})v_{\varepsilon}\,\d x\d t\\ 
        & =&\displaystyle\iint_{Q}{\mathcal C}_{\varepsilon}(x,\,t,\,\D\bu_{\varepsilon})(\nabla[\Psi(v_{\varepsilon})],\,\nabla\zeta)\,\d x\d t,\\ E_{2}&\coloneqq & \displaystyle\iint_{Q}{\mathcal C}_{\varepsilon}(x,\,t,\,\D\bu_{\varepsilon})(\nabla v_{\varepsilon},\,\nabla v_{\varepsilon})\psi^{\prime}(v_{\varepsilon})v_{\varepsilon}\zeta \,\d x\d t,\\ E_{3}&\coloneqq & \displaystyle\iint_{Q}{\mathcal A}_{\varepsilon}(x,\,t,\,\D\bu_{\varepsilon})({\D}^{2}\bu_{\varepsilon},\,{\D}^{2}\bu_{\varepsilon})\psi(v_{\varepsilon})\zeta \,\d x\d t,\\ E_{4}&\coloneqq & \displaystyle\iint_{Q}\left(v_{\varepsilon}^{p}\left(1+v_{\varepsilon}^{1-p} \right)^{2}+\lvert\buf_{\varepsilon}\rvert^{2}v_{\varepsilon}^{2-p} \right) \left(\psi(v_{\varepsilon})+\psi^{\prime}(v_{\varepsilon})v_{\varepsilon}\right)\zeta\,\d x\d t, \\ E_{5}&\coloneqq & \displaystyle\iint_{Q}\left(v_{\varepsilon}^{p}\left(1+v_{\varepsilon}^{1-p}\right)+\lvert \buf_{\varepsilon}\rvert v_{\varepsilon} \right) \psi(v_{\varepsilon})\lvert\nabla\zeta\rvert \,\d x\d t,\\ 
        E_{6} &=&\displaystyle\iint_{Q}v_{\varepsilon}^{2}\psi(v_{\varepsilon})\zeta\,\d x\d t.\end{array}\]
        Then, there exists a constant $C=C({\mathcal D})\in(1,\,\infty)$ such that there holds
    \begin{equation}\label{Eq (Section 4): Weak Form of v-epsilon}
    E_{0}+E_{1}+\frac{1}{4}(E_{2}+E_{3})\le C(E_{4}+E_{5}+E_{6}).
    \end{equation}
    \end{lemma}
    \begin{remark}\label{Rmk: Section 4} \upshape
        We must keep in mind that the computations in Sections \ref{Section: Weak Formulations}--\ref{Section: Gradient Bounds} are formal, in the sense that $\partial_{t}\partial_{x_{\sigma}}\bu_{\varepsilon}$ and $\D^{2}\bu_{\varepsilon}$ are treated as some sort of \textit{function}.
        In particular, the computations in the proof of Lemma \ref{Lemma (Section 4): Basic Weak Form} and the resulting estimate (\ref{Eq (Section 4): Weak Form of v-epsilon}) involves the integral of $v_{\varepsilon}^{2}$, which appears to be critical when $p\in(1,\,2)$.
        These formal computations, however, can be justified by noting $\bu_{\varepsilon}\in L^{2,\,\infty}(Q)^{N}\subset L^{2}(Q)^{N}$ and utilizing the difference quotient method (see \cite[Chapter VIII]{DiBenedetto-monograph MR1230384} or \cite[\S 3]{Strunk preprint}), as well as the Steklov average.
        This strategy works for (\ref{Eq (Section 2): Approximate System}), but seems invalid for the original problem (\ref{Eq (Section 1): General System}), since the $(1,\,p)$-Laplace operator lacks any uniform ellipticity on the facet, which fact may prevent us from deducing difference quotient estimates.
    \end{remark}
    \begin{remark}\label{Remark (Section 4): Positivity of E2 and E3} \upshape
        By (\ref{Eq (Section 2): Ellipticity of Bilinear Forms}), the integrals $E_{2}$ and $E_{3}$ respectively satisfy
        \begin{equation}\label{Eq (Section 4): Positivity of E2 and E3}
            E_{2}\ge \lambda_{0}\iint_{Q}v_{\varepsilon}^{p-1}\lvert \nabla v_{\varepsilon}\rvert_{\bg}^{2}\psi^{\prime}(v_{\varepsilon})\zeta\,\d x\d t,\quad \text{and}\quad E_{3}\ge \lambda_{0}\iint_{Q}v_{\varepsilon}^{p-2}\lvert \D^{2}\bu_{\varepsilon}\rvert_{\bg}^{2}\psi(v_{\varepsilon})\zeta \,\d x\d t.
        \end{equation}
        This fact is often carefully used to carry out absorbing arguments.
    \end{remark}
    \begin{proof}
    For each $\nu,\,\sigma\in\{\,1,\,\dots\,,\,n\,\}$, we formally test $\bphi\coloneqq ( [\zeta\psi(v_{\varepsilon})\gamma_{\nu\sigma}\partial_{x_{\sigma}}u_{\varepsilon}^{j}]_{x_{\sigma}})_{j}$ into (\ref{Eq (Section 2): Approximate Weak Form}), and sum over $\nu,\,\sigma\in\{\,1,\,\dots\,,\,n\,\}$. 
    Then, we have 
    \begin{align}\label{Eq (Section 4): Weak form mid}
    &\iint_{Q}\partial_{t}\partial_{x_{\nu}}u_{\varepsilon}^{j}\cdot \zeta\gamma_{\nu\sigma}\partial_{x_{\sigma}}u_{\varepsilon}^{j}\psi(v_{\varepsilon})\,\d x\d t+\iint_{Q}\left[a_{s}g_{s}(v_{\varepsilon}^{2})\gamma_{\alpha\beta}\partial_{x_{\alpha}} u_{\varepsilon}^{j} \right]_{x_{\nu}} \left[ \zeta\gamma_{\nu\sigma} \partial_{x_{\sigma}} u_{\varepsilon}^{j}\psi(v_{\varepsilon}) \right]_{x_{\beta}}\,\d x\d t \nonumber \\&\quad\quad +\iint_{Q}f_{\varepsilon}^{j}\left[ \zeta\gamma_{\nu\sigma} \partial_{x_{\sigma}} u_{\varepsilon}^{j}\psi(v_{\varepsilon}) \right]_{x_{\nu}}\,\d x\d t=0,
    \end{align}
    where we use the convention to sum over $\alpha,\,\beta,\,\nu,\,\sigma\in\{\,1,\,\dots\,,\,n\,\}$, $j\in\{\,1,\,\dots\,,\,N\,\}$, and $s\in\{\,1,\,p\,\}$.
    By the identities \(2v_{\varepsilon}\partial_{t}v_{\varepsilon}=\partial_{t}v_{\varepsilon}^{2}=2\gamma_{\nu\sigma}\partial_{t}\partial_{x_{\nu}}u_{\varepsilon}^{j}\partial_{x_{\sigma}}u_{\varepsilon}^{j}+\partial_{t}\gamma_{\nu\sigma}\partial_{x_{\nu}}u_{\varepsilon}^{j}\partial_{x_{\sigma}}u_{\varepsilon}^{j}\), and $\partial_{t}[\Psi(v_{\varepsilon})]=\psi(v_{\varepsilon})v_{\varepsilon}\partial_{t} v_{\varepsilon}$, we integrate by parts in time to compute 
    \[\iint_{Q}\partial_{t}\partial_{x_{\nu}}u_{\varepsilon}^{j}\cdot \zeta\gamma_{\nu\sigma}\partial_{x_{\sigma}}u_{\varepsilon}^{j}\psi(v_{\varepsilon})\,\d x\d t=E_{0}-\frac{1}{2}\iint_{Q}\zeta\psi(v_{\varepsilon})\partial_{t}\gamma_{\nu\sigma}\partial_{x_{\nu}}u_{\varepsilon}^{j}\partial_{x_{\sigma}}u_{\varepsilon}^{j}\,\d x\d t\ge E_{0}-C({\mathcal D})E_{6},\]
    where (\ref{Eq (Section 1): Bound Assumptions for Coefficients}) is used to deduce the last estimate.
    For each fixed $\alpha,\,\beta\in\{\,1,\,\dots\,n\,\}$, we abbreviate
    \[F_{\alpha}\coloneqq
    \gamma_{\kappa\lambda}\partial_{x_{\kappa}} u_{\varepsilon}^{k}\partial_{x_{\alpha}x_{\lambda}}u_{\varepsilon}^{k},\quad V_{\beta}\coloneqq 
    \partial_{x_{\beta}}\gamma_{\kappa\lambda}\partial_{x_{\kappa}}u_{\varepsilon}^{k}\partial_{x_{\lambda}}u_{\varepsilon}^{k}.\]
    where we sum over $\kappa,\,\lambda\in\{\,1,\,\dots\,,\,n\,\}$ and $k\in\{\,1,\,\dots\,,\,N\,\}$. Then, the ${\mathbb R}^{n}$-valued mappings ${\mathbf F}\coloneqq (F_{1},\,\dots\,,\,F_{n})$, ${\mathbf V}\coloneqq (V_{1},\,\dots\,,\,V_{n})$.
    satisfies
    \begin{equation}\label{Eq (Section 4): F and V}
    2v_{\varepsilon}\nabla v_{\varepsilon}=\nabla v_{\varepsilon}^{2}=2{\mathbf F}+{\mathbf V}.
    \end{equation}
    With (\ref{Eq (Section 4): F and V}) in mind, for each $s\in\{\,1,\,p\,\}$, we compute
    \begin{align*}
    \left[a_{s}g_{s}(v_{\varepsilon}^{2})\gamma_{\alpha\beta}\partial_{x_{\alpha}} u_{\varepsilon}^{j} \right]_{x_{\nu}}&=a_{s}\left( g_{s}(v_{\varepsilon}^{2})\gamma_{\alpha\beta}\partial_{x_{\alpha}x_{\nu}}u_{\varepsilon}^{j}+2g_{s}^{\prime}(v_{\varepsilon}^{2})F_{\nu}\gamma_{\alpha\beta}u_{\varepsilon}^{j} \right)\\&\quad +a_{s}g_{s}^{\prime}(v_{\varepsilon}^{2})V_{\nu} \gamma_{\alpha\beta}\partial_{x_{\alpha}}u_{\varepsilon}^{j}+g_{s}(v_{\varepsilon}^{2})\partial_{x_{\alpha}}u_{\varepsilon}^{j}\partial_{x_{\nu}}(a_{s}\gamma_{\alpha\beta})\\ &\eqqcolon {\mathbf I}_{1,\,s}+{\mathbf I}_{2,\,s}+{\mathbf I}_{3,\,s},
    \end{align*}
    and
    \begin{align*}
    \left[ \zeta\gamma_{\nu\sigma} \partial_{x_{\sigma}} u_{\varepsilon}^{j}\psi(v_{\varepsilon}) \right]_{x_{\beta}} &=\gamma_{\nu\sigma}\partial_{x_{\sigma}x_{\beta}}u_{\varepsilon}^{j}\psi(v_{\varepsilon})\zeta+\gamma_{\nu\sigma}\partial_{x_{\sigma}}u_{\varepsilon}^{j}\frac{2F_{\beta}+V_{\beta}}{2v_{\varepsilon}}\psi^{\prime}(v_{\varepsilon})\zeta\\
    &\quad +\gamma_{\nu\sigma}\partial_{x_{\sigma}}u_{\varepsilon}^{j}\partial_{x_{\beta}}\zeta\psi(v_{\varepsilon})+\partial_{x_{\beta}}\gamma_{\nu\sigma}\partial_{x_{\sigma}}u_{\varepsilon}^{j}\psi(v_{\varepsilon})\zeta\\
    &\eqqcolon {\mathbf J}_{1}+{\mathbf J}_{2}+{\mathbf J}_{3}+{\mathbf J}_{4}
    \end{align*}
    where ${\mathbf I}_{l,\,s}\,(l=1,\,2,\,3)$ and ${\mathbf J}_{m}\,(m=1,\,2,\,3,\,4)$ are ${\mathbb R}^{Nn^{2}}$-valued tensors.
    Recalling the definition of ${\mathcal A}_{\varepsilon}$ and ${\mathcal C}_{\varepsilon}$, and noting that (\ref{Eq (Section 4): F and V}) implies $4v_{\varepsilon}^{2}{\mathcal C}_{\varepsilon}(\nabla v_{\varepsilon},\,\nabla v_{\varepsilon})={\mathcal C}_{\varepsilon}(2{\mathbf F}+{\mathbf V},\,2{\mathbf F}+{\mathbf V})=4{\mathcal C}_{\varepsilon}({\mathbf F},\,{\mathbf F})+4{\mathcal C}_{\varepsilon}({\mathbf F},\,{\mathbf V})+{\mathcal C}_{\varepsilon}({\mathbf V},\,{\mathbf V})$, we have
    \[({\mathbf I}_{1,\,1}+{\mathbf I}_{1,\,p})\cdot {\mathbf J}_{1}={\mathcal A}_{\varepsilon}(x,\,t,\,\D\bu_{\varepsilon})({\D}^{2}\bu_{\varepsilon},\,{\D}^{2}\bu_{\varepsilon})\psi(v_{\varepsilon})\zeta,\]
    and
    \begin{align*}
    &({\mathbf I}_{1,\,1}+{\mathbf I}_{1,\,p})\cdot {\mathbf J}_{2}\\ &=\left({\mathcal C}_{\varepsilon}(x,\,t,\,\D\bu_{\varepsilon})({\mathbf F},\,{\mathbf F})+\frac{1}{2}{\mathcal C}_{\varepsilon}(x,\,t,\,\D\bu_{\varepsilon})({\mathbf F},\,{\mathbf V})\right)\frac{\psi^{\prime}(v_{\varepsilon})}{v_{\varepsilon}}\zeta\\
    &={\mathcal C}_{\varepsilon}(x,\,t,\,\D\bu_{\varepsilon})(\nabla v_{\varepsilon},\,\nabla v_{\varepsilon})\psi^{\prime}(v_{\varepsilon})v_{\varepsilon}\zeta\\
    &\quad -\left(\frac{1}{4}{\mathcal C}_{\varepsilon}(x,\,t,\,\D\bu_{\varepsilon})({\mathbf V},\,{\mathbf V})+\frac{1}{2}{\mathcal C}_{\varepsilon}(x,\,t,\,\D\bu_{\varepsilon})({\mathbf F},\,{\mathbf V}) \right)\frac{\psi^{\prime}(v_{\varepsilon})}{v_{\varepsilon}}\zeta\\
    &= \frac{1}{2}{\mathcal C}_{\varepsilon}(x,\,t,\,\D\bu_{\varepsilon})(\nabla v_{\varepsilon},\,\nabla v_{\varepsilon})\psi^{\prime}(v_{\varepsilon})v_{\varepsilon}\zeta\\ 
    &\quad +\frac{1}{2}{\mathcal C}_{\varepsilon}(x,\,t,\,\D\bu_{\varepsilon})({\mathbf F},\,{\mathbf F})\frac{\psi^{\prime}(v_{\varepsilon})}{v_{\varepsilon}}\zeta-\frac{1}{8}{\mathcal C}_{\varepsilon}(x,\,t,\,\D\bu_{\varepsilon})({\mathbf V},\,{\mathbf V})\frac{\psi^{\prime}(v_{\varepsilon})}{v_{\varepsilon}}\zeta.
    \end{align*}
    For the integrands that involve $\partial_{x_{\beta}}\zeta$, we use (\ref{Eq (Section 2): Ellipticity of Bilinear Forms}) and (\ref{Eq (Section 4): F and V}) to compute
    \begin{align*}
        ({\mathbf I}_{1,\,1}+{\mathbf I}_{1,\,p})\cdot {\mathbf J}_{3}&= {\mathcal C}_{\varepsilon}(x,\,t,\,\D\bu_{\varepsilon})({\mathbf F},\,\nabla \zeta)\psi(v_{\varepsilon})\\
     &={\mathcal C}_{\varepsilon}(x,\,t,\,\D\bu_{\varepsilon})(\nabla[\Psi(v_{\varepsilon})],\,\nabla \zeta)\psi(v_{\varepsilon})-\frac{1}{2}{\mathcal C}_{\varepsilon}(x,\,t,\,\D\bu_{\varepsilon})({\mathbf V},\,\nabla \zeta)\psi(v_{\varepsilon})\\
     &\ge {\mathcal C}_{\varepsilon}(x,\,t,\,\D\bu_{\varepsilon})(\nabla[\Psi(v_{\varepsilon})],\,\nabla \zeta)\psi(v_{\varepsilon})-Cv_{\varepsilon}^{p}\left(1+v_{\varepsilon}^{1-p}\right)\lvert\nabla\zeta\rvert\psi(v_{\varepsilon}).
    \end{align*}
    By (\ref{Eq (Section 4): Positivity of E2 and E3}) and Young's inequality, the remaining integrands are estimated as follows;
    \begin{align*}
    &\left(\lvert {\mathbf I}_{2,\,1}\rvert+\lvert {\mathbf I}_{3,\,1}\rvert+\lvert {\mathbf I}_{2,\,p}\rvert+\lvert {\mathbf I}_{3,\,p}\rvert\right)\left(\lvert {\mathbf J}_{1}\rvert+\lvert {\mathbf J}_{2}\rvert+\lvert {\mathbf J}_{3}\rvert+\lvert {\mathbf J}_{4}\rvert \right)\\ 
    &\le C\left(v_{\varepsilon}^{p-1}+1 \right)\left(\lvert {\D}^{2}\bu_{\varepsilon}\rvert_{\bg}\psi(v_{\varepsilon})+\lvert {\mathbf F}\rvert_{\bg}\psi^{\prime}(v_{\varepsilon}) \right)\zeta\\
    &\quad +C\left(v_{\varepsilon}^{p-1}+1 \right)v_{\varepsilon}\left(\psi(v_{\varepsilon})\zeta+\psi^{\prime}(v_{\varepsilon})v_{\varepsilon}\zeta+\psi(v_{\varepsilon})\lvert\nabla\zeta\rvert \right)\\
    &\le \frac{\lambda_{0}}{4}v_{\varepsilon}^{p-2}\lvert {\D}^{2}\bu_{\varepsilon}\rvert_{\bg}^{2}\psi(v_{\varepsilon}) \zeta +\frac{\lambda_{0}}{2}v_{\varepsilon}^{p-3}\lvert {\mathbf F}\rvert_{\bg}^{2}\psi^{\prime}(v_{\varepsilon})\zeta\\ 
    &\quad +\frac{C}{\lambda_{0}}v_{\varepsilon}^{p}\left(1+v_{\varepsilon}^{1-p}\right)^{2}(\psi(v_{\varepsilon})+\psi^{\prime}(v_{\varepsilon})v_{\varepsilon}) \zeta\\ &\quad\quad +Cv_{\varepsilon}^{p}\left(1+v_{\varepsilon}^{1-p} \right)\left(\psi(v_{\varepsilon})\zeta+\psi^{\prime}(v_{\varepsilon})v_{\varepsilon}\zeta+\psi(v_{\varepsilon})\lvert\nabla\zeta\rvert \right)\\ 
    &\le \frac{1}{4}{\mathcal A}_{\varepsilon}(x,\,t,\,\D\bu_{\varepsilon})({\D}^{2}\bu_{\varepsilon},\,{\D}^{2}\bu_{\varepsilon})\psi(v_{\varepsilon})\zeta+\frac{1}{2}{\mathcal C}_{\varepsilon}(x,\,t,\,\D\bu_{\varepsilon})({\mathbf F},\,{\mathbf F})\frac{\psi^{\prime}(v_{\varepsilon})}{v_{\varepsilon}}\zeta\\
    &\quad +Cv_{\varepsilon}^{p}\left(1+v_{\varepsilon}^{1-p}\right)^{2}\left(\psi(v_{\varepsilon})\zeta+\psi^{\prime}(v_{\varepsilon})v_{\varepsilon}\zeta\right)+Cv_{\varepsilon}^{p}\left(1+v_{\varepsilon}^{1-p}\right)\psi(v_{\varepsilon})\lvert\nabla\zeta\rvert,
    \end{align*}
    \begin{align*}
        (\lvert \mathbf{I}_{1,\,1}\rvert+\lvert \mathbf{I}_{1,\,p}\rvert )\cdot \lvert \mathbf{J}_{4}\rvert&\le Cv_{\varepsilon}^{p-1}\left(1+v_{\varepsilon}^{1-p}\right)\lvert {\D}^{2}\bu_{\varepsilon}\rvert_{\bg}\psi(v_{\varepsilon})\zeta\\ 
    &\le \frac{1}{4}{\mathcal A}_{\varepsilon}(x,\,t,\,\D\bu_{\varepsilon})({\D}^{2}\bu_{\varepsilon},\,{\D}^{2}\bu_{\varepsilon})\psi(v_{\varepsilon})\zeta+\frac{C}{\lambda_{0}}v_{\varepsilon}^{p}\left(1+v_{\varepsilon}^{1-p}\right)^{2}\psi(v_{\varepsilon})\zeta.
    \end{align*}
    Hence, from (\ref{Eq (Section 4): Weak form mid}) we have
    \begin{align*}
    &E_{0}+E_{1}+\frac{1}{2}E_{2}+\frac{1}{2}E_{3}
    \\ 
    &\le 
    C\left[\iint_{Q}v_{\varepsilon}^{p}\left(1+v_{\varepsilon}^{1-p} \right)^{2} \left(\psi(v_{\varepsilon})+\psi^{\prime}(v_{\varepsilon})v_{\varepsilon}\right)\zeta \,\d x\d t+\iint_{Q}v_{\varepsilon}^{p}\left(1+v_{\varepsilon}^{1-p}\right)\psi(v_{\varepsilon})\lvert\nabla\zeta\rvert\,\d x\d t\right]\\ 
    &\quad +\frac{1}{2}\iint_{Q}\partial_{t}\gamma_{\nu\sigma}\partial_{x_{\sigma}}u_{\varepsilon}^{j}\partial_{x_{\nu}}u_{\varepsilon}^{j}\psi(v_{\varepsilon})\zeta \,\d x\d t+ \frac{1}{8}\iint_{Q} {\mathcal C}_{\varepsilon}(x,\,t,\,\D\bu_{\varepsilon})({\mathbf V},\,{\mathbf V})\frac{\psi^{\prime}(v_{\varepsilon})}{v_{\varepsilon}} \zeta\,\d x\d t\\
    &\quad\quad  +\iint_{Q}f_{\varepsilon}^{j}\left[ \zeta\gamma_{\nu\sigma} \partial_{x_{\sigma}} u_{\varepsilon}^{j}\psi(v_{\varepsilon}) \right]_{x_{\nu}}\,\d x\d t.
    \end{align*} 
    It is easy to compute
    \[\frac{1}{8}\iint_{Q} {\mathcal C}_{\varepsilon}(x,\,t,\,\D\bu_{\varepsilon})({\mathbf V},\,{\mathbf V})\frac{\psi^{\prime}(v_{\varepsilon})}{v_{\varepsilon}} \zeta\,\d x\d t \le C\iint_{Q}v_{\varepsilon}^{p+1}\left(1+v_{\varepsilon}^{1-p} \right)\psi^{\prime}(v_{\varepsilon})\zeta\,\d x\d t,\]
    and 
    \begin{align*}
    &\iint_{Q}f_{\varepsilon}^{j}\left[ \zeta\gamma_{\nu\sigma} \partial_{x_{\sigma}} u_{\varepsilon}^{j}\psi(v_{\varepsilon}) \right]_{x_{\nu}}\,\d x\d t\\
    &\le C\iint_{Q}\lvert \buf_{\varepsilon}\rvert v_{\varepsilon}\psi(v_{\varepsilon})\lvert \nabla\zeta \rvert\,\d x\d t+C\iint_{Q}\lvert\buf_{\varepsilon}\rvert\lvert {\D}^{2}\bu_{\varepsilon}\rvert_{\bg}\psi(v_{\varepsilon})v_{\varepsilon}\zeta\,\d x\d t\\ 
    &\quad +C\iint_{Q}\lvert \buf_{\varepsilon}\rvert \lvert \nabla v_{\varepsilon}\rvert_{\bg}\psi^{\prime}(v_{\varepsilon})v_{\varepsilon}\zeta\,\d x\d t+C\iint_{Q}\lvert \buf_{\varepsilon}\rvert\psi(v_{\varepsilon})v_{\varepsilon} \zeta\,\d x\d t\\ 
    &\le \frac{1}{4}(E_{2}+E_{3})+\frac{C}{\lambda_{0}}\iint_{Q}\lvert \buf_{\varepsilon}\rvert^{2}v_{\varepsilon}^{2-p}\left(\psi(v_{\varepsilon})+\psi^{\prime}(v_{\varepsilon})v_{\varepsilon} \right)\zeta\,\d x\d t\\
    &\quad +C\iint_{Q}\left(v_{\varepsilon}^{p}+\lvert \buf_{\varepsilon}\rvert^{2}v_{\varepsilon}^{2-p} \right)\psi(v_{\varepsilon})\zeta\,\d x\d t
    +C\iint_{Q}\lvert \buf_{\varepsilon}\rvert v_{\varepsilon}\psi(v_{\varepsilon})\lvert \nabla\zeta \rvert\,\d x\d t
    \end{align*}
    by (\ref{Eq (Section 4): Positivity of E2 and E3}) and Young's inequality.
    Combining these estimates, we conclude (\ref{Eq (Section 4): Weak Form of v-epsilon}).
    \end{proof}
    To see the left-hand side of (\ref{Eq (Section 4): Weak Form of v-epsilon}), we observe the two basic and important properties of $\bu_{\varepsilon}$.
    Firstly, from $E_{0}$ and $E_{1}$, we can realize that the scalar-valued function $\Psi(v_{\varepsilon})$ is a subsolution to a certain parabolic equation.
    Secondly, from $E_{2}$ and $E_{3}$, we can deduce some energy estimates related to $\nabla v_{\varepsilon}$ and ${\D}^{2}\bu_{\varepsilon}$ respectively.

 \subsection{Estimates for subsolutions}
 In this subsection, we choose $\psi_{2,\,k}$ with suitable $k\ge \delta>0$.
 From Lemma \ref{Lemma (Section 4): Basic Weak Form}, it follows that the non-negative functions \((v_{\varepsilon}-k)_{+}^{2}+k^{2}\) and \((v_{\varepsilon}-\delta)_{+}^{2}\) are subsolutions to certain uniformly parabolic equations.
 For these subsolutions, we provide the Caccioppoli estimates (Lemmata \ref{Lemma (Section 4): Caccioppoli estimates for W-k}--\ref{Lemma (Section 4): Caccioppoli estimates for U-delta-epsilon}).

 Lemma \ref{Lemma (Section 4): Caccioppoli estimates for W-k} is to be used in showing local gradient bounds in Section \ref{Section: Gradient Bounds}.
 \begin{lemma}\label{Lemma (Section 4): Caccioppoli estimates for W-k}
     Fix a positive constant $k\in\lbrack 1,\,\infty)$.
     Let $\bu_{\varepsilon}$ be a weak solution to (\ref{Eq (Section 2): Approximate System}) in ${\mathcal Q}$, and fix $Q_{R}=I_{R}\times B_{R}=I_{R}(t_{0})\times B_{R}(x_{0})\Subset {\mathcal Q}$.
     If the function $W_{k}\coloneqq \sqrt{(v_{\varepsilon}-k)_{+}^{2}+k^{2}}$ satisfies $W_{k}\in L^{p+2\alpha}(Q_{R})\cap L^{2+2\alpha}(Q_{R})$ for some $\alpha\in\lbrack 0,\,\infty)$, then there hold $\eta W^{\alpha+1}\in L^{2,\,\infty}(Q_{R})$ and $\eta W_{k}^{\alpha+p/2}\in L^{2}(I_{R};\,W_{0}^{1,\,2}(B_{R}))$ for any non-negative function $\eta\in C_{\mathrm c}^{1}(Q_{R})$.
     Moreover, there exists a constant $C=C({\mathcal D})\in(0,\,\infty)$ such that
     \begin{align}\label{Eq (Section 4): Energy estimate for Moser Iteration}
        & \esssup_{\tau\in I_{R}}\int_{B_{R}\times \{\tau\}}\left(\eta W_{k}^{\alpha+1}\right)^{2} \,\d x\d t+\iint_{Q_{R}}\left\lvert \nabla \left(\eta W_{k}^{\alpha+p/2}\right) \right\rvert^{2}\,\d x\d t \nonumber\\ 
        &\le C(1+\alpha)^{2}\iint_{Q_{R}}\left(W_{k}^{p}+W_{k}^{2}+W_{k}^{2-p}\lvert \buf_{\varepsilon}\rvert^{2}\right)W_{k}^{2\alpha}\eta^{2}\,\d x\d t\nonumber\\
        &\quad  +C(1+\alpha)^{2}\iint_{Q_{R}}\left(W_{k}^{p}\lvert\nabla\eta\rvert^{2}+W_{k}^{2}\lvert\partial_{t}\eta^{2}\rvert\right)W_{k}^{2\alpha}\,\d x\d t.
    \end{align}
 \end{lemma}
 In the proof of Lemma \ref{Lemma (Section 4): Caccioppoli estimates for W-k}, we note that the function $W_{k}$ satisfies
 \begin{equation}\label{Eq (Section 5): v vs w}
     v_{\varepsilon}\le W_{k}\quad \text{in}\quad Q_{R}\quad \text{and} \quad W_{k}\le cv_{\varepsilon}\quad \text{in}\quad Q_{R}\cap \{v_{\varepsilon}>k\},
 \end{equation}
 where $c\in(1,\,\infty)$ is a universal constant.
 Therefore, for the integrands whose supports are contained in $\{v_{\varepsilon}>k\}$, we may replace $v_{\varepsilon}$ by $W_{k}$ if necessary.
 \begin{proof}
     We apply Lemma \ref{Lemma (Section 4): Basic Weak Form} with $\psi\coloneqq \psi_{2,\,k}$, so that we may choose $\Psi_{2,\,k}(v_{\varepsilon})=W_{k}^{2}$. 
     We test $\zeta\coloneqq \eta^{2}\phi \psi_{3,\,2\alpha,\,l}(W_{k})$ into (\ref{Eq (Section 4): Weak Form of v-epsilon}) with $l\in(k,\,\infty)$, where $\phi\colon \lbrack t_{0}-R^{2},\, t_{0}\rbrack\to\lbrack 0,\,1\rbrack$ is a non-increasing function satisfying $\phi(t_{0})=0$.
     Since $\psi_{2,\,k}(\sigma)+\sigma\psi_{2,\,k}^{\prime}(\sigma)=\chi_{\{\sigma>k\}}$ and $\lvert \nabla W_{k}\rvert\le \lvert \nabla v_{\varepsilon} \rvert$, we have
     \begin{align*}
         &E_{1}+\frac{1}{4}(E_{2}+E_{3})\\ 
         &\ge c(1+\alpha)\iint_{Q}\lvert \nabla W_{k}\rvert^{2}W_{k}^{p-2}W_{k,\,l}^{2\alpha}\eta^{2}\phi\,\d x\d t-C\iint_{Q}W_{k}^{p-2}W_{k,\,l}^{2\alpha}\lvert \nabla W_{k}\rvert\lvert\nabla\eta\rvert\eta\,\d x\d t\\ 
         &\ge \frac{c}{2}(1+\alpha)\iint_{Q}W_{k}^{p-2}W_{k,\,l}^{2\alpha}\eta^{2}\phi\,\d x\d t-C\iint_{Q_{R}}W_{k}^{p}W_{k,\,l}^{2\alpha}\lvert \nabla\eta\rvert^{2}\,\d x\d t
     \end{align*}
     for some $c=c({\mathcal D})\in(0,\,1)$ and $C=C({\mathcal D})\in(1,\,\infty)$, 
     where $W_{k,\,l}\coloneqq W_{k}\wedge l$.
     The right-hand side of (\ref{Eq (Section 4): Weak Form of v-epsilon}) is computed as 
     \begin{align*}
         C(E_{4}+E_{5}+E_{6})&\le C(1+\alpha)\iint_{Q}\left(W_{k}^{p}+W_{k}^{2}+W_{k}^{2-p}\lvert\buf_{\varepsilon}\rvert^{2}\right)W_{k,\,l}^{2\alpha-2}W_{k}^{2}\eta^{2}\phi\,\d x\d t\\ 
         &\quad +\frac{c}{4}(1+\alpha)\iint_{Q}W_{k}^{p-2}W_{k,\,l}^{2\alpha}\eta^{2}\phi\,\d x\d t
     \end{align*}
     by Young's inequality. Rewriting $\psi_{3,\,2\alpha,\,l}(W_{k})\partial_{t}W_{k}^{2}=2\partial_{t}\left(\Psi_{3,\,2\alpha,\,l}(W_{k})\right)$ and using (\ref{Eq (Section 2): Psi-3}), we obtain
     \begin{align*}
        &-\iint_{Q}W_{k,\,l}^{2\alpha+2}\eta^{2}\partial_{t}\phi\,\d x\d t+(1+\alpha)^{2}\iint_{Q}W_{k}^{p-2}\lvert \nabla W_{k}\rvert^{2}W_{k,\,l}^{2\alpha}\eta^{2}\phi\,\d x\d t\\ 
        &\le C(1+\alpha)^{2}\iint_{Q}\left(W_{k}^{p}\left[\eta^{2}+\lvert \nabla\eta \rvert^{2}\right]+W_{k}^{2}\left[\eta^{2}+\eta\lvert\partial_{t}\eta\rvert\right] \right)W_{k,\,l}^{2\alpha} \,\d x\d t\\ 
        &\quad +C(1+\alpha)^{2}\iint_{Q}\lvert\buf_{\varepsilon}\rvert^{2} W_{k}^{2-p}\eta^{2}\cdot W_{k}^{2}W_{k,\,l}^{2\alpha-2}\,\d x\d t.
     \end{align*}
     We note that the last integral makes sense by the inclusions $\buf_{\varepsilon}\in L^{\infty}(\Omega_{T})^{N}$ and $W_{k}\in L^{p+2\alpha}(Q)\cap L^{2+2\alpha}(Q)$.
     Suitably choosing $\phi=\phi(t)$, and letting $l\to\infty$, we easily conclude (\ref{Eq (Section 4): Energy estimate for Moser Iteration}) by the monotone convergence theorem.
 \end{proof}
 Lemma \ref{Lemma (Section 4): Caccioppoli estimates for U-delta-epsilon} states that $U_{\delta,\,\varepsilon}\coloneqq \lvert \G_{2\delta,\,\varepsilon}(\D\bu_{\varepsilon}) \rvert_{\bg}^{2}$ belongs to a parabolic De Giorgi class. This result is used in Section \ref{Section: Non-Degenerate}.
 \begin{lemma}\label{Lemma (Section 4): Caccioppoli estimates for U-delta-epsilon}
     Let all of the assumptions in Proposition \ref{Proposition (Section 4): Degenerate Case}, except (\ref{Eq (Section 4): Measure Assumption: Degenerate}), be in force.
     Fix a parabolic cylinder $Q=B_{R}(x_{0})\times (\tau_{0},\,\tau_{1} \rbrack \subset Q_{\rho}(x_{0},\,t_{0})$.
     We also fix non-negative functions $\eta=\eta(x,\,t)\in C_{\mathrm c}^{1}(Q)$ and $\widetilde{\eta}=\widetilde{\eta}(x) \in C_{\mathrm c}^{1}(B_{R})$.
     Then, for any $k\in(0,\,\infty)$, we have 
     \begin{align}\label{Eq (Section 4): De Girogi Energy 1}
         &\esssup_{\tau\in(\tau_{0},\,\tau_{1})}\iint_{B_{R}}(U_{\delta,\,\varepsilon}-k)_{+}^{2}\eta^{2}\,\d x\d t+\iint_{Q}\lvert \nabla(\eta (U_{\delta,\,\varepsilon}-k)_{+}) \rvert^{2}\,\d x\d t\nonumber \\ 
         &\le C({\mathcal D},\,\delta,\,M)\left[\iint_{Q}(U_{\delta,\,\varepsilon}-k)_{+}^{2}\left(\eta^{2}+\lvert\nabla\eta\rvert^{2}+\lvert\partial_{t}\eta^{2}\rvert\right)\,\d x\d t+\mu^{4}\iint_{A_{k}}\left(1+\lvert\buf_{\varepsilon}\rvert^{2}\right)\eta^{2}\,\d x\d t\right],
     \end{align}
     where $A_{k}\coloneqq \{(x,\,t)\in Q\mid U_{\delta,\,\varepsilon}(x,\,t)>k\}$, and 
     \begin{align}\label{Eq (Section 4): De Giorgi Energy 2}
         &\int_{B_{R}\times \{\tau\}}(\widetilde{\eta}(U_{\delta,\,\varepsilon}-k)_{+})^{2}\,\d x  -\int_{B_{R}\times \{\tau_{0}\}} (\widetilde{\eta}(U_{\delta,\,\varepsilon}-k)_{+})^{2}\,\d x \nonumber \\ 
         &\le C({\mathcal D},\,\delta,\,M)\left[\iint_{Q}(U_{\delta,\,\varepsilon}-k)_{+}^{2}\left(\widetilde{\eta}^{2}+\lvert\nabla\widetilde{\eta}\rvert^{2} \right)\,\d x\d t+\mu^{4}\iint_{A_{k}}\left(1+\lvert\buf_{\varepsilon}\rvert^{2} \right)\widetilde{\eta}^{2}\,\d x\d t \right]
     \end{align}
     for a.e.~$\tau\in(\tau_{0},\,\tau_{1})$.
 \end{lemma}
 \begin{proof}
     We apply Lemma \ref{Lemma (Section 4): Basic Weak Form} with $\psi=\psi_{2,\,\delta}$, so that we may take $\Psi_{2,\,\delta}(v_{\varepsilon})=U_{\delta,\,\varepsilon}$.
     Deleting the non-negative integrals $E_{2}$ and $E_{3}$, and utilizing (\ref{Eq (Section 4): Bound Assumption of v-epsilon}) and (\ref{Eq (Section 4): delta vs mu}), we obtain
     \begin{align*}
         &-\iint_{Q}U_{\delta,\,\varepsilon}\partial_{t}\zeta\,\d x\d t+\iint_{Q}{\mathcal C}_{\varepsilon}(x,\,t,\,\D\bu_{\varepsilon})(\nabla U_{\delta,\,\varepsilon},\,\nabla\zeta)\,\d x\d t\\ 
         &\le C({\mathcal D},\,\delta,\,M)\mu^{4}\left[\iint_{Q}(1+\lvert \buf_{\varepsilon}\rvert^{2})\zeta\,\d x\d t+\iint_{Q}\lvert\buf_{\varepsilon}\rvert\lvert\nabla\zeta\rvert\,\d x\d t \right],
     \end{align*}
     where we note that all of the integrals are supported in $\{\delta\le v_{\varepsilon}\le M\}$, where the matrix ${\mathcal C}_{\varepsilon}(x,\,t,\,\D\bu_{\varepsilon})$ becomes uniformly elliptic in the classical sense.
     With this in mind, we test $\zeta\coloneqq (U_{\delta,\,\varepsilon}-k)_{+}\eta^{2}\phi$ into this weak formulation, where $\phi\colon \lbrack \tau_{0},\,\tau_{1} \rbrack\to\lbrack 0,\,1\rbrack$ is a non-increasing function satisfying $\phi(\tau_{1})=0$.
     Carrying out standard absorbing arguments, we have
     \begin{align*}
        &-\iint_{Q}(U_{\delta,\,\varepsilon}-k)_{+}^{2}\eta^{2}\partial_{t}\phi\,\d x\d t+\iint_{Q}\lvert \nabla (\eta(U_{\delta,\,\varepsilon}-k)_{+})\rvert^{2}\phi\,\d x\d t\\ 
        &\le C({\mathcal D},\,\delta,\,M)\left[\iint_{Q}(U_{\delta,\,\varepsilon}-k)_{+}^{2}\left(\eta^{2}+\lvert\nabla\eta\rvert^{2}+\lvert\partial_{t}\eta^{2}\rvert\right)\,\d x\d t+\mu^{4}\iint_{Q}\left(1+\lvert\buf_{\varepsilon}\rvert^{2}\right)\eta^{2}\,\d x\d t\right].
    \end{align*}
     Suitably choosing $\phi$, we easily deduce (\ref{Eq (Section 4): De Girogi Energy 1}).
     To prove (\ref{Eq (Section 4): De Giorgi Energy 2}), we define a piecewise linear function $\phi_{\widetilde{\varepsilon}}\colon \lbrack\tau_{0},\,\tau_{1} \rbrack\to \lbrack 0,\,1\rbrack$ as $\phi_{\widetilde{\varepsilon}}(t)\coloneqq (\min\{\,1,\,(t-\tau_{0})/\widetilde{\varepsilon},\,-(t-\tau)/\widetilde{\varepsilon} \,\})_{+}$, which converges to the characteristic function $\chi_{(\tau_{0},\,\tau)}$ as $\widetilde{\varepsilon}\to 0$.
     Testing $\zeta\coloneqq (U_{\delta,\,\varepsilon}-k)_{+}\widetilde{\eta}^{2}\phi_{\widetilde{\varepsilon}}$, and making absorptions yield
     \begin{align*}
        &-\iint_{Q}(U_{\delta,\,\varepsilon}-k)_{+}^{2}\widetilde{\eta}^{2}\partial_{t}\phi_{\widetilde{\varepsilon}}\,\d x\d t+\iint_{Q}\lvert \nabla (\widetilde{\eta}(U_{\delta,\,\varepsilon}-k)_{+})\rvert^{2}\phi_{\widetilde{\varepsilon}}\,\d x\d t\\ 
        &\le C({\mathcal D},\,\delta,\,M)\left[\iint_{Q}(U_{\delta,\,\varepsilon}-k)_{+}^{2}\left(\widetilde{\eta}^{2}+\lvert\nabla \widetilde{\eta}\rvert^{2} \right)\,\d x\d t+\mu^{4}\iint_{Q}\left(1+\lvert\buf_{\varepsilon}\rvert^{2}\right)\widetilde{\eta}^{2}\,\d x\d t\right].
    \end{align*}
     Deleting the second integral on the left-hand side, and letting $\widetilde{\varepsilon}\to 0$, we obtain (\ref{Eq (Section 4): De Giorgi Energy 2}).
 \end{proof}

 \subsection{Energy estimates}
 In this subsection, we choose $\psi$ as $\psi_{4,\,\alpha,\,l}$ or $\psi_{5}$ to deduce local $L^{2}$-estimates for second-order spatial derivatives (Lemmata \ref{Lemma (Section 4): Energy estimates for higher integrability of v-epsilon}--\ref{Lemma (Section 4): L2-estimates for spatial derivatives}).
 Here we infer an inequality
 \begin{equation}\label{Eq (Section 4): nabla v-epsilon vs Hessian}
     \lvert \nabla v_{\varepsilon}\rvert\le C({\mathcal D})\left(\lvert \D^{2}\bu_{\varepsilon}\rvert+v_{\varepsilon} \right),
 \end{equation}
 which is easy to deduce by (\ref{Eq (Section 4): F and V}).

 Lemma \ref{Lemma (Section 4): Energy estimates for higher integrability of v-epsilon} is used to prove higher integrability of a gradient when $p\in(1,\,2)$.
 \begin{lemma}\label{Lemma (Section 4): Energy estimates for higher integrability of v-epsilon}
     Let $p\in(1,\,2)$, and $u_{\varepsilon}$ be a weak solution to (\ref{Eq (Section 2): Approximate System}) in ${\mathcal Q}\Subset \Omega_{T}$.
     Fix a subcylinder $Q=B\times I\Subset {\mathcal Q}$ and a cut-off function $\eta\in C_{\mathrm c}^{1}(Q;\,\lbrack 0,\,1\rbrack)$.
     If $\buf_{\varepsilon}\in L^{\infty}({\mathcal Q})^{N}$ and $v_{\varepsilon}\in L^{p+\alpha}(Q)$ hold for some $\alpha\in\lbrack 0,\,\infty)$, then we have 
     \begin{align}\label{Eq (Section 4); Qualitative Higher Integrability}
         &\iint_{Q}v_{\varepsilon}^{p-2}\left(\lvert \D\bu_{\varepsilon}\rvert^{2}\psi_{4,\,\alpha,\,l}(v_{\varepsilon})+\lvert \nabla v_{\varepsilon}\rvert^{2}\psi_{4,\,\alpha,\,l}^{\prime}(v_{\varepsilon})v_{\varepsilon}\right)\eta^{2}\,\d x\d t\nonumber
         \\ & \le C \iint_{Q}\left[\psi_{4,\,\alpha,\,l}(v_{\varepsilon})v_{\varepsilon}^{2}\left(\eta^{2}+\lvert\nabla\eta \rvert^{2} +\lvert\partial_{t}\eta^{2}\rvert \right)+v_{\varepsilon}^{\alpha+p}\left(\eta^{2}+\lvert\nabla\eta\rvert^{2}\right)+\lvert \buf_{\varepsilon}\rvert^{2}v_{\varepsilon}^{\alpha+2-p}\eta^{2}\right]\,\d x\d t
     \end{align}
     for all $l\in(1,\,\infty)$ with $C=C({\mathcal D},\,\alpha)\in(1,\,\infty)$.
     Moreover, if $\buf_{\varepsilon}$ satisfies (\ref{Eq (Section 4): Uniform Bounds F and U}) and $v_{\varepsilon}\in L^{2+\alpha}(Q)$ holds with $\alpha\in\lbrack 0,\,(pq-4)/2)$, then we have      
     \begin{align}\label{Eq (Section 4); Quantitative Higher Integrability}
        &\iint_{Q}v_{\varepsilon}^{p-2}\left(\lvert \D\bu_{\varepsilon}\rvert^{2}\psi_{4,\,\alpha}(v_{\varepsilon})+\lvert \nabla v_{\varepsilon}\rvert^{2}\psi_{4,\,\alpha}^{\prime}(v_{\varepsilon})v_{\varepsilon}\right)\eta^{2}\,\d x\d t\nonumber
        \\ & \le C\left[ \iint_{Q}\left[\psi_{4,\,\alpha}(v_{\varepsilon})v_{\varepsilon}^{2}\left(\eta^{2}+\lvert\nabla\eta \rvert^{2}+ \lvert\partial_{t}\eta^{2}\rvert \right)+v_{\varepsilon}^{\alpha+p}\left(\eta^{2}+\lvert\nabla\eta\rvert^{2}\right)\right]\,\d x\d t+F^{\frac{2(\alpha+2)}{p}}+1\right]
     \end{align}
     with $C=C({\mathcal D},\,\alpha)\in(1,\,\infty)$.
 \end{lemma}
 \begin{proof}
     We apply Lemma \ref{Lemma (Section 4): Basic Weak Form} with $\psi=\psi_{4,\,\alpha,\,l}$, and $\zeta\coloneqq \eta^{2}\phi$.
     For this choice, we that the supports of the integrals are contained in $\{v_{\varepsilon}>1\}$.
     By (\ref{Eq (Section 2): Control of Psi}),  (\ref{Eq (Section 2): Growth of psi-4}),  we compute (\ref{Eq (Section 4): Weak Form of v-epsilon}) as follows;
     \begin{align*}
         &-\iint_{Q}\Psi_{4,\,\alpha,\,l}(v_{\varepsilon})\eta^{2}\partial_{t}\phi\,\d x\d t+\frac{\lambda_{0}}{4}\iint_{Q}v_{\varepsilon}^{p-2} \left(\lvert \D^{2}\bu_{\varepsilon}\rvert_{\bg}^{2}\psi_{4,\,\alpha,\,l}(v_{\varepsilon})+\lvert\nabla v_{\varepsilon}\rvert_{\bg}^{2}\psi_{4,\,\alpha,\,l}^{\prime}(v_{\varepsilon})v_{\varepsilon} \right)\eta^{2}\phi\,\d x\d t \nonumber \\ 
         &\le C\iint_{Q}\psi_{4,\,\alpha,\,l}(v_{\varepsilon})v_{\varepsilon}^{2} \left(\eta^{2}+\lvert \partial_{t}\eta^{2}\rvert\right) \phi\,\d x\d t+ C\iint_{Q} v_{\varepsilon}^{p-2}\left(\lvert \D^{2} \bu_{\varepsilon}\rvert_{\bg}+v_{\varepsilon}\right)\lvert\nabla\eta\rvert\eta\psi_{4,\,\alpha,\,l}(v_{\varepsilon})v_{\varepsilon}\phi \, \d x\d t\nonumber\\ 
         &\quad +C\iint_{Q}\left(v_{\varepsilon}^{p}+\lvert \buf_{\varepsilon}\rvert^{2}v_{\varepsilon}^{2-p}\right)\left(\psi_{4,\,\alpha,\,l}(v_{\varepsilon})+\psi_{4,\,\alpha,\,l}^{\prime}(v_{\varepsilon})v_{\varepsilon} \right)\eta^{2}\phi\,\d x\d t \\ 
         &\quad \quad +C\iint_{Q}\left(v_{\varepsilon}^{p}+\lvert \buf_{\varepsilon}\rvert v_{\varepsilon} \right) \psi_{4,\,\alpha,\,l}(v_{\varepsilon})\lvert\nabla\eta\rvert\eta\phi \,\d x\d t\\ 
         &\le \frac{\lambda_{0}}{8}\iint_{Q}v_{\varepsilon}^{p-2}\lvert \D^{2}\bu_{\varepsilon}\rvert_{\bg}^{2}\psi_{4,\,\alpha,\,l}(v_{\varepsilon})\eta^{2}\phi \,\d x\d t+C({\mathcal D})\iint_{Q}\psi_{4,\,\alpha,\,l}(v_{\varepsilon})v_{\varepsilon}^{2} \left(\eta^{2}+\lvert \partial_{t}\eta^{2}\rvert+\lvert \nabla\eta\rvert^{2} \right) \phi\,\d x\d t\\ 
         &\quad +C({\mathcal D},\,\alpha)\left[\iint_{Q}v_{\varepsilon}^{p+\alpha}\left(\eta^{2}+\lvert \nabla \eta\rvert^{2}\right)\phi\,\d x\d t+\iint_{Q}\lvert\buf_{\varepsilon}\rvert^{2}v_{\varepsilon}^{\alpha+2-p}\eta^{2}\phi\,\d x\d t\right]
     \end{align*}
     Recalling (\ref{Eq (Section 1): Matrix Gamma}), we have
     \begin{align}\label{Eq (Section 4): MID ENERGY}
         &-\iint_{Q}\Psi_{4,\,\alpha,\,l}(v_{\varepsilon})\eta^{2}\partial_{t}\phi\,\d x\d t+\iint_{Q}v_{\varepsilon}^{p-2} \left(\lvert \D^{2}\bu_{\varepsilon}\rvert^{2}\psi_{4,\,\alpha,\,l}(v_{\varepsilon})+\lvert\nabla v_{\varepsilon}\rvert^{2}\psi_{4,\,\alpha,\,l}^{\prime}(v_{\varepsilon})v_{\varepsilon} \right)\eta^{2}\phi\,\d x\d t \nonumber\\ 
         &\le C({\mathcal D},\,\alpha)\iint_{Q}\left[\psi_{4,\,\alpha,\,l}(v_{\varepsilon})v_{\varepsilon}^{2}\left(\eta^{2}+\lvert\nabla\eta \rvert^{2} +\lvert\partial_{t}\eta^{2}\rvert \right)+v_{\varepsilon}^{\alpha+p}\left(\eta^{2}+\lvert\nabla\eta\rvert^{2}\right)+\lvert \buf_{\varepsilon}\rvert^{2}v_{\varepsilon}^{\alpha+2-p}\eta^{2}\right]\,\d x\d t.
     \end{align}
     Deleting the first non-negative integral of (\ref{Eq (Section 4): MID ENERGY}), and suitably choosing $\phi$, we easily conclude (\ref{Eq (Section 4); Qualitative Higher Integrability}).
     When $v_{\varepsilon}\in L^{2+\alpha}(Q)$, then we may let $l\to \infty$ in (\ref{Eq (Section 4): MID ENERGY}). 
     The standard choices of $\phi$ imply
     \begin{align*}
         &\esssup_{\tau\in I}\int_{B\times\{\tau\}}(v_{\varepsilon}-1)_{+}^{\alpha+2}\eta^{2} \,\d x\d t+\iint_{Q}v_{\varepsilon}^{p-2}\left(\lvert \D\bu_{\varepsilon}\rvert^{2}\psi_{4,\,\alpha}(v_{\varepsilon})+\lvert \nabla v_{\varepsilon}\rvert^{2}\psi_{4,\,\alpha}^{\prime}(v_{\varepsilon})v_{\varepsilon}\right)\eta^{2}\,\d x\d t \\ 
         &\le C({\mathcal D},\,\alpha) \iint_{Q}\left[\psi_{4,\,\alpha}(v_{\varepsilon})v_{\varepsilon}^{2}\left(\eta^{2}+\lvert\nabla\eta \rvert^{2} +\lvert\partial_{t}\eta^{2}\rvert \right)+v_{\varepsilon}^{\alpha+p}\left(\eta^{2}+\lvert\nabla\eta\rvert^{2}\right)+\lvert \buf_{\varepsilon}\rvert^{2}v_{\varepsilon}^{\alpha+2-p}\eta^{2}\right]\,\d x\d t,
     \end{align*}
     where (\ref{Eq (Section 2): Psi-4}) is also used.
     By H\"{o}lder's inequality and Young's inequality, the last integral is estimated as follows;
     \begin{align*}
         C({\mathcal D},\,\alpha)\iint_{Q}\lvert \buf_{\varepsilon}\rvert^{2}v_{\varepsilon}^{\alpha+2-p}\eta^{2} \, \d x\d t&\le C({\mathcal D},\,\alpha)\left[1+\esssup_{\tau\in I}\int_{B\times \{\tau\}}(v_{\varepsilon}-1)_{+}^{\alpha+2}\eta^{2}\,\d x \right]^{1-\frac{p}{\alpha+2}}F^{2}\\ 
         &\le \esssup_{\tau\in I}\int_{B\times\{\tau\}}(v_{\varepsilon}-1)_{+}^{\alpha+2}\eta^{2} \,\d x\d t+C({\mathcal D},\,\alpha)\left(F^{\frac{2(\alpha+2)}{p}}+1 \right),
     \end{align*}
    which computations make sence since $(\alpha+2-p) \widehat{q} <\alpha+2$. From these estimates, we obtain (\ref{Eq (Section 4); Quantitative Higher Integrability}).
 \end{proof}
 Lemma \ref{Lemma (Section 4): L2-estimates for spatial derivatives} provides local $L^{2}$-estimates of $\D({\mathbf G}_{p,\,\varepsilon}(\D\bu_{\varepsilon}))=\D(v_{\varepsilon}^{p-1}\D\bu_{\varepsilon})$, where $\mathbf{G}_{p,\,\varepsilon}$ is given in Lemma \ref{Lemma (Section 2): G-p-epsilon}. 
 This result plays an important role as a starting point of Section \ref{Section: Non-Degenerate}.
 \begin{lemma}\label{Lemma (Section 4): L2-estimates for spatial derivatives}
     Let all of the assumptions of Proposition \ref{Proposition (Section 4): Non-Degenerate Case} be in force, except (\ref{Eq (Section 4): Measure Assumption: Non-Degenerate}).
     Then, we have
     \begin{equation}\label{Eq (Section 4): Energy 1}
         \fiint_{Q_{\sigma\rho}}\left\lvert \D\left(v_{\varepsilon}^{p-1}\D\bu_{\varepsilon} \right) \right\rvert^{2}\,\d x\d t\le \frac{C\mu^{2p}}{\sigma^{n+2}\rho^{2}}\left[\frac{1}{(1-\sigma)^{2}}+\left(1+F^{2}\right)\rho^{2\beta}\right],
     \end{equation}
     \begin{equation}\label{Eq (Section 4): Energy 2}
         \lvert Q_{\sigma\rho}\rvert^{-1}\iint_{S_{\sigma\rho}}\left\lvert \D\left(v_{\varepsilon}^{p-1}\D\bu_{\varepsilon} \right) \right\rvert^{2}\,\d x\d t\le \frac{C\mu^{2p}}{\sigma^{n+2}\rho^{2}}\left[\frac{\nu}{(1-\sigma)^{2}}+\frac{\left(1+F^{2}\right)\rho^{2\beta}}{\nu} \right],
     \end{equation}
     for any $\sigma\in(0,\,1)$ and $\nu\in(0,\,1/4)$, where $C=C({\mathcal D},\,\delta,\,M)\in(1,\,\infty)$ is a constant.
     \end{lemma}
     \begin{proof}
     Recalling (\ref{Eq (Section 4): nabla v-epsilon vs Hessian}), we compute $\left\lvert \D\left({\mathbf G}_{p,\,\varepsilon}(\D\bu_{\varepsilon})\right) \right\rvert^{2}\le Cv_{\varepsilon}^{2(p-1)}\left(\lvert \D^{2}\bu_{\varepsilon}\rvert_{\bg}^{2}+v_{\varepsilon}^{2}\right)$.
     To prove (\ref{Eq (Section 4): Energy 1})--(\ref{Eq (Section 4): Energy 2}), we apply Lemma \ref{Lemma (Section 4): Basic Weak Form} with $Q=Q_{\rho}(x_{0},\,t_{0})$ and $\psi(\sigma)\coloneqq \sigma^{p}{\widetilde \psi}(\sigma)$, where $\widetilde{\psi}$ is chosen as either ${\widetilde \psi}(\sigma)\equiv 1$ or ${\widetilde \psi}(\sigma)=(\sigma-\delta-(1-2\nu)\mu)_{+}^{2}$.
     We choose $\zeta\coloneqq \eta^{2}\phi_{\widetilde{\varepsilon}}$ as a test function into (\ref{Eq (Section 4): Weak Form of v-epsilon}), where $\phi$ is chosen as $\phi_{\widetilde{\varepsilon}}(t)=\min \{\,1,\,-(t-t_{0})/\widetilde{\varepsilon} \,\}\,(t_{0}-\rho^{2}\le t\le t_{0})$ with $\widetilde{\varepsilon}$ being sufficiently small.
     Noting $\partial_{t}\phi_{\widetilde{\varepsilon}}\le 0$, we deduce
     \begin{align*}
     &\frac{1}{2}(E_{2}+E_{3}) \\ 
     &\le \iint_{Q}\Psi(v_{\varepsilon})\phi_{\widetilde{\varepsilon}}\lvert \partial_{t}\eta^{2}\rvert \,\d x\d t+2\iint_{Q}\left\lvert {\mathcal C}_{\varepsilon}(x,\,t,\,\D\bu_{\varepsilon})(\nabla v_{\varepsilon},\,\nabla\eta ) \right\rvert\eta\psi(v_{\varepsilon})v_{\varepsilon}\phi_{\widetilde{\varepsilon}}\,\d x\d t \\ 
     &\quad +C(E_{4}+E_{6})+C\iint_{Q}\left(v_{\varepsilon}^{p}\left(1+v_{\varepsilon}^{1-p} \right)+\lvert\buf_{\varepsilon}\rvert v_{\varepsilon} \right)\psi(v_{\varepsilon})\eta\lvert\nabla\eta\rvert\phi_{\widetilde{\varepsilon}}\,\d x\d t\\
     &\le \iint_{Q}\Psi(v_{\varepsilon})\lvert \partial_{t}\eta^{2}\rvert \,\d x\d t+\frac{1}{2}E_{2}+2\iint_{Q}{\mathcal C}_{\varepsilon}(x,\,t,\,\D\bu_{\varepsilon})(\nabla\eta,\,\nabla\eta)\frac{\psi(v_{\varepsilon})^{2}v_{\varepsilon}}{\psi^{\prime}(v_{\varepsilon})}\,\d x\d t\\ 
     &\quad +C\iint_{Q}v_{\varepsilon}^{p-1}\frac{\psi(v_{\varepsilon})^{2}}{\psi^{\prime}(v_{\varepsilon})}\lvert \nabla\eta\rvert^{2}\,\d x\d t+C\iint_{Q}\left(v_{\varepsilon}^{2}+v_{\varepsilon}^{p}\left(1+v_{\varepsilon}^{1-p} \right)^{2}+\lvert\buf_{\varepsilon}\rvert^{2}v_{\varepsilon}^{2-p} \right)\left(\psi(v_{\varepsilon})+\psi^{\prime}(v_{\varepsilon})v_{\varepsilon}\right)\,\d x\d t,
     \end{align*}
     where $\Psi$ is given by (\ref{Eq (Section 2): Def of Psi}) with $C=0$.
     By our choice of $\psi$, the identities 
     \[\psi(v_{\varepsilon})+\psi^{\prime}(v_{\varepsilon})v_{\varepsilon}=v_{\varepsilon}^{p}\left((p+1)\widetilde\psi(v_{\varepsilon})+v_{\varepsilon}{\widetilde \psi}^{\prime}(v_{\varepsilon})\right),\quad  \text{and}\quad \psi(v_{\varepsilon})^{2}\psi^{\prime}(v_{\varepsilon})^{-1}=\frac{v_{\varepsilon}^{p+1}{\widetilde \psi}(v_{\varepsilon})^{2}}{p\widetilde\psi(v_{\varepsilon})+v_{\varepsilon}{\widetilde \psi}^{\prime}(v_{\varepsilon})} \]
     are easily checked.
     Letting $\widetilde{\varepsilon}\to 0$, we have 
     \begin{align*}
     &\frac{\lambda_{0}}{2}\iint_{Q_{\rho}}v_{\varepsilon}^{2p-2}\lvert {\D}^{2}\bu_{\varepsilon}\rvert_{\bg}^{2}{\widetilde\psi}(v_{\varepsilon})\eta^{2} \,\d x\d t\\&\le 
     \iint_{Q_{\rho}}\Psi(v_{\varepsilon})\lvert \partial_{t}\eta^{2}\rvert \,\d x\d t
     +C\iint_{Q_{\rho}}\frac{(v_{\varepsilon}^{p-1}+1)v_{\varepsilon}^{p+1}{\widetilde \psi}(v_{\varepsilon})^{2}}{p{\widetilde \psi}(v_{\varepsilon})+v_{\varepsilon}{\widetilde \psi}^{\prime}(v_{\varepsilon})}\lvert\nabla\eta\rvert^{2} \,\d x\d t
     \\
     &\quad 
     +C\iint_{Q_{\rho}}\left(v_{\varepsilon}^{p+2}+v_{\varepsilon}^{2p}+\left(1+\lvert\buf_{\varepsilon}\rvert^{2} \right)v_{\varepsilon}^{2} \right)\left((p+1){\widetilde \psi}(v_{\varepsilon})+v_{\varepsilon}{\widetilde \psi}^{\prime}(v_{\varepsilon}) \right)\,\d x\d t.
     \end{align*}
     Choosing ${\widetilde\psi}(\sigma)\equiv 1$ implies $\Psi(\sigma)\equiv \sigma^{p+2}/(p+2)$.
     Since $v_{\varepsilon}^{l}\le (2\mu)^{l}\le C_{\delta,\,M,\,l}\mu^{2p}$ holds for each fixed $l\in(0,\,\infty)$, we compute
     \begin{align*}
         \iint_{Q_{\sigma\rho}}\left\lvert \D\left( v_{\varepsilon}^{p-1}\D\bu_{\varepsilon}\right) \right\rvert^{2}\,\d x\d t&
         \le C\mu^{2p}\left[\frac{\lvert Q_{\rho}\rvert}{(1-\sigma)^{2}\rho^{2}}+\iint_{Q_{\rho}}\left(1+\lvert\buf_{\varepsilon}\rvert^{2} \right)\,\d x\d t \right]\\ 
         &\le \frac{C\mu^{2p}}{\rho^{2}}\left[\frac{1}{(1-\sigma)^{2}}+(1+F^{2})\rho^{2\beta} \right]\lvert Q_{\rho}\rvert
     \end{align*}
     by H\"{o}lder's inequality.
     Dividing the resulting inequality by $\lvert Q_{\sigma\rho}\rvert=\sigma^{n+2}\lvert Q_{\rho}\rvert$ completes the proof of (\ref{Eq (Section 4): Energy 1}).
     Choosing $\widetilde\psi(\sigma)=(\sigma-\delta-(1-2\nu)\mu)_{+}^{2}$ yields 
     \[\Psi(v_{\varepsilon})\le (\mu+\delta)^{p+1} \int_{0}^{\mu+\delta}\,(\tau-\delta-(1-2\nu)\mu)_{+}^{2} \d \tau\le C(p,\,\delta,\,M)\nu^{3}\mu^{2p+2}.\]
     Since $\widetilde{\psi}(v_{\varepsilon})\ge (\nu\mu)^{2}$ holds in $S_{\rho}$ and $(1-2\nu)\mu>\mu/2$ follows from $\nu\in(0,\,1/4)$, we have
     \begin{align*}
         (\nu\mu)^{2}\iint_{S_{\sigma\rho}}\left\lvert \D\left(v_{\varepsilon}^{p-1}\D\bu_{\varepsilon} \right) \right\rvert^{2}\,\d x\d t&\le C\mu^{2p}\left[\frac{\nu^{3}\lvert Q_{\rho}\rvert}{(1-\sigma)^{2}\rho^{2}}+\nu\iint_{Q_{\rho}}\left(1+\lvert\buf_{\varepsilon}\rvert^{2} \right)\,\d x\d t \right]\\ 
         &\le \frac{C\mu^{2p+2}}{\rho^{2}}\left[\frac{\nu^{3}}{(1-\sigma)^{2}}+\nu\left(1+F^{2}\right)\rho^{2\beta} \right]\lvert Q_{\rho}\rvert,
     \end{align*}
     from which (\ref{Eq (Section 4): Energy 2}) is concluded.
    \end{proof}

\section{Gradient Bounds}\label{Section: Gradient Bounds}
In Section \ref{Section: Gradient Bounds}, we aim to prove Theorem \ref{Theorem (Section 4): Gradient bounds} from Lemmata \ref{Lemma (Section 4): Caccioppoli estimates for W-k} and \ref{Lemma (Section 4): Energy estimates for higher integrability of v-epsilon}.
\subsection{Higher integrability estimates for $p\in(1,\,2)$}
For $p\in(1,\,2)$, we use (\ref{Eq (Section 4): Weak Sol Bound for Singular p}) and Lemma \ref{Lemma (Section 4): Energy estimates for higher integrability of v-epsilon} to deduce a higher integrability lemma (Lemma \ref{Lemma (Section 5): Higher Integrability of v-epsilon}) as a preliminary.
For this topic, we infer to \cite[Lemma 7.5]{BDGLS} and \cite[Lemmata 3.9--3.10]{Strunk preprint}, which provide similar results for parabolic $p$-Laplace problems.
\begin{lemma}\label{Lemma (Section 5): Higher Integrability of v-epsilon}
    Let $\buf_{\varepsilon}\in L^{\infty}({\mathcal Q})^{N}$ satisfy (\ref{Eq (Section 4): Uniform Bounds F and U}) for some $F\in(0,\,\infty)$.
    Let $\bu_{\varepsilon}$ be a weak solution to (\ref{Eq (Section 2): Approximate System}) in $Q=Q_{R}\Subset {\mathcal Q}$ with $p\in(1,\,2)$.
    Then, $v_{\varepsilon}\in L_{\mathrm{loc}}^{\pi}(Q)$ for any $\pi\in(p,\,\infty)$.
    Moreover, for any $\pi\in (p,\,pq/2)$ and $\theta\in(0,\,1)$, we have 
    \begin{equation}\label{Eq (Section 5): Quantitative Higher Integrability}
        \iint_{Q_{\theta R}}v_{\varepsilon}^{\pi}\,{\mathrm d}x{\mathrm d}t\le C({\mathcal D},\,F,\,\pi,\,\theta)\left[\iint_{Q_{R}}\left(v_{\varepsilon}^{p}+1\right)\,{\mathrm d}x{\mathrm d}t+1\right]
    \end{equation}
\end{lemma}

\begin{proof}
    From $\buf_{\varepsilon}\in L^{\infty}({\mathcal Q})^{N}$ and (\ref{Eq (Section 4); Qualitative Higher Integrability}), we firstly prove $v_{\varepsilon}\in L_{\mathrm{loc}}^{\pi}(Q)$ for any $\pi\in(p,\,\infty)$.
    Let $v_{\varepsilon}\in L^{p+\alpha}(Q_{\widehat{R}})$ for some $\alpha\in\lbrack 0,\,\infty)$ and ${\widehat R}\in (0,\,R\rbrack$.
    Then, we would like to show
    \begin{equation}\label{Eq (Section 5): Qualitative Higher Integrability}
        \iint_{Q_{\theta R}}\psi_{4,\,\alpha,\,l}(v_{\varepsilon})v_{\varepsilon}^{2}\, \d x\d t  \le \frac{C({\mathcal D},\,\alpha,\,\lVert\buf_{\varepsilon} \rVert_{L^{\infty}(Q)})}{[(1-\theta)R]^{2}} \iint_{Q_{R}}\left(v_{\varepsilon}^{p+\alpha}+1\right)  \,\d x\d t
    \end{equation}
    for any $\theta\in(0,\,1)$.
    Arbitrarily fix $\theta\le \theta_{1}<\theta_{2}\le1$, and suitably choose a cut-off function $\eta\in C_{\mathrm c}^{1}(Q_{\theta_{2}R};\,\lbrack 0,\,1\rbrack)$.
    Then, integrating by parts with respect to the space variable, we have
    \begin{align*}
        &\iint_{Q_{\theta_{1}R}}\psi_{4,\,\alpha,\,l}(v_{\varepsilon})v_{\varepsilon}^{2}\,\d x\d t\le \iint_{Q_{\theta_{2}R}}v_{\varepsilon}^{\alpha}\, \d x\d t+ \iint_{Q_{\theta_{2}R}}\eta^{2}\psi_{4,\,\alpha,\,l}(v_{\varepsilon})\lvert \D\bu_{\varepsilon}\rvert^{2}\,\d x\d t \\ 
        &\le \iint_{Q_{\theta_{2}R}}v_{\varepsilon}^{\alpha}\, \d x\d t+2M_{0}\iint_{Q_{\theta_{2}R}}\lvert \nabla\eta\rvert\eta \psi_{4,\,\alpha,\,l}(v_{\varepsilon})\lvert \D\bu_\varepsilon\rvert\,\d x\d t\\ &\quad +c_{n}M_{0}\left[\iint_{Q_{\theta_{2}R}}\eta^{2}\left(\psi_{4,\,\alpha,\,l}(v_{\varepsilon})\lvert \D^{2}\bu_{\varepsilon}\rvert+\psi_{4,\,\alpha,\,l}^{\prime}(v_{\varepsilon})v_{\varepsilon}\lvert\nabla v_{\varepsilon}\rvert \right) \,\d x\d t \right]\\ 
        &\le C(M_{0})\iint_{Q_{\theta_{2}R}}\left(v_{\varepsilon}^{\alpha+1}+1\right)\lvert \nabla\eta\rvert\,\d x\d t\\ 
        &\quad +c_{n}M_{0}\left[\iint_{Q_{R}}v_{\varepsilon}^{p-2}\left[\psi_{4,\,\alpha,\,l}(v_{\varepsilon})\lvert \D^{2}\bu_{\varepsilon}\rvert^{2}+\psi_{4,\,\alpha,\,l}^{\prime}(v_{\varepsilon})v_{\varepsilon} \lvert \nabla v_{\varepsilon} \rvert^{2} \right]\eta^{2}\,\d x\d t \right]^{1/2}\\ 
        &\quad\quad  \cdot \left[\iint_{Q_{\theta_{2}R}}v_{\varepsilon}^{2-p}\left(\psi_{4,\,\alpha,\,l}(v_{\varepsilon})+\psi_{4,\,\alpha,\,l}^{\prime}(v_{\varepsilon})v_{\varepsilon} \right)\eta^{2}\,\d x\d t \right]^{1/2}\\ 
        &\le \frac{1}{2}\iint_{Q_{\theta_{2}R}}\psi_{4,\,\alpha,\,l}(v_{\varepsilon})v_{\varepsilon}^{2}\,\d x\d t+\frac{C({\mathcal D},\,\alpha,\,M_{0},\,\lVert \buf_{\varepsilon}\rVert_{L^{\infty}(Q)})}{[(\theta_{2}-\theta_{1})R]^{2}} \iint_{Q_{\theta_{2}R}}\left(v_{\varepsilon}^{p+\alpha}+1 \right)\,\d x\d t,
    \end{align*}
    where we have used (\ref{Eq (Section 4); Qualitative Higher Integrability}) and Young's inequality to deduce the last estimate.
    The claim (\ref{Eq (Section 5): Qualitative Higher Integrability}) follows from the above estimate and Lemma \ref{Lemma (Section 2): Interpolation Abosorbing lemma}.
    Letting $l\to\infty$ in (\ref{Eq (Section 5): Qualitative Higher Integrability}) and repeatedly using the resulting estimate, we deduce $v_{\varepsilon}\in L_{\mathrm{loc}}^{\pi}(Q_{R})$ for any $\pi\in(p,\,\infty)$, and therefore we are now allowed to use (\ref{Eq (Section 4); Quantitative Higher Integrability}).
    Carrying out similar computations, we have 
    \begin{align*}
        &\iint_{Q_{\theta_{1}R}}\psi_{4,\,\alpha}(v_{\varepsilon})v_{\varepsilon}^{2}\,\d x\d t\le \iint_{Q_{\theta_{2}R}}v_{\varepsilon}^{\alpha}\,\d x\d t+\iint_{Q_{R}}\eta^{2}\psi_{4,\,\alpha}(v_{\varepsilon})\lvert \D\bu_{\varepsilon}\rvert^{2}\,\d x\d t\\ 
        &\le \frac{1}{2}\iint_{Q_{\theta_{2}R}}\psi_{4,\,\alpha}(v_{\varepsilon})v_{\varepsilon}^{2}\,\d x\d t+C({\mathcal D},\,\alpha,\,M_{0})\left[\frac{1}{[(\theta_{2}-\theta_{1})R]^{2}}\iint_{Q_{\theta_{2}R}}\left(v_{\varepsilon}^{p+\alpha}+1\right)\,\d x\d t+F^{\frac{2(\alpha+2)}{p}}+1 \right].
    \end{align*}
    Applying Lemma \ref{Lemma (Section 2): Interpolation Abosorbing lemma} again, and using Young's inequality, we have 
    \begin{align*}
        &\iint_{Q_{\theta R}}v_{\varepsilon}^{p+2}\,\d x\d t\le\iint_{Q_{\theta R}}(\psi_{4,\,\alpha}(v_{\varepsilon})v_{\varepsilon}^{2}+v_{\varepsilon}^{\alpha+1})\,\d x\d t \\ 
        &\le C({\mathcal D},\,\alpha,\,M_{0})\left( [(1-\theta)R]^{-2} \iint_{Q_{R}}\left(v_{\varepsilon}^{p+\alpha}+1 \right)\,\d x\d t+F^{\frac{2(\alpha+2)}{p}}+1\right)
    \end{align*}
    for all $\theta\in(0,\,1)$. Repeatedly using the resulting estimate completes the proof of (\ref{Eq (Section 5): Quantitative Higher Integrability}).
\end{proof}
 \subsection{Local $L^{\infty}$-estimates of spatial gradients}
 \begin{proposition}\label{Proposition (Section 5): Reversed Holder}
     Let $\bu_{\varepsilon}$ be a weak solution to (\ref{Eq (Section 2): Approximate System}).
     For $p\in(1,\,p_{\mathrm c}\rbrack$, let $v_{\varepsilon}\in L^{{\widetilde p}}(Q_{R})$ be additionally in force with
     \begin{equation}\label{Eq (Section 5): Higher Integrability assumption of v-epsilon}
         p_{\mathrm c}\le \frac{n(2-p)}{2}<{\widetilde p}<\infty,\quad \text{and}\quad 2\le {\widetilde p}<\infty.
     \end{equation}
     Then, we have $v_{\varepsilon}\in L_{\mathrm{loc}}^{\infty}(Q)$.
     Moreover, we respectively have 
     \[
        \esssup_{Q_{\theta R}}v_{\varepsilon}\le \frac{C({\mathcal D})}{(1-\theta)^{e/d}}\left[1+\lVert \buf_{\varepsilon}\rVert_{L^{q,\,r}(\Omega_{T})}^{p\pi}+\fiint_{Q_{R}}v_{\varepsilon}^{p}\,\d x\d t \right]^{1/d}
     \]
     when $p\in(p_{\mathrm c},\,\infty)$, and  
     \[
         \esssup_{Q_{\theta R}}v_{\varepsilon}\le \frac{C({\mathcal D},\,\widetilde{p})}{(1-\theta)^{e/d}}\left[1+\lVert \buf_{\varepsilon}\rVert_{L^{q,\,r}(\Omega_{T})}^{\widetilde{p}\pi}+\fiint_{Q_{R}}v_{\varepsilon}^{\widetilde p}\,\d x\d t \right]^{1/d}
     \]
     when $p\in(1,\,p_{\mathrm c}\rbrack$. 
     Here the positive exponents $\pi$, $d$, $e$ are defined as $\pi\coloneqq \max\{\,1/(p-1),\,2/p\,\}\in(0,\,\infty)$, 
     \[d\coloneqq \left\{\begin{array}{cc}
         2 & (p\ge 2), \\ (n+2)(p-p_{\mathrm c}) & (n\ge 3\text{ and }p_{\mathrm c}<p<2),\\ 2-\sigma(2-p) & (n=2\text{ and }1=p_{\mathrm c}<p<2),\\ \widetilde{p}-n(2-p)/2 &(1<p\le p_{\mathrm c}),
     \end{array}  \right.\quad  e\coloneqq \left\{\begin{array}{cc}
         n+2 & (n\ge 3),\\ 2\sigma & (n=2),
     \end{array} \right.  \]
     where we fix a constant $\sigma>2$ that is close to $2$, so that the following inequalities are satisfied;
     \[2<\sigma<1+\frac{q}{2\widehat{r}} \quad \text{and} \quad 2<\sigma<\frac{2}{(2-p)_{+}}.\]
 \end{proposition}
 \begin{proof}
     We write ${\widehat p}\coloneqq \max\{\,p,\,2\,\}\in\lbrack 2,\,\infty)$.
     We choose $\kappa\coloneqq 1+2/n\in(1,\,2)$ for $n\ge 3$. 
     When $n=2$, we note $\kappa\coloneqq \sigma^{\prime}\in(1,\,2)$ satisfies $\frac{2}{(\kappa-1)q}+\frac{2}{r}<1$. 
     We define $\widetilde{\kappa}\coloneqq 2-\kappa\in(0,\,1)$ for $n=2$, and formally set $\widetilde{\kappa}\coloneqq 0$ for $n\ge 3$, so that $\kappa+\widetilde{\kappa}=1+2/n$ holds for any $n\ge 2$.
     It should be noted that $e=2\kappa^{\prime}$ automatically holds by the definition of $\kappa$.
     Instead of $v_{\varepsilon}$, we consider the function $W_{k}$, defined as in Lemma \ref{Lemma (Section 4): Caccioppoli estimates for W-k} with $k\coloneqq 1+\lVert \buf_{\varepsilon}\rVert_{L^{q,\,r}(Q_{R})}^{\pi}\ge 1$.
     We would like to show that this $W_{k}$ satisfies
     \begin{equation}\label{Eq (Section 5): Local Gradient Bound for super}
         \esssup_{Q_{\theta R}}\,W_{k}\le \left(\frac{C({\mathcal D})}{(1-\theta)^{2\kappa^{\prime}}}\fiint_{Q_{R}}W_{k}^{p}\,\d x\d t\right)^{1/d}
     \end{equation}
     when $p\in(p_{\mathrm c},\,\infty)$, and 
     \begin{equation}\label{Eq (Section 5): Local Gradient Bound for sub}
         \esssup_{Q_{\theta R}}\,W_{k}\le \left(\frac{C({\mathcal D},\,\widetilde{p})}{(1-\theta)^{2\kappa^{\prime}}}\fiint_{Q_{R}}W_{k}^{\widetilde{p}}\,\d x\d t \right)^{1/d}
     \end{equation}
     when $p\in(1,\,p_{\mathrm c}\rbrack$ respectively.
     Then, the desired estimates are easily deduced from (\ref{Eq (Section 5): v vs w}), (\ref{Eq (Section 5): Local Gradient Bound for super})--(\ref{Eq (Section 5): Local Gradient Bound for sub}) and our choice of $k\ge 1$.
     
     For preliminaries, we would like to show that $W_{k}$ satisfies 
     \begin{equation}\label{Eq (Section 5): Reversed Holder for W-k}
         \iint_{Q_{R_{2}}}W_{k}^{\kappa\beta-(\kappa-1)({\widehat p}-2)-({\widehat p}-p)}\,\d x\d t \le \left[\frac{C\beta^{\gamma}R^{2\widetilde\kappa/\kappa}}{(R_{1}-R_{2})^{2}} \iint_{Q_{R_{1}}}W_{k}^{\beta}\,\d x\d t\right]^{\kappa}
     \end{equation}
     for any $\beta\in\lbrack {\widehat p},\,\infty)$ and $0<R_{2}<R_{1}\le R$, provided $W_{k}\in L^{\beta}(Q_{R_{1}})$.
     It should be noted that the assumption $v_{\varepsilon}\in L^{2}(Q_{R})\cap L^{p}(Q_{R})$ is not restrictive for $p\in(1,\,2)$ by Lemma \ref{Lemma (Section 5): Higher Integrability of v-epsilon}.
     In particular, we may let $W_{k}\in L^{\widehat{p}}(Q_{R})$.
     Let $\eta\in C_{\mathrm c}^{1}(Q_{R})$ be non-negative, and use (\ref{Eq (Section 4): Energy estimate for Moser Iteration}) with $\alpha\coloneqq (\beta-{\widehat p})/2\in\lbrack 0,\,\infty)$.
     We introduce ${\widetilde \varphi}_{1}\coloneqq \eta W_{k}^{\alpha+1}\in L^{2,\,\infty}(Q_{R})$, ${\widetilde \varphi}_{2}\coloneqq \eta W_{k}^{\alpha+p/2}\in L^{p}(I_{R};\,W_{0}^{1,\,p}(B_{R}))$, and the non-negative function $h_{k}\coloneqq W_{k}^{-\widetilde{\pi}}\lvert \buf_{\varepsilon}\rvert^{2}\in L^{q/2,\,r/2}(Q_{R})$ with $\widehat{\pi}\coloneqq 2/\pi=\min\{\,2(p-1),\,p\,\}\in(0,\,p\rbrack$, so that the inequality $\lVert h_{k}\rVert_{L^{q/2,\,r/2}(Q_{R})}\le 1$ is satisfied by the definition of $k$.
     Then, (\ref{Eq (Section 4): Energy estimate for Moser Iteration}) is rewritten as
     \begin{align*}
        &\esssup_{t\in I_{R}}\int_{B_{R}}\lvert{\widetilde\varphi}_{1}\rvert^{2}\,\d x\d t+\iint_{Q_{R}}\lvert \nabla{\widetilde\varphi}_{2}\rvert^{2}\,\d x\d t\\ 
        &\le C(1+\alpha)^{2}\left[\iint_{Q_{R}}W_{k}^{2\alpha+{\widehat p}}\left(\eta^{2}+\lvert\nabla\eta\rvert^{2}+\eta\lvert\partial_{t}\eta\rvert \right)\,\d x\d t+\iint_{Q_{R}}h_{k}\varphi_{1}^{2}\,\d x\d t \right]
     \end{align*}
     with $\varphi_{1}\in L^{2,\,\infty}(Q_{R})$ defined as $\varphi_{1}\coloneqq {\widetilde \varphi}_{1}$ for $p\in \lbrack 2,\,\infty)$, and $\varphi_{1}\coloneqq {\widetilde \varphi}_{2}$ for $p\in(1,\,2)$.
     We use (\ref{Eq (Section 2): Parabolic absorbing}) with $s=2$, $(\pi_{1},\,\pi_{2})=(q/2,\,r/2)$.
     Here we choose $\varphi_{1}=\varphi_{2}={\widetilde\varphi}_{1}$ for $p\in\lbrack 2,\,\infty)$, and $\varphi_{1}\le \varphi_{2}\coloneqq \widetilde{\varphi}_{1}$ for $p\in(1,\,2)$ in carrying out absorbing artuments.
     By (\ref{Eq (Section 2): Parabolic Sobolev}) and Remark \ref{Remark (Section 2): Scaling Constant}, we obtain
     \[\iint_{Q_{R}}{\varphi}_{1}^{2(\kappa-1)}\varphi_{2}^{2}\,\d x\d t\le \left[C(1+\alpha)^{\gamma}R^{2\widetilde{\kappa}/\kappa}\iint_{Q_{R}}W_{k}^{2\alpha+{\widehat p}}\left(\eta^{2}+\lvert\nabla\eta\rvert^{2}+\lvert\partial_{t}\eta\rvert \right)\,\d x\d t \right]^{\kappa}.\]
     Here the exponent $\gamma=\gamma(n,\,\kappa,\,q,\,r)\in \lbrack 2,\,\infty)$ is determined by Lemma \ref{Lemma (Secion 2): Interpolation among parabolic function spaces}.
     Suitably choosing a cut-off function $\eta\in C_{\mathrm c}^{1}(Q_{R_{1}};\,\lbrack 0,\,1\rbrack)$, we conclude (\ref{Eq (Section 5): Reversed Holder for W-k}).
     
     For each $l\in{\mathbb Z}_{\ge 0}$, we set $R_{l}\coloneqq \theta R+(1-\theta)2^{-l}R$, and consider
     \(Y_{l}\coloneqq \left(\iint_{Q_{R_{l}}}W_{k}^{p_{l}}\,\d x\d t \right)^{1/p_{l}}\),
     where the sequence $\{p_{l}\}_{l=0}^{\infty}\subset \lbrack {\widehat p},\,\infty)$ satisfies $p_{l+1}=\kappa p_{l}-(\kappa-1)({\widehat p}-2)-({\widehat p}-p)$ for all $l\in{\mathbb Z}_{\ge 0}$, and $p_{0}\in \lbrack \widehat{p},\,\infty)$ is chosen later.
     It is easy to calculate $p_{l}=\kappa^{l}(p_{0}-c_{0})+c_{0}$ with $c_{0}\coloneqq \widehat{p}-2+(\widehat{p}-p)/(\kappa-1)$.
     The reversed H\"{o}lder inequality (\ref{Eq (Section 5): Reversed Holder for W-k}) allows us to apply Lemma \ref{Lemma (Section 2): Moser Iteration Lemma} with $\mu\coloneqq p_{0}-c_{0}$, $A\coloneqq C({\mathcal D})(1-\theta)^{-2}R^{-2(1-\widetilde{\kappa}/\kappa)}$ and $B\coloneqq 4\kappa^{\gamma\kappa}$, provided $p_{0}>c_{0}$ and $p_{0}\ge \widehat{p}$.
     The resulting estimate from Lemma \ref{Lemma (Section 2): Moser Iteration Lemma} is 
     \begin{equation}\label{Eq (Section 5): Result from Iteration}
        \esssup_{Q_{\theta R}}\, W_{k}\le \limsup_{l\to\infty} Y_{l}\le A^{\frac{\kappa^{\prime}}{\mu}}B^{\frac{(\kappa^{\prime})^{2}}{\mu}}Y_{0}^{\frac{p_{0}}{\mu}}=\left(\frac{C({\mathcal D})}{(1-\theta)^{2\kappa^{\prime}}}\fiint_{Q_{R}}W_{k}^{p_{0}}\,\d x\d t \right)^{1/\mu}.
     \end{equation}
     For $p\in\lbrack 2,\,\infty)$, which yields $c_{0}=p-2$, we choose $p_{0}\coloneqq p$ and therefore we have $\mu=2$. 
     The claim (\ref{Eq (Section 5): Local Gradient Bound for super}) is easily deduced by (\ref{Eq (Section 5): Result from Iteration}). 
     For $p\in(p_{\mathrm c},\,2)$, which implies $c_{0}=(2-p)/(\kappa-1)$, we choose $p_{0}\coloneqq 2$.
     Then, the corresponding $\mu=2-(2-p)/(\kappa-1)$ satisfies $\mu>2-p>0$.
     In fact, we have $\mu-2+p=(n+2)(p-p_{\mathrm c})/2>0$ for $n\ge 3$, and $\mu-2+p=-\sigma(2-p)+2>0$ when $n=2$.
     Combining (\ref{Eq (Section 5): Result from Iteration}) with Young's inequality, we have
     \begin{align*}
         \esssup_{Q_{\theta R}}W_{k}&\le \left(\frac{C({\mathcal D})}{(1-\theta)^{2\kappa^{\prime}}}\fiint_{Q_{R}}W_{k}^{2}\,\d x\d t \right)^{1/\mu}\\ 
         &\le \frac{1}{2}\esssup_{Q_{R}}W_{k}+\left(\frac{C({\mathcal D})}{(1-\theta)^{2\kappa^{\prime}}}\fiint_{Q_{R}}W_{k}^{p}\,\d x\d t \right)^{1/(\mu-2+p)}
     \end{align*}
     It should be noted that these computations can be made, even when $\theta R$ and $R$ are respectively replaced by $\theta_{1}R$ and $\theta_{2}R$ with $\theta\le \theta_{1}<\theta_{2}\le 1$.
     Therefore, (\ref{Eq (Section 5): Local Gradient Bound for super}) follows from Lemma \ref{Lemma (Section 2): Interpolation Abosorbing lemma}.
     In the remaining case $p\in(1,\,p_{\mathrm c})$, which clearly yields $n\ge 3$ and $c_{0}=n(2-p)/2$, we choose $p_{0}\coloneqq {\widetilde p}\ge 2$, so that $\mu=p_{0}-c_{0}>0$ holds.
     Then, (\ref{Eq (Section 5): Local Gradient Bound for sub}) immediately follows from (\ref{Eq (Section 5): Result from Iteration}).
 \end{proof}
 
 \subsection{Proofs of Theorem \ref{Theorem (Section 4): Gradient bounds} and a corollary}
 We provide the proof of Theorem \ref{Theorem (Section 4): Gradient bounds}.
 \begin{proof}[Proof of Theorem \ref{Theorem (Section 4): Gradient bounds}]
     The case $p\in(p_{\mathrm c},\,\infty)$ is clear by Proposition \ref{Proposition (Section 5): Reversed Holder} and (\ref{Eq (Section 4): Uniform Bounds F and U}).
     For $p\in(1,\,p_{\mathrm c}\rbrack$, we fix a subdomain $\widehat{Q}$ that satisfies $\widetilde{\mathcal Q}\Subset \widehat{\mathcal Q}\Subset {\mathcal Q}$.
     By Lemma \ref{Lemma (Section 5): Higher Integrability of v-epsilon}, $v_{\varepsilon}\in L^{{\widetilde p}}(\widehat{Q})$ holds for some new exponent ${\widetilde p}$ satisfying ${\widetilde p}\in(n(2-p)/2,\,pq/2)$ and ${\widetilde p}\ge 2$.
     In particular, there exists a constant $C=C(D,\,U,\,F,\,M_{0},\,{\widetilde p},\,\widehat{\mathcal Q},\,{\mathcal Q})$ such that $\lVert v_{\varepsilon}\rVert_{L^{\widetilde{p}}(\widehat{\mathcal Q}) }\le C$ holds uniformly for $\varepsilon\in(0,\,1)$.
     The proof for $p\in(1,\,p_{\mathrm c}\rbrack$ is also completed by this bound estimate and Proposition \ref{Proposition (Section 5): Reversed Holder}.
 \end{proof}

 Following \cite[Lemma 3.1]{BDMi}, we deduce Lemma \ref{Lemma (Section 5): Local Holder continuity of u-epsilon} from Theorem \ref{Theorem (Section 4): Gradient bounds}.
 \begin{lemma}\label{Lemma (Section 5): Local Holder continuity of u-epsilon}
     Under the assumptions of Theorem \ref{Theorem (Section 4): Gradient bounds}, $\bu_{\varepsilon}$ is $(1,\,1/2)$-H\"{o}lder continuous in each fixed ${\widetilde Q}\Subset {\mathcal Q}$, uniformly for $\varepsilon\in(0,\,1)$.
 \end{lemma}
 \begin{proof}
     For given $Q_{\rho}(x_{0},\,t_{0})=B_{\rho}(x_{0},\,t_{0})\times I_{\rho}(t_{0})\subset {\widetilde Q}$, we choose and fix a non-negative function $\widetilde{\eta}\in C_{\mathrm c}^{1}(B_{\rho})$ satisfying 
     \[\fint_{B_{\rho}(x_{0})}\widetilde{\eta}\,\d x=1,\quad \lVert \widetilde{\eta}\rVert_{L^{\infty}(B_{\rho})}+\rho\lVert\nabla \widetilde{\eta}\rVert_{L^{\infty}(B_{\rho})}\le C(n).\]
     For each $t\in I_{\rho}(t_{0})$, we define 
     \[\widetilde{\bu}_{\varepsilon}(t)\coloneqq \fint_{B_{\rho}(x_{0})}\widetilde{\eta}\bu_{\varepsilon}\,\d x\in{\mathbb R}^{N}.\]
     We claim the following estimate;
     \begin{equation}\label{Eq (Section 5): Claim for Local Holder}
        \sup_{\tau_{1},\,\tau_{2}\in I_{\rho}(t_{0})}\,\lvert \widetilde{\bu}_{\varepsilon}(\tau_{1})-\widetilde{\bu}_{\varepsilon}(\tau_{2}) \rvert^{2} \le C({\mathcal D},\,F,\,\mu_{0})\rho.
     \end{equation}
     To prove (\ref{Eq (Section 5): Claim for Local Holder}), we may let $t_{0}-\rho^{2}<\tau_{1}<\tau_{2}<t_{0}$ without loss of generality.
     We choose a piecewise linear function as 
     \begin{equation}\label{Eq (Section 5): Piecewise Linear}
        \phi_{\widetilde{\varepsilon}}(t)\coloneqq \left(\min\left\{\,1,\,(t-\tau_{1})/\widetilde{\varepsilon},\,-(t-\tau_{2})/\widetilde{\varepsilon} \,\right\}\right)_{+}
     \end{equation}
     for $t\in I_{\rho}(t_{0})$, where we will later let $\widetilde{\varepsilon}\to 0$.
     For each $i\in\{\,1,\,\dots\,,\,N\,\}$, we test $\bphi\coloneqq (\delta^{ij}\widetilde{\eta}\phi_{\widetilde{\varepsilon}})_{j}$ into (\ref{Eq (Section 2): Approximate Weak Form}).
     Integrating by parts and summing over $i\in\{\,1,\,\dots\,,\,N\,\}$, we have 
     \begin{align*}
         \left\lvert \iint_{Q_{\rho}} \partial_{t}\phi_{\widetilde{\varepsilon}} \widetilde{\eta} \bu_{\varepsilon}\,\d x\d t \right\rvert&\le \gamma_{0}^{-2}\iint_{Q_{\rho}}\lvert \A_{\varepsilon}(x,\,t,\,\D\bu_{\varepsilon}) \rvert\lvert \nabla \widetilde{\eta}\rvert\,\d x\d t+\iint_{Q_{\rho}}\lvert \buf_{\varepsilon}\rvert \widetilde{\eta}\,\d x\d t \\ &\le C({\mathcal D},\,\mu_{0})\rho^{n+1}+C({\mathcal D})F\rho^{n+1+\beta}\le C({\mathcal D},\,F,\,\mu_{0})\lvert B_{\rho}\rvert\rho,
     \end{align*}
     where we have used H\"{o}lder's inequality, (\ref{Eq (Section 1): Growth of p-Laplace-type operator}), (\ref{Eq (Section 1): Bound Assumptions for Coefficients}), (\ref{Eq (Section 1): Growth of 1-Laplace-type operator}), and (\ref{Eq (Section 4): Local Gradient bound}).
     Letting $\widetilde{\varepsilon}\to 0$ completes the proof of (\ref{Eq (Section 5): Claim for Local Holder}).
     Using (\ref{Eq (Section 4): Local Gradient bound}), (\ref{Eq (Section 5): Claim for Local Holder}), and the Poincar\'{e}--Sobolev inequality, we have
     \begin{align*}
        & \fiint_{Q_{\rho}}\lvert \bu_{\varepsilon}-(\bu_{\varepsilon})_{Q_{\rho}} \rvert^{2}\,\d x\d t\\ 
        &\le 3\left[\fiint_{Q_{\rho}}\lvert\bu_{\varepsilon}(x,\,t) -\widetilde{\bu}_{\varepsilon}(t) \rvert^{2}\,\d x\d t+\fiint_{I\times I}\lvert\widetilde{\bu}_{\varepsilon}(t)-\widetilde{\bu}_{\varepsilon}(\tau) \rvert^{2}\,\d t\d\tau+\fiint_{Q_{\rho}}\lvert \widetilde{\bu}_{\varepsilon}(t)-(\bu_{\varepsilon})_{Q_{\rho}} \rvert^{2}\,\d x\d t \right]\\
        &\le C({\mathcal D})\left[\rho^{2}\fiint_{Q_{\rho}}\lvert \D\bu_{\varepsilon}\rvert^{2}\,\d x\d t+\sup_{\tau_{1},\,\tau_{2}\in I_{\rho}}\lvert \widetilde{\bu}_{\varepsilon}(\tau_{1})-\widetilde{\bu}_{\varepsilon}(\tau_{2})\rvert^{2}\right]\le C({\mathcal D},\,F,\,\mu_{0})\rho^{2}.
     \end{align*}
     From this Campanato-type growth estimate, the local H\"{o}lder continuity of $\bu_{\varepsilon}$ is easy to conclude.
 \end{proof}

\section{Degenerate case}\label{Section: Degenerate}
 Section \ref{Section: Degenerate} is dedicated to showing Proposition \ref{Proposition (Section 4): Degenerate Case}.
 The starting point is Lemma \ref{Lemma (Section 4): Caccioppoli estimates for U-delta-epsilon}, which states that the function $U_{\delta,\,\varepsilon}$ is in a certain parabolic De Giorgi class.
 A fundamental source of the parabolic De Giorgi class is \cite[Chapter II, \S 7]{LSU MR0241822}, based on De Giorgi's truncation and level set estimates. 
 We follow \cite[Chapter II \S 7]{LSU MR0241822} to give the proof of Proposition \ref{Proposition (Section 4): Degenerate Case}.
 However, some proofs are briefly sketched, since they can be completed very similarly to \cite[\S 5]{T-supercritical}, where the case $q=r>n+2$ is discussed.
 Removing the condition $q=r$ makes us give a slightly different proof in Lemma \ref{Lemma (Section 6): Density Lemma 2}, where we have to use another decay lemma (Lemma \ref{Lemma (Section 2): Geometric Decay Lemma}).
 \subsection{Expansion of positivity}
 In the proof of Proposition \ref{Proposition (Section 4): Degenerate Case}, we often assume
 \begin{equation}\label{Eq (Section 6): Inclusion Assumption}
    \widetilde{I}_{\frac{3}{2}\rho}(\gamma;\,\tau_{0})=\left(\tau_{0},\,\tau_{0}+\frac{9}{4}\gamma\rho^{2}\right] 
    \subset (t_{0}-2\rho^{2},\,t_{0}\rbrack=I_{\sqrt{2}\rho}(t_{0}),
 \end{equation}
 \begin{equation}\label{Eq (Section 6): Assumtion for Exp of Pos}
    \left\lvert \left\{ x\in B_{\rho}\mathrel{}\middle|\mathrel{}U_{\delta,\,\varepsilon}(x,\,t)\le (1-\widehat{\nu})\mu^{2} \right\} \right\rvert\ge \widetilde{\nu}\lvert B_{\rho}\rvert\quad \text{for a.e.~}t\in \widetilde{I}_{\frac{3}{2}\rho}(\gamma;\,\tau_{0}),
 \end{equation}
 hold for some $\gamma\in(0,\,\nu/9)$ and $\widehat{\nu},\,\widetilde{\nu}\in(0,\,1\rbrack$. 
 These conditions are verified by following Lemma \ref{Lemma (Section 6): Expansion of Positivity}.
 \begin{lemma}\label{Lemma (Section 6): Expansion of Positivity}
     If all of the assumptions of Proposition \ref{Proposition (Section 4): Degenerate Case} hold, then there exist $\rho_{\star}=\rho_{\star}({\mathcal D},\,\delta,\,M)\in(0,\,1)$, $\tau_{0}\in(t_{0}-\rho^{2},\,t_{0}-\nu\rho^{2}/2)$, $\gamma_{\star}=\gamma_{\star}({\mathcal D},\,\delta,\,M,\,\nu)\in(0,\,\nu/9)$ and $\widehat{\nu}=\widehat{\nu}(n,\,\nu)\in(0,\,\nu/2)$ such that (\ref{Eq (Section 6): Inclusion Assumption})--(\ref{Eq (Section 6): Assumtion for Exp of Pos}) hold with $\widetilde{\nu}=\nu/8$ and $\gamma=\gamma_{\star}$, provided $\rho\le \rho_{\star}$.
 \end{lemma}
 \begin{proof}
    It suffices to prove (\ref{Eq (Section 6): Assumtion for Exp of Pos}), since (\ref{Eq (Section 6): Inclusion Assumption}) immediately follows from $\tau_{0}\in (t_{0}-\rho^{2},\,t_{0}-\nu\rho^{2}/2)$ and $\gamma\in(0,\,\nu/9)$.
    We also note that there exists a number $\tau_{0}\in (t-\rho^{2},\,t_{0}-\nu\rho^{2}/2)$ such that 
    \begin{equation}\label{Eq (Section 6): Positivity at tau0}
        \lvert \left\{x\in B_{\rho}\mid U_{\delta,\,\varepsilon}(x,\,\tau_{0})\le (1-\nu)\mu^{2} \right\}\rvert\ge \frac{\nu}{2}\lvert B_{\rho}\rvert,
    \end{equation}
    since otherwise we would have 
    \[\lvert Q_{\rho}\setminus S_{\rho}\rvert\le \lvert\{(x,\,t)\in Q_{\rho}\mid U_{\delta,\,\varepsilon}(x,\,t)\le (1-\nu)\mu^{2} \}\rvert \le \left(1-\frac{\nu}{2}\right)\rho^{2}\cdot \frac{\nu}{2} \lvert B_{\rho}\rvert+\frac{\nu}{2}\rho^{2}\cdot \lvert B_{\rho}\rvert\le \nu\lvert Q_{\rho}\rvert,\] 
    which would cause a contradiction with (\ref{Eq (Section 4): Measure Assumption: Degenerate}).
    Let $\sigma\in(0,\,1)$ and $\theta_{0}\in (0,\,1/2)$ be chosen later, and corresponding to these numbers, we fix a cutoff function $\widetilde{\eta}\in C_{\mathrm c}^{1}(B_{\rho};\,\lbrack 0,\,1\rbrack)$ satisfying $\widetilde{\eta}|_{B_{(1-\sigma)\rho}}\equiv 1$, and introduce the super-level set $A_{\theta_{0}\nu,\,R}(\tau)\coloneqq \left\{x\in B_{R}(x_{0})\mathrel{}\middle|\mathrel{} U_{\delta,\,\varepsilon}(x,\,\tau)>(1-\theta_{0}\nu)\mu^{2} \right\}$ for $R\in(0,\,\rho\rbrack$ and $\tau\in \widetilde{I}_{\frac{3}{2}\rho}(\gamma;\,\tau_{0})$.
    We use (\ref{Eq (Section 4): De Giorgi Energy 2}) with $Q=B_{\rho}\times \widetilde{I}_{\frac{3}{2}\rho}(\gamma;\,\tau_{0})$ and $k\coloneqq (1-\nu)\mu^{2}$, which implies $(U_{\delta,\,\varepsilon}-k)_{+}\le \nu\mu^{2}$.
    Keeping in mind that $U_{\delta,\,\varepsilon}-k\ge (1-\theta_{0})\nu\mu^{2}$ in $A_{\theta_{0}\nu,\,(1-\theta)\rho}$ and using (\ref{Eq (Section 6): Positivity at tau0}), we have 
    \begin{align*}
        &(1-\theta_{0})^{2}(\nu\mu^{2})^{2} \lvert A_{\theta_{0}\nu,\,(1-\sigma)\rho}(\tau) \rvert\\ 
        &\le (\nu\mu^{2})^{2}\lvert \left\{x\in B_{\rho}\mid U_{\delta,\,\varepsilon}(x,\,\tau_{0})>(1-\nu)\mu^{2} \right\}\rvert\\ 
        &\quad +C({\mathcal D},\,\delta,\,M)\left[\frac{(\nu\mu)^{2}}{(\sigma\rho)^{2}}\lvert Q\rvert+\mu^{4}\left(1+F^{2}\right)\lvert B_{\rho}(x_{0})\rvert^{1/\widehat{q}}\left\lvert \widetilde{I}_{\frac{3}{2}\rho}(\gamma;\,\tau_{0})\right\rvert^{1/\widetilde{r}} \right]\\ 
        &\le (\nu\mu^{2})^{2}\left(1-\frac{\nu}{2}+\frac{C\gamma}{\sigma^{2}}+\frac{C\left(1+F^{2}\right)\gamma^{1/\widehat{r}}\rho^{2\beta}}{\nu^{2}} \right)\lvert B_{\rho}\rvert
    \end{align*}
    for some constant $C_{0}=C_{0}({\mathcal D},\,\delta,\,M)\in(1,\,\infty)$.
    Combining the above estimate with $\lvert A_{\theta_{0}\nu,\,\rho}(\tau)\rvert\le \lvert B_{\rho}\setminus B_{(1-\sigma)\rho}\rvert+\lvert A_{\theta_{0}\nu,\,(1-\sigma)\rho}(\tau) \rvert$ yields
    \[\lvert A_{\theta_{0}\nu,\,\rho}(\tau)\rvert\le \left(n\sigma+\frac{1-\nu/2}{(1-\theta_{0})^{2}}+\frac{C_{0}\gamma}{(1-\theta_{0})^{2}\sigma^{2}}+\frac{C_{0}\left(1+F^{2}\right)\gamma^{1/\widehat{r}}\rho^{2\beta}}{\nu^{2}} \right)\lvert B_{\rho}\rvert.\]
    Determining $\sigma\in(0,\,1)$, $\theta_{0}\in(0,\,1/2)$, $\gamma\in(0,\,\nu/9)$, and $\rho_{\ast}\in(0,\,1)$ such that 
    \[n\sigma\le \frac{\nu}{24},\quad \frac{1-\nu/2}{(1-\theta_{0})^{2}}\le 1-\frac{\nu}{4},\quad \frac{C_{0}\gamma}{\sigma^{2}(1-\theta_{0})^{2}}\le \frac{\nu}{24},\quad \text{and}\quad \frac{C_{0}\left(1+F^{2} \right)\gamma^{1/\widehat{r}} \rho_{\star}^{2\beta} }{\nu^{2}(1-\theta_{0})^{2}}\le \frac{\nu}{24}\]
    are all satisfied, we deduce $\lvert A_{\theta_{0}\nu,\,\rho}(\tau) \rvert\le (1-\nu/8)\lvert B_{\rho}\rvert=(1-\widetilde{\nu})\lvert B_{\rho}\rvert$ for a.e.~$\tau\in \widetilde{I}_{\frac{3}{2}\rho}(\gamma;\,\tau_{0})$, provided $\rho\le \rho_{\star}$.
    This completes the proof of (\ref{Eq (Section 6): Assumtion for Exp of Pos}) with $\widehat{\nu}\coloneqq \theta_{0}\nu\in(0,\,\nu/2)$. 
 \end{proof}
 \subsection{Density of level sets}
 \begin{lemma}\label{Lemma (Section 6): Density Lemma 1}
     In addition to the assumptions of Proposition \ref{Proposition (Section 4): Degenerate Case}, let (\ref{Eq (Section 6): Inclusion Assumption})--(\ref{Eq (Section 6): Assumtion for Exp of Pos}) and $\rho^{\beta}\le 2^{-l_{\star}}\gamma^{-\frac{1}{2\widehat{r}}} \widehat{\nu}$ hold for some fixed $\tau_{0}\in (t_{0}-\rho^{2},\,t_{0}-\nu\rho^{2}/2)$, $l_{\star}\in{\mathbb N}$, $\gamma\in(0,\,\nu/9)$, and $\widehat{\nu},\,\widetilde{\nu}\in(0,\,1\rbrack$.
     Then, we have 
     \begin{equation}\label{Eq (Section 6): Density Result 1}
         \left\lvert \left\{(x,\,t)\in Q_{\frac{3}{2}\rho}(\gamma;\,\tau_{0})\mathrel{}\middle|\mathrel{}U_{\delta,\,\varepsilon}(x,\,t)\ge \left(1-2^{-l_{\star}}\widehat{\nu} \right)\mu^{2}  \right\} \right\rvert\le \frac{C_{\star}}{\widetilde{\nu}\sqrt{\gamma l_{\star}}} \left\lvert \widetilde{Q}_{\frac{3}{2}\rho}(\gamma;\,\tau_{0}) \right\rvert
     \end{equation}
     for some $C_{\star}=C_{\ast}({\mathcal D},\,F,\,\delta,\,M)\in (1,\,\infty)$.
 \end{lemma}
 We recall a well-known isoperimetric inequality of the form
 \begin{equation}\label{Eq (Section 6): Isoperimetric Ineq}
     (d-c)\lvert\{x\in B_{\rho}\mid w(x)>d\} \rvert\le \frac{C(n)\rho^{n+1}}{\lvert\{x\in B_{\rho}\mid w(x)<c\} \rvert}\int_{B_{\rho}\cap\{c<w<d\}}\lvert \nabla w\rvert\,\d x,
 \end{equation}
 which holds for any function $w\in W^{1,\,1}(B_{\rho})$ and real numbers $-\infty<c<d<\infty$.
 \begin{proof}
     For each $l\in{\mathbb Z}_{\ge 0}$, we introduce $k_{l}\coloneqq \left( 1-2^{-l}\widehat{\nu}\right)\mu^{2}$ and $A_{l}\coloneqq \{(x,\,t)\in \widetilde{Q}_{\frac{3}{2}\rho}(\gamma;\,\tau_{0})\mid U_{\delta,\,\varepsilon}(x,\,t)>k_{l} \}$.
     Since $k_{l+1}-k_{l}=2^{-l-1}\widehat{\nu}\mu^{2}$ clearly holds for any $l\in{\mathbb Z}_{\ge 0}$, using (\ref{Eq (Section 6): Assumtion for Exp of Pos}) and applying (\ref{Eq (Section 6): Isoperimetric Ineq}) to $U_{\delta,\,\varepsilon}(\,\cdot\,,\,t)\in W^{1,\,1}(B_{\frac{3}{2}\rho})$ yield
     \[\frac{\widehat{\nu}\widetilde{\nu}\mu^{2}}{2^{l+1}}\left\lvert \left\{ x\in B_{\frac{3}{2}\rho}\mathrel{}\middle|\mathrel{}U_{\delta,\,\varepsilon}>k_{l+1} \right\}\right\rvert \le C(n)\rho\int_{B_{\frac{3}{2}\rho}\cap \{k_{l}<U_{\delta,\,\varepsilon}(\,\cdot\,,\,t)<k_{l+1}\}}\lvert \nabla(U_{\delta,\,\varepsilon}(x,\,t)-k_{l})_{+} \rvert\,\d x\d t\]
     for a.e.~$t\in I_{\frac{3}{2}\rho}(\gamma;\,\tau_{0})$.
     In particular, using the Cauchy--Schwarz inequality, we easily deduce 
     \[\frac{(\widehat{\nu}\widetilde{\nu})^{2}\mu^{4} }{4^{l+1}}\lvert A_{l+1}\rvert^{2}\le C(n)\rho^{2}\left(\lvert A_{l}\rvert-\lvert A_{l+1}\rvert \right)\iint_{Q_{\frac{3}{2}\rho}(\gamma;\,\tau_{0})}\lvert \nabla(U_{\delta,\,\varepsilon}-k_{l})_{+} \rvert^{2}\,\d x\d t.\]
     We suitably choosing a cutoff function $\eta\in C_{\mathrm c}^{1}(Q_{2\rho}(x_{0},\,t_{0});\,\lbrack 0,\,1\rbrack)$ satisfying $\eta|_{\widetilde{Q}_{\frac{3}{2}\rho}(\gamma;\,x_{0},\,t_{0})}\equiv 1$.
     Using (\ref{Eq (Section 4): De Girogi Energy 1}) with $Q=Q_{2\rho}(x_{0},\,t_{0})$ and $k=k_{l}$, we compute the last integral as 
     \begin{align*}
        &\iint_{\widetilde{Q}_{\frac{3}{2}\rho(\gamma;\,\tau_{0})}}\lvert \nabla(U_{\delta,\,\varepsilon}-k_{l})_{+} \rvert^{2}\,\d x\d t\\ 
        &\le C({\mathcal D},\,\delta,\,M)\left[4^{-l}\widehat{\nu}^{2}\mu^{4}\left\lvert Q_{2\rho}(x_{0},\,t_{0}) \right\rvert+\mu^{4}\left(1+F^{2}\right)\lvert B_{2\rho}(x_{0}) \rvert^{1/\widehat{q}}\left\lvert I_{2\rho}(t_{0}) \right\rvert^{1/\widehat{r}} \right]\\ 
        &\le \frac{C({\mathcal D},\,F,\,\delta,\,M)(\widehat{\nu}\mu^{2})^{2}}{4^{l}\gamma\rho^{2}}\left[1+\frac{4^{l}\gamma^{1/\widehat{r}}\rho^{2\beta}}{\widehat{\nu}^{2}} \right] \left\lvert \widetilde{Q}_{\frac{3}{2}\rho}(\gamma;\,\tau_{0}) \right\rvert,
     \end{align*}
     by $(U_{\delta,\,\varepsilon}-k_{l})_{+}\le 2^{-l}\widehat{\nu}\mu^{2}$.
     If $\rho^{\beta}\le 2^{-l_{\star}}\widehat{\nu}\gamma^{-1/(2\widehat{r})}$ holds, then the last two estimates imply
     \[\lvert A_{l+1}\rvert^{2}\le \frac{C_{\star}^{2}}{\gamma\widetilde{\nu}^{2}}\left\lvert \widetilde{Q}_{\frac{3}{2}\rho}(\gamma;\,\tau_{0}) \right\rvert (\lvert A_{l}\rvert -\lvert A_{l+1}\rvert)\]
     for every $l\in\{\,0,\,\dots\,,\,l_{\star}-1 \,\}$. Therefore, we deduce
     \[l_{\star}\lvert A_{l_{\star}}\rvert^{2}\le \sum_{l=0}^{l_{\star}-1}\lvert A_{l}\rvert^{2}\le \frac{C_{\star}^{2}}{\gamma\widetilde{\nu}^{2}}\left\lvert \widetilde{Q}_{\frac{3}{2}\rho}(\gamma;\,\tau_{0}) \right\rvert \lvert A_{0}\rvert\le \frac{C_{\star}^{2}}{\gamma\widetilde{\nu}^{2}}\left\lvert \widetilde{Q}_{\frac{3}{2}\rho}(\gamma;\,\tau_{0}) \right\rvert^{2}.\]
     The desired estimate (\ref{Eq (Section 6): Density Result 1}) immediately follows from the above inequality.
 \end{proof}
 \begin{lemma}\label{Lemma (Section 6): Density Lemma 2}
     In addition to the assumptions of Proposition \ref{Proposition (Section 4): Degenerate Case}, let 
     \begin{equation}\label{Eq (Section 7): level set Assumption}
        \left\lvert \left\{ (x,\,t)\in\widetilde{Q}_{\frac{3}{2}\rho}(\gamma;\,x_{0},\,\tau_{0})\mathrel{}\middle|\mathrel{}U_{\delta,\,\varepsilon}(x,\,t)\ge (1-\nu_{0})\mu^{2} \right\} \right\rvert\le \alpha\left\lvert \widetilde{Q}_{\frac{3}{2}\rho}(\gamma;\,x_{0},\,\tau_{0}) \right\rvert
     \end{equation}
     hold for some $\alpha,\,\gamma,\,\nu_{0}\in(0,\,1)$.
     There exists sufficiently small $\alpha_{0}=\alpha_{0}({\mathcal D},\,F,\,\delta,\,M,\,\gamma)\in(0,\,1)$ and $\rho_{\star}=\rho_{\star}({\mathcal D},\,F,\,\delta,\,M)\in(0,\,1)$ such that $\alpha\le \alpha_{0}$ and $\rho^{\beta}\le \nu_{0}$ imply
     \begin{equation}\label{Eq (Section 7): Density Result 2}
         \esssup_{\widehat{Q}}\,U_{\delta,\,\varepsilon}\le \left(1-\frac{\nu_{0}}{2}\right)\mu^{2},\quad \text{where}\quad \widehat{Q}\coloneqq B_{\rho}(x_{0})\times I_{\rho}\left(\gamma;\,\tau_{0}+\frac{9}{4}\gamma\rho^{2}\right).
     \end{equation}
 \end{lemma}
 We use Lemmata \ref{Lemma (Section 2): Geometric Decay Lemma}--\ref{Lemma (Section 2): Dimensionless Lemma} to prove Lemma \ref{Lemma (Section 6): Density Lemma 2}.
 Although \cite[Chapter II, Lemma 7.2]{LSU MR0241822} provides a similar result, our proof is slightly different. In particular, we treat the dimensionless quantities introduced in Lemma \ref{Lemma (Section 2): Dimensionless Lemma}.
 \begin{proof}
     We set $\upsilon\coloneqq \frac{1}{n+2}$, $\varkappa\coloneqq \min\left\{\,\frac{{\widehat q}-1}{2},\,\frac{{\widehat r}-1}{2},\,\frac{2\beta}{n+2\beta} \,\right\}\in(0,\,1)$.
     For each $l\in{\mathbb Z}_{\ge 0}$, we choose $k_{l}\coloneqq [1-(2^{-1}+2^{-l-1})\nu_{0}]\mu^{2}\in\lbrack (1-\nu_{0})\mu^{2},\, (1-\nu_{0}/2)\mu^{2})$, $\rho_{l}\coloneqq \left(1+2^{-l-1}\right)\rho\in(\rho,\,3\rho/2\rbrack$, $\tau_{l}\coloneqq \tau_{0}+\frac{9}{4}\gamma\rho^{2}-\gamma\rho_{l}^{2}\in \lbrack \tau_{0},\,\tau_{0}+\frac{5}{4}\gamma\rho^{2})$, $I_{l}\coloneqq (\tau_{l},\,\tau_{0}+\frac{9}{4}\gamma\rho^{2} \rbrack\subset (t_{0}-4\rho^{2},\,t_{0}\rbrack$, and
     \[Y_{l}\coloneqq \frac{\lVert A_{l} \rVert_{L^{1}(I_{l})}}{\lvert Q_{l}\rvert}\in\lbrack0,\,1\rbrack ,\quad Z\coloneqq \gamma^{\frac{c_{q,\,r}n}{n+2}}\frac{\lVert A_{l}^{1/{\widehat q}} \rVert_{L^{{\widehat r}}(I_{l})}}{\lvert Q_{l}\rvert^{\frac{n+2\beta}{n+2}}}\in\lbrack 0,\,1\rbrack,\]
     where $c_{q,\,r}\coloneqq 1/{\widehat q}-1/{\widehat r}$, and $A_{l}(t)\coloneqq \lvert\{x\in B_{l}\mid U_{\delta,\,\varepsilon}(x,\,t)>k_{l} \} \rvert$ for $t\in I_{l}$.
     Fix $l\in{\mathbb Z}_{\ge 0}$, and choose a suitable cut-off function $\eta\in C_{\mathrm c}^{1}(Q_{l};\,\lbrack 0,\,1\rbrack)$ such that $\eta|_{Q_{l+1}}\equiv 1$.
     Then, we have
     \begin{align*}
         2^{-(l+2)}\nu_{0}\mu^{2}\lVert A_{l+1} \rVert_{L^{1}(I_{l+1})} &=(k_{l+1}-k_{l})\lVert A_{l+1} \rVert_{L^{1}(I_{l+1})}\le \lVert \eta(U_{\delta,\,\varepsilon}-k_{l})_{+}\rVert_{L^{1}(I_{l})}\\ 
         &\le \left(\iint_{Q_{l}}\left[\eta(U_{\delta,\,\varepsilon}-k_{l})_{+}\right]^{2+\frac{4}{n}}\,{\mathrm d}X \right)^{\frac{n}{2(n+2)}}  \lVert A_{l}\rVert_{L^{1}(I_{l})}^{1-\frac{n}{2(n+2)}}\\ 
         &\le \frac{C({\mathcal D},\,F,\,\delta,\,M) \nu_{0}\mu^{2}}{\sqrt{\gamma}} \left[\frac{2^{l}\lVert A_{l}\rVert_{L^{1}(I_{l})}^{1/2}}{\rho}+\frac{\sqrt{\gamma}}{\nu_{0}}\lVert A^{1/{\widehat q}}\rVert_{L^{{\widehat r}}(I_{l})}^{1/2} \right]\lVert A_{l}\rVert_{L^{1}(I_{l})}^{\frac{1}{2}+\frac{1}{n+2}}.
     \end{align*}
     by (\ref{Eq (Section 2): PS-Ineq for any s}) with $s=2$ and (\ref{Eq (Section 4): De Girogi Energy 1}).
     Dividing by $\lvert Q_{l+1}\rvert$ and using (\ref{Eq (Section 2): Dimensionless Comparisons}) and (\ref{Eq (Section 2): c-q-r}), we get
     \[Y_{l+1}\le \frac{C({\mathcal D},\,F,\,\delta,\,M)\cdot 4^{l}}{\gamma^{1/2-1/(n+2)}}\left[Y_{l}^{1+\frac{1}{n+2}}+\frac{\gamma^{1/(2\widehat{r})}\rho^{\beta}}{\nu_{0}}Y_{l}^{\frac{1}{n+2}}Z_{l}^{\frac{1+\min\{\,\widehat{q},\,\widehat{r}\,\}}{2}} \right].\]
     where $c_{1}\in\lbrack 0,\,\infty)$ depends on $q$ and $r$.
     We use (\ref{Eq (Section 2): Embedding for level set}) and (\ref{Eq (Section 4): De Girogi Energy 1}) to compute
     \begin{align*}
         4^{-(l+2)}(\nu_{0}\mu^{2})^{2}\lVert A_{l+1}^{1/{\widehat q}} \rVert_{L^{{\widehat r}}(I_{l+1})}&=(k_{l+1}-k_{l})^{2}\lVert A_{l+1}^{1/{\widehat q}} \rVert_{L^{{\widehat r}}(I_{l+1})} \\ 
         &\le \lVert \eta(U_{\delta,\,\varepsilon}-k_{l})_{+} \rVert_{L^{2{\widehat q},\,2{\widehat r}}(Q_{l})}^{2} \\ 
         &\le \frac{C({\mathcal D},\,F,\,\delta,\,M)(\nu_{0}\mu^{2})^{2} }{\gamma} \left[\frac{4^{l}}{\rho^{2}} \lVert A_{l}\rVert_{L^{1}(I_{l})}+\frac{\gamma}{\nu_{0}^{2}}\lVert A_{l}^{1/{\widehat q}} \rVert_{L^{{\widehat r}}(I_{l})} \right]\lVert A^{1/{\widehat q}}\rVert_{L^{\widehat r}(I_{l})}^{\frac{2\beta}{n+2\beta}}. 
     \end{align*}
     Dividing by $\lvert Q_{l+1}\rvert^{\frac{n+2\beta}{n+2}}$, and using (\ref{Eq (Section 2): Dimensionless Comparisons}) and (\ref{Eq (Section 2): c-q-r}), we have
     \[Z_{l+1}\le \frac{C({\mathcal D},\,F,\,\delta,\,M)\cdot 16^{l}}{\gamma^{1-2/(n+2)}}\left[Y_{l}+\frac{\gamma^{1/\widehat{r}}\rho^{2\beta}}{\nu_{0}^{2}} Z_{l}^{1+\frac{2\beta}{n+2\beta}} \right].\]
     Let $\rho^{\beta}\le \gamma^{1/(2\widehat{r})}\nu_{0}$, so that $Y_{l}$ and $Z_{l}$ satisfies the recursice inequalities found in Lemma \ref{Lemma (Section 2): Geometric Decay Lemma} with $A\coloneqq C({\mathcal D},\,F,\,\delta,\,M) \gamma^{-n/(n+2)}$, $B\coloneqq 16$, $\upsilon\coloneqq \frac{1}{n+2}$, $\varkappa\coloneqq \min\{\,(\widehat{q}-1)/2,\,(\widehat{r}-1)/2,\,2\beta/(n+2\beta) \,\}\in(0,\,1)$.
     Keeping in mind that $Y_{0}\le \alpha$ and $Z_{0}\le C(n,\,q,\,r)Y_{0}^{\min\{1/{\widehat q},\,1/{\widehat r}\}}$ hold by (\ref{Eq (Section 7): level set Assumption}) and (\ref{Eq (Section 2): Dimensionless Comparisons}), we choose a sufficiently small $\alpha_{0}=\alpha_{0}({\mathcal D},\,\delta,\,M,\,\gamma)\in(0,\,1)$ such that all of the assumptions of Lemma \ref{Lemma (Section 2): Geometric Decay Lemma} are satisfied.
     In particular, we conclude (\ref{Eq (Section 7): Density Result 2}) as a consequence of Lemma \ref{Lemma (Section 2): Geometric Decay Lemma}.
 \end{proof}

 \subsection{Proof of Proposition \ref{Proposition (Section 4): Degenerate Case}}
 We outline the proof of Proposition \ref{Proposition (Section 4): Degenerate Case}.
 Although the detailed discussions appear almost similar to \cite[Proposition 2.9]{T-supercritical}, we provide the sketch of the proof for the reader's convenience.
 \begin{proof}[Proof of Proposition \ref{Proposition (Section 4): Degenerate Case}]
     Hereinafter we at least let $\rho\le \rho_{\star}$, where $\rho_{\star}({\mathcal D},\,\delta,\,M,\,\nu)\in(0,\,1)$ is given by Lemma \ref{Lemma (Section 6): Expansion of Positivity}.
     By Lemma \ref{Lemma (Section 6): Expansion of Positivity}, there exists $\tau_{0}\in (t_{0}-\rho^{2},\,t_{0}-\nu\rho^{2}/2)$ such that (\ref{Eq (Section 6): Inclusion Assumption})--(\ref{Eq (Section 6): Assumtion for Exp of Pos}) hold with $\gamma=\gamma_{\star}({\mathcal D},\,\delta,\,M,\,\nu)\in(0,\,\nu/9)$, $(\widehat{\nu},\,\widetilde{\nu})=(\theta_{0}\nu,\,\nu/8)$ for some $\theta_{0}=\theta_{0}(n,\,\nu)\in(0,\,1/2)$.
     Corresponding to $\tau_{0}$ and $\gamma_{\star}$, we set $A\coloneqq (t_{0}-\tau_{0})\rho^{-2}\in(\nu/2,\,1)$ and define $i_{\star}\in{\mathbb N}$ as the unique natural number satisfying $9\gamma_{\star}i_{\star}\ge 4A>9\gamma_{\star}(i_{\star}-1)$, which implies $i_{\star}\ge 4A/(9\gamma_{\star})\eqqcolon d_{\star}\in(1,\,\infty)$.
     We choose $\gamma_{\star\star}\coloneqq 4A/(9i_{\star})\in (2\nu/(9i_{\star}),\,\gamma_{\star}\rbrack$, and determine the natural number $l_{\star}=l_{\star}({\mathcal D},\,\delta,\,M,\,\nu)\in{\mathbb N}$ such that both 
     \[\frac{C_{\star}({\mathcal D},\,\delta,\,M,\,\nu)}{(\nu/8)\sqrt{\gamma_{\star\star}l_{\star}}} \le \alpha_{0}({\mathcal D},\,\delta,\,M,\,\gamma_{\star\star})\quad \text{and} \quad \frac{C_{\star}({\mathcal D},\,\delta,\,M,\,\nu)}{\sqrt{(\gamma_{\star\star}/3)l_{\star}}} \le \alpha_{0}({\mathcal D},\,\delta,\,M,\,\gamma_{\star\star}/3)\]
     hold, where $C_{\star}\in(1,\,\infty)$ and $\alpha_{0}\in (0,\,1)$ are given by Lemmata \ref{Lemma (Section 6): Density Lemma 1}--\ref{Lemma (Section 6): Density Lemma 2}. Finally, we determine the radius $\widetilde{\rho}\in(0,\,\rho_{\star}\rbrack$ that satisfies
     \[\widetilde{\rho}^{\beta}\le 2^{-[(9i_{\star}-8)l_{\star}+9(l_{\star}-1)]}\gamma_{\star\star}^{-1/(2\widehat{r})} \theta_{0}\nu.\]
     Hereinafter, let $\rho\le \widetilde{\rho}$ be in force.
     We would like to prove that 
     \begin{equation}\label{Eq (Section 6): Induction Claim}
         \esssup_{B_{\frac{3}{2}\rho}(x_{0})\times I_{k}}\,U_{\delta,\,\varepsilon}\le \left( 1-2^{-(l_{\star}+1)k}\theta_{0}\nu \right)\mu^{2}
     \end{equation}
     for every $k\in\{\,1,\,\dots\,,\,9i_{\star}-8\,\}$, where $I_{k}\coloneqq (\tau_{0}+\frac{5}{4}\gamma_{\star\star}\rho^{2},\, T_{k}\rbrack\subset (\tau_{0}+\frac{5}{4}\gamma_{\star\star}\rho^{2},\,t_{0}\rbrack$ with $T_{k}\coloneqq \tau_{0}+(2+k/4)\gamma_{\star\star}\rho^{2}\in\lbrack \tau_{0}+\frac{9}{4}\gamma_{\star\star}\rho^{2},\, t_{0}\rbrack$.
     For $k=1$, we apply Lemma \ref{Lemma (Section 6): Density Lemma 1}--\ref{Lemma (Section 6): Density Lemma 2} with $(\gamma,\,\widetilde{\nu},\,\widehat{\nu},\,\nu_{0})=(\gamma_{\star\star},\,\nu/8,\,\theta_{0}\nu,\,2^{-l_{\star}}\theta_{0}\nu)$.
     As a result, the proof of (\ref{Eq (Section 6): Induction Claim}) is completed for the special case $i_{\star}=1$.
     When $i_{\star}>1$, let (\ref{Eq (Section 6): Induction Claim}) holds for all $k\le k_{0}\in\{\,1,\,\dots\,,\,9l_{\star}-9\,\}$.
     We replace $\widetilde{I}_{\frac{3}{2}\rho}(\gamma;\,\tau_{0})$ by $\widetilde{I}_{\frac{3}{2}\rho}(\gamma;\,T_{k_{0}}-\frac{1}{4}\gamma_{\star\star}\rho^{2})$, and apply Lemmata \ref{Lemma (Section 6): Density Lemma 1}--\ref{Lemma (Section 6): Density Lemma 2} with $(\gamma,\,\widetilde{\nu},\,\widehat{\nu},\,\nu_{0})\coloneqq (\gamma_{\star\star}/3,\,1,\,2^{-(l_{\star}+1)k_{0}}\theta_{0}\nu,\,2^{-(l_{\star}+1)k_{0}-l_{\star}}\theta_{0}\nu)$.
     The resulting estimate yields (\ref{Eq (Section 6): Induction Claim}) with $k=k_{0}+1$, which completes the proof of (\ref{Eq (Section 6): Induction Claim}) by induction.
     Noting $T_{9i_{\star}-8}=t_{0}$ and $\tau_{0}+\frac{5}{4}\gamma_{\star\star}\rho^{2}-t_{0}=(1-5/(9i_{\star}))(\tau_{0}-t_{0})\le \frac{4}{9}(\tau_{0}-t_{0})\le -(\sqrt{\nu}\rho/3)^{2}$, from (\ref{Eq (Section 6): Induction Claim}) we conclude 
     \[\esssup_{Q_{\sqrt{\nu}\rho/3}(x_{0},\,t_{0})}\,\lvert \G_{\delta,\,\varepsilon}(\D\bu_{\varepsilon})\rvert_{\bg}^{2}\le \left(1-2^{-(l_{\star}+1)(9i_{\star}-8)}\theta_{0}\nu\right)\mu^{2}\le \left(1-2^{-(l_{\star}+1)(9d_{\star}-8)}\theta_{0}\nu\right)\mu^{2}.\]
     Hence, we conclude (\ref{Eq (Section 4): Oscillation result}) with $\kappa\coloneqq \max\left\{\,(\sqrt{\nu}/6)^{\beta},\,\sqrt{1-2^{-(l_{\star}+1)(9d_{\star}-8)}\theta_{0}\nu} \right\}\in\lbrack (\sqrt{\nu}/6)^{\beta},\, 1)$.
 \end{proof}

 \section{Non-degenerate case}\label{Section: Non-Degenerate}
 Section \ref{Section: Non-Degenerate} aims to show Proposition \ref{Proposition (Section 4): Non-Degenerate Case}.
 Under the assumptions in Proposition \ref{Proposition (Section 4): Non-Degenerate Case}, we would like to show growth estimates for the $L^{2}$-mean oscillation, defined as
 \[\Phi(\tau\rho)\coloneqq \fiint_{Q_{\tau\rho}}\left\lvert \D\bu_{\varepsilon}-(\D\bu_{\varepsilon})_{Q_{\tau\rho}}\right\rvert^{2}\,\d x\d t\quad \text{for }\tau\in(0,\,1\rbrack.\]
 \subsection{Lower estimates for an integral average}
 The first step is to deduce a lower bound estimate of $\lvert (\D\bu_{\varepsilon})_{Q_{\rho}(x_{0},\,t_{0})}\rvert_{\bg(x_{0},\,t_{0})}$ from (\ref{Eq (Section 4): Measure Assumption: Non-Degenerate}).
 Here we keep in mind the following inequalities;
 \begin{equation}\label{Eq (Section 7): Non-Degenerate over S}
 \lvert \D\bu_{\varepsilon}\rvert_{\bg}>\frac{3}{4}\delta+(1-\nu)\mu,\quad    \lvert \D\bu_{\varepsilon}\rvert>\gamma_{0}\left(\frac{3}{4}\delta+(1-\nu)\mu\right)\quad \text{a.e.~in }S_{\rho},
 \end{equation}
 since $(\delta/4)+\gamma_{0}^{-1}\lvert \D\bu_{\varepsilon}\rvert\ge (\delta/4)+\lvert \D\bu_{\varepsilon} \rvert_{\bg}>\varepsilon+\lvert \D\bu_{\varepsilon} \rvert_{\bg} \ge v_{\varepsilon}>\delta+(1-\nu)\mu$ over $S_{\rho}$.
 \begin{lemma}\label{Lemma (Section 7): Energy decay from level set}
 Under the assumptions of Proposition \ref{Proposition (Section 4): Non-Degenerate Case}, where $\nu\in(0,\,1/4)$ is arbitrarily fixed, there exists a constant $C_{\dagger}=C_{\dagger}({\mathcal D},\,\delta,\,M)\in(1,\,\infty)$ such that
    \begin{equation}\label{Eq (Section 7): Non-Degenerate Energy 1}
        \Phi(\sigma\rho)\le C_{\dagger}\mu^{2}\left[(1-\sigma)+\frac{\sqrt{\nu}}{(1-\sigma)^{3}}+\frac{(1+F)\rho^{\beta}}{(1-\sigma)^{2}\sqrt{\nu}}\right]
    \end{equation}
 holds for any $\sigma\in(0,\,1)$.
 \end{lemma}
 \begin{proof}
    Fix $\sigma\in(0,\,1)$, and define the ${\mathbb R}^{Nn}$-valued functions in $I_{\rho}$ as follows;
    \[{\boldsymbol\Sigma}_{1}(t)\coloneqq \fint_{B_{\sigma\rho}}\D\bu_{\varepsilon}(x,\,t)\,{\mathrm d}x,\,{\boldsymbol\Sigma}_{p}(t)\coloneqq \fint_{B_{\sigma\rho}}\left(v_{\varepsilon}^{p-1}\D\bu_{\varepsilon} \right)(x,\,t)\,{\mathrm d}x,\,{\boldsymbol\Theta}(t)\coloneqq {\mathbf G}_{p,\,\varepsilon}^{-1}({\boldsymbol\Sigma}_{p}(t)).\]
    The proof of (\ref{Eq (Section 7): Non-Degenerate Energy 1}) is reduced to the following estimates;
    \begin{equation}\label{Eq (Section 7): Non-Degenerate Energy 2}
        \fiint_{Q_{\sigma\rho}}\left\lvert \D\bu_{\varepsilon}(x,\,t)-{\boldsymbol\Theta}(t) \right\rvert\,{\mathrm d}x{\mathrm d}t\le \frac{C\mu}{\sigma^{n+2}}\left(\frac{\sqrt{\nu}}{1-\sigma}+\frac{(1+F)\rho^{\beta}}{\sqrt{\nu}}\right),
    \end{equation}
    \begin{equation}\label{Eq (Section 7): Non-Degenerate Energy 3}
        \esssup_{\tau_{1},\,\tau_{2}\in I_{\sigma\rho}}\lvert {\boldsymbol\Sigma}_{1}(\tau_{1})-{\boldsymbol\Sigma}_{1}(\tau_{2}) \rvert\le \frac{C\mu}{\sigma^{n}}\left[(1-\sigma)+\frac{\sqrt{\nu}}{(1-\sigma)^{3}}+\frac{(1+F)\rho^{\beta}}{(1-\sigma)^{2}\sqrt{\nu}} \right].
    \end{equation}
    In fact, we compute 
    \begin{align*}
        \Phi(\sigma\rho)&\le 2\fiint_{Q_{\sigma\rho}}\left\lvert \D\bu_{\varepsilon}-{\boldsymbol\Sigma}_{1}(t) \right\rvert^{2}\,\d x\d t+2\fiint_{Q_{\sigma\rho}}\left\lvert {\boldsymbol\Sigma}_{1}(t)-(\D \bu_{\varepsilon})_{Q_{\sigma\rho}} \right\rvert^{2}\,\d x\d t\\
        &\le 2\fiint_{Q_{\sigma\rho}}\left\lvert \D\bu_{\varepsilon}-{\boldsymbol\Theta}(t) \right\rvert^{2}\,\d x\d t+2\fiint_{Q_{\sigma\rho}}\left\lvert {\boldsymbol\Sigma}_{1}(t)-({\boldsymbol\Sigma}_{1})_{I_{\sigma\rho}} \right\rvert^{2}\,\d x\d t\\ 
        &\le C\mu\left[\fiint_{Q_{\sigma\rho}}\left\lvert \D\bu_{\varepsilon}(x,\,t)-{\boldsymbol\Theta}(t) \right\rvert\,{\mathrm d}x{\mathrm d}t+\esssup_{\tau_{1},\,\tau_{2}\in I_{\sigma\rho}}\lvert {\boldsymbol\Sigma}_{1}(\tau_{1})-{\boldsymbol\Sigma}_{1}(\tau_{2}) \rvert\right].
    \end{align*}
    Combining this estimate with (\ref{Eq (Section 7): Non-Degenerate Energy 2})--(\ref{Eq (Section 7): Non-Degenerate Energy 3}) implies
    \[
        \Phi(\sigma\rho)\le \frac{C\mu^{2}}{\sigma^{n+2}}\left[(1-\sigma)+\frac{\sqrt{\nu}}{(1-\sigma)^{3}}+\frac{(1+F)\rho^{\beta}}{(1-\sigma)^{2}\sqrt{\nu}}\right].
    \]
    Noting $\lvert Q_{\rho}\setminus Q_{\sigma\rho}\rvert=(1-\sigma^{n+2})\lvert Q_{\rho}\rvert\le (n+2)(1-\sigma)\lvert Q_{\rho}\rvert$, and 
    \begin{align*}
        \Phi(\rho)&\le \fiint_{Q_{\rho}}\lvert \D\bu_{\varepsilon}-(\D\bu_{\varepsilon})_{Q_{\sigma\rho}} \rvert^{2}\,\d x\d t\\ 
        &=\sigma^{n+2}\Phi(\sigma\rho)+\lvert Q_{\rho}\rvert^{-1}\iint_{Q_{\rho}\setminus Q_{\sigma\rho}}\lvert \D\bu_{\varepsilon}-(\D\bu_{\varepsilon})_{Q_{\sigma\rho}}\rvert^{2}\,\d x\d t,
    \end{align*}
    we find $C_{\dagger}\in(1,\,\infty)$ satisfying (\ref{Eq (Section 7): Non-Degenerate Energy 1}).

    In the proof of (\ref{Eq (Section 7): Non-Degenerate Energy 2})--(\ref{Eq (Section 7): Non-Degenerate Energy 3}), we note that (\ref{Eq (Section 4): Measure Assumption: Non-Degenerate}) implies 
    \begin{equation}\label{Eq (Section 7): Sublevel-set is small}
        \lvert Q_{\sigma\rho}\setminus S_{\sigma\rho} \rvert\le \sigma^{-(n+2)}\nu\lvert Q_{\sigma\rho}\rvert\quad \text{for any }\sigma\in(0,\,1).
    \end{equation}
    Recalling Lemma \ref{Lemma (Section 2): G-p-epsilon} and using (\ref{Eq (Section 1): Matrix Gamma}), (\ref{Eq (Section 7): Sublevel-set is small}), we compute the left-hand side of (\ref{Eq (Section 7): Non-Degenerate Energy 2}) as follows;
    \begin{align*}
        &\frac{1}{\lvert Q_{\sigma\rho}\rvert}\left(\iint_{S_{\sigma\rho}}\left\lvert \D\bu_{\varepsilon}(x,\,t)-{\boldsymbol\Theta}(t) \right\rvert\,{\mathrm d}x{\mathrm d}t+\iint_{Q_{\sigma\rho}\setminus S_{\sigma\rho}}\left\lvert \D\bu_{\varepsilon}(x,\,t)-{\boldsymbol\Theta}(t) \right\rvert\,{\mathrm d}x{\mathrm d}t \right)\\ 
        &\le \frac{C}{\mu^{p-1}}\frac{1}{\lvert Q_{\sigma\rho}\rvert}\iint_{S_{\sigma\rho}}\left\lvert v_{\varepsilon}^{p-1}\D\bu_{\varepsilon}-{\boldsymbol\Sigma}_{p}(t) \right\rvert\,\d x\d t+\frac{C\nu\mu}{\sigma^{n+2}}\\ 
        &\le \frac{C}{\mu^{p-1}}\cdot(\sigma\rho)\fiint_{Q_{\sigma\rho}}\left\lvert \D\left[v_{\varepsilon}^{p-1}\D\bu_{\varepsilon} \right] \right\rvert\,\d x\d t+\frac{C\nu\mu}{\sigma^{n+2}}.
    \end{align*}
    We use (\ref{Eq (Section 4): Energy 1})--(\ref{Eq (Section 4): Energy 2}), (\ref{Eq (Section 7): Sublevel-set is small}), and H\"{o}lder's inequality to compute
    \begin{align*}
        &\fiint_{Q_{\sigma\rho}}\left\lvert \D\left[v_{\varepsilon}^{p-1}\D\bu_{\varepsilon} \right] \right\rvert\,\d x\d t\\ 
        &\le \left(\frac{1}{\lvert Q_{\sigma\rho}\rvert}\iint_{S_{\sigma\rho}}\left\lvert \D\left[v_{\varepsilon}^{p-1}\D\bu_{\varepsilon} \right] \right\rvert^{2}\,\d x\d t \right)^{\frac{1}{2}}+\frac{\nu^{\frac{1}{2}}}{\sigma^{\frac{n}{2}+1}}\left(\fiint_{Q_{\sigma\rho}}\left\lvert \D\left[v_{\varepsilon}^{p-1}\D\bu_{\varepsilon} \right] \right\rvert^{2}\,\d x\d t\right)^{\frac{1}{2}}\\ 
        &\le \frac{C\mu^{p}}{\sigma^{\frac{n}{2}+1}\rho}\left[\frac{\sqrt{\nu}}{1-\sigma}+\frac{(1+F)\rho^{\beta}}{\sqrt{\nu}}+\frac{\sqrt{\nu}}{\sigma^{\frac{n}{2}+1}}\left(\frac{1}{1-\sigma}+(1+F)\rho^{\beta}\right) \right].
    \end{align*}
    Combining these estimates, we obtain (\ref{Eq (Section 7): Non-Degenerate Energy 2}).
    To prove (\ref{Eq (Section 7): Non-Degenerate Energy 3}), we may let $\tau_{1}<\tau_{2}$. 
    For sufficiently small $\widetilde{\varepsilon}>0$, which will be diminished later, we define the piecewise linear function $\phi_{\widetilde{\varepsilon}}\colon I_{\sigma\rho}\to \lbrack 0,\,1\rbrack$ as (\ref{Eq (Section 5): Piecewise Linear}).
    We set $\widetilde\sigma\coloneqq (\sigma+1)/2\in (1/2,\,1)\cap(\sigma,\,1)$ and choose a cut-off function $\eta\in C_{\mathrm c}^{2}(B_{\sigma\rho};\,\lbrack 0,\,1\rbrack)$ that is supported in $B_{\widetilde{\sigma}\rho}$, and satisfies $\eta|_{B_{\sigma\rho}}=1$ and $\lVert \nabla\eta \rVert_{L^{\infty}(B_{\rho})}^{2}+\lVert \nabla^{2}\eta \rVert_{L^{\infty}(B_{\rho})}\le \frac{c}{(1-\sigma)^{2}\rho^{2}}$.
    We fix $j\in\{\,1,\,\dots\,N\,\}$ and $\kappa\in\{\,1,\,\dots\,,\,n\,\}$, and test a cut-off function $\bphi$, defined as $\bphi(x,\,t)\coloneqq (\delta^{ij}\partial_{x_{\kappa}}(\phi_{\widetilde{\varepsilon}}(t)\eta(x)))_{i}$, into (\ref{Eq (Section 2): Approximate Weak Form}).
    Integrating by parts, letting $\widetilde{\varepsilon}\to 0$, and recalling our choice of $\eta$, we have
    \begin{align*}
        &\left\lvert ({\boldsymbol\Sigma}(\tau_{1})-{\boldsymbol\Sigma}(\tau_{2}))_{\kappa}^{j} \right\rvert\\ 
        &\le \frac{\gamma_{0}^{-2}}{\lvert B_{\sigma\rho}\rvert}\iint_{Q_{{\widetilde\sigma}\rho}}\left\lvert {\mathbf A}_{\varepsilon}(x,\,t,\,\D\bu_{\varepsilon})-{\mathbf A}_{\varepsilon}(x,\,t,\,{\boldsymbol\Theta}(t)) \right\rvert\lvert \nabla^{2}\eta\rvert\,\d x\d t\\ 
        &\quad +
        \frac{\gamma_{0}^{-2}}{\lvert B_{\sigma\rho}\rvert}\iint_{Q_{{\widetilde\sigma}\rho}}\lvert {\mathbf A}_{\varepsilon}(x,\,t,\,{\boldsymbol\Theta}(t))-{\mathbf A}_{\varepsilon}(x_{0},\,t,\,{\boldsymbol\Theta}(t))\rvert\lvert \nabla^{2}\eta\rvert\,\d x\d t\\ 
        &\quad \quad +\frac{1}{\lvert B_{\sigma\rho}\rvert}\iint_{Q_{\sigma\rho}}\lvert \buf_{\varepsilon}\rvert\lvert\nabla\eta\rvert\,\d x\d t+C\mu\frac{\lvert B_{{\widetilde\sigma}\rho}\setminus B_{\sigma\rho} \rvert}{\lvert B_{\sigma\rho}\rvert}\\ 
        &\le \frac{C{\widetilde\sigma}^{n+2}}{\sigma^{n}(1-\sigma)^{2}}\fiint_{Q_{{\widetilde\sigma}\rho}}\lvert {\mathbf A}_{\varepsilon}(x,\,t,\,\D\bu_{\varepsilon})-{\mathbf A}_{\varepsilon}(x,\,t,\,{\boldsymbol\Theta}(t)) \rvert\,\d x\d t\\ 
        &\quad +\frac{C\mu\rho}{\sigma^{n}(1-\sigma)^{2}}
        +\frac{CF\rho^{\beta}}{\sigma^{n}(1-\sigma)}+C(1-\sigma)\mu.
    \end{align*}
    Decomposing $Q_{{\widetilde\sigma}\rho}=S_{\widetilde\sigma\rho}\cup(Q_{{\widetilde\sigma}\rho}\setminus S_{\widetilde\sigma\rho})$, and using (\ref{Eq (Section 2): Continuity outside the origin}), (\ref{Eq (Section 7): Non-Degenerate Energy 2}), and (\ref{Eq (Section 7): Sublevel-set is small}) with $\sigma$ replaced by $\widetilde\sigma$, we compute
    \begin{align*}
        &\fiint_{Q_{{\widetilde\sigma}\rho}}\lvert {\mathbf A}_{\varepsilon}(x,\,t,\,\D\bu_{\varepsilon})-{\mathbf A}_{\varepsilon}(x,\,t,\,{\boldsymbol\Theta}(t)) \rvert\,\d x\d t\\ 
        &\le C\mu\frac{\lvert Q_{\widetilde\sigma\rho}\setminus S_{\widetilde\sigma\rho} \rvert}{\lvert Q_{\widetilde\sigma\rho}\rvert}+C\fiint_{Q_{\widetilde\sigma\rho}}\lvert \D\bu_{\varepsilon}(x,\,t)-{\boldsymbol\Theta}(t) \rvert\,\d x\d t\\
        &\le \frac{C\nu\mu}{\widetilde\sigma^{n+2}}+\frac{C\mu}{\widetilde\sigma^{n+2}}\left(\frac{\sqrt{\nu}}{1-\widetilde{\sigma}}+\frac{(1+F)\rho^{\beta}}{\sqrt{\nu}} \right).    
    \end{align*}
    Recalling our choice of ${\widetilde\sigma}$, and summing over $j\in \{\,1,\,\dots\,,\,N\,\},\,\kappa\in\{\,1,\,\dots\,,\,n\,\}$ we conclude (\ref{Eq (Section 7): Non-Degenerate Energy 3}).    
 \end{proof}
 \begin{lemma}\label{Lemma (Section 7): Energy Decay and Average Estimate}
    Let $u_{\varepsilon}$ be a weak solution to (\ref{Eq (Section 2): Approximate System}) with $\varepsilon\in(0,\,\delta/4)$, and assume that (\ref{Eq (Section 4): Uniform Bounds F and U}) and (\ref{Eq (Section 4): Bound Assumption of v-epsilon})--(\ref{Eq (Section 4): delta vs mu}) hold for some constants $F$ and $M$.
    For each $\theta\in(0,\,\gamma_{0}^{2}/256)$, set
    \[\sigma\coloneqq 1-\frac{\theta}{3C_{\dagger}},\quad \nu\coloneqq \left(\frac{\theta(1-\sigma)^{3}}{3C_{\dagger}}\right)^{2}=\left(\frac{\theta}{3C_{\dagger}}\right)^{8}<\frac{\gamma_{0}^{16}}{10^{23}},\quad \rho_{\ast}\coloneqq\min\left\{\,\frac{1}{16c_{\dagger}},\,\left(\frac{(1-\sigma)^{2}\sqrt{\nu}}{3C_{\dagger}(1+F)}\right)^{\frac{1}{\beta}}\,\right\},\] 
    where $c_{\dagger}\in(0,\,\infty)$ and $C_{\dagger}\in(1,\,\infty)$ are the constants respectively from Lemmata \ref{Lemma (Section 2): Continuity on Norms} and \ref{Lemma (Section 7): Energy decay from level set}.
    Then, we have
    \begin{equation}\label{Eq (Section 7): Measure Result Final}
        \Phi(\rho)\le \theta\mu^{2},\quad \text{and}\quad \lvert (\D\bu_{\varepsilon})_{Q_{\rho}(x_{0},\,t_{0})}\rvert_{\bg(x_{0},\,t_{0})}\ge \frac{\mu}{2}+\delta,
    \end{equation}
    provided (\ref{Eq (Section 4): Measure Assumption: Non-Degenerate}) and $\rho\le \rho_{\ast}$.
 \end{lemma}
 \begin{proof}
    The former claim of (\ref{Eq (Section 7): Measure Result Final}) immediately follows from (\ref{Eq (Section 7): Non-Degenerate Energy 1}) and our choice of $\sigma$, $\nu$ and $\rho_{\ast}$.
    To prove the latter one, we use (\ref{Eq (Section 4): Measure Assumption: Non-Degenerate}) and (\ref{Eq (Section 7): Non-Degenerate over S}) to compute
    \begin{align*}
        \fiint_{Q_{\rho}}\lvert \D\bu_{\varepsilon} \rvert_{\bg}\,\d x\d t&\ge \frac{\lvert S_{\rho}\rvert}{\lvert Q_{\rho}\rvert}\essinf_{S_{\rho}}\,\lvert\D\bu_{\varepsilon}\rvert_{\bg}\ge (1-\nu)\left((1-\nu)\mu+\frac{3}{4}\delta \right)\\& \ge\frac{3-11\nu}{4}\mu+\delta\ge \frac{5}{8}\mu+\delta,
    \end{align*}
    where we note that $\mu>\delta$ and $\nu\le 1/22$ yield the last two inequalities.
    We use the Cauchy--Schwarz inequality, (\ref{Eq (Section 2): Continuity of Norm}), and the former claim of (\ref{Eq (Section 7): Measure Result Final}) to obtain
    \begin{align*}
        &\left\lvert \fiint_{Q_{\rho}}\lvert \D\bu_{\varepsilon}\rvert_{\bg}\,\d x\d t-\lvert (\D\bu_{\varepsilon})_{Q\rho} \rvert_{\bg(x_{0},\,t_{0})} \right\rvert\\ 
        &\le\fiint_{Q_{\rho}}\lvert \D\bu_{\varepsilon}-(\D\bu_{\varepsilon})_{Q_{\rho}} \rvert_{\bg}\,\d x\d t+\fiint_{Q_{\rho}}\left\lvert\lvert (\D\bu_{\varepsilon})_{Q_{\rho}}\rvert_{\bg}-\lvert(\D\bu_{\varepsilon})_{Q_{\rho}}\rvert_{\bg(x_{0},\,t_{0})} \right\rvert\,\d x\d t\\ 
        &\le \left(\gamma_{0}^{-1}\sqrt{\theta}+c_{\dagger}\rho_{\ast}\right)\mu\le \frac{\mu}{16}+\frac{\mu}{16}= \frac{\mu}{8}.
    \end{align*}
    Combining these estimates with the triangle inequality completes the proof of the latter claim.
 \end{proof}

 \subsection{Higher integrability and comparison estimates}
 Throughout this subsection, we do not necessarily assume (\ref{Eq (Section 4): Measure Assumption: Non-Degenerate}).
 Here we aim to deduce comparison estimates with some classical heat flows (Lemmata \ref{Lemma (Section 7): Comparison Lemma}--\ref{Lemma (Section 7): Energy Estimate from Perturbation}), after deducing a higher integrability estimate (Lemma \ref{Lemma (Section 7): Higher integrability}).
 \begin{lemma}\label{Lemma (Section 7): Higher integrability}
    Let all of the assumptions of Proposition \ref{Proposition (Section 4): Non-Degenerate Case}, except (\ref{Eq (Section 4): Measure Assumption: Non-Degenerate}), be in force.
    If $\bz_{0}\in{\mathbb R}^{Nn}$ satisfies
    \begin{equation}\label{Eq (Section 7): Assumption for Higher Integrability}
        \frac{\mu}{4}+\delta\le \lvert \bz_{0} \rvert_{\bg(x_{0},\,t_{0})}\le \frac{M}{\gamma_{0}},
    \end{equation}
    then there exists a sufficiently small constant $\vartheta=\vartheta({\mathcal D},\,\delta,\,M)>0$ such that
    \begin{equation}\label{Eq (Section 7): Higher Integrability}
        \fiint_{Q_{\rho/2}} \lvert\D\bu_{\varepsilon}-\bz_{0} \rvert^{2(1+\vartheta)}\,\d x\d t\le C\left[\fiint_{Q_{\rho}}\lvert \D\bu_{\varepsilon}-\bz_{0}\rvert^{2}\,\d x\d t+\left(1+F^{2}\right)\rho^{2\beta}\right]^{1+\vartheta}.
    \end{equation}
 \end{lemma}
 \begin{proof}
    It suffices to prove that the reversed H\"{o}lder inequality
    \begin{align}\label{Eq (Section 7): Claim for Gehring}
        \fiint_{Q_{R}}\lvert \D\bu_{\varepsilon}-\bz_{0} \rvert^{2}\,\d x\d t&\le A\left[\left(\fiint_{Q_{2R}}\lvert\D\bu_{\varepsilon}-\bz_{0} \rvert^{\frac{n+2}{2n}}\,\d x\d t \right)^{\frac{n}{n+2}}+\fiint_{Q_{2R}}\lvert \rho \left(1+\lvert\buf_{\varepsilon}\rvert \right)\rvert^{2}\,\d x\d t \right]\nonumber\\ 
        &\quad +\tau\fiint_{Q_{2R}}\lvert \D\bu_{\varepsilon}-\bz_{0} \rvert^{2}\,\d x\d t.
    \end{align}
    holds for any $Q_{2R}(y_{0},\,s_{0})\subset Q_{\rho}(x_{0},\,t_{0})$.
    Here, $\tau\in(0,\,1)$ is arbitrarily fixed, and the bound $A=A({\mathcal D},\,\delta,\,M,\,\tau)\in(1,\,\infty)$ is independent of $\varepsilon\in(0,\,\delta/4)$.
    Then, the parabolic version of Gehring's lemma implies the existence of the positive exponent $\vartheta=\vartheta({\mathcal D},\,\delta,\,M)$ satisfying
    \[\fiint_{Q_{\rho/2}}\lvert \D\bu_{\varepsilon}-\bz_{0} \rvert^{2(1+\vartheta)}\,\d x\d t\le C\left[\left(\fiint_{Q_{\rho}}\lvert \D\bu_{\varepsilon}-\bz_{0}\rvert^{2}\,\d x\d t \right)^{1+\vartheta}+\fiint_{Q_{\rho}}\lvert \rho \left(1+\lvert\buf_{\varepsilon}\rvert \right)\rvert^{2(1+\vartheta)}\,\d x\d t\right]\] 
    for some constant $C=C({\mathcal D},\,\delta,\,M)\in(0,\,\infty)$.
    Without loss of generality, we may let $2(1+\vartheta)$ be close to $2$, so that H\"{o}lder's inequality can be applied to $\lvert \rho \left(1+\lvert\buf_{\varepsilon}\rvert \right)\rvert^{2(1+\vartheta)}$.
    Hence, (\ref{Eq (Section 7): Higher Integrability}) is easily concluded.

    Fix $Q_{2R}=Q_{2R}(y_{0},\,s_{0})=B_{2R}(y_{0})\times I_{2R}(s_{0})\subset Q_{\rho}(x_{0},\,t_{0})$, and let $\eta\in C_{\mathrm c}^{1}(B_{2R};\,\lbrack 0,\,1\rbrack)$ and $\widetilde{\eta}\in C_{\mathrm c}^{1}(I_{2R}(s_{0});\,\lbrack 0,\,1\rbrack)$ be arbitrarily chosen.
    We introduce the ${\mathbb R}^{N}$-valued functions
    \[\bw_{\varepsilon}(x,\,t)\coloneqq \bu_{\varepsilon}(x,\,t)-\bz_{0}(x-x_{0}),\quad \widetilde{\bw}_{\varepsilon}(t)\coloneqq \left(\int_{B_{2R}}\eta^{2}\,\d x\right)^{-1}\int_{B_{2R}}\bw_{\varepsilon}(x,\,t)\,\d x\]
    for $(x,\,t)\in Q_{2R}(y_{0},\,s_{0})$.
    We consider the function $\bphi$ of the form $\bphi\coloneqq \eta^{2}\widetilde{\eta}^{2}\phi(\bw_{\varepsilon}-\widetilde{\bw}_{\varepsilon})$, where $\phi\colon \lbrack s_{0}-(2R)^{2},\,s_{0}\rbrack\to\lbrack 0,\,1\rbrack$ is a non-increasing function that satisfies $\phi(s_{0})=0$.
    Then, this $\bphi$ satisfies $\int_{B_{2R}}\boldsymbol{\varphi}(x,\,t)\,\d x=0$ for a.e.~$t\in (s_{0}-(2R)^{2},\,s_{0})$, which implies $\int_{B_{2R}}\partial_{t}\boldsymbol{\varphi}(x,\,t)\,\d x=0$ for a.e.~$t\in (s_{0}-(2R)^{2},\,s_{0})$.
    Therefore, $\bphi$ is an admissible test function into the weak formulation
    \begin{align*}
        &-\int_{s_{0}-(2R)^{2}}^{s_{0}}\langle \bw_{\varepsilon}-\widetilde{\bw}_{\varepsilon}\mid \partial_{t}\bphi \rangle\,\d t +\iint_{Q_{2R}}\left\langle \A_{\varepsilon}(x,\,t,\,\D\bu)-\A_{\varepsilon}(x,\,t,\,\bz_{0})\mathrel{}\middle|\mathrel{}\D\bphi\right\rangle_{\bg}\,\d x\d t\\ 
        &=\iint_{Q_{2R}}\langle \buf_{\varepsilon}\mid \bphi\rangle\,\d x\d t-\iint_{Q_{2R}}\left\langle \A_{\varepsilon}(x,\,t,\,\bz_{0}) \mathrel{} \middle|\mathrel{} \D\bphi\right\rangle_{\bg}-\left\langle \A_{\varepsilon}(x_{0},\,t,\,\bz_{0})\mathrel{}\middle|\mathrel{}\D\bphi \right\rangle_{\bg(x_{0},\,t)}\,\d x\d t.
    \end{align*}
    By (\ref{Eq (Section 7): Assumption for Higher Integrability}), we are allowed to use not only (\ref{Eq (Section 2): Continuity of A-epsilon in x}) but also (\ref{Eq (Section 2): Ellipticity outside the origin})--(\ref{Eq (Section 2): Continuity outside the origin}). 
    With this in mind, we integrate by parts and carry out an absorbing argument to get
    \begin{align*}
        &-\iint_{Q_{2R}}\lvert \bw_{\varepsilon}-\widetilde{\bw}_{\varepsilon}\rvert^{2}\eta^{2}\widetilde{\eta}^{2}\partial_{t}\phi\,\d x\d t+\iint_{Q_{2R}}\lvert \D\bw_{\varepsilon}\rvert^{2}\eta^{2}\widetilde{\eta}^{2}\phi\,\d x\d t\\ 
        &\le C({\mathcal D},\,\delta,\,M)\left(\iint_{Q_{2R}}\lvert \bw_{\varepsilon}-\widetilde{\bw}_{\varepsilon}\rvert^{2}\left(\lvert\nabla\eta\rvert^{2}+\lvert\partial_{t}\widetilde{\eta} \rvert+R^{-2}\right)\,\d x\d t+R^{2}\iint_{Q_{2R}}\left(1+\lvert \buf_{\varepsilon} \rvert^{2}\right)\eta^{2}\widetilde{\eta}^{2}\,\d x\d t \right).
    \end{align*}
    Noting $\D\bw_{\varepsilon}=\D\bu_{\varepsilon}-\bz_{0}$ and $2R\le \rho$, we conclude (\ref{Eq (Section 7): Claim for Gehring}) from the last estimate (see \cite[Lemma 6.1]{BDLS-parabolic}, \cite[\S 2]{Giaquinta-Struwe MR0652852}, or \cite[Lemma 6.3]{T-supercritical} for the detailed discussions).
 \end{proof}

 We consider a comparison function $\bv_{\varepsilon}$, which satisfies a classical linear system.
 \begin{lemma}\label{Lemma (Section 7): Comparison Lemma}
    Let $\bu_{\varepsilon}$ be a weak solution to (\ref{Eq (Section 2): Approximate System}) in $\widetilde{\mathcal Q}\Subset {\mathcal Q}\Subset \Omega_{T}$ and let (\ref{Eq (Section 4): Bound Assumption of v-epsilon})--(\ref{Eq (Section 4): delta vs mu}) be in force.
    Assume that the condition 
    \begin{equation}\label{Eq (Section 7): Average Assumption}
        \delta+\frac{\mu}{4}\le \lvert (\D\bu_{\varepsilon})_{Q_{\rho}(x_{0},\,t_{0})}\rvert_{\bg(x_{0},\,t_{0})}
    \end{equation}
    is satisfied.
    Then, there uniquely exists a function $\bv_{\varepsilon}\in \bu_{\varepsilon}+X_{0}^{2}(I_{\rho/2}(t_{0});\,B_{\rho/2}(x_{0}))$ that satisfies 
    \begin{equation}\label{Eq: (Section 7) Dirichlet Boundary}
        -\iint_{Q_{\rho/2}}\langle \bv_{\varepsilon}\mid \partial_{t}{\boldsymbol\varphi}\rangle\,\d x\d t+\iint_{Q_{\rho/2}} {\mathcal B}_{\varepsilon}(x_{0},\,t,\,(\D\bu_{\varepsilon})_{Q_{\rho}})(\D\bv_{\varepsilon},\,\D{\boldsymbol\varphi})\,{\mathrm d}x{\mathrm d}t=0
    \end{equation}
    for any ${\boldsymbol\varphi}\in C_{\mathrm c}^{1}(Q_{\rho/2})^{N}$, and $(\bv_{\varepsilon}-\bu_{\varepsilon})(\,\cdot\,,\,t_{0}-(\rho/2)^{2})=0$ in $L^{2}(B_{\rho/2}(x_{0}))^{N}$.
    Moreover, there exists a constant $C\in(1,\,\infty)$, depending at most on data, such that we have
    \begin{equation}\label{Eq: (Section 7) Basic Heat Estimates}
        \fiint_{Q_{\sigma\rho}}\lvert \D\bv_{\varepsilon}-(\D\bv_{\varepsilon})_{Q_{\sigma\rho}} \rvert^{2}\,\d x\d t\le C\sigma^{2}\fiint_{Q_{\rho/2}}\lvert \D\bv_{\varepsilon}-(\D\bv_{\varepsilon})_{Q_{\rho/2}}\rvert^{2}\,\d x\d t
    \end{equation}
    for all $\sigma\in(0,\,1/2\rbrack$, and
    \begin{equation}\label{Eq: (Section 7) Comparison Estimate}
        \fiint_{Q_{\rho/2}}\lvert \D\bu_{\varepsilon}-\D\bv_{\varepsilon} \rvert^{2}\,\d x\d t\le C({\mathcal D},\,\delta,\,M)\left[\omega\left(\sqrt{\Phi(\rho)}\right)^{\frac{\vartheta}{1+\vartheta}}\Phi(\rho) +\left(1+F^{2}\right)\rho^{2\beta} \right],
    \end{equation}
    where the positive exponent $\vartheta$ is as in Lemma \ref{Lemma (Section 7): Higher integrability}, and the concave function $\omega$ is determined by Lemma \ref{Lemma (Section 2): Continuity of Hessian} with $c_{1}=\delta/4$ and $c_{2}=M_{0}/\gamma_{0}$.
 \end{lemma}
 \begin{proof}
    By (\ref{Eq (Section 1): Matrix Gamma}), (\ref{Eq (Section 4): Bound Assumption of v-epsilon}) and (\ref{Eq (Section 7): Average Assumption}), the matrix $\bz_{0}\coloneqq (\D\bu_{\varepsilon})_{Q_{\rho}}\in{\mathbb R}^{Nn}$ satisfies (\ref{Eq (Section 7): Assumption for Higher Integrability}).
    Hence, ${\mathcal B}_{\varepsilon}(x_{0},\,t,\,\bz_{0})$ is uniformly elliptic in the classical sense for a.e.~$t\in (t_{0}-(\rho_{0}/2)^{2},\,t_{0})$, and
    Therefore, the unique existence of (\ref{Eq: (Section 7) Dirichlet Boundary}) under the parabolic Dirichlet boundary is clear (see \cite[Chapitre 2]{Lions-monotone MR0259693} or \cite[Chapter III]{Showalter MR1422252}).
    Since ${\mathcal B}_{\varepsilon}(x_{0},\,t,\,\bz_{0})$ is independent of the spatial variable $x$, we can deduce classical Caccioppoli estimates of higher spatial derivatives of $\bv_{\varepsilon}$.
    The basic estimate (\ref{Eq: (Section 7) Basic Heat Estimates}) is an immediate consequence of these estimates, which are found in the proof of \cite[Lemma 6.3]{BDLS-parabolic} (see also \cite[Lemma 5.1]{Campanato MR0213737} for the classical results in the time-independent cases).

    The weak formulations (\ref{Eq (Section 2): Approximate Weak Form}) and (\ref{Eq: (Section 7) Dirichlet Boundary}) show that the identity
    \begin{align*}
        &-\iint_{Q_{\rho/2}}(\bu_{\varepsilon}-\bv_{\varepsilon})\partial_{t}{\boldsymbol\varphi}\,\d x\d t+\iint_{Q_{\rho/2}}{\mathcal B}_{\varepsilon}(x_{0},\,t,\,\bz_{0})(\D\bu_{\varepsilon}-\D\bv_{\varepsilon},\,\D{\boldsymbol\varphi})\,{\mathrm d}x{\mathrm d}t\\ 
        &=\iint_{Q_{\rho/2}}\left({\mathbf J}_{1}(x,\,t)+{\mathbf J}_{2}(x,\,t) \right)\,\d x\d t+\iint_{Q_{\rho/2}}\langle\buf_{\varepsilon}\mid{\boldsymbol\varphi}\rangle\,\d x\d t
    \end{align*} 
    holds for all ${\boldsymbol\varphi}\in C_{\mathrm c}^{1}(Q_{\rho/2};\,{\mathbb R}^{N})$,
    where 
    \[{\mathbf J}_{1}(x,\,t)\coloneqq {\mathcal B}_{\varepsilon}(x_{0},\,t,\,\bz_{0})(\D\bu_{\varepsilon}-\bz_{0},\,\D{\boldsymbol\varphi})-\left\langle {\mathbf A}_{\varepsilon}(x_{0},\,t,\,\D\bu_{\varepsilon})-{\mathbf A}_{\varepsilon}(x_{0},\,t,\,\bz_{0})\mathrel{}\middle|\mathrel{}\D{\boldsymbol\varphi} \right\rangle_{\bg(x_{0},\,t)},\]
    \[{\mathbf J}_{2}(x,\,t)\coloneqq \langle {\mathbf A}_{\varepsilon}(x_{0},\,t,\,\D\bu_{\varepsilon})\mid \D\bphi \rangle_{\bg(x_{0},\,t)}-\langle \A_{\varepsilon}(x,\,t,\,\D\bu_{\varepsilon})\mid \D\bphi \rangle_{\bg(x,\,t)}.\]
    We test $\bphi\coloneqq \phi_{\widetilde{\varepsilon}}(\bu_{\varepsilon}-\bv_{\varepsilon})$ into this weak formulation, where $\phi_{\widetilde{\varepsilon}}(t)\coloneqq \min \{\,1,\,(t_{0}-t)/\widetilde{\varepsilon} \,\}$ for $t\in I_{\rho/2}(t_{0})$ with $\widetilde{\varepsilon}>0$ being sufficiently small.
    Integrating by parts, and using H\"{o}lder's inequality and the Poincar\'{e} inequality, we have 
    \begin{align*}
        &-\fiint_{Q_{\rho/2}}\lvert \bu_{\varepsilon}-\bv_{\varepsilon} \rvert^{2}\partial_{t}\phi_{\widetilde{\varepsilon}}\,\d x\d t+\fiint_{Q_{\rho/2}}\lvert \D\bu_{\varepsilon}-\D\bv_{\varepsilon}\rvert^{2}\phi_{\widetilde{\varepsilon}}\,\d x\d t\\ 
        &\le C\fiint_{Q_{\rho/2}}\omega(\lvert \D\bu_{\varepsilon}-\bz_{0} \rvert)\lvert \D\bu_{\varepsilon}-\bz_{0}\rvert \lvert \D\bu_{\varepsilon}-\D\bv_{\varepsilon}\rvert\phi_{\widetilde{\varepsilon}} \,\d x\d t\\ 
        &\quad +C\rho\fiint_{Q_{\rho/2}}\lvert \D\bu_{\varepsilon}-\D\bv_{\varepsilon}\rvert\phi_{\widetilde{\varepsilon}}\,\d x\d t+C\fiint_{Q_{\rho/2}}\lvert \buf_{\varepsilon}\rvert\lvert\bu_{\varepsilon}-\bv_{\varepsilon}\rvert\phi_{\widetilde{\varepsilon}}\, \d x\d t \\ 
        &\le C\left[\left(1+F^{2}\right)\rho^{2\beta}+\fiint_{Q_{\rho/2}}\omega(\lvert \D\bu_{\varepsilon}-\bz_{0}\rvert)^{2}\lvert \D\bu_{\varepsilon}-\bz_{0}\rvert^{2}\phi_{\widetilde{\varepsilon}}\,\d x\d t \right]^{1/2}\left(\fiint_{Q_{\rho/2}}\lvert \D\bu_{\varepsilon}-\D\bu_{\varepsilon}\rvert^{2}\phi_{\widetilde{\varepsilon}}\,\d x\d t \right)^{1/2}
    \end{align*}
    where we have used Lemmata \ref{Lemma (Section 2): Continuity of Hessian}--\ref{Lemma (Section 2): Continuity on Norms} and (\ref{Eq (Section 4): Bound Assumption of v-epsilon}) to estimate ${\mathbf J}_{1}$ and ${\mathbf J}_{2}$.
    Deleting the first integral in the right-hand side, making absorptions, and finally letting ${\widetilde\varepsilon}\to 0$, we have 
    \begin{align*}
        &\fiint_{Q_{\rho/2}}\lvert \D\bu_{\varepsilon}-\D\bv_{\varepsilon}\rvert^{2}\,\d x\d t\\ 
        &\le C(1+F^{2})\rho^{2\beta}+C\left[\omega\left(\frac{2M}{\gamma_{0}}\right)^{\frac{\vartheta+2}{\vartheta}}\fiint_{Q_{\rho/2}}\omega(\lvert \D\bu_{\varepsilon}-\bz_{0}\rvert)\,\d x\d t \right]^{\frac{\vartheta}{1+\vartheta}}\left[\fiint_{Q_{\rho/2}}\lvert \D\bu_{\varepsilon}-\bz_{0} \rvert^{2(1+\vartheta)}\,\d x\d t \right]^{\frac{1}{1+\vartheta}}
    \end{align*}
    with $\vartheta>0$ given by Lemma \ref{Lemma (Section 7): Higher integrability}.
    We apply Jensen's inequality to the concave function $\omega$, and use the Cauchy--Schwarz inequality, and (\ref{Eq (Section 7): Higher Integrability}) to deduce (\ref{Eq: (Section 7) Comparison Estimate}).
 \end{proof}
 \begin{lemma}\label{Lemma (Section 7): Energy Estimate from Perturbation}
    For each $\sigma\in(0,\,1/2)$, there exists a sufficiently small $\theta_{0}=\theta_{0}({\mathcal D},\,\delta,\,M,\,\sigma)\in(0,\,\gamma_{0}^{2}/256)$ such that the first result of (\ref{Eq (Section 7): Measure Result Final}) with $\theta\le \theta_{0}$ implies 
    \begin{equation}\label{Eq (Section 7): Energy Decay from Perturbation}
        \Phi(\sigma\rho)\le C_{\ast}\left[\sigma^{2}\Phi(\rho)+\frac{1+F^{2}}{\sigma^{n+2}}\rho^{2\beta} \right]
    \end{equation}
    with $C_{\ast}=C_{\ast}({\mathcal D},\,\delta,\,M)\in(1,\,\infty)$.
 \end{lemma}
 \begin{proof}
    Let $\bv_{\varepsilon}$ be the function given in Lemma \ref{Lemma (Section 7): Comparison Lemma}.
    By (\ref{Eq: (Section 7) Basic Heat Estimates})--(\ref{Eq: (Section 7) Comparison Estimate}), we compute
    \begin{align*}
        \Phi(\sigma\rho)&\le \fiint_{Q_{\sigma\rho}}\lvert \D\bu_{\varepsilon}-(\D\bv_{\varepsilon})_{Q_{\sigma\rho}}\rvert^{2}\,\d x\d t\\ 
        &\le 2\left[\frac{1}{(2\sigma)^{n+2}}\fiint_{Q_{\rho/2}}\lvert\D\bu_{\varepsilon}-\D\bv_{\varepsilon} \rvert^{2}\,\d x\d t+C\sigma^{2}\fiint_{Q_{\rho/2}}\lvert \D\bv_{\varepsilon}-(\D\bv_{\varepsilon})_{Q_{\rho/2}}\rvert^{2}\,\d x\d t\right]\\ 
        &\le C\left[ \left(\sigma^{2}+\frac{1}{\sigma^{n+2}}\right)\fiint_{Q_{\rho/2}}\lvert \D\bu_{\varepsilon}-\D\bv_{\varepsilon} \rvert^{2}\,\d x\d t+\sigma^{2}\Phi(\rho) \right]\\
        &\le \frac{C_{\ast}}{2}\left[\left(\sigma^{2}+\frac{\omega(\theta^{1/2}M)^{\frac{\vartheta}{1+\vartheta}}}{\sigma^{n+2}} \right)\Phi(\rho)+\frac{1+F^{2}}{\sigma^{n+2}}\rho^{2\beta} \right].
    \end{align*}
    Choosing $\theta_{0}\in(0,\,\gamma_{0}^{2}/256)$ such that $\omega(\theta_{0}^{1/2}M)^{\frac{\vartheta}{1+\vartheta}}\le \sigma^{n+4}$, we conclude (\ref{Eq (Section 7): Energy Decay from Perturbation}).
 \end{proof}

  \subsection{Proof of Proposition \ref{Proposition (Section 4): Non-Degenerate Case}}
  We conclude Section \ref{Section: Non-Degenerate} by giving the proof of Proposition \ref{Proposition (Section 4): Non-Degenerate Case}.
  \begin{proof}[Proof of Proposition \ref{Proposition (Section 4): Non-Degenerate Case}]
    We first note that if $\bG_{2\delta,\,\varepsilon}(x_{0},\,t_{0})$ is well-defined, then (\ref{Eq (Section 4): Growth of Gamma-2delta-epsilon}) is easily shown by (\ref{Eq (Section 1): Matrix Gamma}) and (\ref{Eq (Section 4): Bound Assumption of v-epsilon}).
    Let $c_{\dagger\dagger}$ and $c_{\dagger\dagger\dagger}$ be the positive constants found in Lemmata \ref{Lemma (Section 2): Continuity of G-2delta-epsilon} and \ref{Lemma (Section 2): Campanato Integral Growth Lemma}.
    We choose and fix $\sigma\in(0,\,1/2)$, $\theta\in(0,\,\gamma_{0}^{2}/256)$ such that 
    \[\max\left\{\,\sigma^{\beta},\,C_{\ast}\sigma^{2(1-\beta)}\, \right\}\le \frac{1}{2},\quad \theta\le \min\left\{\,\frac{\gamma_{0}^{2}}{256},\, \theta_{0},\, \frac{\sigma^{n+2+2\beta}}{c_{\dagger\dagger}^{2}c_{\dagger\dagger\dagger}}\right\},\]
    where the constants $C_{\ast}=C_{\ast}({\mathcal D},\,\delta,\,M)\in(1,\,\infty)$ and $\theta_{0}=\theta_{0}({\mathcal D},\,\delta,\,M,\,\sigma)\in(0,\,\gamma_{0}^{2}/256)$ are determined by Lemma \ref{Lemma (Section 7): Energy Estimate from Perturbation}.
    Let the ratio $\nu=\nu({\mathcal D},\,\delta,\,M,\,\theta_{0})$ and the radius $\rho_{\ast}=\rho_{\ast}({\mathcal D},\,\delta,\,M,\,\theta_{0})\in(0,\,1)$ be given by Lemma \ref{Lemma (Section 7): Energy Decay and Average Estimate}, and finally determine $\widehat\rho\in(0,\,1)$ that satisfies
    \[\widehat{\rho}\le \rho_{\ast},\quad \text{and}\quad C_{\ast}(1+F^{2})\widehat{\rho}^{2\beta}\le \frac{\theta\sigma^{n+2+2\beta}}{2}.\]
    Hereinafter we let all of the assumptions in Proposition \ref{Proposition (Section 4): Non-Degenerate Case} be satisfied with $\rho\in(0,\,\widehat\rho\rbrack$.
    By induction for $k\in{\mathbb Z}_{\ge 0}$, we would like to prove the following (\ref{Eq (Section 7): Induction Claim 1})--(\ref{Eq (Section 7): Induction Claim 2});
    \begin{equation}\label{Eq (Section 7): Induction Claim 1}
        \Phi(\rho_{k})\le \sigma^{2k\beta}\theta\mu^{2},
    \end{equation}
    \begin{equation}\label{Eq (Section 7): Induction Claim 2}
        \left\lvert (\D\bu_{\varepsilon})_{Q_{\rho_{k}}} \right\rvert_{\bg(x_{0},\,t_{0})}\ge \delta+\left(\frac{1}{2}-\frac{1}{8}\sum_{j=1}^{k-1}2^{-j}\right)\mu\ge \delta+\frac{\mu}{4},
    \end{equation}
    where $\rho_{k}\coloneqq \sigma^{k}\rho$.
    The assumption (\ref{Eq (Section 4): Measure Assumption: Non-Degenerate}) enables us to apply Lemma \ref{Lemma (Section 7): Energy Decay and Average Estimate}, from which (\ref{Eq (Section 7): Induction Claim 1})--(\ref{Eq (Section 7): Induction Claim 2}) with $k=0$ immediately follows.
    Assume that the claims (\ref{Eq (Section 7): Induction Claim 1})--(\ref{Eq (Section 7): Induction Claim 2}) are valid for an arbitrarily fixed $k\in{\mathbb Z}_{\ge 0}$.
    Then, we are allowed to apply Lemma \ref{Lemma (Section 7): Energy Estimate from Perturbation} with $\rho=\rho_{k}$. 
    Combining with the induction hypothesis (\ref{Eq (Section 7): Induction Claim 1}), we have 
    \[
        \Phi(\rho_{k+1})\le C_{\ast}\left[\sigma^{2}\Phi(\rho_{k})+\frac{1+F^{2}}{\sigma^{n+2}}\rho_{k}^{2\beta}\right]\le C_{\ast}\sigma^{2(1-\beta)}\cdot \sigma^{2\beta}\Phi(\rho_{k})+\frac{C_{\ast}(1+F^{2})\widehat{\rho}^{2\beta}}{\sigma^{n+2}}\cdot \sigma^{2\kappa\beta}\le \sigma^{2(k+1)\beta} \theta\mu^{2}.  
    \]
    This result and the Cauchy--Schwarz inequality imply
    \begin{align*}
        \left\lvert (\D\bu_{\varepsilon})_{Q_{\rho_{k+1}}} -(\D\bu_{\varepsilon})_{Q_{\rho_{k}}} \right\rvert_{\bg(x_{0},\,t_{0})}
        &\le \gamma_{0}^{-1}\fiint_{Q_{\rho_{k+1}}}\left\lvert \D\bu_{\varepsilon}-(\D\bu_{\varepsilon})_{Q_{\rho_{k}}} \right\rvert\,\d x\d t \\ 
        &\le \gamma_{0}^{-1}\sqrt{\Phi(\rho_{k+1})} \le \frac{\sigma^{k\beta}\sqrt{\theta}}{\gamma_{0}}\mu \le \frac{2^{-k}\mu}{8}.
    \end{align*}
    By the triangle inequality and the induction hypothesis (\ref{Eq (Section 7): Induction Claim 2}), we get 
    \[\left\lvert (\D\bu_{\varepsilon})_{Q_{\rho_{k+1}}}\right\rvert_{\bg(x_{0},\,t_{0})}\ge \delta+\left(\frac{1}{2}-\frac{1}{8}\sum_{j=1}^{k-1}2^{-j}\right)\mu-\frac{2^{-k}\mu}{8},\]
    which completes the induction proof of (\ref{Eq (Section 7): Induction Claim 1})--(\ref{Eq (Section 7): Induction Claim 2}).

    For every $\tau\in(0,\,1\rbrack$, there uniquely exists $k\in{\mathbb Z}_{\ge 0}$ such that $\rho_{k+1}<\tau\rho\le \rho_{k}$.
    Using Lemma \ref{Lemma (Section 2): Continuity of G-2delta-epsilon} and (\ref{Eq (Section 7): Induction Claim 1}), and noting $\sigma^{k}<\tau/\sigma$ by our choice of $k$, we have 
    \begin{align*}
        &\fiint_{Q_{\tau\rho}}\left\lvert \boldsymbol{\mathcal G}_{2\delta,\,\varepsilon}(\D\bu_{\varepsilon})-(\boldsymbol{\mathcal G}_{2\delta,\,\varepsilon}(\D\bu_{\varepsilon}))_{Q_{\rho}} \right\rvert^{2}\,\d x\d t\le \fiint_{Q_{\tau\rho}}\left\lvert \boldsymbol{\mathcal G}_{2\delta,\,\varepsilon}(\D\bu_{\varepsilon})-\boldsymbol{\mathcal G}_{2\delta,\,\varepsilon}((\D\bu_{\varepsilon})_{Q_{\rho}}) \right\rvert^{2}\,\d x\d t\\ 
        &\le c_{\dagger\dagger}^{2}\fiint_{Q_{\tau\rho}}\left\lvert\D\bu_{\varepsilon}-(\D\bu_{\varepsilon})_{Q_{\rho_{k}}} \right\rvert^{2}\,\d x\d t \le \frac{c_{\dagger\dagger}^{2}}{\sigma^{n+2}}\fiint_{Q_{\rho_{k}}}\left\lvert\D\bu_{\varepsilon}-(\D\bu_{\varepsilon})_{Q_{\rho_{k}}} \right\rvert^{2}\,\d x\d t\\ 
        &\le c_{\dagger\dagger}^{2}\sigma^{2k\beta-(n+2)}\theta\mu^{2}\le c_{\dagger\dagger}^{2}\tau^{2\beta}\cdot \theta\sigma^{-(n+2+2\beta)}\mu^{2}\le \frac{\tau^{2\beta}}{c_{\dagger\dagger\dagger}}\mu^{2}.
    \end{align*}
    The existence of $\boldsymbol{\Gamma}_{2\delta,\,\varepsilon}(x_{0},\,t_{0})\in{\mathbb R}^{Nn}$ follows from Lemma \ref{Lemma (Section 2): Campanato Integral Growth Lemma} and the last estimate. 
    Moreover, by Lemma \ref{Lemma (Section 2): Campanato Integral Growth Lemma} and our choice of $\theta$, the limit $\boldsymbol{\Gamma}_{2\delta,\,\varepsilon}(x_{0},\,t_{0})$ satisfies (\ref{Eq (Section 4): Campanato-Growth}), which completes the proof.
 \end{proof}
\section{Convergence for the parabolic Dirichlet problems}\label{Section: Convergence Result}
In Section \ref{Section: Convergence Result}, we aim to construct the weak solution of
\begin{equation}\label{Eq (Section 8): Dirichlet Parabolic}
    \left\{\begin{array}{rclcc}
        \partial_{t}u^{j}-\partial_{x_{\beta}}\left(\gamma_{\alpha\beta}a_{s}(x,\,t)g_{s}(\lvert\D\bu_{k}\rvert_{\bg}^{2})\partial_{x_{\alpha}}u^{j} \right)&=&f^{j}& \text{in}& \Omega_{T},\\
        \bu&=&\bv&\textrm{on}&\partial_{\mathrm p}\Omega_{T},
    \end{array} \right.
\end{equation}
the definition of which is given as follows.
\begin{definition}\label{Definition (Section 8): Weak sol Dirichlet boundary}
    For given $\buf\in L^{2}(\Omega_{T})^{N}\cap L^{p^{\prime}}(0,\,T;\,V_{0}^{\prime})^{N}$ and $\bu_{\star}\in X^{p}(0,\,T;\,\Omega)^{N}\cap C(\lbrack 0,\,T\rbrack;\,L^{2}(\Omega))^{N}$, a function $\bu\in \bu_{\star} +X_{0}^{p}(0,\,T;\,\Omega)^{N}$ is called the \textit{weak} solution of (\ref{Eq (Section 8): Dirichlet boundary}) when $\bu$ is a weak solution to (\ref{Eq (Section 1): General System}) in the sense of Definition \ref{Definition (Section 1): A weak solution}, and satisfies $\bu(\,\cdot\,,\,0)=\bu_{\star}(\,\cdot\,,\,0)$ in $L^{2}(\Omega)^{N}$.
\end{definition}
Section \ref{Section: Convergence Result} provides two results.
Firstly, we show a priori stability estimates for the parabolic Dirichlet problem (\ref{Eq (Section 8): Dirichlet Parabolic}), with respect to the external force term $\buf$ and the boundary datum $\bu_{\star}$.
Secondly, we aim to prove that the weak solution of (\ref{Eq (Section 8): Dirichlet Parabolic}) is constructed as a limit function of $\bu_{\varepsilon}\in \bu_{\star}+X_{0}^{p}(0,\,T;\,\Omega)^{N}$, the unique solution of (\ref{Eq (Section 2): Approximate System})--(\ref{Eq (Section 2): Parabolic Dirichlet Boundary}).

The monotonicity of the $(1,\,p)$-Laplace operator plays an important role in Section \ref{Section: Convergence Result}.
More precisely, there exists a constant $c=c({\mathcal D})\in(0,\,1)$ such that 
\begin{align}\label{Eq (Section 8): Strong Monotonicity of A-epsilon}
    \nonumber &\left\langle \A_{\varepsilon}(x,\,t,\,\bz_{1})-\A_{\varepsilon}(x,\,t,\,\bz_{2})\mid \bz_{1}-\bz_{2} \right\rangle_{\bg(x,\,t)}\\ 
    &\ge \left\{\begin{array}{cc}
        c\lvert \bz_{1}-\bz_{2}\rvert^{p} & (2\le p<\infty),\\ c\left(\varepsilon^{2}+\lvert \bz_{1}\rvert^{2}+\lvert \bz_{2}\rvert^{2}\right)^{p/2-1}\lvert \bz_{1}-\bz_{2}\rvert^{2}& (1<p<2)
       \end{array}  \right.
\end{align}
\begin{align}\label{Eq (Section 8): Strong Monotonicity of A}
    \nonumber&\left\langle \A(x,\,t,\,\bz_{1},\,{\mathbf Z}_{1})-\A(x,\,t,\,\bz_{2},\,{\mathbf Z}_{2})\mid \bz_{1}-\bz_{2} \right\rangle_{\bg(x,\,t)}\\ 
    &\ge \left\{\begin{array}{cc}
     c\lvert \bz_{1}-\bz_{2}\rvert^{p} & (2\le p<\infty),\\ c\left(\lvert \bz_{1}\rvert^{2}+\lvert \bz_{2}\rvert^{2}\right)^{p/2-1}\lvert \bz_{1}-\bz_{2}\rvert^{2}& (1<p<2)
    \end{array}  \right.
\end{align}
hold for $(x,\,t)\in\Omega_{T}$, $\bz_{1},\,\bz_{2}\in{\mathbb R}^{Nn}$, and ${\mathbf Z}_{1},\,{\mathbf Z}_{2}\in{\mathbb R}^{Nn}$ satisfying ${\mathbf Z}_{k}\in\partial_{\bg(x,\,t)}\lvert\,\cdot\,\rvert_{\bg(x,\,t)}(\bz_{k})$ ($k\in\{\,1,\,2\,\}$).
The first estimate (\ref{Eq (Section 8): Strong Monotonicity of A-epsilon}) follows from the ellipticity of the biliniear form ${\mathcal B}_{\varepsilon}(x,\,t,\,\bz)$, as described in (\ref{Eq (Section 2): Ellipticity of Bilinear Forms}).
The second estimate (\ref{Eq (Section 8): Strong Monotonicity of A}) is similarly shown by using the strong monotonicity of the $p$-Laplace-type operator and the monotonicity of the subdifferential $\partial_{\bg(x,\,t)}\lvert\,\cdot\,\rvert_{\bg(x,\,t)}$.
We additionally note that 
\begin{equation}\label{Eq (Section 8): Coercivity}
    \langle \A(x,\,t,\,\bz,\,{\mathbf Z})\mid \bz\rangle_{\bg(x,\,t)}\ge \lambda_{0}\lvert \bz \rvert^{p}+\gamma_{0}\lvert \bz\rvert
\end{equation}
holds for $(x,\,t)\in\Omega_{T}$, $\bz\in{\mathbb R}^{Nn}$, and ${\mathbf Z}\in{\mathbb R}^{Nn}$ satisfying ${\mathbf Z}\in\partial_{\bg(x,\,t)}\lvert\,\cdot\,\rvert_{\bg(x,\,t)}(\bz)$.
This is easily shown by (\ref{Eq (Section 1): Ellipticity of p-Laplace-type operator}) and Euler's identity $\langle{\mathbf Z}\mid \bz \rangle_{\bg(x,\,t)}=\lvert \bz\rvert_{\bg(x,\,t)}$ (see also \cite[Theorem 1.8]{Andreu-Vaillo et al}).
We also note that for any $\bphi_{1},\,\bphi_{2}\in C(\lbrack 0,\,T\rbrack;\,L^{2}(\Omega))^{N}\cap L^{p}(0,\,T;\,V(\Omega))^{N}$ with $\bphi_{1}-\bphi_{2}\in L^{p}(0,\,T;\,V_{0}(\Omega))^{N}$, there holds  
\begin{align}\label{Eq (Section 8): L-p V-0}
    &\lVert \bphi_{1}-\bphi_{2} \rVert_{L^{p}(0,\,T;\,V_{0}(\Omega))}^{p}\nonumber \\ 
    &\le c_{p}\left(\lVert \D\bphi_{1} \rVert_{L^{p}(\Omega_{T})}^{p}+\lVert \D\bphi_{2} \rVert_{L^{p}(\Omega_{T})}^{p} \right)+F_{p}(T)\sup_{\tau\in (0,\,T)}\lVert (\bphi_{1}-\bphi_{2})(\,\cdot\,,\,\tau) \rVert_{L^{2}(\Omega)}^{p},
\end{align}
where $F_{p}(T)\coloneqq T$ when $p\le p_{\mathrm c}$ and otherwise $F_{p}(T)=0$.
\subsection{A priori stability estimates}
We discuss the stability estimates concerning the Dirichlet boundary problem
\begin{equation}\label{Eq (Section 8): Dirichlet boundary}
    \left\{\begin{array}{rclcc}
        \partial_{t}u_{k}^{j}-\partial_{x_{\beta}}\left(\gamma_{\alpha\beta}a_{s}(x,\,t)g_{s}(\lvert\D\bu_{k}\rvert_{\bg}^{2})\partial_{x_{\alpha}}u_{k}^{j} \right)&=&f_{k}^{j}& \text{in}& \Omega_{T},\\
        \bu_{k}&=&\bv_{k}&\textrm{on}&\partial_{\mathrm p}\Omega_{T},
    \end{array} \right.
\end{equation}
for $k\in\{\,1,\,2\,\}$.
Here $\buf_{1},\,\buf_{2}\in L^{2}(\Omega_{T})^{N}\cap L^{p^{\prime}}(0,\,T;\,V_{0}^{\prime})^{N}$ and $\bv_{1},\,\bv_{2}\in X^{p}(0,\,T;\,\Omega)^{N}\cap C(\lbrack 0,\,T\rbrack;\,L^{2}(\Omega))^{N}$ are given.
Using (\ref{Eq (Section 8): Strong Monotonicity of A}), we would like to show Lemma \ref{Lemma (Section 8): Stability Estimates}.
Hereinafter, for notational simplicity, we often abbreviate $\A(\D\bu_{k})$ as
$\A(x,\,t,\,\D\bu_{k},\,{\mathbf Z}_{k})$, where the given mapping ${\mathbf Z}_{k}\in L^{\infty}(\Omega_{T})^{Nn}$ is assumed to satisfy ${\mathbf Z}_{k}\in\partial_{\bg(x,\,t)}\lvert\,\cdot\,\rvert_{\bg(x,\,t)}(\D\bu_{k}(x,\,t))$ for a.e.~$(x,\,t)\in\Omega_{T}$.
\begin{lemma}\label{Lemma (Section 8): Stability Estimates}
    Let $\buf_{1},\,\buf_{2}\in L^{2,\,1}(\Omega_{T})^{N}\cap L^{p^{\prime}}(0,\,T;\,V_{0}^{\prime})^{N}$ and $\bv_{1},\,\bv_{2}\in X^{p}(0,\,T;\,\Omega)^{N}\cap C(\lbrack 0,\,T\rbrack;\,L^{2}(\Omega))^{N}$.
    For each $k\in\{\,1,\,2\,\}$, we consider $\bu_{k}\in \bv_{k}+X_{0}^{p}(0,\,T;\,\Omega)^{N}$ the weak solution of (\ref{Eq (Section 8): Dirichlet boundary}).
    Then, we have 
    \begin{align}\label{Eq (Section 8): Boundedness of Du when perturbing f}
        &\sup_{0<\tau<T}\int_{\Omega\times \{\tau\}}\lvert\bu_{k}-\bv_{k}\rvert^{2} \,\d x+ \iint_{\Omega_{T}}\lvert \D\bu_{k}\rvert^{p}\,\d x\d t\nonumber \\ 
        & \le C\left[ \lVert\partial_{t}\bv_{k} \rVert_{L^{p^{\prime}}(0,\,T;\,V_{0}^{\prime})}^{p^{\prime}}+\lVert\D\bv_{k} \rVert_{L^{p}(\Omega_{T})}^{p}+\lVert \buf_{k}\rVert_{L^{2,\,1}(\Omega_{T})}^{2}+1 \right]
    \end{align}
    for each $k\in\{\,1,\,2\,\}$, and
    \begin{align}\label{Eq (Section 8): Stability estimates when perturbing f}
        &\sup_{0<\tau<T}\int_{\Omega\times\{\tau\} }\lvert \bu_{1}-\bu_{2}\rvert^{2}\,\d x+\iint_{\Omega_{T}}\langle \A(\D\bu_{1})-\A(\D\bu_{2}),\,\D\bu_{1}-\D\bu_{2}\rangle_{\bg}\,\d x\d t\nonumber\\ 
        &\le  C\lVert \buf_{1}-\buf_{2}\rVert_{L^{2,\,1}(\Omega_{T})}^{2}+ C\left(\lVert\A(\D\bu_{1}) \rVert_{L^{p^{\prime}}(\Omega_{T})}+\lVert\A(\D\bu_{2}) \rVert_{L^{p^{\prime}}(\Omega_{T})} \right)\lVert \D\bv_{1}-\D\bv_{2} \rVert_{L^{p}(\Omega_{T})}\nonumber \\ 
        &\quad +C\lVert \partial_{t}\bv_{1}-\partial_{t}\bv_{2} \rVert_{L^{p^{\prime}}(0,\,T;\,V_{0}^{\prime}(\Omega))}\left(\lVert \bu_{1}-\bu_{2} \rVert_{L^{p}(0,\,T;\,V_{0}(\Omega))}+\lVert \bv_{1}-\bv_{2} \rVert_{L^{p}(0,\,T;\,V_{0}(\Omega))}\right)\nonumber\\ &\quad \quad +C\sup_{0<\tau<T}\int_{\Omega\times\{\tau\}}\lvert\bv_{1}-\bv_{2} \rvert^{2}\,\d x
    \end{align}
    for some $C=C({\mathcal D},\,\Omega,\,T)\in(0,\,\infty)$.
    In particular, for given $\buf\in L^{2,\,1}(\Omega_{T})^{N}\cap L^{p^{\prime}}(0,\,T;\,V_{0}^{\prime})^{N}$ and $\bu_{\star}\in X^{p}(0,\,T;\,\Omega)^{N}\cap C(\lbrack 0,\,T\rbrack;\,L^{2}(\Omega))^{N}$, the weak solution of (\ref{Eq (Section 8): Dirichlet Parabolic}) is unique.
\end{lemma}
\begin{proof}
    We introduce $\bw_{k}\coloneqq \bu_{k}-\bv_{k}\in X_{0}^{p}(0,\,T;\,\Omega)^{N}$ for each $k\in\{\,1,\,2\,\}$.
    Let $\phi\colon \lbrack 0,\,T\rbrack\to\lbrack 0,\,1\rbrack$ be a non-increasing function that satisfies $\phi(T)=0$.
    We first prove (\ref{Eq (Section 8): Boundedness of Du when perturbing f}) by testing $\bphi\coloneqq \phi \bw_{k}$ into the weak formulation
    \[\int_{0}^{T}\langle\partial_{t}\bw_{k},\,\bphi\rangle\,\d t+\iint_{\Omega_{T}}\left\langle{\mathbf A}(x,\,t,\,\D\bu_{k},\,{\mathbf Z}_{k})\mathrel{}\middle|\mathrel{}\D\bphi  \right\rangle_{\bg}\,\d x\d t=\iint_{\Omega_{T}}\langle \buf_{k}\mid \bphi \rangle\,\d x\d t-\int_{0}^{T}\langle\partial_{t}\bv_{k},\,\bphi \rangle\,\d t.\]
    Integrating by parts, applying (\ref{Eq (Section 8): L-p V-0}) with $(\bphi_{1},\,\bphi_{2})=(\bw_{k},\,0)$, and using Young's inequality, we have 
    \begin{align*}
        &-\frac{1}{2}\iint_{\Omega_{T}}\lvert \bw_{k} \rvert^{2}\partial_{t}\phi\,\d x\d t+\iint_{\Omega_{T}}\left\langle {\mathbf A}(\D\bu_{k})\mathrel{}\middle|\mathrel{}\D\bu_{k} \right\rangle_{\bg}\phi\,\d x\d t\\ 
        &\le \iint_{\Omega_{T}}\lvert {\mathbf A}(\D\bu_{k}) \rvert_{\bg}\lvert \D\bv_{k} \rvert_{\bg}\,\d x\d t+\iint_{\Omega_{T}}\lvert \buf_{k}\rvert\lvert\bw_{k} \rvert\,\d x\d t+\int_{0}^{T}\lVert \partial_{t}\bv_{k}(t) \rVert_{V_{0}^{\prime}}\lVert \bw_{k}(t) \rVert_{V_{0}}\,\d t  \\
        &\le \sigma\left(\sup_{0<\tau<T}\int_{\Omega\times \{\tau\}}\lvert \bw_{k} \rvert^{2}\,\d x+\iint_{\Omega_{T}}\left(1+\lvert \D\bu_{k} \rvert^{p}\right)\,\d x\d t \right) \\ 
        &\quad +C({\mathcal D},\,\sigma)\left[\lVert \partial_{t}\bv_{k}\rVert_{L^{p^{\prime}}(0,\,T;\,V_{0}^{\prime})}^{p^{\prime}}+\lVert \D\bv_{k}\rVert_{L^{p}(\Omega_{T})}^{p}+ \lVert\buf_{k} \rVert_{L^{2,\,1}(\Omega_{T})}^{2}+L_{p}(T) \right]
    \end{align*}
    for any $\sigma\in(0,\,\infty)$, where $L_{p}(T)\coloneqq T^{2/(2-p)}\,(p\le p_{\mathrm c})$ and otherwise $L_{p}(T)\equiv 0$.
    Recalling (\ref{Eq (Section 8): Coercivity}), and suitably choosing $\phi$ and $\sigma>0$, we conclude (\ref{Eq (Section 8): Boundedness of Du when perturbing f}).

    Testing $\bphi\coloneqq \phi(\bw_{1}-\bw_{2})$ into the above weak formulation for each $k\in\{\,1,\,2\,\}$, we have
    \begin{align*}
        &-\frac{1}{2}\iint_{\Omega_{T}}\lvert \bw_{1}-\bw_{2} \rvert^{2}\partial_{t}\phi\,\d x\d t+\iint_{\Omega_{T}}\left\langle \A(\D\bu_{1})-\A(\D\bu_{2})\mathrel{}\middle|\mathrel{}\D\bu_{1}-\D\bu_{2} \right\rangle_{\bg}\phi\,\d x\d t\\ 
        &\le \iint_{\Omega_{T}}\left(\lvert \A(\D\bv_{1}) \rvert_{\bg}+\lvert \A(\D\bv_{2})\rvert_{\bg}\right)\lvert \D\bv_{1}- \D\bv_{2} \rvert_{\bg}\,\d x\d t+\iint_{\Omega_{T}}\lvert\buf_{\varepsilon}\rvert\lvert\bw_{1}-\bw_{2} \rvert\,\d x\d t\\ 
        &\quad +\int_{0}^{T}\lVert \partial_{t}\bv_{1}-\partial_{t}\bv_{2} \rVert_{V_{0}^{\prime}(\Omega)}\lVert \bw_{1}-\bw_{2} \rVert_{V_{0}(\Omega)} \,\d t \\
        &\le \sigma\left(\sup_{0<\tau<T}\int_{\Omega\times \{\tau\}}\lvert \bw_{1}-\bw_{2}\rvert^{2}\,\d x \right)+C(\sigma)T\lVert \buf_{1}-\buf_{2}\rVert_{L^{2,\,1}(\Omega_{T})}^{2}\\
        &\quad + C({\mathcal D})\left(\lVert\A(\D\bu_{1}) \rVert_{L^{p^{\prime}}(\Omega_{T})}+\lVert\A(\D\bu_{2}) \rVert_{L^{p^{\prime}}(\Omega_{T})} \right)\lVert \D\bv_{1}-\D\bv_{2} \rVert_{L^{p}(\Omega_{T})}\nonumber \\ 
        &\quad +C({\mathcal D})\lVert \partial_{t}\bv_{1}-\partial_{t}\bv_{2} \rVert_{L^{p^{\prime}}(0,\,T;\,V_{0}^{\prime}(\Omega))}\lVert \bw_{1}-\bw_{2} \rVert_{L^{p}(0,\,T;\,V_{0}(\Omega))}
    \end{align*}
    for any $\sigma\in(0,\,\infty)$.
    Choosing suitably $\phi$ and sufficiently small $\sigma>0$, we have
    \begin{align*}
    &    \sup_{0<\tau<T}\int_{\Omega\times \{\tau\}}\lvert \bw_{1}-\bw_{2}\rvert^{2}\,\d x+\iint_{\Omega_{T}}\left\langle \A(\D\bu_{1})-\A(\D\bu_{2})\mathrel{}\middle|\mathrel{}\D\bu_{1}-\D\bu_{2}\right\rangle_{\bg}\,\d x\d t\\ 
    &\le C(T)\lVert \buf_{1}-\buf_{2}\rVert_{L^{2,\,1}(\Omega)}^{2}+C({\mathcal D})\left(\lVert\A(\D\bu_{1}) \rVert_{L^{p^{\prime}}(\Omega_{T})}+\lVert\A(\D\bu_{2}) \rVert_{L^{p^{\prime}}(\Omega_{T})} \right)\lVert \D\bv_{1}-\D\bv_{2} \rVert_{L^{p}(\Omega_{T})}\nonumber \\ 
    &\quad +C\lVert \partial_{t}\bv_{1}-\partial_{t}\bv_{2} \rVert_{L^{p^{\prime}}(0,\,T;\,V_{0}^{\prime}(\Omega))}\lVert \bw_{1}-\bw_{2} \rVert_{L^{p}(0,\,T;\,V_{0}(\Omega))}.
    \end{align*}
    The desired estimate (\ref{Eq (Section 8): Stability estimates when perturbing f}) is easily deduced by using the triangle inequality for $L^{p}(0,\,T;\,V_{0}(\Omega))$ and the parallelogram law for $L^{2}(\Omega)$.
    The uniqueness of the weak solution is clear by (\ref{Eq (Section 8): Stability estimates when perturbing f}).
\end{proof}
As a corollary from the stability estimate, Lemma \ref{Lemma (Section 8): Convergence from stability} follows.
\begin{lemma}\label{Lemma (Section 8): Convergence from stability}
    Let the sequences $\{\buf_{k}\}_{k=1}^{\infty}\subset L^{2,\,1}(\Omega_{T})^{N}\cap L^{p^{\prime}}(0,\,T;\,V_{0}^{\prime})^{N}$ and $\{\bv_{k}\}_{k=1}^{\infty}\subset X^{p}(0,\,T;\,\Omega)^{N}\cap C(\lbrack 0,\,T\rbrack;\,L^{2}(\Omega))^{N}$ admit the limit functions $\buf\in L^{2}(\Omega_{T})^{N}\cap L^{p^{\prime}}(0,\,T;\,V_{0}^{\prime})^{N}$ and $\bv_{\star}\in X^{p}(0,\,T;\,\Omega)^{N}\cap C(\lbrack 0,\,T\rbrack;\,L^{2}(\Omega))^{N}$ such that
    \[\begin{array}{rclcc}
        \buf_{k}&\to&\buf &\text{in}&L^{2,\,1}(\Omega_{T})^{N},\\ 
        \buf_{k}&\rightharpoonup & \buf&\text{in}&L^{p^{\prime}}(0,\,T;\,V_{0}^{\prime}(\Omega))^{N},\\ 
        \D\bv_{k}&\to&\D\bu_{\star}& \textrm{in}& L^{p}(\Omega_{T})^{Nn},\\ 
        \partial_{t}\bv_{k}&\to&\partial_{t}\bu_{\star}& \textrm{in}& L^{p^{\prime}}(0,\,T;\,V_{0}^{\prime}(\Omega))^{N},\\ 
        \bv_{k}&\to&\bu_{\star}&\text{in}&C(\lbrack 0,\,T\rbrack;\,L^{2}(\Omega))^{N}.
    \end{array}\]
    For each $k\in{\mathbb N}$, let $\bu_{k}\in \bv_{k}+X_{0}^{p}(0,\,T;\,\Omega)^{N}$ satisfy (\ref{Eq (Section 8): Dirichlet boundary}). 
    Then, there uniquely exists a function $\bu\in \bu_{\star}+X_{0}^{p}(0,\,T;\,\Omega)^{N}$ such that
    \[\begin{array}{rclcc}
        \D\bu_{k} &\to& \D\bu& \text{in}& L^{p}(\Omega_{T})^{Nn},\\ 
        \partial_{t}\bu_{k}&\rightharpoonup& \partial_{t}\bu  &\text{in}&L^{p^{\prime}}(0,\,T;\,V_{0}^{\prime}(\Omega))^{N}, \\ 
        \bu_{k}&\to &\bu & \text{in}&C(\lbrack 0,\,T\rbrack;\,L^{2}(\Omega))^{N},
    \end{array}\]
    by relabelling a sequence if necessary.
    Moreover, $\bu$ is the weak solution of (\ref{Eq (Section 8): Dirichlet Parabolic}).
\end{lemma}
\begin{proof}
    Using (\ref{Eq (Section 8): Boundedness of Du when perturbing f}) and the assumptions of Lemma \ref{Lemma (Section 8): Convergence from stability}, we can check that $\D\bu_{k}\in L^{p}(\Omega_{T})^{Nn}$, $\A(\D\bu_{k})\in L^{p^{\prime}}(\Omega_{T})^{Nn}$, $\partial_{t}\bu_{k}\in L^{p^{\prime}}(0,\,T;\,V_{0}^{\prime}(\Omega))^{N}$, and $\bu_{k}-\bv_{k},\,\bu_{k}-\bu_{l}\in L^{p}(0,\,T;\,V_{0}(\Omega))^{N}$ are uniformly bounded for $k,\,l\in{\mathbb N}$.
    In particular, by (\ref{Eq (Section 8): Strong Monotonicity of A}) and (\ref{Eq (Section 8): Stability estimates when perturbing f}), where we also use H\"{o}lder's inequality when $p\in(1,\,2)$, we conclude that $\{\bu_{k}\}_{k=1}^{\infty}\subset C(\lbrack 0,\,T\rbrack;\,L^{2}(\Omega))^{N}$ and $\{\D\bu_{k}\}_{k=1}^{\infty}\subset L^{p}(\Omega_{T})^{Nn}$ are Cauchy sequences.
    Hence, by taking a subsequence if necessary, we conclude all of the convergence results in Lemma \ref{Lemma (Section 8): Convergence from stability}.
    The identity $\bu(\,\cdot\,,\,0)=\bu_{\star}(\,\cdot\,,\,0)$ in $L^{2}(\Omega)^{N}$ is clear by the third convergence result.
    Using Lemma \ref{Lemma (Section 2): Fundamental Lemma for convergence of solutions} \ref{Item 3/3}, we conclude that $\bu$ is the weak solution of (\ref{Eq (Section 8): Dirichlet Parabolic}).
\end{proof}

\subsection{Convergence of the approximate solutions}
\begin{proposition}\label{Lemma (Section 8): Convergence of Approximate solutions}
    Fix $\bu_{\star}\in X^{p}(0,\,T;\,\Omega)^{N}\cap C(\lbrack 0,\,T\rbrack;\,L^{2}(\Omega))^{N}$. 
    Assume that $\buf_{\varepsilon}\in L^{p^{\prime}}(0,\,T;\,V_{0}^{\prime})^{N}\cap L^{\infty}(\Omega_{T})^{N}$ satisfies
    \begin{equation}\label{Eq (Section 8): Assumption in Convergence of Approximate Solutions}
        \buf_{\varepsilon} \stackrel{\star}{\rightharpoonup} \buf\quad \textrm{in}\quad L^{q,\,r}(\Omega_{T}),\quad \text{and}\quad \buf_{\varepsilon}\rightharpoonup \buf \quad \text{in}\quad L^{p^{\prime}}(0,\,T;\,V_{0}^{\prime})^{N}
    \end{equation}
    as $\varepsilon\to 0$, where $(q,\,r)$ satisfies (\ref{Eq (Section 1): Condition for q,r}).
    Let $\bu_{\varepsilon}$ be the weak solution of (\ref{Eq (Section 2): Approximate System})--(\ref{Eq (Section 2): Parabolic Dirichlet Boundary}).  
    When $n\ge 3$ and $p\in(1,\,p_{\mathrm c}\rbrack$, let $\bu_{\star}\in L^{\infty}(\Omega_{T})^{N}$ be also in force.
    Then, there exists a sequence $\varepsilon_{k}$ such that $\varepsilon_{k}\to 0$,
    \[\partial_{t}\bu_{\varepsilon_{k}}\rightharpoonup \partial_{t}\bu_{0}\quad \text{in}\quad L^{p^{\prime}}(0,\,T;\,V_{0}^{\prime}),\quad \text{and}\quad \D\bu_{\varepsilon}\to \D\bu_{0}\quad \text{in}\quad L^{p}(\Omega_{T})^{Nn}\]
    hold for some unique function $\bu_{0}\in \bu_{\star}+X_{0}^{p}(0,\,T;\,\Omega)$.
    Moreover, $\bu_{0}$ is the weak solution of (\ref{Eq (Section 8): Dirichlet Parabolic}).
\end{proposition}
The proof of Proposition \ref{Lemma (Section 8): Convergence of Approximate solutions} mainly consists of two parts. The first is to show $\bu_{\varepsilon}-\bu_{\star}$ is bounded in $X_{0}^{p}(0,\,T;\,\Omega)^{N}$, so that the limit function $\bu_{0}\in \bu_{\star}+X_{0}^{p}(0,\,T;\,\Omega)^{N}$ is constructed by a standard weak compactness argument.
The second is to show the strong $L^{p}$-convergence of the spatial gradient.
As mentioned in Section \ref{Section: Introduction}, we treat the external force term in two different ways, depending on the value of $p$.
For $p\in(p_{\mathrm c},\,\infty)$, we appeal to the Aubin--Lions lemma to use the compact embedding $X_{0}^{p}(0,\,T;\,\Omega)\subset L^{p}(0,\,T;\,L^{2}(\Omega))$, which is guaranteed by the compact embedding $W_{0}^{1,\,p}(\Omega)\subset L^{2}(\Omega)$.
The computations in the case $p>p_{\mathrm c}$ are based on the slight modification of \cite[Lemma 3.1]{Kuusi-Mingione MR2916967}, which provides strong convergence results for approximate $p$-Laplace flows. 
There, some absorption method is carefully used, although no parabolic compact embedding is used at all.
For $p\in(1,\,p_{\mathrm c}\rbrack$, we never rely on any parabolic compact embedding, since $V_{0}(\Omega)$ is no longer compactly embedded into $L^{2}(\Omega)$.
Instead, we recall some uniform a priori estimates such as Proposition \ref{Proposition (Section 3): A Weak Maximum Principle} and Corollary \ref{Lemma (Section 5): Local Holder continuity of u-epsilon} to deduce the strong convergence of $\bu_{\varepsilon}$.
Since we appeal to the weak maximum principle (Proposition \ref{Proposition (Section 3): A Weak Maximum Principle}), we have to require $\bu_{\star}$ to be in $L^{\infty}$.
This assumption is not restrictive in the proof of Theorem \ref{Theorem (Section 1): Gradient continuity}, since it suffices to use Proposition \ref{Proposition (Section 3): Local Boundedness for subcritical cases} and to consider local approximate problems in a smaller domain.

\begin{proof}
    We first show the following uniform bound estimate;
    \begin{align}\label{Eq (Section 8): Bound for Du-epsilon}
        \nonumber&\sup_{\tau\in(0,\,T)}\lVert (\bu_{\varepsilon}-\bu_{\star})(\,\cdot\,,\,\tau) \rVert_{L^{2}(\Omega)}^{2}+\lVert \D\bu_{\varepsilon} \rVert_{L^{p}(\Omega_{T})}^{p}\\
        &\le C({\mathcal D},\,\Omega,\,T)\left(\lVert\partial_{t}\bu_{\star} \rVert_{L^{p^{\prime}}(0,\,T;\,V_{0}^{\prime})}^{p^{\prime}}+\lVert \D\bu_{\star}\rVert_{L^{p}(\Omega_{T})}^{p}+\lVert \buf_{\varepsilon}\rVert_{L^{2,\,1}(\Omega_{T})}^{2}+1 \right).
    \end{align}
    Let $\phi\colon \lbrack 0,\,T\rbrack\to\lbrack 0,\,1\rbrack$ be a non-increasing function satisfying $\phi(T)=0$.
    We test $\bphi\coloneqq \phi(\bu_{\varepsilon}-\bu_{\star})$ into (\ref{Eq (Section 2): Approximate System}), and integrate by parts.
    By Young's inequality, we have 
    \begin{align*}
        &-\frac{1}{2}\iint_{\Omega_{T}}\lvert\bu_{\varepsilon}-\bu_{\star}\rvert^{2}\partial_{t}\phi\,\d x\d t+\iint_{\Omega_{T}}\langle \A_{\varepsilon}(\D\bu_{\varepsilon})\mid \D\bu_{\varepsilon} \rangle_{\bg}\phi\,\d x\d t  \\ 
        &=\iint_{\Omega_{T}}\langle\A_{\varepsilon}(\D\bu_{\varepsilon})\mid \D\bu_{\star} \rangle_{\bg}\phi\,\d x\d t+\iint_{\Omega_{T}}\langle\buf_{\varepsilon}\mid \bu_{\varepsilon}-\bu_{\star} \rangle\phi\,\d x\d t -\int_{0}^{T}\langle\partial_{t}\bu_{\star},\,\bu_{\varepsilon}-\bu_{\star} \rangle\phi\,\d t\\ 
        &\le \sigma\left(\sup_{\tau\in(0,\,T)}\lVert (\bu_{\varepsilon}-\bu_{\star})(\,\cdot\,,\,\tau) \rVert_{L^{2}(\Omega)}^{2}+\lVert \D\bu_{\varepsilon} \rVert_{L^{p}(\Omega_{T})}^{p}+\lVert \bu_{\varepsilon}-\bu_{\star} \rVert_{L^{p}(0,\,T;\,V_{0})}^{p} \right)+C({\mathcal D})\lVert \D\bu_{\star} \rVert_{L^{1}(\Omega_{T})} \\ 
        &\quad +C({\mathcal D},\,\sigma)\left(\lVert \D\bu_{\star} \rVert_{L^{p}(\Omega_{T})}^{p}+T\lVert \buf_{\varepsilon}\rVert_{L^{2,\,1}(\Omega_{T})}^{2}+\lVert \partial_{t}\bu_{\star}\rVert_{L^{p^{\prime}}(0,\,T;\,V_{0}^{\prime})}^{p^{\prime}} \right)
    \end{align*}
    for any $\sigma\in(0,\,\infty)$.
    Recalling (\ref{Eq (Section 2): Coercivity of A-epsilon}), and suitably choosing $\phi=\phi(t)$ and a sufficiently small number $\sigma>0$, we conclude (\ref{Eq (Section 8): Bound for Du-epsilon}). 
    We note that (\ref{Eq (Section 8): L-p V-0}) with $(\bphi_{1},\,\bphi_{2})=(\bu_{\varepsilon},\,\bu_{\star})$ and (\ref{Eq (Section 8): Assumption in Convergence of Approximate Solutions})--(\ref{Eq (Section 8): Bound for Du-epsilon}) imply that $\bu_{\varepsilon}-\bu_{\star}\in L^{p}(0,\,T;\,V_{0})$ is uniformly bounded. We also have 
    \[\lVert\partial_{t}\bu_{\varepsilon} \rVert_{L^{p^{\prime}}(0,\,T;\,V_{0}^{\prime})}\le C(p,\,\gamma_{0})\left(\lVert \A_{\varepsilon}(\D\bu_{\varepsilon})\rVert_{L^{p^{\prime}}(\Omega_{T})}+\lVert \buf_{\varepsilon} \rVert_{L^{p^{\prime}}(0,\,T;\,V_{0}^{\prime})}\right)\le C({\mathcal D},\,\bu_{\star},\,F).\]
    Thanks to these uniform bound estimates, we construct a function $\bu_{0}\in \bu_{\star}+X_{0}^{p}(0,\,T;\,\Omega)$ satisfying  
    \begin{equation}\label{Eq (Section 8): 3 Weak convergence results}
        \left\{ 
        \begin{array}{ccccc}
            \D\bu_{\varepsilon_{k}}&\rightharpoonup& \D\bu_{0}& \text{in}& L^{p}(\Omega_{T};\,{\mathbb R}^{Nn}),\\ 
            \bu_{\varepsilon_{k}}-\bu_{\star}&\rightharpoonup& \bu_{0}-\bu_{\star} & \text{in}&  L^{p}(0,\,T;\,V_{0})^{N},\\
            \partial_{t}\bu_{\varepsilon_{k}}& \rightharpoonup& \partial_{t}\bu_{0}&  \text{in}&  L^{p^{\prime}}(0,\,T;\,V_{0}^{\prime})^{N},    
        \end{array}\right.
    \end{equation}
    for some decreasing sequence $\{\varepsilon_{k}\}_{k=0}^{\infty}\subset (0,\,1)$ such that $\varepsilon_{k}\to 0$ as $k\to\infty$. 
    We note that the identity $(\bu_{0}-\bu_{\star})(\,\cdot\,,\,0)=0$ in $L^{2}(\Omega)^{N}$ is easy to prove by the second and the third weak convergence results of (\ref{Eq (Section 8): 3 Weak convergence results}). 
    We would like to show 
    \begin{equation}\label{Eq (Section 8): Claim for Strong Convergence}
        \iint_{\Omega_{T}}\langle \A_{\varepsilon_{k}} (x,\,t,\,\D\bu_{\varepsilon_{k}})- \A_{\varepsilon_{k}}(x,\,t,\,\D\bu_{0})\mid \D\bu_{\varepsilon_{k}}-\D\bu_{0} \rangle_{\bg}\to 0
    \end{equation}
    as $k\to\infty$, relabelling a sequence if necessary.
    Then, recalling (\ref{Eq (Section 8): Strong Monotonicity of A-epsilon}), and using H\"{o}lder's inequality and (\ref{Eq (Section 8): Bound for Du-epsilon}) when $p\in(1,\,2)$, we conclude
    \begin{equation}\label{Eq (Section 8): Strong Convergence of Approx sol}
        \D\bu_{\varepsilon_{k}}\to \D\bu_{0}\quad \text{in}\quad L^{p}(\Omega_{T};\,{\mathbb R}^{Nn})
    \end{equation}
    from (\ref{Eq (Section 8): Claim for Strong Convergence}).
    To prove (\ref{Eq (Section 8): Claim for Strong Convergence}), we choose $\overline{\phi}(t)\coloneqq (1-t/T)_{+}\in\lbrack 0,\,1\rbrack$ or $\phi_{\widetilde{\varepsilon}}(t)\coloneqq \min\{\,1,\,-(t-T)/\widetilde{\varepsilon} \,\}$ for $t\in\lbrack 0,\,T\rbrack$, where we will let $\widetilde{\varepsilon}\to 0$.
    We test $\bphi\coloneqq \overline{\phi}(\bu_{\varepsilon}-\bu_{0})$ or $\bphi\coloneqq \phi_{\widetilde{\varepsilon}}(\bu_{\varepsilon}-\bu_{0})$ into (\ref{Eq (Section 2): Approximate Weak Form}).
    Integrating by parts, deleting some non-negative integral if necessary, and finally letting $\widetilde{\varepsilon}\to 0$, we have
    \begin{align*}
       &{\mathbf L}_{1}(\varepsilon)+{\mathbf L}_{2}(\varepsilon)\\ 
       &\coloneqq -\frac{1}{2}\iint_{\Omega_{T}}\lvert \bu_{\varepsilon}-\bu_{0}\rvert^{2}\partial_{t}\overline{\phi}\,\d x\d t+\iint_{\Omega_{T}} \langle \A_{\varepsilon}(x,\,t,\,\D\bu_{\varepsilon})-\A_{\varepsilon}(x,\,t,\,\D\bu_{0}) \mid \D\bu_{\varepsilon}-\D\bu_{0} \rangle_{\bg}\,\d x\d t\\ 
       &\le \iint_{\Omega_{T}}\langle\buf_{\varepsilon}\mid \bu_{\varepsilon}-\bu_{0} \rangle\phi\,\d x\d t-\int_{0}^{T}\langle \partial_{t}\bu_{0},\,\bu_{\varepsilon}-\bu_{0}\rangle\phi\,\d t+\iint_{\Omega_{T}}\langle \A_{\varepsilon}(x,\,t,\,\D\bu_{0})\mid \D\bu_{\varepsilon}-\D\bu_{0} \rangle_{\bg}\phi\,\d x\d t\\ 
       &\eqqcolon {\mathbf R}_{1}(\varepsilon)-{\mathbf R}_{2}(\varepsilon)+{\mathbf R}_{3}(\varepsilon),
    \end{align*}
    where $\phi\coloneqq 1+\overline{\phi}$.
    It suffices to show $\limsup\limits_{k\to\infty}{\mathbf L}_{2}(\varepsilon_{k})\le 0$ to complete the proof of (\ref{Eq (Section 8): Claim for Strong Convergence}), since ${\mathbf L}_{2}(\varepsilon)$ is non-negative.
    We note that ${\mathbf L}_{1}(\varepsilon)=(2T)^{-1}\lVert\bu_{\varepsilon}-\bu_{0} \rVert_{L^{2}(\Omega_{T})}^{2}\ge 0$ by the definition of $\overline{\phi}$.
    Also, $\lim\limits_{k\to \infty}{\mathbf R}_{2}(\varepsilon_{k})=\lim\limits_{k\to \infty}{\mathbf R}_{3}(\varepsilon_{k})=0$ is clear by (\ref{Eq (Section 8): 3 Weak convergence results}) and Lemma \ref{Lemma (Section 2): Fundamental Lemma for convergence of solutions} \ref{Item 1/3}.
    To complete the proof of (\ref{Eq (Section 8): Claim for Strong Convergence}), we deal with ${\mathbf R}_{1}(\varepsilon)$ in two different approaches, depending on whether $p>p_{\mathrm c}$ or not.
    
    For $p>p_{\mathrm c}$, the Aubin--Lions lemma allows us to use the compact embedding $X_{0}^{p}(0,\,T;\,\Omega)\subset L^{p}(0,\,T;\,L^{2}(\Omega))$.
    Hence, by taking a subsequence if necessary, we may let 
    \begin{equation}\label{Eq (Section 8): Strong conv from compact embedding}
        \bu_{\varepsilon_{k}}\to \bu_{0}\quad \text{in}\quad L^{2,\,p}(\Omega_{T})^{N}=L^{p}(0,\,T;\,L^{2}(\Omega))^{N}.
    \end{equation}
    We choose a sufficiently small number $\pi\in(0,\,1)$ such that the exponents
    \[\widetilde{q}\coloneqq \left\{\begin{array}{cc} \displaystyle\frac{2q\pi}{(1+\pi)q-2} & (q<\infty),\\ \displaystyle\frac{2\pi}{1+\pi} & (q=\infty), \end{array} \right. \quad\text{and}\quad {\widetilde r}\coloneqq \left\{\begin{array}{cc} \displaystyle\frac{2r\pi}{(1+\pi)r-2} & (r<\infty),\\ \displaystyle\frac{2\pi}{1+\pi} & (r=\infty), \end{array} \right.\]
    satisfy ${\widetilde q}\le 2$ and ${\widetilde r}\le p$ respectively, where $q>2$ and $r>2$ are used to diminish $\widetilde{q}$ and $\widetilde{r}$ respectively.
    We use H\"{o}lder's inequality and Young's inequality to compute
    \begin{align*}
        \lvert{\mathbf R}_{1}(\varepsilon)\rvert&\le \iint_{\Omega_{T}}\lvert \buf_{\varepsilon}\rvert\lvert \bu_{\varepsilon}-\bu_{0}\rvert^{\pi}\cdot\lvert \bu_{\varepsilon}-\bu_{0}\rvert^{1-\pi}\,{\mathrm d}x{\mathrm d}t \\
        &\le\left(\iint_{\Omega_{T}}\lvert \buf_{\varepsilon}\rvert^{\frac{2}{1+\pi}}\lvert \bu_{\varepsilon}-\bu_{0}\rvert^{\frac{2\pi}{1+\pi}} \,{\mathrm d}x{\mathrm d}t \right)^{\frac{1+\pi}{2}}\left(\iint_{\Omega_{T}}\lvert \bu_{\varepsilon}-\bu_{0} \rvert^{2}\,\d x\d t \right)^{\frac{1-\pi}{2}}  \\ 
        &\le \frac{1}{2T}\iint_{\Omega_{T}}\lvert \bu_{\varepsilon}-\bu_{0}\rvert^{2}\,{\mathrm d}x{\mathrm d}t+C(\pi)T^{\frac{1-\pi}{1+\pi}}\iint_{\Omega_{T}}\lvert \buf_{\varepsilon}\rvert^{\frac{2}{1+\pi}}\lvert \bu_{\varepsilon}-\bu_{0}\rvert^{\frac{2\pi}{1+\pi}}\,{\mathrm d}x{\mathrm d}t\\ 
        &= {\mathbf L}_{1}(\varepsilon)+C(\pi)T^{\frac{1+\pi}{1-\pi}}{\mathbf R}_{4}(\varepsilon), \quad \text{where}\quad {\mathbf R}_{4}(\varepsilon)\coloneqq \iint_{\Omega_{T}}\lvert \buf_{\varepsilon}\rvert^{\frac{2}{1+\pi}}\lvert \bu_{\varepsilon}-\bu_{0}\rvert^{\frac{2\pi}{1+\pi}}\,\d x\d t.
    \end{align*}
    By H\"{o}lder's inquality and our choice of the exponents $\widetilde{q}$ and $\widetilde{r}$, we get
    \[
       0\le  {\mathbf R}_{4}(\varepsilon) \le \int_{0}^{T}\lVert \buf_{\varepsilon}(\,\cdot\,,\,t)\rVert_{L^{q}(\Omega)}^{\frac{2}{1+\pi}}\lVert (\bu_{\varepsilon}-\bu_{0})(\,\cdot\,,\,t) \rVert_{L^{\widetilde{q}}(\Omega)}^{\frac{2\pi}{1+\pi}}\,{\mathrm d}t \le \lVert \buf_{\varepsilon}\rVert_{L^{q,\,r}(\Omega_{T})}^{\frac{2}{1+\pi}}\lVert \bu_{\varepsilon}-\bu_{0}\rVert_{L^{\widetilde{q},\, \widetilde{r}}(\Omega_{T})}^{\frac{2\pi}{1+\pi}}.
    \]
    Combining (\ref{Eq (Section 8): Strong conv from compact embedding}) and the continuous embedding $L^{2,\,p}(\Omega_{T})\subset L^{\widetilde{q},\,\widetilde{r}}(\Omega_{T})$ with the above inequality implies ${\mathbf R}_{4}(\varepsilon_{k})\to 0$.
    Hence, we have 
    \[\limsup_{k\to \infty}{\mathbf L}_{2}(\varepsilon_{k})\le -\lim_{k\to\infty}{\mathbf R}_{2}(\varepsilon_{k})+\lim_{k\to\infty}{\mathbf R}_{3}(\varepsilon_{k})+C(\pi)T^{\frac{1+\pi}{1-\pi}}\lim_{k\to\infty}{\mathbf R}_{4}(\varepsilon_{k}) =0,\]
    which completes the proof of (\ref{Eq (Section 8): Claim for Strong Convergence}) for $p\in(p_{\mathrm c},\,\infty)$.

    For $p\le p_{\mathrm c}$, by taking a subsequence if necessary, we would like to prove the strong convergence
    \begin{equation}\label{Eq (Section 8): Strong conv from a priori regularity}
        \bu_{\varepsilon_{k}}\to \bu_{0}\quad \text{in}\quad L^{\pi}(\Omega_{T})^{N}\quad \text{for any }\pi\in(1,\,\infty).
    \end{equation}
    by utilizing the regularity results in Sections \ref{Section: L-infty} and \ref{Section: Gradient Bounds}.
    In fact, by the assumption $\bu_{\star}\in L^{\infty}(\Omega)^{N}$ and the uniform bound of $\bu_{\varepsilon}\in L^{2}(\Omega_{T})^{N}$, which is clear by (\ref{Eq (Section 8): Bound for Du-epsilon}), we obtain the uniform bound of $\bu_{\varepsilon}\in L^{\infty}(\Omega_{T})^{N}$ by the weak maximum principle (Proposition \ref{Proposition (Section 3): A Weak Maximum Principle}).
    By Lemma \ref{Lemma (Section 5): Local Holder continuity of u-epsilon}, $\bu_{\varepsilon}$ is uniformly continuous in any fixed subcylinder of $\Omega_{T}$.
    Therefore by the Aezel\`{a}--Ascoli theorem and a diagonal argument, it is easy to deduce the everywhere convergence $\bu_{\varepsilon_{k}}\to \bu_{0}$ in $\Omega_{T}$ by taking a subsequence if necessary.
    Hence, ({\ref{Eq (Section 8): Strong conv from a priori regularity}}) immediately follows from the bounded convergence theorem.
    Using (\ref{Eq (Section 8): Assumption in Convergence of Approximate Solutions}) and (\ref{Eq (Section 8): Strong conv from a priori regularity}) with $\pi\ge \max\{\,q^{\prime},\,r^{\prime} \,\}$, we easily verify $\lim\limits_{k\to\infty}{\mathbf R}_{1}(\varepsilon_{k})=0$. 
    Dropping ${\mathbf L}_{1}(\varepsilon)\ge 0$, we conclude $\limsup\limits_{k\to\infty}{\mathbf L}_{2}(\varepsilon_{k})\le 0$. Hence, (\ref{Eq (Section 8): Claim for Strong Convergence}) is shown also for $p\in(1,\,p_{\mathrm c}\rbrack$.

    Since (\ref{Eq (Section 8): Strong Convergence of Approx sol}) is verified, we are allowed to use Lemma \ref{Lemma (Section 2): Fundamental Lemma for convergence of solutions} \ref{Item 3/3}.
    As a consequence, we conclude that the limit function $\bu_{0}$ is the weak solution of (\ref{Eq (Section 8): Dirichlet Parabolic}), which uniquely exists by Lemma \ref{Lemma (Section 8): Stability Estimates}.
\end{proof}
\begin{remark}\label{Remark: Approximation on boundary data} \upshape
    We give the two remarks on Proposition \ref{Lemma (Section 8): Convergence of Approximate solutions}.
    \begin{enumerate}
        \item Proposition \ref{Lemma (Section 8): Convergence of Approximate solutions}, as well as Proposition \ref{Proposition (Section 3): A Weak Maximum Principle}, is valid even when $\Omega_{T}$ is replaced by a parabolic subcylinder ${\mathcal Q}\Subset \Omega_{T}$. 
        This fact is used in the proof of Theorem \ref{Theorem (Section 1): Gradient continuity}, particularly for $p\in(1,\, p_{\mathrm c}\rbrack$.
        \item Proposition \ref{Lemma (Section 8): Convergence of Approximate solutions} provides an existence result of (\ref{Eq (Section 8): Dirichlet Parabolic}), although we have to require some regularity assumptions on $\buf$ or $\bu_{\star}$. 
        Such technical conditions can be removed by utilizing a priori stability estimates shown in Lemma \ref{Lemma (Section 8): Stability Estimates}.
        When $p\in(1,\,p_{\mathrm c}\rbrack$, however, we have to let $\bu_{\star}\in X^{p}(0,\,T;\,\Omega)^{N}\cap C(\lbrack 0,\,T\rbrack;\,L^{2}(\Omega))^{N}$ admit an approximate sequence $\bv_{m}\,(m\in{\mathbb N})$ of the class $X^{p}(0,\,T;\,\Omega)^{N}\cap C(\lbrack 0,\,T\rbrack;\,L^{2}(\Omega))^{N}\cap L^{\infty}(\Omega_{T})^{N}$ such that all of the convergence assumptions in Lemma \ref{Lemma (Section 8): Convergence from stability} are satisfied.
        Although this approximation property appears to hold, in this paper we do not discuss the details.
    \end{enumerate}
\end{remark}
\subsection{Existence results for the Dirichlet boundary problem}
From Lemmata \ref{Lemma (Section 8): Stability Estimates}--\ref{Lemma (Section 8): Convergence from stability} and Proposition \ref{Lemma (Section 8): Convergence of Approximate solutions}, we would like to show the following result.
\begin{corollary}\label{Corollary (Section 8)}
    Let $\buf\in L^{p^{\prime}}(0,\,T;\,V_{0}^{\prime}(\Omega))^{N}\cap L^{2,\,1}(\Omega_{T})^{N}$ and $\bu_{\star}\in C(\lbrack 0,\,T\rbrack;\,L^{2}(\Omega))^{N}\cap X^{p}(0,\,T;\,\Omega)^{N}$ be given. 
    When $p\in(1,\,p_{\mathrm c}\rbrack$, assume that the function $\bu_{\star}$ admits the approximation property as in Remark \ref{Remark: Approximation on boundary data}.
    Then, the weak solution $\bu\in \bu_{\star}+X_{0}^{p}(0,\,T;\,\Omega)^{N}$ of (\ref{Eq (Section 8): Dirichlet Parabolic}), in the sense of Definition \ref{Definition (Section 8): Weak sol Dirichlet boundary}, uniquely exists.
\end{corollary}
\begin{proof}
    Since the uniqueness of (\ref{Eq (Section 8): Dirichlet Parabolic}) is already shown by Lemma \ref{Lemma (Section 8): Stability Estimates}, it suffices to construct the weak solution of (\ref{Eq (Section 8): Dirichlet Parabolic}).
    We first consider $p\in(p_{\mathrm c},\,\infty)$. 
    We let $\bv_{k}\equiv \bu_{\star}$, and suitably choose $\buf_{k}\in L^{\infty}(\Omega_{T})^{N}$ such that all of the convergence assumptions in Lemma \ref{Lemma (Section 8): Convergence from stability} are satisfied.
    Then by Proposition \ref{Lemma (Section 8): Convergence of Approximate solutions}, the weak solution of (\ref{Eq (Section 8): Dirichlet boundary}) uniquely exists, say $\bu_{k}\in \bv_{k}+X_{0}^{p}(0,\,T;\,V_{0}(\Omega))^{N}=\bu_{\star}+X_{0}^{p}(0,\,T;\,V_{0}(\Omega))^{N}$.
    Letting $k\to\infty$ and using Lemma \ref{Lemma (Section 8): Convergence from stability} complete the proof of Corollary \ref{Corollary (Section 8)}.
    The remaining case $p\in(1,\,p_{\mathrm c}\rbrack$ is similarly shown, with $\bu_{\star}$ also approximated by $\bv_{k}$ that is in the class $X^{p}(0,\,T;\,\Omega)^{N}\cap C(\lbrack 0,\,T\rbrack;\,L^{2}(\Omega))^{N}\cap L^{\infty}(\Omega_{T})^{N}$.
    This condition is required in our proof, since the proof of Proposition \ref{Lemma (Section 8): Convergence of Approximate solutions} for $p\in(1,\,p_{\mathrm c}\rbrack$ essentially relies on the weak maximum principle.
    \end{proof}

\section{Gradient continuity}\label{Section: Gradient Continuity}
 \subsection{Proof of Theorem \ref{Theorem (Section 4): Holder Truncated-Gradient Continuity}}
 We would like to show Theorem \ref{Theorem (Section 4): Holder Truncated-Gradient Continuity}.
 We note that the basic strategy is almost the same with \cite[Theorem 2.8]{T-supercritical}, except the fact that we often compare $\lvert\, \cdot\,\rvert_{\bg}$ with $\lvert\,\cdot\,\rvert$.
 For the reader's convenience, we provide the outline of the proof. 
 \begin{proof}[Proof of Theorem \ref{Theorem (Section 4): Holder Truncated-Gradient Continuity}]
     For given $\delta\in(0,\,1)$ and $M\coloneqq 1+\mu_{0}\in(1,\,\infty)$, we choose and fix $\nu\in(0,\,10^{-23}\gamma_{0}^{16})$ and ${\widehat\rho}\in(0,\,1)$ as in Proposition \ref{Proposition (Section 4): Non-Degenerate Case}.
     Hereinafter we set $\sigma\coloneqq \sqrt{\nu}/6$.
     Corresponding to this $\nu$, we choose $\kappa\in\lbrack (\sigma^{\beta},\,1)$ and $\widetilde{\rho}\in(0,\,1)$ as in Proposition \ref{Proposition (Section 4): Degenerate Case}.
     It suffices to prove that $\bG_{2\delta,\,\varepsilon}(x,\,t)$, defined as (\ref{Eq (Section 4): Limit Average of G-2delta-epsilon}), exists for every $(x,\,t)$, and this limit satisfies
     \begin{equation}\label{Eq (Section 9): Claim for Gamma-2delta-epsilon}
         \fiint_{Q_{\tau \rho}(x_{0},\,t_{0})}\lvert \G_{2\delta,\,\varepsilon}(\D\bu_{\varepsilon})-\bG_{2\delta,\,\varepsilon}(x_{0},\,t_{0}) \rvert^{2}\,\d x\d t\le C({\mathcal D},\,\delta,\,M)\mu_{0}^{2}\tau^{2\alpha}\quad \text{for all }\tau\in(0,\,1\rbrack,
     \end{equation}
     provided $Q_{2\rho}(x_{0},\,t_{0})\subset \widetilde{\mathcal Q}$ and $\rho\le \overline{\rho}\coloneqq \min\{\,\widetilde{\rho},\,\widehat{\rho}\,\}$.
     Here the H\"{o}lder exponent $\alpha$ is defined as
     \begin{equation}
         \alpha\coloneqq \frac{\log \kappa}{\log \sigma}\in(0,\,\beta),\quad \text{so that}\quad \kappa=\sigma^{\alpha}\text{ holds}.
     \end{equation}
     The local H\"{o}lder continuity estimate of $\bG_{2\delta,\,\varepsilon}$ is easily shown from (\ref{Eq (Section 9): Claim for Gamma-2delta-epsilon}) (see \cite[Theorem 2.8]{T-supercritical} for the detailed computations).
     Hereinafter we assume $Q_{2\rho}(x_{0},\,t_{0})\subset \widetilde{\mathcal Q}$ and $\rho\le \overline{\rho}$.
     For each $l\in{\mathbb Z}_{\ge 0}$, we define $\rho_{l}\coloneqq \sigma^{l}\rho$ and $\mu_{l}\coloneqq \kappa^{l}\mu_{0}$.
     We define the set
     \[{\mathcal N}\coloneqq \{l\in{\mathbb Z}_{\ge 0}\mid \mu_{l}>\delta\text{ and }\lvert S_{\rho_{l},\,\mu_{l}}\rvert\le (1-\nu)\lvert Q_{\rho_{l}}\rvert\}\]
     Since $\mu_{l}\to 0$ as $l\to\infty$, ${\mathcal N}$ is a proper subset of ${\mathbb Z}_{\ge 0}$.
     We define $l_{\star}\in{\mathbb Z}_{\ge 0}$ as the minimum number of the non-empty set ${\mathbb Z}_{\ge 0}\setminus{\mathcal N}$.
     Repeatedly applying Proposition \ref{Proposition (Section 4): Degenerate Case}, we have 
     \begin{equation}\label{Eq (Section 9): Consequence from De Giorgi}
         \esssup_{Q_{2\rho_{l}}}\,\lvert \G_{\delta,\,\varepsilon}(\D\bu_{\varepsilon})\rvert_{\bg}\le \mu_{l}.
     \end{equation}
     for all $l\in\{\,0,\,\dots\,,\,l_{\star}\,\}$. 
     For the number $l_{\star}$, we consider the two possible cases. 
    
     The first case is when $\mu_{l_{\star}}\le \delta$ holds, which clearly yields $\G_{2\delta,\,\varepsilon}(\D\bu_{\varepsilon})\equiv 0$ in $Q_{\rho_{l_{\star}}}$.
     Therefore, the limit $\G_{2\delta,\,\varepsilon}(x_{0},\,t_{0})$ obviously exists with $\G_{2\delta,\,\varepsilon}(x_{0},\,t_{0})=0$.
     Moreover, by (\ref{Eq (Section 9): Consequence from De Giorgi}), it is easy to deduce
     \begin{equation}\label{Eq (Section 9): Consequence from Truncation Trick}
         \esssup_{Q_{2\rho_{l}}}\,\lvert \G_{2\delta,\,\varepsilon}(\D\bu_{\varepsilon})\rvert_{\bg} \le \mu_{l}
     \end{equation}
     for all $l\in{\mathbb Z}_{\ge 0}$. 
     For each $\tau\in(0,\,1)$, there uniquely exists $l\in{\mathbb Z}_{\ge 0}$ satisfying $\sigma^{l+1} <\tau\le \sigma^{l}$.
     Using (\ref{Eq (Section 1): Matrix Gamma}) and (\ref{Eq (Section 9): Consequence from Truncation Trick}), we have 
     \begin{align*}
         \fiint_{Q_{\tau\rho}(x_{0},\,t_{0})}\lvert \G_{2\delta,\,\varepsilon}(x_{0},\,t_{0}) \rvert^{2}\,\d x\d t&\le \left(\frac{\mu_{l}}{\gamma_{0}}\right)^{2}=\frac{\sigma^{2\alpha l}\mu^{2}}{\gamma_{0}^{2}}\le \frac{\mu_{0}^{2}\tau^{2\alpha}}{\sigma^{2\alpha}\gamma_{0}^{2}}.
     \end{align*}
     Recalling $\G_{2\delta,\,\varepsilon}(x_{0},\,t_{0})=0$, we conclude (\ref{Eq (Section 9): Claim for Gamma-2delta-epsilon}).

     The second case is when both $\mu_{l_{\star}}>\delta$ and $\lvert S_{\rho_{l_{\star}},\,\mu_{l_{\star}}}\rvert> (1-\nu)\lvert Q_{\rho_{l_{\star}}}\rvert$ hold, which allows us to apply Proposition \ref{Proposition (Section 4): Non-Degenerate Case} to a small cylinder $Q_{\rho_{l_{\star}}}(x_{0},\,t_{0})$ with $\mu=\mu_{l_{\star}}$.
     Therefore, the limit $\G_{2\delta,\,\varepsilon}(x_{0},\,t_{0})$ exists, and this limit satisfies
     \begin{equation}\label{Eq (Section 9): Bound Estimate of G-2delta-epsilon}
         \lvert\bG_{2\delta,\,\varepsilon}(x_{0},\,t_{0}) \rvert\le \frac{\mu_{l_{\star}} }{\gamma_{0}},
     \end{equation}
     Moreover, for any $\tau\in(0,\,\sigma^{l_{\star}}\rbrack$, we have 
     \[\fiint_{Q_{\tau\rho}} \lvert \G_{2\delta,\,\varepsilon}(\D\bu_{\varepsilon})- \bG_{2\delta,\,\varepsilon}(x_{0},\,t_{0}) \rvert^{2}\,\d x\d t\le \left(\frac{\tau}{\sigma^{l_{\star}}}\right)^{2\beta}\mu_{l_{\star}}^{2}\le \left(\frac{\tau}{\sigma^{l_{\star}}}\right)^{2\alpha}\mu_{l_{\star}}^{2}=\mu_{0}^{2}\tau^{2\alpha},\]
     which yields (\ref{Eq (Section 9): Claim for Gamma-2delta-epsilon}).
     To complete the remaining range $\tau\in(\sigma^{l_{\star}},\,1\rbrack$, we note that (\ref{Eq (Section 9): Consequence from Truncation Trick}) is valid for all $l\in\{\,0,\,\dots\,,\,l_{\star}\,\}$, which is clear by (\ref{Eq (Section 9): Consequence from De Giorgi}).
     Choosing the unique number $l\in\{\,0,\,\dots\,,\,l_{\star}-1\,\}$ satisfying $\sigma^{l+1}<\tau\le \sigma^{l}$ and using (\ref{Eq (Section 9): Bound Estimate of G-2delta-epsilon}), we have 
     \begin{align*}
         \fiint_{Q_{\tau\rho}}\lvert \G_{2\delta,\,\varepsilon}(\D\bu_{\varepsilon})- \bG_{2\delta,\,\varepsilon}(x_{0},\,t_{0}) \rvert^{2}\,\d x\d t
         &\le 2\left(\lvert \bG_{2\delta,\,\varepsilon}(x_{0},\,t_{0}) \rvert^{2}+\esssup_{Q_{\rho_{l}}}\,\lvert \G_{2\delta,\,\varepsilon} (\D\bu_{\varepsilon}) \rvert^{2} \right)\\ 
         &\le \frac{4\mu_{l}^{2}}{\gamma_{0}^{2}}=\frac{4\sigma^{2\alpha l}\mu^{2}}{\gamma_{0}^{2}}\le \frac{4}{\sigma^{2\alpha}\gamma_{0}^{2}}\mu_{0}^{2}\tau^{2\alpha},
     \end{align*}
     which completes the proof.
 \end{proof}

 \subsection{Proof of Theorem \ref{Theorem (Section 1): Gradient continuity}}
 In the last subsection, we would like to prove Theorem \ref{Theorem (Section 1): Gradient continuity}.
 \begin{proof}[Proof of Theorem \ref{Theorem (Section 1): Gradient continuity}]
     We choose and fix $\{\buf_{\varepsilon}\}_{0<\varepsilon<1}\subset L^{\infty}(\Omega_{T})^{N}$ such that (\ref{Eq (Section 2): Weak Conv of f-epsilon}) holds.
     We would like to prove that for each fixed $\delta\in(0,\,1)$, the truncated gradient ${\G}_{2\delta}(\D\bu)$ is H\"{o}lder continuous in each fixed $Q^{(0)}\Subset \Omega_{T}$, whose continuity estimate may depend on $\delta$.
     From this regularity result, the continuity of $\D\bu$ in $\Omega_{T}$ easily follows by letting $\delta\to 0$.

     We first consider $p\in(p_{\mathrm c},\,\infty)$. 
     Let $\bu_{\varepsilon}$ be the weak solution of (\ref{Eq (Section 2): Approximate System})--(\ref{Eq (Section 2): Parabolic Dirichlet Boundary}) with $\bu_{\star}=\bu$.
     Then, by Theorems \ref{Theorem (Section 4): Gradient bounds}--\ref{Theorem (Section 4): Holder Truncated-Gradient Continuity} and a standard covering argument, we find out that ${\G}_{2\delta,\,\varepsilon}(\D\bu_{\varepsilon})$ is uniformly H\"{o}lder continuous in $Q^{(0)}$, independent of $\varepsilon\in(0,\,\delta/4)$.
     By the Aezel\`{a}--Ascoli theorem, we can construct a uniform convergence limit of ${\G}_{2\delta,\,\varepsilon}(\D\bu_{\varepsilon})$ over $Q^{(0)}$ as $\varepsilon$ tends to $0$. 
     Let $\bv_{2\delta}$ denote this uniform limit.
     By Proposition \ref{Lemma (Section 8): Convergence of Approximate solutions}, ${\G}_{2\delta,\,\varepsilon}(\D\bu_{\varepsilon})\to \G_{2\delta}(\D\bu)$ a.e.~in $\Omega_{T}$.
     From these convergence results, it follows that the identity $\bv_{2\delta}={\G}_{2\delta}(\D\bu)$ holds a.e.~in $Q^{(0)}$, which completes the proof.

     The case $p\in (1,\,p_{\mathrm c}\rbrack$ is also similarly shown by local approximation arguments.
     More precisely, we fix the parabolic subcylinders $Q^{(0)}\Subset Q^{(1)}\coloneqq B\times I\Subset Q^{(2)}\Subset \Omega_{T}$.
     Let $\bu_{\varepsilon}\in \bu+X_{0}^{p}(I;\,B)$ be the weak solution of (\ref{Eq (Section 2): Approximate System})--(\ref{Eq (Section 2): Parabolic Dirichlet Boundary}) with $\Omega_{T}=\Omega\times (0,\,T)$ replaced by $Q^{(1)}=B\times I$.
     The higher regularity assumption $\bu\in L^{\varsigma}(Q^{(2)})^{N}$ with $\varsigma>\varsigma_{\mathrm c}\ge 2$ and Propositions \ref{Proposition (Section 3): Local Boundedness for subcritical cases} yield $\bu\in L^{\infty}(Q^{(1)})^{N}$.
     Hence by Proposition \ref{Proposition (Section 3): A Weak Maximum Principle}, we also have $\bu_{\varepsilon}\in L^{\infty}(Q^{(1)})^{N}$, whose estimate is independent of $\varepsilon$.
     Applying Theorems \ref{Theorem (Section 4): Gradient bounds}--\ref{Theorem (Section 4): Holder Truncated-Gradient Continuity} and using Proposition \ref{Lemma (Section 8): Convergence of Approximate solutions} with $\Omega_{T}$ replaced by $Q^{(1)}$, we similarly conclude the H\"{o}lder continuity of ${\G}_{2\delta}(\D\bu)$ in $Q^{(0)}$.
 \end{proof}

 From Proposition \ref{Proposition (Section 5): Reversed Holder}, we have local gradient bounds of a weak solution for the supercritical case.
 \begin{corollary}
     Fix $p\in(p_{\mathrm c},\,\infty)$ and $\buf\in L^{q,\,r}(\Omega_{T})^{N}\cap L^{p^{\prime}}(0,\,T;\,V_{0}^{\prime})^{N}$.
     Assume that $\bu$ is a weak solution to (\ref{Eq (Section 1): General System}).
     Then, for any fixed $Q_{R}=Q_{R}(x_{0},\,t_{0})\Subset \Omega_{T}$ with $R\in(0,\,1\rbrack$, we have
     \[\sup_{Q_{R/2}} \lvert \D\bu_{\varepsilon}\rvert\le C({\mathcal D})\left[1+\lVert \buf_{\varepsilon}\rVert_{L^{q,\,r}(\Omega_{T})}^{p\pi}+\fiint_{Q_{R}}\lvert \D\bu\rvert^{p}\,\d x\d t \right]^{1/d},\]
     where the exponents $\pi$ and $d$ are given by Proposition \ref{Proposition (Section 5): Reversed Holder}. 
 \end{corollary}
 \begin{proof}
     Without loss of generality, we may let $q$ and $r$ be finite, since the remaining case is clear by H\"{o}lder's inequality.
     We choose $\buf_{\varepsilon}\in L^{\infty}(\Omega_{T})^{N}$ that satisfies
     \[\buf_{\varepsilon}\to \buf\quad \text{in}\quad L^{q,\,r}(\Omega_{T})^{N}\quad \text{and} \quad \buf_{\varepsilon}\rightharpoonup \buf\quad \text{in}\quad L^{p^{\prime}}(0,\,T;\,V_{0}^{\prime}(\Omega))\quad \text{as}\quad \varepsilon\to 0.\]
     Let $\bu_{\varepsilon}$ be the weak solution of (\ref{Eq (Section 2): Approximate System})--(\ref{Eq (Section 2): Parabolic Dirichlet Boundary}) with $\bu_{\star}=\bu$.
     Then, Proposition \ref{Proposition (Section 5): Reversed Holder} clearly yields 
     \[\esssup_{Q_{R/2}}\,\lvert \D\bu_{\varepsilon}\rvert\le C({\mathcal D})\left[1+\lVert \buf_{\varepsilon}\rVert_{L^{q,\,r}(\Omega_{T})}^{p\pi}+\fiint_{Q_{R}}\lvert \D\bu_{\varepsilon}\rvert^{p}\,\d x\d t \right]^{1/d}.\]
     Since both $\D\bu_{\varepsilon}\to \D\bu$ in $L^{p}(\Omega_{T})$ and $\D\bu\in C^{0}(\Omega_{T};\,{\mathbb R}^{Nn})$ are shown by Proposition \ref{Lemma (Section 8): Convergence of Approximate solutions} and Theorem \ref{Theorem (Section 1): Gradient continuity}, the proof is completed by letting $\varepsilon\to 0$.
 \end{proof}
 
 \addcontentsline{toc}{section}{References}
\bibliographystyle{siamplain}

\end{document}